\documentclass[11pt, twoside]{amsart}
\usepackage{inputenc}
\scrollmode 
\usepackage{inputenc}
\usepackage{amssymb}
\usepackage[matrix,arrow,curve]{xy}
\usepackage{amsmath}
\usepackage[french,english]{babel}
\usepackage[dvips]{xcolor}

\newtheorem{thm}{Theorem}[section]
\newtheorem{prop}[thm]{Proposition}
\newtheorem{lemma}[thm]{Lemma}
\newtheorem{cor}[thm]{Corollary}

\theoremstyle{definition}
\newtheorem{definition}[thm]{Definition}
\newtheorem{def-thm}[thm]{Definition-Theorem}
\newtheorem{rem}[thm]{Remark}
\newtheorem{ex}[thm]{Example}

\numberwithin{equation}{section}

\def\ad{\operatorname{ad}}

\def\braid{\operatorname{braid}}

\def\diag{\operatorname{diag}}
\def\dim{\operatorname{dim}}

\def\dual{\operatorname{dual}}

\def\ev{\operatorname{ev}}
\def\Ev{\operatorname{Ev}}

\def\id{\operatorname{id}}

\def\nil{\operatorname{nil}}
\def\Norm{\operatorname{Norm}}

\def\op{\operatorname{op}}
\def\PGL{\operatorname{PGL}}

\def\red{\operatorname{red}}
\def\sgn{\operatorname{sgn}}
\def\sl{\operatorname{sl}}
\def\SL{\operatorname{SL}}

\def\T{\operatorname{T}}
\def\th{\operatorname{th}}

\newcommand{\doublesubscript}[3]{
\displaystyle\mathop{\displaystyle #1_{#2}}_{#3}}
\begin{document}

\title
{Cluster $\mathcal X$-varieties for dual Poisson-Lie groups I}

\author{Renaud Brahami}
\address{Section of Mathematics, University of Geneva, 2-4 rue du Li\`{e}vre, c.p. 64, 1211
Gen\`{e}ve 4, Switzerland
}
\email{Renaud.Brahami@unige.ch}

\maketitle

\begin{abstract}
We associate a family of cluster $\mathcal X$-varieties to the dual
Poisson-Lie group~$G^*$ of a complex semi-simple Lie group $G$ of adjoint type
given with the standard Poisson structure. This family is described by the
$W$-permutohedron associated to the Lie algebra
$\mathfrak g$ of $G$: vertices being labeled by cluster
$\mathcal X$-varieties and edges by new Poisson birational isomorphisms,
on appropriate seed $\mathcal X$-tori, called \emph{saltation}.
The underlying combinatorics is based on a factorization
of the Fomin-Zelevinsky twist maps into mutations and other new Poisson
birational isomorphisms on seed $\mathcal X$-tori called \emph{tropical
mutations} (because they are obtained by a tropicalization of the mutation
formula), associated to an enrichment of the combinatorics on double words
of the Weyl group $W$ of~$G$.
\end{abstract}

\tableofcontents

\section{Introduction}

The rising of the cluster combinatorics goes back to two sources:
Berenstein, Fomin and Zelevinsky from one hand (\cite{FZ1}, \cite{FZ2},
\cite{BFZbruhat},\cite{FZ4}) in their study of total positivity,
and Fock and Goncharov on the other hand (\cite{FGdilog}, \cite{FGcluster},
\cite{FGdilog2}) in their higher Teichm\"{u}ller theory.
These structures quickly spread to diverse mathematical areas
such as: quiver representations, Poisson geometry,
integrable systems, convex polytops, tropical geometry,
and so on.
In this paper and its sequel \cite{RB2}, we use them to sharpen the
geometry of dual Poisson-Lie groups. Therefore, according to the
quantum duality principle \cite{STSPoisson},
these two papers can be seen as the semi-classical
starting point towards a cluster combinatorics describing the quantized universal enveloping
algebra $\mathcal{U}_{q}(\mathfrak{g})$ associated to a complex semi-simple Lie
algebra $\mathfrak{g}$.

Let us recall that a Lie group $G$ given with a Poisson structure is called
\emph{Poisson-Lie group} if the multiplication $m:G\times G\to G$
is a Poisson map, when the set $G\times G$ is given the Poisson product
structure.
Let $(\mathfrak{g},\mathfrak{g}^*)$ be the tangent Lie bialgebra of $G$.
According to the standard theory $(\mathfrak{g^*},\mathfrak{g})$ is also
a Lie bialgebra, and hence the Lie group associated with $\mathfrak{g}^*$ is
again a Poisson-Lie group called the \emph{dual Poisson-Lie group} of $G$ and denoted $G^*$.
If the Lie bialgebra $(\mathfrak{g},\mathfrak{g}^*)$ is factorizable,
the Poisson-Lie group $G^*$ can be embedded as a dense subset $G_0$ of
$G$, when this one is given the appropriate Poisson structure.
The symplectic leaves of $G$ are then the $G^*$-orbits on $G$ via the dressing
transformations and the symplectic leaves of $G^*$ are
the conjugacy classes in $G$ \cite{STS}. Let us denote this appropriate
Poisson structure $\pi_*$ when $G$ is a complex semi-simple
Lie group given with the standard Poisson structure $\pi_G$,
that is a Sklyanin bracket associated to the standard $r$-matrix
of the Belavin-Drinfeld classification. In that case,
the dual of the Poisson-Lie group $(G,\pi_G)$ may be identified
with a subgroup in the direct product of two opposite Borel subgroups
$B$ and $B_-$ of $G$, and we denote it $(G^*,\pi_{G^*})$.

When $G$ is a real split semisimple Lie group with trivial center, the geometry of
$(G,\pi_G)$ has been described by Fock and Goncharov via the combinatorics involved
in cluster $\mathcal X$-variety. Recall that a cluster $\mathcal X$-variety is a Poisson variety
obtained by gluing a set of tori
along some specific bi-rational isomorphisms called ($\mathcal X$-)mutations.
{Each torus is given a log-canonical Poisson structure, that is a set of
coordinates $x_i$ and a skew-symmetric matrix $\widehat{\varepsilon}$, with
generic integer values, such that $\{x_i,x_j\}=\widehat{\varepsilon}_{ij}x_ix_j$}.
Because mutations are Poisson maps relative to these log-canonical Poisson
structures, cluster $\mathcal X$-varieties are naturally given a kind of
Darboux coordinates. In \cite{FGcluster}, using a coarser Poisson stratification
of $(G,\pi_G)$ into {double Bruhat cells} $G^{u,v}$, defined as the intersection
of the cells $BuB$ and $B_-vB_-$, where $u,v$ belong to the Weyl group $W$ of $G$,
Fock and Goncharov have constructed canonical Poisson birational maps, called evaluation maps,
of cluster $\mathcal X$-varieties into $(G^{u,v},\pi_G)$ (one map $\ev_{\mathbf i}$ for each seed
$\mathcal X$-torus ${\mathcal X}_{\mathbf i}$ associated with a double reduced word
$\mathbf i$ associated to the pair $u,v\in W$); this construction provides for $(G,\pi_G)$ a
natural set of rational canonical coordinates. Canonical maps $\mu_{\mathbf i\to\mathbf j}$
associated with different double reduced words $\mathbf i$, $\mathbf j$ are given
by a composition of mutations simply related to the composition of generalized $d$-moves
linking the double reduced words $\mathbf i$ and $\mathbf j$.

In the present paper, using a key result of Evens and Lu \cite{EL},
we adapt the construction of Fock and Goncharov
to study the dual Poisson-Lie group $(G^*,\pi_{G^*})$ of $(G,\pi_G)$
when $G$ is a complex semi-simple of adjoint type. It turns that the
description of $(G^*,\pi_{G^*})$ requires not one but a family of cluster
$\mathcal X$-varieties indexed by the Weyl group $W$ of $G$.
This family is in fact described by the $W$-permutohedron associated to the
Lie algebra $\mathfrak g$ of $G$: vertices being labeled by cluster
$\mathcal X$-varieties and edges by new Poisson birational isomorphisms,
on appropriate seed $\mathcal X$-tori, called \emph{saltation}.
Roughly speaking, we associate to every cluster variety of this family a
twisted evaluation, i.e. a composition of an evaluation map as above with a
new map called twisted maps which generalizes the birational isomorphisms constructed
by Evens and Lu in their study of Grothendieck resolutions \cite{EL}, and use
saltations to relate them.
The combinatorics underlying this result rely on two new moves on double words
added to the previous generalized $d$-moves; these moves are
called \emph{$\tau$-moves} and \emph{dual moves}. The maps on seed
$\mathcal X$-tori associated to {dual moves} are {saltations}, whereas
the maps on seed $\mathcal X$-tori associated to {$\tau$-moves} are obtained by
a tropicalization of the mutation formulas and therefore called
\emph{tropical mutations}. In fact, one of the key technical result here is
an explicit factorization of the Fomin-Zelevinsky twist maps given in
\cite{FZtotal} in terms mutations and tropical mutations.

In the sequel \cite{RB2} of this paper, we will see how to use tropical
mutations to include the De-Concini-Kac-Procesi Poisson automorphisms on
$(G_0,\pi_*)$ in the story.

Here is the organization of the paper. We fix the notation, give backgrounds
on semi-simple Lie algebras and dual Poisson-Lie groups, and recall the basic
definitions leading to the notion of cluster $\mathcal X$-variety
in Section \ref{section:Preliminaries}.
We show, in a way useful for our purposes, how to naturally attach a cluster
$\mathcal X$-variety ${\mathcal X}_{|\mathbf i|}$ to every double Bruhat cells
$(G^{u,v},\pi_G)$ via evaluation maps associated to any double (reduced) word
$\mathbf i$ in Section
\ref{section:ClusterG} (this section sums-up results of \cite{FGcluster}).
In Section \ref{subsection:evdual}, we introduce new evaluation maps and new
seeds related to double reduced
words to state an analog of the previous construction of Fock and Goncharov for
the dual Poisson-Lie group $(G_0,\pi_*)$;
although the result in this section are strongly generalized
in Section \ref{section:Loop}, it is the occasion to give a flavor
of our construction without using the machinery of
generalized cluster transformations and saltations later developed.
In Section \ref{section:tropical},
we enlarge the combinatorics on double words and on their related
seed $\mathcal X$-tori by introducing respectively new moves
called $\tau$-moves and related birational
Poisson isomorphisms on seed $\mathcal X$-tori called tropical
mutations; this enables us to describe the Fomin-Zelevinsky twist maps
and their variations in terms on mutations and tropical mutations.
In Section \ref{section:taucombi}, we use the $W$-permutehedron associated
to the Lie algebra $\mathfrak g$ to study
the combinatorics on double reduced words generated by generalized $d$-moves
and enriched with tropical moves, as well as the related combinatorics on cluster
$\mathcal X$-varieties; the idea is to prepare the ground for the
cluster combinatorics related to twisted evaluations and dual Poisson-Lie
groups, developed in Section \ref{section:Loop}.
In Section \ref{section:twistedev}, we generalized the results of Section
\ref{subsection:evdual}: we introduce {twisted evaluation} and adapt the
combinatorics of the previous section
to get a family of cluster $\mathcal X$-varieties ${\mathcal X}_{w}$, associated
to each element $w$ of $W$ parameterizing $(BB_-,\pi_*)$. (The results of Section
\ref{subsection:evdual} are rediscovered by setting $w=1$.)
In Section \ref{section:Loop}, we relate the previous twisted evaluations
by cluster transformations and the birational
Poisson isomorphisms called saltations; as a corollary, we get a parametrization of
the dual Poisson Lie-group $(BB_-,\pi_*)$ by a family of cluster
$\mathcal X$-varieties; moreover, the cluster $\mathcal X$-varieties of this
family are related by saltations described by the $1$-skeleton
of the $W$-permutohedron $P_W$.
In Section \ref{section:evaluationdual},
we start by giving an alternative way to describe twist maps with mutations
and tropical mutations and provide evaluations for $(G^*,\pi_{G^*})$ in the
spirit of the Kirillov-Reshetikhin multiplicative formula for the quantum $R$-matrix
associated to ${\mathcal U}_q(\mathfrak{g})$;
moreover, birational Poisson isomorphisms using to pass from the positive part
to the negative part of $(G^*,\pi_{G^*})$ (and vice-versa) are easily encoded by paths
on the $1$-skeleton of $W$-permotohedron relating the identity and the longest element $w_0$
of $W$. Finally, we apply all our construction to
the very special case of $G=SL(2,\mathbb C)$ in Section \ref{section:SL2},
and, as a conclusion, we give the quantization of this elementary construction by
considering the cluster combinatorics associated to the quantized universal enveloping
algebra $\mathcal{U}_q(\mathfrak{g})$ of the Lie algebra
$\mathfrak{g}=\sl(2,\mathbb C)$.

\section{Preliminaries}
\label{section:Preliminaries}
We fix the notation, give backgrounds on semi-simple Lie algebras and dual
Poisson-Lie groups, and recall the basic definitions leading to the notion
of cluster $\mathcal X$-variety.
\subsection{Backgrounds on semi-simple Lie algebras}\label{section:Preliminaries1}
Let $\mathfrak g$ be a complex semi-simple Lie algebra  of rank $l$,
$A$ its Cartan matrix, {and $G$ its Lie group of adjoint type}. Fix a
Borel subgroup $B\subset G$, let $B_-$ be the opposite Borel subgroup, $H=B\cap B_-$
the associated Cartan subgroup and $N$ (resp. $N_-$) the unipotent radical of $B$ (resp. $B_-$).
Let $\mathfrak{h},\mathfrak{n},\mathfrak{n}_-\subset \mathfrak{g}$ be the
Cartan and nilpotent subalgebras of  $\mathfrak{g}$, corresponding respectively
to $H$, $N$ and $N_-$. In the following, we will denote
$[1,l]:=\{1,\dots,l\}$.

Let
$\alpha_1,\dots,\alpha_l$ be the simple roots of $\mathfrak g$, and let
$\omega_1, \omega_2, \ldots, \omega_l\in{\mathfrak h}^*$ be the corresponding
{fundamental weights}.
For every $i\in[1,l]$, let $(e_i,f_i,h_i)$ be the Chevalley generators
of $\mathfrak g$; they generate a Lie subalgebra $\mathfrak{g}_{\alpha_i}$
of $\mathfrak{g}$. In particular, we have $\omega_j(h_k) = \delta_{jk}$
for every $j, k \in[1,l]$. Let us recall that the \emph{weight lattice}
$P$ is the set of all weights $\gamma \in \mathfrak{h}^*$
such that $\gamma (h_i) \in \mathbb{Z}$ for all $i$. So the group $P$ has a
$\mathbb{Z}$-basis formed by the {fundamental weights}.
Every weight $\gamma \in P$ gives rise to a multiplicative character
$a \mapsto a^\gamma$ of the maximal torus $H$; this character
is given by $\exp (h)^\gamma = e^{\gamma (h)}$, with $h \in \mathfrak{h}$.

The Lie algebra $\mathfrak g$
being semi-simple, its Cartan matrix $A$ is invertible and we can introduce
a new basis $\{h^i,1\leq i\leq l\}$ on $\mathfrak h$ putting
\begin{equation}\label{equ:hbasis}
h^i:=\sum{(A^{-1})}_{ij}\ h_j\ .
\end{equation}

Let  $D=\diag(d_1,\dots,d_l)$ be the diagonal matrix associated with the set of
Cartan symmetrizers; we put $\widehat{a}_{ij}=d_ia_{ij}=a_{ij}d_j$.
For every $x\in\mathbb C$ and $i\in[1,l]$, we define the group elements
$$\begin{array}{cccccc}
E^i=\exp(e_i),& F^i=\exp(f_i),& H_{i}(x)=\exp(\log(x)h_{i}),&
H^{i}(x)=\exp(\log(x)h^{i})
\end{array}$$
related respectively
to the generators $e_i$, $f_i$, $h_i$ and $h^i$ of $\mathfrak{g}$.
The canonical inclusions $\varphi_i: \SL(2,\mathbb C)
\hookrightarrow G$ give in particular the following equalities for every nonzero
complex number $x$.
\begin{equation}\label{equ:defphi}
\begin{array}{cc}
\varphi_i\left(
\begin{array}{cc}
1 & x\\
0 & 1
\end{array}
\right) =H^i(x)E^iH^i(x^{-1}),
& \varphi_i\left(
\begin{array}{cc}
1 & 0\\
x & 1
\end{array}
\right) =H^i(x^{-1})F^iH^i(x)\ ,
\end{array}
\end{equation}
\begin{equation}\label{equ:defphi2}
\begin{array}{cc}
\mbox{and}&
\varphi_i\left(
\begin{array}{cc}
x & 0\\
0 & x^{-1}
\end{array}
\right) =H_i(x)=\displaystyle\prod_{j\in[1,l]}H^j(x)^{a_{ij}}\ .
\end{array}
\end{equation}

We denote by $W$ the Weyl group of $G$.
As an abstract group, $W$ is a finite Coxeter group of rank $l$ generated by
the set of \emph{simple reflections} $S=\{s_1, \dots, s_l\}$; it acts on
$\mathfrak h^*$, $\mathfrak h$ and the Cartan subgroup $H$ by
\begin{equation}\label{equ:actionW}
\begin{array}{ccccc}
s_i (\gamma) = \gamma - \gamma (\alpha_i^\vee) \alpha_i&,&
s_i (h) = h - \alpha_i (h) \alpha_i^\vee&\mbox{and}&
a^{w(\gamma)}=(\widehat w^{-1}a\widehat{w})^{\gamma}
\end{array}
\end{equation}
for every $\gamma \in \mathfrak h^*$, $h \in \mathfrak h$,
$w\in W$ and $a\in H$.
Recall now that a \emph{reduced word}
for $w\in W$ is an expression for $w$ in the generators belonging to $S$, which is
minimal in length among all such expressions for $w$. Let us denote $\ell(w)$
this minimal length and $R(w)$ the set of reduced words associated to $w$. As
usual, the notation $w_0$ will refer to the longest word of $W$.

Let us denote $\Pi$ the set of positive roots of the Lie algebra $\mathfrak g$.
It is well-known that if $i_1\dots i_{\ell(w_0)}$ is a reduced
expression for $w_0$, then
$$\Pi=\{\alpha_{i_1},s_{i_1}(\alpha_{i_2}),\dots, s_{i_1}
\dots s_{i_{\ell(w_0)-1}}(\alpha_{i_{\ell(w_0)}})\}\ ,$$
each positive root occurring exactly one in the right-hand side. There are
automorphisms $T_1,\dots T_l$ of $\mathfrak g$ such that
\begin{equation}\label{equ:classicalbraid}
\begin{array}{lcc}
T_i(e_i)=-f_i\ ,\quad T_i(f_i)=-e_i\ ,\quad T_i(h_j)=h_j-a_{ji}h_i\ , \\
\\
T_i(e_j)=(-a_{ij})!^{-1}(\ad_{e_i})^{-a_{ij}}(e_j)\ ,\quad\mbox{if } i\neq j\ ,\\
\\
T_i(f_j)=(-a_{ij})!^{-1}(\ad_{f_i})^{-a_{ij}}(f_j)\ ,\quad\mbox{if } i\neq j\ ,
\end{array}
\end{equation}
where $\ad_a(b)=[a,b]$ for every $a,b\in\mathfrak g$.
To any positive root $\beta=s_{i_1}s_{i_2}\dots s_{i_{k-1}}({\alpha}_{i_k})\in\Pi$,
$i_1\dots i_{\ell(w_0)}$ being a reduced expression of the longest word $w_0$ of $W$,
we associate the \emph{positive} and \emph{negative root vectors}
\begin{equation}\label{equ:autoLusztig}
\begin{array}{ccc}
e_{\beta}=T_{i_1}T_{i_2}\dots T_{i_{k-1}}(e_{i_k})&\mbox{and}
&f_{\beta}=T_{i_1}\dots T_{i_{k-1}}(f_{i_k})\ .
\end{array}
\end{equation}
Let us also recall that $W$ can also be seen as the subgroup $\Norm_G(H)/H$ of $G$.
Thus, to every simple reflection $s_i\in S$ we associate the group element
$$\widehat{s_i}=\varphi_i\left(
\begin{array}{rr}
0&-1\\
1&0
\end{array}
\right).
$$
We can choose  representatives in $G$ for every element of $W$ by setting
$\widehat{w_1w_2}=\widehat{w_1}\widehat{w_2}$ for every $w_1,w_2\in W$
as long as $\ell(w_1)+\ell(w_2)=\ell(w_1w_2)$. We finally have
the \emph{Bruhat decompositions} associated to $G$:
\begin{equation}\label{bruhatdec}
G=\bigcup_{u\in W}B\widehat{u}B=\bigcup_{v\in W}B_-\widehat{v}B_-
=\bigcup_{w\in W}B\widehat{w}B_-\ .
\end{equation}

\subsection{Backgrounds on dual Poisson-Lie groups}
A Lie group $G$ given with a Poisson structure is called
\emph{Poisson-Lie group} if the multiplication $m:G\times G\longrightarrow G$
is a Poisson map, when the set $G\times G$ is given the Poisson product
structure. In the following, we will focus on the following
Poisson-Lie groups.
\subsubsection{The Poisson-Lie group $(G,\pi_{G})$}
According to the Belavin-Drinfel'd classification, the so-called
\emph{standard} classical $r$-matrix is given by the formula
$$r=\sum e_{\alpha}\wedge f_{\alpha}\in\mathfrak{g}\wedge\mathfrak{g}\ ,$$
where the summation is done over all positive roots.
Let $\langle\ ,\ \rangle$ be the Killing form on $\mathfrak{g}$.
For every $x\in\mathfrak{g}$, $X\in G$ and every function $f\in{\mathcal F}(G)$,
the left and right gradients are defined respectively by
$$
\begin{array}{ccc}
\langle \nabla f(X), x\rangle =\frac{d}{dt}|_{t=0}f(e^{tx}X)& \text{and}&
\langle \nabla 'f(X), x\rangle =\frac{d}{dt}|_{t=0}f(Xe^{tx})\ .
\end{array}$$
If $f,g\in{\mathcal F}(G)$, let $\pi_G$ be the following Poisson structure
on $G$ given by the \emph{Sklyanin bracket} which transforms $G$ into a Poisson-Lie group.
\begin{equation}\label{equ:Sbracket}
\{f,g\}_G=\frac{1}{2}(\langle \nabla f\otimes\nabla g,r\rangle
-\langle \nabla ' f\otimes\nabla ' g,r\rangle )\ .
\end{equation}
A Poisson stratification of $(G,\pi_G)$ is obtained by using the Bruhat decompositions
given by (\ref{bruhatdec}). As in \cite[Section 1.2]{FZtotal}, we associate to any
$u,v\in W$ the \emph{double Bruhat cell} $G^{u,v}\subset G$ defined by
$$
G^{u,v}=B\widehat{u}B\cap B_-\widehat{v}B_-\ .
$$
Each double Bruhat cell $G^{u,v}$ has dimension $\ell(u)+\ell(v)+l$ by
\cite[Theorem 1.1]{FZtotal}.
The following result leads to the Poisson stratification
\begin{equation}\label{equ:doubleBruhatstrat}
G=\displaystyle\bigcup_{u,v\in W}G^{u,v}\ .
\end{equation}

\begin{prop}\cite{HKKR, KZ, Rh}\label{prop:decG} For every $u,v\in W$,
the double Bruhat cells $G^{u,v}$ are the $H$-orbits, by the right-multiplication
action, of the symplectic leaves of $(G,\pi_G)$.
\end{prop}

\subsubsection{The Poisson-Lie group $(G^*,\pi_{G^*})$}Let us recall
that any multiplicative Poisson bracket on $G$ identically vanishes at
its unit  element $e\in G$; its linearization at $e$ gives rise to
the structure of a Lie algebra on the dual space $\mathfrak{g}^*$;
multiplicativity then implies that the dual of the commutator map
$[\,,\,]\colon \mathfrak{g}^*\times \mathfrak{g}^*\to
\mathfrak{g}^*$ is a 1-cocycle on $\mathfrak{g}$. A pair
$(\mathfrak{g},\, \mathfrak{g}^*)$ with these properties is called a
\emph{Lie bialgebra}. Because of an equivalence
between the category of Poisson-Lie groups (whose morphisms are Lie
group homomorphisms which are also Poisson mappings) and the
category of Lie bialgebras (whose morphisms are homomorphisms of Lie
algebras such that their duals are homomorphisms of the dual
algebras), finite dimensional
Poisson-Lie groups always come by pair. The Poisson-Lie group
associated to the Lie bialgebra $(\mathfrak{g}^*,\mathfrak{g})$ is
called \emph{the Poisson-Lie group dual to $G$}.

Let $\mathfrak{b}_{\pm}\subset\mathfrak g$ be the opposite Borel subalgebras
associated to $B_{\pm}$. The dual Lie algebra $\mathfrak{g}^*$ associated to
the standard $r$-matrix may be identified with the following subalgebra of
$\mathfrak{b}_{+}\oplus\mathfrak{b}_{-}$:
$$ \mathfrak{g}^* =\{(X_+,X_-)\in \mathfrak{b}_+\oplus \mathfrak{b}_-\;|\;\diag X_++\diag X_-=0\}\ .$$
Here, the application  $\diag:\mathfrak{g}\to\mathfrak{h}$ denotes the projection
of $\mathfrak{g}$ on its Cartan subalgebra. It can be lifted-up to give a projection
from the Lie group $G$ to its Cartan subgroup $H$, also denoted $\diag$.
 The Lie group $G^*$ associated with $\mathfrak{g}^*$ may be
identified with the following subgroup in $B_+\times B_-$,
\begin{equation}\label{equ:defG*double}
G^*=\{(b_+,b_-)\in B_+\times B_-| \diag b_+\cdot \diag b_-=I\}\ .
\end{equation}
It carries a natural Poisson bracket which makes it a Poisson-Lie
group; this is the dual Poisson-Lie group $(G^*,\pi_{G^*})$ of $G$.
Let $t\in{\mathfrak g}\otimes{\mathfrak g}$ be the Casimir element
of $\mathfrak{g}$, given by the following formula, and let us denote $r_{\pm}=r\pm t$.
$$t=\frac{1}{2}(\sum_{i,j\in[1,l]}({\widehat A}^{-1})_{ij}
h_i\otimes h_j+\sum_{i\in[1,l]}(e_i\otimes f_i+f_i\otimes e_i))\ .$$

\begin{prop}[\cite{STS}]\label{prop:STSPoisson}
Let us equip $G$ with the Poisson structure $\pi_*$ given by
$$
\{f,g\}_*=\frac{1}{2}(\langle \nabla f\otimes\nabla g,r\rangle
+\langle \nabla ' f\otimes\nabla ' g,r\rangle )
-\langle \nabla f\otimes\nabla' g,r_+\rangle -\langle \nabla ' f\otimes\nabla g,r_-\rangle\ .
$$
The map $\phi:(G^*,\pi_{G^*})\to(BB_-,\pi_*):(b_+,b_-)\mapsto b_+b_-^{-1}$
is a Poisson covering of degree $2^l$ of Poisson manifolds.
\end{prop}

\begin{prop}[\cite{STS}]\label{prop:conj}The following conjugation
action is a Poisson action.
$$
\begin{array}{rcl}
(G,\pi_G)\times(G,\pi_0)&\longrightarrow&(G,\pi_0)\\
(g,h)&\longmapsto&ghg^{-1}
\end{array}$$
\end{prop}

Following \cite{EL}, let us now give a Poisson stratification for $(G,\pi_*)$.
A \emph{regular class function}
on $G$ is a regular function on $G$ that is
invariant under conjugation. Two elements $g_1, g_2 \in G$ are said to be in the
same \emph{Steinberg fiber} if $f(g_1) = f(g_2)$ for
every regular class function $f$ on $G$.
For $t \in H$, let $F_t$ be the Steinberg fiber containing $t$.
By the Jordan decomposition
of elements in $G$, every Steinberg fiber is of the form
$F_t$ for some $t \in H$. Moreover, the equality
$F_{t'} = F_t$ is satisfied if and only if there exists
$w\in W$ such that $t'=w(t)$, where $W$ acts on $H$ by the formula
(\ref{equ:actionW}).
The group $G$ has therefore the decompositions
\begin{equation}\label{equ:decG^*}
\begin{array}{cccc}
G=\displaystyle\bigcup_{t\in H,w\in W}F_{t,w}
=\displaystyle\bigsqcup_{t\in H\backslash W,w\in W}F_{t,w}
&&\mbox{where}&F_{t,w}:=B\widehat{w}B_-\bigcap F_t\ .
\end{array}
\end{equation}

\begin{prop}\cite[Proposition 3.3]{EL} For every $t\in H$ and $v\in W$:
\begin{itemize}
\item
$F_{t,w}$ is a non empty irreducible subvariety of $G$ with dimension equal to
$\dim(G)-l-\ell(w)$.
\item
$F_{t,v}$ is a finite union of $H$-orbits, for the conjugation action,
of the symplectic leaves of $(G,\pi_*)$.
\end{itemize}
\end{prop}

\subsection{Backgrounds on cluster $\mathcal{X}$-varieties}\label{section:backcluster}
We recall the definitions
introduced by Fock and Goncharov underlying cluster $\mathcal X$-varieties. We add the
notion of \emph{erasing map}, which has been already used
in \cite{FGcluster} without be named.

\begin{definition}\cite[Definition 1.4]{FGdilog} A \emph{seed}
${\mathbf I}$ is a quadruple $(I, I_0, \varepsilon, d)$ where
\begin{itemize}
\item $I$ is a finite set;

\item $I_0\subset I$;

\item $\varepsilon$ is a matrix $\varepsilon_{ij}$, $i,j\in I$, such
that $\varepsilon_{ij}\in \mathbb{Z}$, unless $i,j\in I_0$;

\item $d=\{d_i\}$, $i \in I$, is a subset of positive integers such that
the matrix $\widehat{\varepsilon}_{ij}=\varepsilon_{ij}d_j$ is skew-symmetric.
\end{itemize}
\end{definition}

Elements of the set $I_0$ are sometimes called \emph{frozen vertices}.
Here, we will not use this terminology, however, for the simple reason that, in
the next, we will allow some birational Poisson isomorphisms in the direction of
these frozen vertices, called {tropical mutations}.
For every real number $x\in\mathbb R$, let us denote
$[x]_+= \max(x,0)$ and
$$
\sgn(x)=
\left\{\begin{array}{cc}
-1 & \text{if $x<0$ ;}\\
0  & \text{if $x=0$ ;}\\
 1 & \text{if $x>0$ .}
\end{array}
\right.
$$

\begin{definition}\cite[Section 1.2]{FGdilog},\cite[Definition 4.2]{FZ1}
Let $\mathbf{I}=(I,I_0,\varepsilon,d)$, $\mathbf {I'}=(I',{I'}_0,{\varepsilon}',d')$
be two seeds, and $k\in I\backslash I_0$. A $\emph{mutation in the direction k}$ is a map
$\mu_k:I\longrightarrow I'$ satisfying the following conditions:
\begin{itemize}
\item
$\mu_k(I_0)={I'}_0$;
\item
${d'}_{\mu_k(i)}=d_i$;
\item
${\varepsilon^{'}}_{\mu_k(i)\mu_k(j)}=\left\{ \begin{array}{lll}
-\varepsilon_{ij}&   \mbox{if}\ i=k\ \mbox{or}\ j=k\ ;\\
\varepsilon_{ij}+\sgn(\varepsilon_{ik})[\varepsilon_{ik}\varepsilon_{kj}]_+& \mbox{if}\  i,j\neq k\ .
\end{array}\right.$
\end{itemize}
\end{definition}

\begin{definition}\cite[Section 1.2]{FGdilog}\label{def:symmetry}
A $\emph{symmetry}$ on a seed $\mathbf{I}=(I,I_0,\varepsilon,d)$ is an
automorphism $\sigma$ of the set $I$ preserving the subset $I_0$, the matrix
$\varepsilon$, and the numbers $d_i$. That is to say:
\begin{itemize}
\item
$\sigma (I_0)=I_0$;
\item
$d_{\sigma (i)}=d_i$;
\item
${\varepsilon}_{\sigma (i)\sigma (j)}={\varepsilon}_{ij}.$
\end{itemize}
\end{definition}

Let $| I|$ be the cardinal of every finite set $I$ and
${\mathbb C}_{\neq 0}$ be the set of non-zero complex numbers.

\begin{definition}\cite[Section 1.2]{FGdilog}\label{def:mutation}
Let ${\mathbf I}$ be a seed. The related \emph{seed} $\mathcal X$-\emph{torus}
$\mathcal{X}_{\mathbf I}$ is the torus $(\mathbb C_{\neq 0})^{| I|}$ with
the Poisson bracket $$\{x_i,x_j\}=\widehat{\varepsilon}_{ij}x_ix_j,$$
where $\{x_i| i\in I\}$ are the standard coordinates on the factors.
The \emph{exchange part} of the seed $\mathcal X$-torus $\mathcal{X}_{\mathbf I}$
is the subtorus obtained by keeping only the $x_j$ for $j\in I\backslash I_0$.

Symmetries and mutations on seeds induce involutive maps between the
corresponding seed $\mathcal X$-tori, which are denoted by the same symbols
${\mu}_k$ and $\sigma$, and given by
\begin{itemize}
\item $
x_{\sigma(i)}=x_i$
\item
$x_{\mu_k(i)}=\left\{ \begin{array}{lll}
{x_k}^{-1}&  \mbox{if}\ i=k\ ;\\
x_ix_k^{[\varepsilon_{ik}]_+}(1+x_k)^{-\varepsilon_{ik}}& \mbox{if}\ i\not=k\ .
\end{array}\right.
$
\end{itemize}
\end{definition}

\begin{definition}\cite[Section 1.2]{FGdilog} A \emph{cluster transformation } linking two seeds
(and two seed $\mathcal X$-tori) is a composition of symmetries and mutations.
Let $\mathbf I$ be a seed.
A $\emph{cluster}$ ${\mathcal X}$-$\emph{variety}$ ${\mathcal X}_{|\mathbf I|}$
is obtained by taking every seed $\mathcal X$-tori obtained from
${\mathcal X}_{\mathbf I}$ by cluster transformations, and gluing them via
the previous bi-rational isomorphisms.
\end{definition}


\begin{definition}\cite{FGcluster}\label{def:amalg}
Let $\mathbf I=(I,I_0,\eta,c)$ and $\mathbf J = (J,J_0,\varepsilon,d)$
be two seeds and let $L$ be a set embedded into both $I_0$ and $J_0$
in a such a way that for any $i,j \in L$ we have $c(i)=d(i)$. Then the
amalgamation of  $\mathbf J$ and $\mathbf I$ is a seed $\mathbf K =
(K,K_0,\zeta,b)$, such that $K=I \cup_L J$, $K_0 = I_0\cup_L J_0$ and
$$\zeta_{ij}=\left\{\begin{array}{lcl}
               0 &\mbox{if}& i\in I\backslash L \mbox{ and } j\in J\backslash L\ ;\\
               0 &\mbox{if}& i\in J\backslash L \mbox{ and } j\in I\backslash L\ ;\\
               \eta_{ij}&\mbox{if}&i\in I\backslash L \mbox{ and } j\in I\backslash L\ ;\\
               \varepsilon_{ij}&\mbox{if}&i\in J\backslash L \mbox{ and } j\in J\backslash L\ ;\\
               \eta_{ij}+\varepsilon_{ij}&\mbox{if}&i,j\in L\ .
                    \end{array}
             \right.
$$
This operation induces a homomorphism ${\mathcal X}_{\mathbf J} \times
{\mathcal X}_{\mathbf I} \to {\mathcal X}_{\mathbf K}$ between the
corresponding seed ${\mathcal X}$-tori given by the rule
$$
z_i = \left\{\begin{array}{lcl}x_i&\mbox{if}&i\in I\backslash L\ ;\\
                               y_i&\mbox{if}&i\in J\backslash  L\ ;\\
                               x_iy_i&\mbox{if}&i\in L\ .
              \end{array}
       \right.
$$
It is easy to check that it respects the Poisson structure and
 commutes with cluster transformations,
thus is defined for the cluster ${\mathcal X}$-varieties, and not only for the seeds.
\end{definition}

\begin{definition}\label{def:erasing} Let $\mathbf{I}=(I,I_0,\varepsilon,d)$ be a seed and
$k\in I$. A \emph{$k$-erasing map} on a seed
$\mathbf{I}=(I,I_0,\varepsilon,d)$ is a
morphism $\varsigma_k$ on $\mathbf{I}$ such that
\begin{itemize}
\item
$\varsigma_k (I_0)=I_0$  and  $\varsigma_k (I)=I\backslash\{k\}$;
\item
$d_{\varsigma_k (i)}=d_i$;
\item
${\varepsilon}_{\varsigma_k (i)\varsigma_k (j)}={\varepsilon}_{ij}.$
\end{itemize}
erasing maps on seeds induce maps between the corresponding seed $\mathcal X$-tori,
which are denoted by the same symbols $\varsigma_k$, and given by
$$x_{\varsigma_k(i)}=x_i\ .$$
\end{definition}

\section{{Cluster $\mathcal X$-varieties} related to $(G,\pi_G)$}
\label{section:ClusterG}
We show, in a way useful for our purposes, how to naturally attach a cluster
$\mathcal X$-variety ${\mathcal X}_{|\mathbf i|}$ to every double Bruhat cells
$(G^{u,v},\pi_G)$ via evaluation maps associated to any double (reduced) word
$\mathbf i$. This section sums-up results of \cite{FGcluster}.

\subsection{Combinatorics on double words of $W$}We start by recalling the combinatorics
on double words of $W$, which is derived from a well-known result of Tits.
Following \cite{FZtotal}, a (reduced) word of
$W\times W$ is called a \emph{double (reduced) word}.
In order to avoid confusions we denote ${\overline 1}, \dots,{\overline l}$ the
indices of the reflections associated to the first copy of $W$, and
$1,\dots,l$ the indices of the reflections associated to the second copy.
A double (reduced) word of $(u,v)$ is nothing else that a shuffle of
a (reduced) word of $u$, written in the alphabet $[{\overline 1},{\overline l}]$,
and of a (reduced) word of  $v$, written in the alphabet $[1,l]$.
We denote $D(u,v)$ and $R(u,v)$ the set of double words and double reduced words of $(u,v)$, respectively.
(Therefore we have the inclusion $R(u,v)\subset D(u,v)$.)
In particular, let $\mathbf 1\in R(1,1)$ be the double word associated
to the unity of $W\times W$.

Let $w\in W$ and $\mathbf i$ be a word of $w$. Following \cite[Section 7]{BZtensor},
we call a \emph{$d$-move} (also named "$\braid$-move" in \cite{BB})
a transformation of $\mathbf i$ that replaces $d$ consecutive entries
$i,j,i,j, \ldots$ by $j,i,j,i,\ldots$,
for some $i$ and $j$ such that $d$ is the order of $s_i s_j\,$, that is:
if $a_{ij}a_{ji} = 0$ (resp.\ $1,2,3$), then $d=2$ (resp.\ $3,4,6$).
We call \emph{$nil$-move} a transformation of $\mathbf i$ that replaces
the string $i\ i$ by the elementary string $i$ for some $i\in[1,l]$.
The following result, called \emph{the Tits theorem}, is standard and can be
found, for example, in \cite[Theorem 3.3.1]{BB}.

\begin{thm}\label{thm:Tits}Let $(W,S)$ be a Coxeter group and $w\in W$.
\begin{itemize}
\item
Any expression for $w$ can be transformed into a
reduced expression for $w$ by a sequence of $nil$-moves and $d$-moves.
\item
Every two reduced words for $w$ can be connected by a sequence of $d$-moves.
\end{itemize}
\end{thm}

Let us say that a letter $i$ of $\mathbf i$ is \emph{positive} if $i\in[1,l]$
and \emph{negative} if $i\in[\overline 1,\overline l]$; a double word $\mathbf i$
will be said to be \emph{positive} (resp. \emph{negative}) if all its letters are
positive (resp. negative). Considering the group $W \times W$,
we conclude that every two double reduced words $\mathbf{i}, \mathbf{j} \in
R(u,v)$ can be obtained from each other by a sequence of
\emph{generalized $d$-moves}, listed in
Figure~\ref{fig:d-moves}. They contains
\begin{itemize}
\item
\emph{positive} $d$-moves for the alphabet $[1,l]$;
\item
\emph{negative} $d$-moves for the alphabet $[\overline 1,\overline l]$;
\item
\emph{mixed} $2$-moves
that interchange two consecutive indices of opposite signs.
\end{itemize}
In the same way, any double word for the couple $(u,v)\in W\times W$ can
be transformed to give any double reduced word of $R(u,v)$ by a sequence of
\emph{generalized $dn$-moves}, including
\begin{itemize}
\item
\emph{positive} $nil$-moves for the alphabet $[1,l]$;
\item
\emph{negative} $nil$-moves for the alphabet $[\overline 1,\overline l]$;
\item
generalized $d$-moves.
\end{itemize}
Positive and negative $\nil$-moves are given by Figure \ref{fig:dn-moves}.
\begin{figure}[ht]
$$\begin{array}{cccccccc}
&\dots\ i\ \overline{j}\ \dots & \rightsquigarrow
& \dots\ \overline{j}\ i\ \dots\\
         \mbox{or}
        & \dots\ \overline{i}\ j\ \dots & \rightsquigarrow
        & \dots\ j\ \overline{i}\ \dots& \mbox{for every }i,j\in[1,l]
        &\\
        &\\
&\dots\ i\ j\ \dots & \rightsquigarrow & \dots\ j\ i\ \dots\\
        \mbox{or}
        & \dots\ \overline i\ \overline j\ \dots
         & \rightsquigarrow & \dots\ \overline j\ \overline i\ \dots
        & \mbox{when }a_{ij}a_{ji}=0\\
        &\\
&\dots\ i\ j\ i\ \dots & \rightsquigarrow & \dots\ j\ i\ j\ \dots\\
        \mbox{or}
        & \dots\ \overline i\ \overline j\ \overline i\ \dots
        & \rightsquigarrow & \dots\ \overline j\ \overline i\ \overline j\ \dots
        & \mbox{when }a_{ij}a_{ji}=1\\
        &\\
&\dots\ i\ j\ i\ j\ \dots & \rightsquigarrow & \dots\ j\ i\ j\ i\ \dots\\
        \mbox{or}
        & \dots\ \overline i\ \overline j\ \overline i\ \overline j\ \dots
        & \rightsquigarrow & \dots\ \overline j\ \overline i\
        \overline j\ \overline i\ \dots
        & \mbox{when }a_{ij}a_{ji}=2\\
        &\\
&\dots\ i\ j\ i\ j\ i\ j\ \dots & \rightsquigarrow
& \dots\ j\ i\ j\ i\ j\ i\ \dots\\
        \mbox{or}
        & \dots\ \overline i\ \overline j\ \overline i\
        \overline j\ \overline i\ \overline j\ \dots
        & \rightsquigarrow & \dots\ \overline j\ \overline i\ \overline j\ \overline i\ \overline j\ \overline i\ \dots
        & \mbox{when }a_{ij}a_{ji}=3
\end{array}
$$
\vspace{-.1in}
\label{fig:d-moves}
\caption{The generalized $d$-moves}
\end{figure}

\begin{figure}[ht]
$$\begin{array}{cccccc}
&\dots\ i\ i\ \dots & \rightsquigarrow
& \dots\ i\ \dots\\
         \mbox{and}
        & \dots\ \overline{i}\ \overline{i}\ \dots & \rightsquigarrow
        & \dots\ \overline{i}\ \dots& \mbox{for every }i\in[1,l]
\end{array}$$
\vspace{-.1in}
\label{fig:dn-moves}
\caption{The positive and negative $nil$-moves}
\end{figure}

\subsection{Quivers and seeds associated to a double word}
\label{section:graph} We then reformulate the procedure to attach
a seed to a double word given by \cite{FGcluster} via gluing
on quivers in the spirit of \cite[Section 13]{FST}.

\subsubsection{Dynkin quivers}
Let $\Gamma_{\mathfrak{g}}$ be the Dynkin diagram of $\mathfrak{g}$, denote its vertex
by $(^{1}_{0})\dots (^{l}_{0})$ and choose some $i\in[1,l]$.
The \emph{elementary Dynkin quiver} $\Gamma_{\mathfrak g}(i)$ is the directed graph
obtained from $\Gamma$ by the following procedure.
\begin{itemize}
\item
Create a new vertex $(^{i}_{1})$ and
call \emph{$i$-vertices} the vertices $(^{i}_{0}), (^{i}_{1})$.
and \emph{$j$-vertex} the vertex $(^{j}_{0})$ for any $j\neq i$.
The vertex $(^{i}_{0})$ and the $j$-vertices are then called \emph{left outlets},
whereas $(^{i}_{1})$ and these $j$-vertices are called \emph{right outlets}.
\item
Erase all the edges of the Dynkin diagram except the ones involving the vertex
$(^{i}_{0})$.
\item
Rely the vertices $(^{i}_{1})$ to the vertex $(^{i}_{0})$ by an arrow such that
$(^{i}_{1})$ is the tail of the arrow and $(^{i}_{0})$ its head.
\item
Create as many arrows between $(^{i}_{1})$ and each remaining vertices
there are between $(^{i}_{0})$ and this remaining vertex. The heads and the
tails of the arrows are directed in such a way that the triangle(s) thus
created is/are oriented.
\end{itemize}
The elementary Dynkin quiver $\Gamma_{\mathfrak g}(\overline i)$
is then obtained from $\Gamma_{\mathfrak g}(i)$ by reversing the orientation
of all the arrows. To these elementary Dynkin quivers, we add the
\emph{trivial Dynkin quiver} $\Gamma_{\mathfrak g}(\mathbf 1)$ obtained
from the Dynkin diagram by removing all the edges; as above the vertices of
this quiver are labeled by $(^{i}_{0})$, with $i\in[1,l]$. (For the trivial
Dynkin quiver $\Gamma_{\mathfrak g}(\mathbf 1)$, the set of left outlets
is, by definition, the set of right outlets.)
Figure \ref{fig:block} and Figure \ref{fig:block2} describe respectively
the elementary Dynkin quivers of the cases $\mathfrak{g}=A_3$ and $\mathfrak{g}=B_2$.
Let us stress that in all our examples, outlets will be marked by unfilled circles.
Other conventions in our drawing will be the following: the type of vertices is given
by a kind of height function where vertices of type $1$ are at the top of the quiver
and vertices of type $l$ at the bottom; moreover, left outlets will always be drawn
at the left of right outlets if both are $k$-vertices but different.

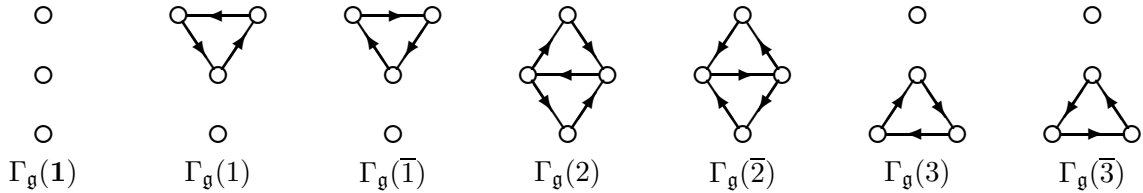
\begin{figure}[htbp]
\begin{center}
\setlength{\unitlength}{1.5pt}
\begin{picture}(20,33)(0,-12)
\thicklines
\put(10,30){\circle{4}}
\put(10,0){\circle{4}}
\put(10,15){\circle{4}}
\put(10,-10){\makebox(0,0){$\Gamma_{\mathfrak{g}}(\mathbf 1)$}}
\end{picture}
\qquad\quad
\begin{picture}(20,33)(0,-12)
\thicklines
\put(2,30){\line(1,0){16}}
\put(18,30){\vector(-1,0){11}}
\put(1,28.5){\line(2,-3){8}}
\put(1,28.5){\vector(2,-3){6}}
\put(11,16.5){\vector(2,3){6}}
\put(11,16.5){\line(2,3){8}}
\put(0,30){\circle{4}}
\put(20,30){\circle{4}}
\put(10,0){\circle{4}}
\put(10,15){\circle{4}}
\put(10,-10){\makebox(0,0){$\Gamma_{\mathfrak{g}}(1)$}}
\end{picture}
\qquad\quad
\begin{picture}(20,33)(0,-12)
\thicklines
\put(2,30){\line(1,0){16}}
\put(2,30){\vector(1,0){11}}
\put(1,28.5){\line(2,-3){8}}
\put(19,28.5){\vector(-2,-3){6}}
\put(9,16.5){\vector(-2,3){6}}
\put(11,16.5){\line(2,3){8}}
\put(0,30){\circle{4}}
\put(20,30){\circle{4}}
\put(10,0){\circle{4}}
\put(10,15){\circle{4}}
\put(10,-10){\makebox(0,0){$\Gamma_{\mathfrak{g}}(\overline 1)$}}
\end{picture}
\qquad\quad
\begin{picture}(20,48)(0,-27)
\thicklines
\multiput(1,1.5)(10,-15){2}{\line(2,3){8}}
\put(1,1.5){\vector(2,3){5}}
\put(11,13.5){\vector(2,-3){6}}
\multiput(1,-1.5)(10,15){2}{\line(2,-3){8}}
\put(1,-1.5){\vector(2,-3){5}}
\put(11,-13.5){\vector(2,3){6}}
\multiput(10,15)(0,-30){2}{\circle{4}}
\multiput(0,0)(20,0){2}{\circle{4}}
\put(2,0){\line(1,0){16}}
\put(18,0){\vector(-1,0){11}}
\put(10,-25){\makebox(0,0){$\Gamma_{\mathfrak{g}}(2)$}}
\end{picture}
\qquad\quad
\begin{picture}(20,48)(0,-27)
\thicklines
\multiput(1,1.5)(10,-15){2}{\line(2,3){8}}
\put(19,1.5){\vector(-2,3){5}}
\put(9,13.5){\vector(-2,-3){6}}
\multiput(1,-1.5)(10,15){2}{\line(2,-3){8}}
\put(19,-1.5){\vector(-2,-3){5}}
\put(9,-13.5){\vector(-2,3){6}}
\multiput(10,15)(0,-30){2}{\circle{4}}
\multiput(0,0)(20,0){2}{\circle{4}}
\put(2,0){\line(1,0){16}}
\put(2,0){\vector(1,0){11}}
\put(10,-25){\makebox(0,0){$\Gamma_{\mathfrak{g}}(\overline 2)$}}
\end{picture}
\qquad\quad
\begin{picture}(20,33)(0,-12)
\thicklines
\put(2,0){\line(1,0){16}}
\put(18,0){\vector(-1,0){11}}
\put(1,1.5){\line(2,3){8}}
\put(1,1.5){\vector(2,3){6}}
\put(11,13.5){\vector(2,-3){6}}
\put(19,1.5){\line(-2,3){8}}
\multiput(0,0)(20,0){2}{\circle{4}}
\put(10,30){\circle{4}}
\put(10,15){\circle{4}}
\put(10,-10){\makebox(0,0){$\Gamma_{\mathfrak{g}}(3)$}}
\end{picture}
\qquad\quad
\begin{picture}(20,33)(0,-12)
\thicklines
\put(2,0){\line(1,0){16}}
\put(2,0){\vector(1,0){11}}
\put(1,1.5){\line(2,3){8}}
\put(19,1.5){\vector(-2,3){6}}
\put(9,13.5){\vector(-2,-3){6}}
\put(19,1.5){\line(-2,3){8}}
\multiput(0,0)(20,0){2}{\circle{4}}
\put(10,30){\circle{4}}
\put(10,15){\circle{4}}
\put(10,-10){\makebox(0,0){$\Gamma_{\mathfrak{g}}(\overline 3)$}}
\end{picture}
\end{center}
\vspace{-.1in}
\caption{Elementary Dynkin quivers for $\mathfrak{g}=A_3$}
\label{fig:block}
\end{figure}

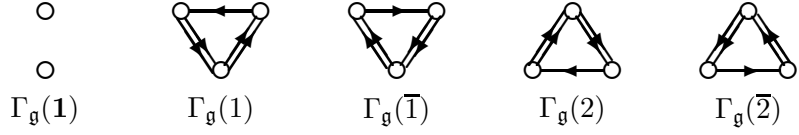
\begin{figure}[htbp]
\begin{center}
\setlength{\unitlength}{1.5pt}
\begin{picture}(20,33)(0,-12)
\thicklines
\put(10,15){\circle{4}}
\put(10,0){\circle{4}}
\put(10,-10){\makebox(0,0){$\Gamma_{\mathfrak{g}}(\mathbf 1)$}}
\end{picture}
\qquad\quad
\begin{picture}(20,33)(0,-12)
\thicklines
\put(2,15){\line(1,0){16}}
\put(18,15){\vector(-1,0){11}}
\put(1,13){\line(2,-3){7.5}}
\put(1,13){\vector(2,-3){6}}
\put(-1,13){\line(2,-3){9}}
\put(-1,13){\vector(2,-3){6.5}}
\put(11,1.5){\vector(2,3){6}}
\put(12,0){\vector(2,3){6.5}}
\put(12,0){\line(2,3){8.5}}
\put(11,1.5){\line(2,3){8}}
\put(0,15){\circle{4}}
\put(20,15){\circle{4}}
\put(10,0){\circle{4}}
\put(10,-10){\makebox(0,0){$\Gamma_{\mathfrak{g}}(1)$}}
\end{picture}
\qquad\quad
\begin{picture}(20,33)(0,-12)
\thicklines
\put(2,15){\line(1,0){16}}
\put(2,15){\vector(1,0){11}}
\put(19,13.5){\line(-2,-3){8}}
\put(19,13.5){\vector(-2,-3){6}}
\put(21,13.5){\line(-2,-3){8.5}}
\put(21,13.5){\vector(-2,-3){6.5}}
\put(9,1.5){\vector(-2,3){6}}
\put(9,1.5){\line(-2,3){8}}
\put(7.5,0.5){\vector(-2,3){6}}
\put(7.5,0.5){\line(-2,3){8.5}}
\put(0,15){\circle{4}}
\put(20,15){\circle{4}}
\put(10,0){\circle{4}}
\put(10,-10){\makebox(0,0){$\Gamma_{\mathfrak{g}}(\overline 1)$}}
\end{picture}
\qquad\quad
\begin{picture}(20,33)(0,-12)
\thicklines
\put(2,0){\line(1,0){16}}
\put(18,0){\vector(-1,0){11}}
\put(1,1.5){\line(2,3){8}}
\put(1,1.5){\vector(2,3){6}}
\put(-1,1.5){\line(2,3){9}}
\put(-1,1.5){\vector(2,3){6.5}}
\put(11,13.5){\vector(2,-3){6}}
\put(11,13.5){\line(2,-3){8}}
\put(12,14.5){\vector(2,-3){6}}
\put(12,14.5){\line(2,-3){8}}
\multiput(0,0)(20,0){2}{\circle{4}}
\put(10,15){\circle{4}}
\put(10,-10){\makebox(0,0){$\Gamma_{\mathfrak{g}}(2)$}}
\end{picture}
\qquad\quad
\begin{picture}(20,33)(0,-12)
\thicklines
\put(2,0){\line(1,0){16}}
\put(2,0){\vector(1,0){11}}
\put(19,1.5){\line(-2,3){8}}
\put(19,1.5){\vector(-2,3){6}}
\put(20.5,2){\line(-2,3){8}}
\put(20.5,2){\vector(-2,3){6}}
\put(9,13.5){\vector(-2,-3){6}}
\put(9,13.5){\line(-2,-3){8}}
\put(8,15){\vector(-2,-3){6.5}}
\put(8,15){\line(-2,-3){8.5}}
\multiput(0,0)(20,0){2}{\circle{4}}
\put(10,15){\circle{4}}
\put(10,-10){\makebox(0,0){$\Gamma_{\mathfrak{g}}(\overline 2)$}}
\end{picture}
\end{center}
\vspace{-.1in}
\caption{Elementary Dynkin quivers for $\mathfrak{g}=B_2$}
\label{fig:block2}
\end{figure}

A quiver $\Gamma$ is called a \emph{Dynkin quiver} if it can
be obtained from a collection of disjoint elementary Dynkin quivers,
coming from the same Dynkin diagram, by the following procedure, called
\emph{amalgamation}. Let $\mathbf i=i_1\dots i_n$ be a double word and
$\Gamma_{\mathfrak g}(i_1),\dots,\Gamma_{\mathfrak g}(i_n)$ be the associated
elementary Dynkin quivers. For every $k\in[1,l]$, we put a total order, called the $k$-order
on the set of the $k$-vertices of all the elementary Dynkin quivers in such a way that
$(^{i_j}_{0})<(^{i_j}_{1})$ for every $j$, and that every $k$-vertices of
$\Gamma_{\mathfrak g}(i_l)$ are lower than the $k$-vertices
of $\Gamma_{\mathfrak g}(i_j)$ when $l<j$.
\begin{itemize}
\item
For every $j\in[1,n-1]$, glue every right outlets of $\Gamma_{\mathfrak g}(i_j)$
to every left outlets of $\Gamma_{\mathfrak g}(i_{j+1})$
in such a way that $k$-vertices are glued together;
\item
relabel each $k$-vertex following the $k$-order, that is $(^{k}_{i})$
is the $i^{\th}$ $k$-vertex, with an increasing numbering, of the set of all
$k$-vertices given by the $k$-order;
\item
redefine the set of outlets: left outlets (resp. right outlets) are vertices
$(^{k}_{i})$ such that $(^{k}_{i})\leq (^{k}_{j})$ (resp. $(^{k}_{j})\leq (^{k}_{i})$)
for every $j$;
\item
if $\Gamma$ contains a pair of edges connecting the same pair of
vertices but going in opposite directions, then remove each such a
pair of edges; (As a result, all arrows linking two vertices point
now in the same direction.)
\item
for every $k_1\neq k_2$, erase arrows linking $k_1$-vertices to
$k_2$-vertices until there is not more arrows linking these vertices
in the resulting graph than edges linking the $k_1^{\th}$ vertex and the
$k_2^{\th}$-vertex in the corresponding Dynkin diagram.
\end{itemize}
The amalgamation process is certainly easier to figure out on examples.
Figure \ref{fig:amalgblock3'}, Figure \ref{fig:amalgblock3''} and
Figure \ref{fig:amalgblock3} describe some amalgamations in the case
$\mathfrak{g}=A_3$.

\begin{figure}[htbp]
\begin{center}
\setlength{\unitlength}{1.5pt}
\begin{picture}(20,40)(0,-27)
\thicklines
{\oval(30,30)[l]}
\put(6,13.5){\vector(2,-3){5}}
\put(16,1.5){\vector(2,3){6}}
\put(6,13.5){\line(2,-3){8}}
\put(16,1.5){\line(2,3){8}}
\put(5,15){\circle{4}}
\put(25,15){\circle{4}}
\put(15,0){\circle{4}}
\put(15,-15){\circle{4}}
\put(7,15){\line(1,0){16}}
\put(23,15){\vector(-1,0){11}}
\put(10,-25){\makebox(0,0){$\Gamma_{\mathfrak{g}}(1)$}}
\end{picture}
\quad
\begin{picture}(10,40)(0,-27)
\thicklines
\put(0,0){$,$}
\end{picture}
\begin{picture}(20,40)(0,-27)
\thicklines
\put(20,0){\oval(30,30)[r]}
\multiput(1,1.5)(10,-15){2}{\line(2,3){8}}
\put(19,1.5){\vector(-2,3){5}}
\put(9,13.5){\vector(-2,-3){6}}
\multiput(1,-1.5)(10,15){2}{\line(2,-3){8}}
\put(19,-1.5){\vector(-2,-3){5}}
\put(9,-13.5){\vector(-2,3){6}}
\multiput(10,15)(0,-30){2}{\circle{4}}
\multiput(0,0)(20,0){2}{\circle{4}}
\put(2,0){\line(1,0){16}}
\put(2,0){\vector(1,0){11}}
\put(10,-25){\makebox(0,0){$\Gamma_{\mathfrak{g}}(\overline 2)$}}
\end{picture}
\qquad\qquad
\begin{picture}(20,40)(0,-27)
\thicklines
\put(0,0){\vector(1,0){20}}
\end{picture}
\quad
\begin{picture}(20,40)(0,-27)
\thicklines
\put(0,15){\circle{4}}
\put(20,15){\circle{4}}
\put(0,0){\circle{4}}
\put(20,0){\circle{4}}
\put(10,-15){\circle{4}}
\put(2,15){\line(1,0){16}}
\put(18,15){\vector(-1,0){11}}
\put(2,0){\line(1,0){16}}
\put(2,0){\vector(1,0){11}}
\put(0,13){\line(0,-1){11}}
\put(0,13){\vector(0,-1){9}}
\put(20,2){\line(0,1){11}}
\put(20,2){\vector(0,1){9}}
\put(9,-13.5){\vector(-2,3){6}}
\put(9,-13.5){\line(-2,3){8}}
\put(19,-1.5){\vector(-2,-3){5}}
\put(19,-1.5){\line(-2,-3){8}}
\put(10,-25){\makebox(0,0){$\Gamma_{\mathfrak{g}}(1\overline 2)$}}
\end{picture}
\end{center}
\vspace{-.1in}
\caption{The amalgamation $(\Gamma_{\mathfrak{g}}(1)
,\Gamma_{\mathfrak{g}}(\overline 2))\mapsto \Gamma_{\mathfrak{g}}(1\overline 2)$
for $\mathfrak{g}=A_3$}
\label{fig:amalgblock3'}
\end{figure}

\begin{figure}[htbp]
\begin{center}
\setlength{\unitlength}{1.5pt}
\begin{picture}(20,40)(0,-27)
\thicklines
{\oval(30,30)[l]}
\multiput(1,1.5)(10,-15){2}{\line(2,3){8}}
\put(1,1.5){\vector(2,3){5}}
\put(11,13.5){\vector(2,-3){6}}
\multiput(1,-1.5)(10,15){2}{\line(2,-3){8}}
\put(1,-1.5){\vector(2,-3){5}}
\put(11,-13.5){\vector(2,3){6}}
\multiput(10,15)(0,-30){2}{\circle{4}}
\multiput(0,0)(20,0){2}{\circle{4}}
\put(2,0){\line(1,0){16}}
\put(18,0){\vector(-1,0){11}}
\put(10,-25){\makebox(0,0){$\Gamma_{\mathfrak{g}}(2)$}}
\end{picture}
\quad
\begin{picture}(10,40)(0,-27)
\thicklines
\put(0,0){$,$}
\end{picture}
\begin{picture}(20,40)(0,-27)
\thicklines
\put(20,0){\oval(30,30)[r]}
\multiput(1,1.5)(10,-15){2}{\line(2,3){8}}
\put(1,1.5){\vector(2,3){5}}
\put(11,13.5){\vector(2,-3){6}}
\multiput(1,-1.5)(10,15){2}{\line(2,-3){8}}
\put(1,-1.5){\vector(2,-3){5}}
\put(11,-13.5){\vector(2,3){6}}
\multiput(10,15)(0,-30){2}{\circle{4}}
\multiput(0,0)(20,0){2}{\circle{4}}
\put(2,0){\line(1,0){16}}
\put(18,0){\vector(-1,0){11}}
\put(10,-25){\makebox(0,0){$\Gamma_{\mathfrak{g}}(2)$}}
\end{picture}
\qquad\qquad
\begin{picture}(20,40)(0,-27)
\thicklines
\put(0,0){\vector(1,0){20}}
\end{picture}
\quad
\begin{picture}(20,40)(0,-27)
\thicklines
\put(21,-13.5){\line(3,2){18}}
\put(21,-13.5){\vector(3,2){12}}
\put(21,13.5){\line(3,-2){18}}
\put(21,13.5){\vector(3,-2){12}}
\put(1,1.5){\line(3,2){18}}
\put(1,1.5){\vector(3,2){12}}
\put(1,-1.5){\line(3,-2){18}}
\put(1,-1.5){\vector(3,-2){12}}
\put(20,15){\circle{4}}
\put(20,-15){\circle{4}}
\put(0,0){\circle{4}}
\put(20,0){\circle*{4}}
\put(40,0){\circle{4}}
\put(2,0){\line(1,0){16}}
\put(18,0){\vector(-1,0){11}}
\put(38,0){\line(-1,0){16}}
\put(38,0){\vector(-1,0){11}}
\put(20,-25){\makebox(0,0){$\Gamma_{\mathfrak{g}}(2 2)$}}
\end{picture}
\end{center}
\vspace{-.1in}
\caption{The amalgamation $(\Gamma_{\mathfrak{g}}(2)
,\Gamma_{\mathfrak{g}}( 2))\mapsto \Gamma_{\mathfrak{g}}(2 2)$
for $\mathfrak{g}=A_3$}
\label{fig:amalgblock3''}
\end{figure}
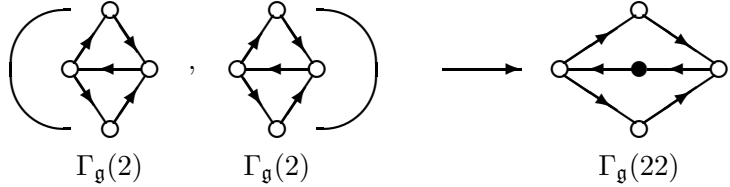

\begin{figure}[htbp]
\begin{center}
\setlength{\unitlength}{1.5pt}
\begin{picture}(20,40)(0,-27)
\thicklines
{\oval(30,30)[l]}
\multiput(1,1.5)(10,-15){2}{\line(2,3){8}}
\put(1,1.5){\vector(2,3){5}}
\put(11,13.5){\vector(2,-3){6}}
\multiput(1,-1.5)(10,15){2}{\line(2,-3){8}}
\put(1,-1.5){\vector(2,-3){5}}
\put(11,-13.5){\vector(2,3){6}}
\multiput(10,15)(0,-30){2}{\circle{4}}
\multiput(0,0)(20,0){2}{\circle{4}}
\put(2,0){\line(1,0){16}}
\put(18,0){\vector(-1,0){11}}
\put(10,-25){\makebox(0,0){$\Gamma_{\mathfrak{g}}(2)$}}
\end{picture}
\quad
\begin{picture}(10,40)(0,-27)
\thicklines
\put(0,0){$,$}
\end{picture}
\begin{picture}(20,40)(0,-27)
\thicklines
\put(20,0){\oval(30,30)[r]}
\multiput(1,1.5)(10,-15){2}{\line(2,3){8}}
\put(19,1.5){\vector(-2,3){5}}
\put(9,13.5){\vector(-2,-3){6}}
\multiput(1,-1.5)(10,15){2}{\line(2,-3){8}}
\put(19,-1.5){\vector(-2,-3){5}}
\put(9,-13.5){\vector(-2,3){6}}
\multiput(10,15)(0,-30){2}{\circle{4}}
\multiput(0,0)(20,0){2}{\circle{4}}
\put(2,0){\line(1,0){16}}
\put(2,0){\vector(1,0){11}}
\put(10,-25){\makebox(0,0){$\Gamma_{\mathfrak{g}}(\overline 2)$}}
\end{picture}
\qquad\qquad
\begin{picture}(20,40)(0,-27)
\thicklines
\put(0,0){\vector(1,0){20}}
\end{picture}
\quad
\begin{picture}(20,40)(0,-27)
\thicklines
\put(39,-1.5){\line(-3,-2){18}}
\put(39,-1.5){\vector(-3,-2){12}}
\put(39,1.5){\line(-3,2){18}}
\put(39,1.5){\vector(-3,2){12}}
\put(1,1.5){\line(3,2){18}}
\put(1,1.5){\vector(3,2){12}}
\put(1,-1.5){\line(3,-2){18}}
\put(1,-1.5){\vector(3,-2){12}}
\put(20,-13){\line(0,1){11}}
\put(20,-13){\vector(0,1){10}}
\put(20,2){\line(0,1){11}}
\put(20,13){\vector(0,-1){10}}
\put(20,15){\circle{4}}
\put(20,-15){\circle{4}}
\put(0,0){\circle{4}}
\put(20,0){\circle*{4}}
\put(40,0){\circle{4}}
\put(2,0){\line(1,0){16}}
\put(18,0){\vector(-1,0){11}}
\put(22,0){\line(1,0){16}}
\put(22,0){\vector(1,0){11}}
\put(20,-25){\makebox(0,0){$\Gamma_{\mathfrak{g}}(2\overline 2)$}}
\end{picture}
\end{center}
\vspace{-.1in}
\caption{The amalgamation $(\Gamma_{\mathfrak{g}}(2)
,\Gamma_{\mathfrak{g}}(\overline 2))\mapsto \Gamma_{\mathfrak{g}}(2\overline 2)$
for $\mathfrak{g}=A_3$}
\label{fig:amalgblock3}
\end{figure}

The resulting graph is called the \emph{Dynkin quiver} $\Gamma_{\mathfrak g}(\mathbf i)$.
Therefore, a Dynkin quiver is associated to every double word $\mathbf i$.
In particular, every elementary Dynkin quiver is a Dynkin quiver, and the elementary Dynkin quiver
$\Gamma_{\mathfrak g}(\mathbf 1)$ does not affect a gluing.

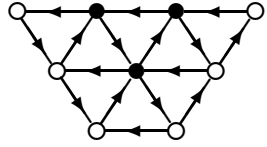
\begin{figure}[htbp]
\begin{center}
\setlength{\unitlength}{1.5pt}
\begin{picture}(20,40)(0,-27)
\thicklines
\put(1,1.7){\line(2,3){8}}
\put(1,1.7){\vector(2,3){6}}
\put(21,1.7){\line(2,3){8}}
\put(21,1.7){\vector(2,3){6}}
\put(41,1.7){\line(2,3){7.5}}
\put(41,1.7){\vector(2,3){6}}
\put(11,-13.3){\line(2,3){7.5}}
\put(11,-13.3){\vector(2,3){6}}
\put(31,-13.3){\line(2,3){7.5}}
\put(31,-13.3){\vector(2,3){6}}
\put(-9,13.3){\line(2,-3){7.5}}
\put(-9,13.3){\vector(2,-3){6}}
\put(11,13.3){\line(2,-3){7.5}}
\put(11,13.3){\vector(2,-3){6}}
\put(31,13.3){\line(2,-3){7.5}}
\put(31,13.3){\vector(2,-3){6}}
\put(1,-1.7){\line(2,-3){7.5}}
\put(1,-1.7){\vector(2,-3){6}}
\put(21,-1.7){\line(2,-3){7.5}}
\put(21,-1.7){\vector(2,-3){6}}
\put(10,15){\circle*{4}}
\put(30,15){\circle*{4}}
\put(-10,15){\circle{4}}
\put(50,15){\circle{4}}
\put(10,-15){\circle{4}}
\put(30,-15){\circle{4}}
\put(0,0){\circle{4}}
\put(20,0){\circle*{4}}
\put(40,0){\circle{4}}
\put(8,15){\vector(-1,0){11}}
\put(8,15){\line(-1,0){16}}
\put(28,15){\vector(-1,0){11}}
\put(28,15){\line(-1,0){16}}
\put(48,15){\vector(-1,0){11}}
\put(48,15){\line(-1,0){16}}
\put(18,0){\vector(-1,0){11}}
\put(18,0){\line(-1,0){16}}
\put(38,0){\vector(-1,0){11}}
\put(38,0){\line(-1,0){16}}
\put(28,-15){\vector(-1,0){11}}
\put(28,-15){\line(-1,0){16}}
\end{picture}
\end{center}
\vspace{-.1in}
\caption{The quiver $\Gamma_{\mathfrak g}(123121)$
for $\mathfrak{g}=A_3$}
\label{fig:amalgblock2}
\end{figure}

It is easy to see that the amalgamation is associative. In particular, the Dynkin
quivers $\Gamma_{\mathfrak g}(\mathbf i)$ and $\Gamma_{\mathfrak g}(\mathbf j)$
can be amalgamated to obtained the Dynkin quiver $\Gamma_{\mathfrak g}(\mathbf{ij})$.
Let us also remark that for every vertex $i$ of a Dynkin quiver $\Gamma$,
there exists $k\in[1,l]$ such that $i$ is a $k$-vertex. Such a $k$ is called the
\emph{vertex-type} of $i$ and is denoted $k(i)$. It is clear that the vertex-type
remains unchanged by the amalgamation. Moreover, for every $j\in [1,l]$, we denote
$N^j(\mathbf i)$ the number of vertices in $\Gamma_{\mathfrak g}(\mathbf i)$ whose
vertex-type is $j$. Stated otherwise, $N^j(\mathbf i)$ is the number of times
the letter $j$ or $\overline j$ appears in the double word $\mathbf i$.

\subsubsection{Seeds associated to double words}\label{section:seeddoublewords}
We derive the seeds constructed by Fock and Goncharov from
the previous Dynkin quivers in the following way.
Let $\mathbf i$ be a double word, and $\Gamma$ be the associated Dynkin quiver.
Let $\overline{B}(\mathbf i)=(\overline{b}_{ij})$ denote the skew-symmetric
matrix whose rows and columns are labeled
by the vertices of~$\Gamma$, and whose entry $\overline{b}_{ij}$
is equal to the number of edges going from $i$ to~$j$
minus the number of edges going from $j$ to~$i$.
The skew-symmetrizable matrix $B(\Gamma)=({b}_{ij})$ is obtained from
$\overline{B}(\mathbf i)$ by the following skew-symmetrizing formula,
involving the vertex-types $k(i)$ and $k(j)$, of $i$ and $j$:
$$d_{k(i)}{b}_{ij}=\overline{b}_{ij}=-\overline{b}_{ji}=-{b}_{ji}d_{k(j)}\ ,$$
where $d_1,\dots d_l$ are the set of non-zero natural numbers that symmetrize the
Cartan matrix.
Let $\mathbf i$ be a double word and $\Gamma_{\mathfrak g}(\mathbf i)$ be
its associated Dynkin quiver. The seed ${\mathbf I}(\mathbf i)=(I(\mathbf i),
I_0(\mathbf i), \varepsilon(\mathbf i), d(\mathbf i))$ is defined in the following
way.
\begin{itemize}
\item
The set $I(\mathbf i)$ is the set of vertices of $\Gamma_{\mathfrak g}(\mathbf i)$,
with a partial order induced by the $k$-orders on the set of $k$-vertices, the sets
$I_0^{\mathfrak R}(\mathbf i),I_0^{\mathfrak L}(\mathbf i)$ are respectively the set
of right outlets and left outlets of $\Gamma_{\mathfrak g}(\mathbf i)$,
and $I_0(\mathbf i)$ is the set of outlets.
Stated otherwise, the set $I(\mathbf i)$
(resp. $I_0^{\mathfrak L}(\mathbf i),I_0^{\mathfrak R}(\mathbf i)$ and $I_0(\mathbf i)$) is the set of all ordered
pairs $(^{j}_{k})$ such that $j\in[1,l]$, and $0\leq k \leq N^{j}({\mathbf i})$ (resp.
$k=0$, $k=N^{j}({\mathbf i})$, and $k\in\{0,N^{j}({\mathbf i})\}$), where $N^{j}({\mathbf i})$
is the number of times the letter $j$ or $\overline j$ appears in $\mathbf i$.
\item
The matrix $\varepsilon(\mathbf i)$ is given by a normalisation of $B(\mathbf i)$:
\begin{equation}\label{equ:b}
\varepsilon(\mathbf i)_{kl}=\left\{
\begin{array}{cc}
\displaystyle\frac{B(\mathbf i)_{kl}}{2}& \mbox{for } k\in I_0(\mathbf i)\mbox{ and } l\in I_0(\mathbf i);\\
\\
B(\mathbf i)_{kl}& \mbox{otherwise}.
\end{array}
\right.
\end{equation}
\item
The multiplier $d(\mathbf i)$ is given by associating to every vertex $j$ the Cartan symetrizer
of its type-vertex, that is:
$$d(\mathbf i)_{j}=d_{k(j)}\ .$$
\end{itemize}

It is easy to translate the amalgamation procedure at the level of seeds:
for every double words $\mathbf i,\mathbf j$,
the \emph{amalgamated seed} $(I(\mathbf i\mathbf j),I_0(\mathbf i\mathbf j),
\varepsilon({\mathbf i\mathbf j}),d({\mathbf i\mathbf j}))$ is defined
in the following way \cite{FGcluster}. The elements of the set
$d({\mathbf i\mathbf j})$ are equal to the corresponding Cartan symmetrizer
as above and the matrix $\varepsilon({\mathbf i\mathbf j})$ is given by

\begin{equation}\label{equ:amaleps}
\varepsilon({\mathbf i\mathbf j})_{\binom i{k}\binom j{l}}=\left\{
\begin{array}{llll}
{\varepsilon}({\mathbf{i}})_{\binom i{k}\binom j{l}}
&\mbox{if}\ k<N^i(\mathbf{i})\mbox{ and }l<N^j(\mathbf{i})\ ;\\
{\varepsilon}({\mathbf{i}})_{\binom i{k}\binom j{l}}
+ {\varepsilon}({\mathbf{j}})_{\binom i{0}\binom j{0}}
&\mbox{if}\ k=N^i(\mathbf{i})\mbox{ and }l=N^j(\mathbf{i})\ ;\\
{\varepsilon}({\mathbf{j}})_{\binom i{k-N^i(\mathbf{i})}
\binom j{l-N^j(\mathbf{i})}}&\mbox{if}\ k>N^i(\mathbf{i})\mbox{ and }l>N^j(\mathbf{i})\ ;\\
0&\mbox{otherwise}\ .
\end{array}
\right.
\end{equation}
In particular, for every $i\in[1,l]$ and $\mathbf i\in\{i,\overline{i}\}$,
the matrices $\varepsilon(\overline i)$ and $\varepsilon(i)$
have their entries labeled by the elements of $I(i)=I(\overline i)$ and are given by
the following equalities and zero otherwise.
\begin{equation}\label{equ:eps}
\begin{array}{cccccc}
\varepsilon(i)_{\binom i{1}\binom j{0}}=\displaystyle\frac{a_{ij}}{2}
=-\varepsilon(i)_{\binom i{0}\binom j{0}},
&&
\varepsilon(\overline i)_{\binom i{1}\binom j{0}}=-\displaystyle\frac{a_{ij}}{2}
=-\varepsilon(\overline i)_{\binom i{0}\binom j{0}}\ .
\end{array}
\end{equation}

\subsection{{Cluster $\mathcal X$-varieties} related to $(G,\pi_G)$}
\label{subsection:FGevaluattion}
We attach a seed $\mathcal X$-torus to each of the previous seeds.
We use them to describe the combinatorics underlying the Poisson geometry
of double Bruhat cells, as in  \cite{FGcluster}.
Let us recall the basis of $\mathfrak h$ given by (\ref{equ:hbasis}) and
the related group elements $H^i(x)$ given by (\ref{equ:defphi2}).
For every $\mathbf i\in\{\mathbf 1,i,\overline i\}$ we
denote ${\mathcal X}_{\mathbf i}$
the seed $\mathcal X$-torus associated to the elementary seed
$(I(\mathbf i),I_0(\mathbf i),\varepsilon({\mathbf i}),d({\mathbf i}))$ and
$\ev_{\mathbf i}:{\mathcal X}_{\mathbf i}\rightarrow G$ the related
evaluation map given by

\begin{align*}
\ev_{\mathbf 1}&\colon \mathbb C_{\neq 0}^{l} \longrightarrow G\hspace{-1,5cm}&&\colon\hspace{-1,5cm}&
\left(x_{\binom{1}{0}},\dots,x_{\binom{i}{0}},\dots,x_{\binom{l}{0}}\right)& \longmapsto \displaystyle\prod_{j}H^j(x_{\binom{j}{0}})\ ,\\
\ev_{i}&\colon \mathbb C_{\neq 0}^{l+1} \longrightarrow G\hspace{+1,5cm}&&\colon\hspace{-1,5cm}&
\left(x_{\binom{1}{0}},\dots,x_{\binom{i}{0}},x_{\binom{i}{1}},x_{\binom{i+1}{0}},\dots,x_{\binom{l}{0}}\right)
 & \longmapsto \displaystyle\prod_{j}H^j(x_{\binom{j}{0}})E^iH^i(x_{\binom{i}{1}})\ ,\\
\ev_{\overline{i}}&\colon \mathbb C_{\neq 0}^{l+1} \longrightarrow G\hspace{-0cm}&&\colon\hspace{-1,5cm} &
\left(x_{\binom{1}{0}},\dots,x_{\binom{i}{0}},x_{\binom{i}{1}},x_{\binom{i+1}{0}},\dots,x_{\binom{l}{0}}\right)
 & \longmapsto \displaystyle\prod_{j}H^j(x_{\binom{j}{0}})F^iH^i(x_{\binom{i}{1}})\ .
\end{align*}

\begin{prop}\cite[Proposition 3.11]{FGcluster}\label{prop:elem}
For every $i\in [1,l]$, $\mathbf i\in\{\mathbf 1,i,\overline{i}\}$
and $(u,v)\in W\times W$ such that $\mathbf i\in R(u,v)$, the evaluation map
$\ev_{\mathbf i}:{\mathcal X}_{\mathbf i}\longrightarrow (G^{u,v},\pi_G)$
is a Poisson birational isomorphism on a Zariski open set of $G^{u,v}$.
\end{prop}

\begin{lemma}\cite{FGcluster}
The amalgamation procedure induces a Poisson homomorphism
$\mathfrak{m}:{\mathcal X}_{\mathbf i}\times{\mathcal X}_{\mathbf j}
\rightarrow{\mathcal X}_{\mathbf i\mathbf j}$ between the corresponding seed $\mathcal X$-tori
given by
\begin{equation}\label{equ:amal}
\begin{array}{ll}
&{z}_{\binom{i}{k}}  =  \left\{
\begin{array}{lll}
x_{\binom{i}{k}}&\mbox{if}\ 0\leq k<N^i(\mathbf{i})\ ;\\
x_{\binom{i}{k}}y_{\binom{i}{0}}&\mbox{if}\ k=N^i(\mathbf{i})\ ;\\
y_{\binom{i}{k-N^i(\mathbf i)}}&\mbox{if}\ N^i(\mathbf{i})<k\leq N^i(\mathbf{i})+N^i(\mathbf{j})\ ,
\end{array}
\right.\\
\end{array}
\end{equation}
where $x_i$, $y_j$ and $z_k$ are the associated variables.
\end{lemma}
Now, let
$\mathbf i=i_1\dots i_k$ be a double word, ${\mathcal X}_{\mathbf i}$
be the seed $\mathcal X$-torus given by the associated amalgamation
$\mathfrak{m}:{\mathcal X}_{{i_1}}\times \dots\times{\mathcal X}_{{i_k}}\longrightarrow {\mathcal X}_{\mathbf i}$,
and $\mathbf z$ be the amalgamated variable $\mathfrak{m}(\mathbf{x_1},\dots,\mathbf{x_k})$.
We define the \emph{evaluation map}
\begin{equation}\label{equ:defev}
\begin{array}{ccc}\ev_{\mathbf i}:{\mathcal X}_{\mathbf i}\rightarrow G:
\mathbf z\mapsto \ev_{i_1}(\mathbf{x_1})\dots \ev_{i_k}(\mathbf{x_k})
&\mbox{where}&\mathbf z=\mathfrak{m}(\mathbf{x_1},\dots,\mathbf{x_k})\ .
\end{array}
\end{equation}
Using the multiplicative property of the Poisson-Lie group $(G,\pi_G)$,
we see that this evaluation is also a Poisson map.
So Proposition \ref{prop:elem} leads to the Poisson statement of
the following result.

\begin{thm}[\cite{FGcluster}]\label{thm:evG}For any $u,v\in W$ and $\mathbf i\in R(u,v)$ the map
$\ev_{\mathbf i}:{\mathcal X}_{\mathbf i}\rightarrow (G^{u,v},\pi_G)$ is a
Poisson birational isomorphism onto a Zariski open set of the double Bruhat cell
$G^{u,v}$.
\end{thm}

We now introduce cluster transformations in the framework.
Let us say that a double reduced word of length $d$ is \emph{$d$-minimal}
if we can perform a generalized $d$-move on it. Let us call \emph{$n$-minimal}
the double words $i$ and $i\ i$ for every $i\in[1,l]\cup[\overline 1,
\overline l]$. Finally, a double word is said to be \emph{$dn$-minimal}
if it is $d$-minimal or $n$-minimal. {To any two $dn$-minimal double
words $\mathbf i$ and $\mathbf{i'}$ related by a generalized $dn$-move}
$\delta:\mathbf i\mapsto\mathbf{i'}$, we associate a cluster
transformation ${\mu}_{\mathbf{i}\rightarrow\mathbf{i'}}:
{\mathcal X}_{\mathbf i}\rightarrow{\mathcal X}_{\mathbf {i'}}$
in the following way:

{\footnotesize
\begin{equation}\label{equ:elemmut}{\mu}_{\mathbf{i}\rightarrow\mathbf{i'}}=
\begin{cases}
\varsigma_{\binom i{1}}\circ\mu_{\binom i{1}}, &\text{
if $\delta$ is a nil-move;}\\
 \mu_{\binom i{1}}, &\text{
if $\delta$ is a move $i\ \overline i\leftrightarrow\overline i\ i$ or a 3-move;}\\
\mu_{\binom i{1}}\mu_{\binom j{1}}\mu_{\binom i{1}},& \text{
if $\delta$ is a 4-move;}\\
\mu_{\binom j{2}}\mu_{\binom i{1}}\mu_{\binom j{1}}
\mu_{\binom j{2}}\mu_{\binom i{2}}\mu_{\binom j{2}}\mu_{\binom i{1}}\mu_{\binom i{2}}
\mu_{\binom j{1}}\mu_{\binom j{2}},&  \text{if $\delta$ is a 6-move;}\\
 \text{the identity map}&\text{otherwise,}
\end{cases}
\end{equation}
}
where, as in \cite{FGcluster}, we have denoted an expression
$\mu_{\mu_{i}(j)}\mu_{i}$ by $\mu_j\mu_k$, an expression
$\mu_{\mu_{\mu_{i}(j)}\mu_{i}(k)}\mu_{\mu_{i}(j)}\mu_{i}$ by
$\mu_k\mu_j\mu_i$, and so on.

Since mutations commute with amalgamation, we may extend these
definitions to any two double words $\mathbf i,\mathbf{i'}\in D(u,v)$
related by a generalized $dn$-move. Finally, if  $\mathbf i,\mathbf j$
are {double words} linked by a sequence $\delta_{\mathbf i\to\mathbf j}$
of generalized $dn$-moves and  $\mathbf{i}\to\mathbf{i_1}\rightarrow
\dots\rightarrow\mathbf{i_{n-1}}\rightarrow\mathbf{j}$ is the associated
chain of elements, we define the cluster transformation
${\mu}_{\mathbf{i}\rightarrow\mathbf{j}}$ as the composition
${\mu}_{\mathbf{i_{n-1}}\rightarrow\mathbf j}\circ \dots\circ
{\mu}_{\mathbf{i}\rightarrow\mathbf{i_1}}$. The following result is
easily derived from Theorem~\ref{thm:Tits}.

\begin{lemma}\label{lemma:transitivedoubleword}A double reduced word $\mathbf j\in R(u,v)$ can be obtained
from a double reduced word $\mathbf i\in R(u',v')$ by a sequence of generalized
$d$-moves $\delta_{\mathbf i\to\mathbf j}$ if and only if the equalities $u'=u$
and $v'=v$ are satisfied.
\end{lemma}

Let $u,v\in W$ and $\mathbf i,\mathbf j\in R(u,v)$.
Because the birational Poisson isomorphism $\mu_{\mathbf i\to\mathbf j}:
{\mathcal X}_{\mathbf i}\to{\mathcal X}_{\mathbf j}$, associated to a sequence
$\delta_{\mathbf i\to\mathbf j}$ of generalized $d$-moves, is a cluster
transformation, we will denote ${\mathcal X}^{u,v}$ the cluster
$\mathcal X$-variety associated to the set $R(u,v)$. Stated otherwise, to any double reduced
word $\mathbf i\in R(u,v)$ corresponds a local chart $({\mathcal X}_{\mathbf i},\varphi_{\mathbf i})$
in the cluster $\mathcal X$-variety ${\mathcal X}^{u,v}$, and the cluster transformation
$\mu_{\mathbf i\to\mathbf j}$ is the transition map between the local charts
$({\mathcal X}_{\mathbf i},\varphi_{\mathbf i})$ and
$({\mathcal X}_{\mathbf j},\varphi_{\mathbf j})$ for any $\mathbf j\in R(u,v)$.
Let us also denote $\iota_{\mathbf i}:R(u,v)\to\mathbf i$ every time we choose
the element $\mathbf i\in R(u,v)$. The diagrams
in Figure \ref{fig:clustervariety0} are therefore commutative.

\begin{figure}[htbp]
\begin{center}
\setlength{\unitlength}{1.5pt}
\qquad
\xymatrix{
&{R(u,v)}\ar@/^0pc/[ld]_{\iota_{\mathbf i}}\ar@/^0pc/[rd]^{\iota_{\mathbf j}}&\\
{\mathbf i}\ar@/_0pc/[rr]_{\delta_{\mathbf i\to\mathbf j}}&&{\mathbf {j}}
}
\qquad
\xymatrix{
&{{\mathcal X}^{u,v}}\ar@/^0pc/[ld]_{\varphi_{\mathbf i}}\ar@/^0pc/[rd]^{\varphi_{\mathbf j}}&\\
{\mathcal X}_{\mathbf i}\ar@/_0pc/[rr]_{\mu_{\mathbf i\to\mathbf j}}&&{\mathcal X}_{\mathbf {j}}
}
\end{center}
\vspace{-.1in}
\caption{The set $R(u,v)$ and the cluster $\mathcal X$-variety ${\mathcal X}^{u,v}$}
\label{fig:clustervariety0}
\end{figure}
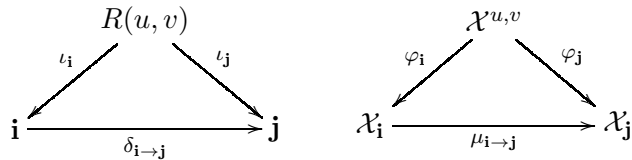

We finally give the way to attach the cluster $\mathcal X$-variety
$\mathcal X^{u,v}$ to the double Bruhat cell $(G^{u,v},\pi_G)$, for
every $u,v\in W$.

\begin{lemma}[\cite{FGcluster}]\label{lemma:fg}For any $u,v\in W$
and $\mathbf i,\mathbf j\in D(u,v)$ such that $\mathbf j$ is obtained from
$\mathbf i$ by a sequence of generalized $dn$-moves, we have
$\ev_{\mathbf{i}}=\ev_{\mathbf{j}}\circ{\mu}_{\mathbf{i}\rightarrow\mathbf j}$.
\end{lemma}

\begin{thm}[\cite{FGcluster}]\label{fg}For any $u,v\in W$
and $\mathbf i,\mathbf j\in R(u,v)$, we have
$\ev_{\mathbf{i}}=\ev_{\mathbf{j}}\circ{\mu}_{\mathbf{i}\rightarrow\mathbf j}$.
\end{thm}

The cluster $\mathcal X$-variety ${\mathcal X}^{u,v}$ is therefore
attached to the double Bruhat cell $(G^{u,v},\pi_G)$ for every $u,v\in W$.
Let us denote ${\mathcal X}_{.}$ the application which associate to any
double word $\mathbf i$ the corresponding seed $\mathcal X$-torus
${\mathcal X}_{\mathbf i}$. We then sum-up Theorem \ref{thm:evG}
and Theorem \ref{fg} by abusively
saying that there exists a Poisson map
$\ev^{u,v}:{\mathcal X}^{u,v}\to(G^{u,v},\pi_{G})$.
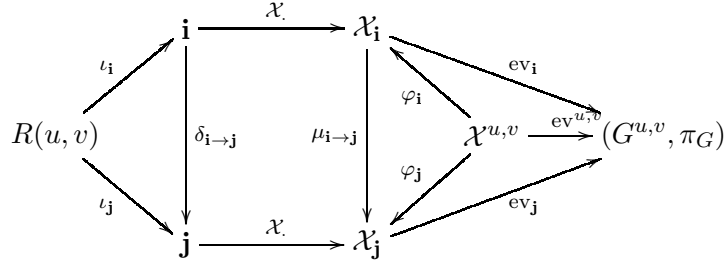
\begin{figure}[htbp]
\begin{center}
\setlength{\unitlength}{1.5pt}
\qquad\qquad
\xymatrix{
&\mathbf i\ar@/_0pc/[dd]^{\delta_{\mathbf i\to\mathbf j}}\ar@/_0pc/[rr]^{{\mathcal X}_.}&&{\mathcal X}_{\mathbf i}
\ar@/_0pc/[dd]_{\mu_{\mathbf i\to\mathbf j}}\ar@/^0pc/[rrd]^{\ev_{\mathbf i}}&&\\
R(u,v)\ar@/^0pc/[ur]^{\iota_{\mathbf i}}\ar@/^0pc/[rd]_{\iota_{\mathbf j}}&&&&{{\mathcal X}^{u,v}}
\ar@/^0pc/[lu]^{\varphi_{\mathbf i}}\ar@/^0pc/[ld]_{\varphi_{\mathbf j}}
\ar@/^0pc/[r]^{\ev^{u,v}}&(G^{u,v},\pi_G)\\
&\mathbf j\ar@/_0pc/[rr]^{{\mathcal X}_.}&&{\mathcal X}_{\mathbf {j}}\ar@/^0pc/[rru]_{\ev_{\mathbf j}}&&
}
\end{center}
\vspace{-.1in}
\caption{The cluster $\mathcal X$-variety ${\mathcal X}^{u,v}$ associated
to $(G^{u,v},\pi_G)$}
\label{fig:clustervariety}
\end{figure}

\section{Truncation maps and cluster $\mathcal X$-varieties related to $(G, \pi_*)$}
\label{subsection:evdual}
We introduce new evaluation maps and new seeds related to double reduced
words to state analogs of Theorem \ref{thm:evG} and Theorem \ref{fg} for
the dual Poisson-Lie group $G^*\subset(G,\pi_*)$.
Although the result in this section will be strongly generalized
in Section \ref{section:Loop}, it is the occasion to introduce truncation
maps and to give a flavor of what will be later called twisted evaluations
and its related combinatorics, without using the machinery of
generalized cluster transformations and saltations that we start to develop
in Section \ref{section:tropical}.

\subsection{Double reduced Bruhat cells and reduced evaluations}
\label{subsection:reducedevaluation}
According to \cite[Section 4.3]{BZtensor}, let $L^{u,v}$ be the
\emph{reduced double Bruhat cell} associated to every $u,v\in W$, that is:
the quotient of double Bruhat cell $G^{u,v}$ by the $H$-right multiplication:
\begin{equation}\label{equ:double reduced}L^{u,v}=G^{u,v}/H.
\end{equation}
We are going to slightly modify the cluster $\mathcal X$-varieties previously
constructed in order to evaluate these double reduced Bruhat cells.
For every double word $\mathbf i$, let $\varsigma_{\mathfrak{R},\mathbf i}$
be the erasing map associated to the set $I^{\mathfrak R}_0(\mathbf i)$ of right
outlets, defined as the product over the set $I_0^{\mathfrak R}(\mathbf i)$ of
the erasing maps $\varsigma_j$ given by Definition \ref{def:erasing}.
\begin{equation}\label{equ:proderase}
\varsigma_{\mathfrak{R},\mathbf i}=\displaystyle\prod_{j\in I^{\mathfrak R}_0(\mathbf i)}
\varsigma_j\ .
\end{equation}
We denote $\mathbf i^{\red}$ the image of
the seed $\mathbf I(\mathbf{i})$ by $\varsigma_{\mathfrak{R},\mathbf i}$, and
$\mathcal{X}_{\mathbf i}^{\red}$ the seed $\mathcal{X}$-torus
associated to the seed $\mathbf i^{\red}$. Therefore, we have
$\varsigma_{\mathfrak{R},\mathbf i}:\mathcal{X}_{\mathbf i}\to
\mathcal{X}_{\mathbf i}^{\red}\ .$
Every cluster transformation ${\mu}_{\mathbf{i}\rightarrow\mathbf j}:
{\mathcal X}_{\mathbf i}\to{\mathcal X}_{\mathbf j}$ canonically leads
to a cluster transformation between the associated seed $\mathcal X$-tori
${\mu}_{\mathbf{i}\rightarrow\mathbf j}^{\red}:
{\mathcal X}_{\mathbf i}^{\red}\to{\mathcal X}_{\mathbf j}^{\red}$ by the relation
$${\mu}_{\mathbf{i}\rightarrow\mathbf j}^{\red}\circ\varsigma_{\mathfrak{R},\mathbf i}
=\varsigma_{\mathfrak{R},\mathbf j}\circ{\mu}_{\mathbf{i}\rightarrow\mathbf j}\ .$$

\begin{definition}Let $\mathbf i$ be a double reduced word.
It is clear that the map
defined on $\mathcal{X}_{\mathbf i}$ and given by
$$\mathbf{x}\mapsto\ev_{\mathbf i}(\mathbf x)\displaystyle
\prod_{j\in[1,l]}H^j(x_{\binom{j}{N^j(\mathbf i)}}^{-1})$$
doesn't depend on $x_k$ when $k\in I^{\mathfrak R}_0(\mathbf i)$, thus we get an evaluation
$\ev_{\mathbf i}^{\red}:{\mathcal X}_{\mathbf i}^{\red}\rightarrow G$ called a
\emph{reduced evaluation}. Stated in a rougher way, the reduced evaluation associated to
the double word $\mathbf i$ is obtained from the evaluation map $\ev_{\mathbf i}$
by setting the cluster variables $x_j$ to $1$ (or forgetting the related Cartan
element $H^i(x_j)$) for every $j\in I_0^{\mathfrak R}(\mathbf i)$.
\end{definition}

When the double word $\mathbf i$ is a double reduced word, Theorem
\ref{thm:evG} and Theorem \ref{fg} are easily adapted to reduced double Bruhat
cells using reduced evaluations. We get the following result.

\begin{cor}\label{cor:evredmut}\label{cor:redmut}For any $u,v\in W$ and $\mathbf i\in R(u,v)$ the map
$\ev_{\mathbf i}^{\red}:{\mathcal X}_{\mathbf i}^{\red}\rightarrow (L^{u,v},\pi_G)$ is a
Poisson birational isomorphism onto a Zariski open set of the double Bruhat cell
$L^{u,v}$.
And for any $u,v\in W$ and $\mathbf i,
\mathbf j\in R(u,v)$, we have $\ev_{\mathbf{i}}^{\red}=
\ev_{\mathbf{j}}^{\red}\circ{\mu}_{\mathbf{i}\rightarrow\mathbf j}^{\red}$.
\end{cor}

\begin{rem}This corollary remains valid even if the Lie group $G$ is not
of adjoint type but simply connected.
\end{rem}

\subsection{Truncation maps on cluster $\mathcal X$-varieties} The cluster
$\mathcal X$-variety we are going to associate to the dual Poisson-Lie group
$(BB_-,\pi_*)$ in the next subsection can be easily obtained from the cluster
$\mathcal X$-variety ${\mathcal X}^{w_0,w_0}$ of Section \ref{section:ClusterG}
by the notions of truncation map and truncated cluster $\mathcal X$-varieties we
are going to introduce now. The underlying idea is to force the apparition of
the Casimir of $(BB_-,\pi_*)$. We start by giving the general setting.

\begin{definition}\label{def:tronc}\label{def:truncatedtorus}Let $\mathbf{I}=(I,I_0,\varepsilon,d)$
be a seed and $J\subset I$. The \emph{truncation map associated to $J$}
is a map $\mathfrak{t}_J:\mathbf I\longrightarrow \mathbf I_J$ such that the seed
$\mathbf I_J=(I,{I}_0,{\varepsilon}',d)$ is given by:
\begin{equation}\label{equ:defJ2}
{{\varepsilon}'}_{ij}=\left\{
\begin{array}{ll}
\varepsilon_{ij}&\text{if } i,j\in I\backslash J;\\
0& \text{otherwise.}
\end{array}
\right.
\end{equation}
To any finite set $J$, let us denote $\mathcal X_J^0$ the seed $\mathcal X$-torus
associated to a seed $(J,J,0,d)$. Let $\mathbf{I}=(I,I_0,\varepsilon,d)$ be a seed
such that $J\subset I$, and $\mathbf{I'}=(I,I_0,\varepsilon',d)$ be the image of
$\mathbf{I}$ by $\mathfrak{t}_J$. To every $\mathbf t\in{\mathcal X}_{J}$,
we associate the subtorus $\mathcal X_{\mathbf I'}(t)\subset\mathcal X_{\mathbf I'}$
constituted of elements $\mathbf x\in\mathcal X_{\mathbf I'}$ such that $x_i=t_i$
for every $i\in J$. It is a Poisson subtorus because of the formula (\ref{equ:defJ2}).
The map $\mathfrak{t}_J$ induces an homomorphism $\mathfrak{t}_{J(t)}:{\mathcal X}_{\mathbf I}
\to{\mathcal X}_{\mathbf{I'}}(t)$ associated to any $t\in\mathcal X_J^0$ and given by
$$x_{\mathfrak{t}_{J(t)}(i)}=\left\{
\begin{array}{ll}
x_{i}&\text{if } i\in I\backslash J\ ;\\
t_i& \text{if } i\in J\ .
\end{array}
\right.
$$
In particular, the \emph{trivial truncated map}, associated to
the empty set $J=\emptyset$, is the identity map.
Now, let $\mathbf I=(I,I_0,\varepsilon,d)$ be a seed and $J$ be
a subset of $I$.
To any cluster transformation $\phi_{\mathbf I\to\mathbf J}:
{\mathcal X}_{\mathbf I}\to{\mathcal X}_{\mathbf J}$, we associate the
following Poisson birational isomorphism $\phi_{\mathbf I_J\to\mathbf J_J}:
{\mathcal X}_{\mathbf I_J}\to{\mathcal X}_{\mathbf J_J}$, called
\emph{truncated cluster transformation} and given by
$$x_{\phi_{\mathbf I_J\to\mathbf J_J}(i)}=\left\{
\begin{array}{ll}
x_{\phi_{\mathbf I\to\mathbf J}(i)}&\text{if } i\in I\backslash J\ ;\\
x_i& \text{if } i\in J\ .
\end{array}
\right.
$$
It is clear that for every $t\in\mathcal X_J^0$ this map admits a restriction
$\phi_{\mathbf I_J\to\mathbf J_J}:
{\mathcal X}_{\mathbf I_J}(t)\to{\mathcal X}_{\mathbf J_J}(t)$
which is also a birational Poisson isomorphism.
\end{definition}

We would like to define truncation maps at the level of cluster
$\mathcal X$-varieties. A sufficient condition is given by the
following immediate result.

\begin{lemma}\label{lemma:obstrunc}Let $\mathbf I=(I,I_0,\varepsilon,d)$ be a seed and
$J\subset I_0$. The following equality is satisfied for every cluster
transformation $\phi_{\mathbf I\to\mathbf J}:{\mathcal X}_{\mathbf I}
\to{\mathcal X}_{\mathbf J}$ and every $t\in{\mathcal X}_J^0$.
\begin{equation}\label{equ:truncmut}\phi_{\mathbf I_J\to\mathbf J_J}
\circ\mathfrak{t}_{J(t)}=\mathfrak{t}_{\phi_{\mathbf I\to\mathbf J}
(J)(t)}\circ\phi_{\mathbf I\to\mathbf J}\ .
\end{equation}
\end{lemma}

\begin{rem}The formula (\ref{equ:truncmut}) is not necessarily true if
the condition $J\subset I_0$ is omitted. Indeed, consider the seed
$\mathbf I(i)$ for any $i\in[1,l]$, and the set $J=\{(^i_1)\}$, with
the cluster transformation $\mu_{\binom{i}{1}}$.
In fact, a generalization of the formula (\ref{equ:truncmut}) will be
the starting point for the definition of the Poisson birational isomorphisms
called saltations and given in Subsection \ref{subsection:saltation}.
\end{rem}

\begin{definition}Let $\mathbf I=(I,I_0,\varepsilon,d)$ be a seed, $J$ be a subset of $I$,
and $\mathbf I'$ be the truncated seed associated to the seed $\mathbf I$ and the set $J$.
The \emph{truncated cluster} ${\mathcal X}$-$\emph{variety}$ ${\mathcal X}_{|\mathbf I'|}$
of the cluster ${\mathcal X}$-variety ${\mathcal X}_{|\mathbf I|}$
is obtained by taking every seed $\mathcal X$-tori obtained from
${\mathcal X}_{\mathbf I'}$ by cluster transformations, and gluing them as usual.
Moreover, if we denote ${\mathcal X}_{|\mathbf I'|}(t)$ the cluster $\mathcal X$-variety
obtained from every seed $\mathcal X$-torus ${\mathcal X}_{\mathbf I'}(t)\subset
{\mathcal X}_{\mathbf I'}$, we have the following Poisson stratification, because
of (\ref{equ:defJ2}).
$${\mathcal X}_{|\mathbf I'|}=\displaystyle\bigcup_{t\in\mathcal X_J^0}
{\mathcal X}_{|\mathbf I'|}(t)\ .$$
\end{definition}

The truncated map $\mathfrak{t}_{J(t)}$ is therefore well-defined at the
level of cluster $\mathcal X$-varieties for every fixed $t\in\mathcal X_J^0$
if $J$ is a subset of $I_0$. It is denoted $\mathfrak{t}_{J(t)}:
{\mathcal X}_{|\mathbf I|}\to{\mathcal X}_{|\mathfrak{t}_{J}(\mathbf I)
|}(t)$.


We now focus on particular truncations on cluster $\mathcal X$-varieties associated
to double words.
For every double word $\mathbf i$, let $[\mathbf i]_{\mathfrak R}$ be the image of the seed
$\mathbf I(\mathbf i)$ by the \emph{right truncation map} $\mathfrak{t}_{\mathbf i_{\mathfrak R}}$ associated
to the set of right outlets $I^{\mathfrak R}_0(\mathbf i)$ of $I(\mathbf i)$. Stated otherwise, for every double
word $\mathbf i$, the seed $[\mathbf i]_{\mathfrak R}=(I(\mathbf i),I_0(\mathbf i),
\eta({\mathbf i}),d({\mathbf i}))$ is the seed defined by the values
\begin{equation}\label{equ:defJ}\eta(\mathbf i)_{ij}=\left\{
\begin{array}{ll}
\varepsilon(\mathbf i)_{ij}&\text{if } i,j\in I(\mathbf i)\backslash I^{\mathfrak R}_0(\mathbf i);\\
0& \text{otherwise.}
\end{array}
\right.
\end{equation}
An example of right truncation map is given by Figure \ref{fig:troncation}.
\begin{figure}[htbp]
\begin{center}
\setlength{\unitlength}{1.5pt}
\begin{picture}(20,48)(0,-27)
\thicklines
\put(1,1.7){\line(2,3){8}}
\put(1,1.7){\vector(2,3){6}}
\put(21,1.7){\line(2,3){8}}
\put(21,1.7){\vector(2,3){6}}
\put(41,1.7){\line(2,3){7.5}}
\put(41,1.7){\vector(2,3){6}}
\put(11,-13.3){\line(2,3){7.5}}
\put(11,-13.3){\vector(2,3){6}}
\put(31,-13.3){\line(2,3){7.5}}
\put(31,-13.3){\vector(2,3){6}}
\put(-9,13.3){\line(2,-3){7.5}}
\put(-9,13.3){\vector(2,-3){6}}
\put(11,13.3){\line(2,-3){7.5}}
\put(11,13.3){\vector(2,-3){6}}
\put(31,13.3){\line(2,-3){7.5}}
\put(31,13.3){\vector(2,-3){6}}
\put(1,-1.7){\line(2,-3){7.5}}
\put(1,-1.7){\vector(2,-3){6}}
\put(21,-1.7){\line(2,-3){7.5}}
\put(21,-1.7){\vector(2,-3){6}}
\put(10,15){\circle*{4}}
\put(30,15){\circle*{4}}
\put(-10,15){\circle{4}}
\put(50,15){\circle{4}}
\put(10,-15){\circle{4}}
\put(30,-15){\circle{4}}
\put(0,0){\circle{4}}
\put(20,0){\circle*{4}}
\put(40,0){\circle{4}}
\put(8,15){\vector(-1,0){11}}
\put(8,15){\line(-1,0){16}}
\put(28,15){\vector(-1,0){11}}
\put(28,15){\line(-1,0){16}}
\put(48,15){\vector(-1,0){11}}
\put(48,15){\line(-1,0){16}}
\put(18,0){\vector(-1,0){11}}
\put(18,0){\line(-1,0){16}}
\put(38,0){\vector(-1,0){11}}
\put(38,0){\line(-1,0){16}}
\put(28,-15){\vector(-1,0){11}}
\put(28,-15){\line(-1,0){16}}
\end{picture}
\qquad\qquad\qquad
\begin{picture}(20,48)(0,-27)
\thicklines
\put(0,0){\vector(1,0){20}}
\put(6,5){$\mathfrak t_{\mathbf i_{\mathfrak R}}$}
\end{picture}
\quad\qquad
\begin{picture}(20,48)(0,-27)
\thicklines
\put(1,1.7){\line(2,3){8}}
\put(1,1.7){\vector(2,3){6}}
\put(21,1.7){\line(2,3){8}}
\put(21,1.7){\vector(2,3){6}}
\put(11,-13.3){\line(2,3){7.5}}
\put(11,-13.3){\vector(2,3){6}}
\put(-9,13.3){\line(2,-3){7.5}}
\put(-9,13.3){\vector(2,-3){6}}
\put(11,13.3){\line(2,-3){7.5}}
\put(11,13.3){\vector(2,-3){6}}
\put(1,-1.7){\line(2,-3){7.5}}
\put(1,-1.7){\vector(2,-3){6}}
\put(10,15){\circle*{4}}
\put(30,15){\circle*{4}}
\put(-10,15){\circle{4}}
\put(50,15){\circle{4}}
\put(10,-15){\circle{4}}
\put(30,-15){\circle{4}}
\put(0,0){\circle{4}}
\put(20,0){\circle*{4}}
\put(40,0){\circle{4}}
\put(8,15){\vector(-1,0){11}}
\put(8,15){\line(-1,0){16}}
\put(28,15){\vector(-1,0){11}}
\put(28,15){\line(-1,0){16}}
\put(18,0){\vector(-1,0){11}}
\put(18,0){\line(-1,0){16}}
\end{picture}
\end{center}
\vspace{-.1in}
\caption{The right truncation map $\mathfrak t_{\mathbf i_{\mathfrak R}}:\Gamma_{A_3}(\mathbf i)
\to\Gamma_{A_3}([\mathbf i]_{\mathfrak R})$
for $\mathbf i=123121$}
\label{fig:troncation}
\end{figure}
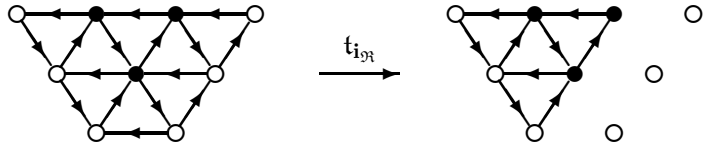
The \emph{right truncated seed $\mathcal X$-torus}
${\mathcal X}_{[\mathbf{i}]_{\mathfrak R}}(t)$ associated to any $t\in H$ is the subset
of ${\mathcal X}_{[\mathbf{i}]_{\mathfrak R}}$ obtained by fixing the
following cluster variables $\mathbf x(\mathfrak R)$ associated to right
outlets
\begin{equation}\label{equ:xR}
\mathbf x(\mathfrak R)=\{x_j|{j\in I^{\mathfrak R}_0(\mathbf i)}\}
\end{equation}
via the equality
\begin{equation}\label{equ:ev1}
\ev_{\mathbf 1}(x_{\binom{1}{N^1(\mathbf i)}},\dots
,x_{\binom{l}{N^l(\mathbf i)}})=t\ .
\end{equation}
Equation (\ref{equ:defJ}) implies that ${\mathcal X}_{\mathbf{i}}(t)$ is a
Poisson submanifold of ${\mathcal X}_{[\mathbf{i}]_{\mathfrak R}}$ and
according to Definition \ref{def:truncatedtorus} the right truncation map
$\mathfrak t_{\mathbf i_{\mathfrak R}}$ on seeds induces
\emph{right truncation maps} on seed $\mathcal X$-tori
$\mathfrak t_{\mathbf i_{\mathfrak R}(t)}:{\mathcal X}_{\mathbf{i}}
\to{\mathcal X}_{[\mathbf{i}]_{\mathfrak R}}(t)$. The following result
is clear.

\begin{lemma}\label{lemma:muR}
Let $\mathbf i$ be a double word. The right truncated seed $\mathcal X$-torus
${\mathcal X}_{[\mathbf{i}]_{\mathfrak R}}$ is Poisson isomorphic to the
direct product ${\mathcal X}_{\mathbf{i}}^{\red}\times{\mathcal X}_{\mathbf{1}}$ of seed
$\mathcal X$-tori. Moreover, the following equality is satisfied for any double word
$\mathbf j$ obtained from $\mathbf i$ by a sequence of generalized $dn$-moves.
\begin{equation}\label{equ:muR}
\begin{array}{rl}
\mu_{[\mathbf i]_{\mathfrak R}\to[\mathbf j]_{\mathfrak R}}:&{\mathcal X}_{[\mathbf{i}]_{\mathfrak R}}
\to{\mathcal X}_{[\mathbf{j}]_{\mathfrak R}}\\
\\
x_{\mu_{[\mathbf i]_{\mathfrak R}\to[\mathbf j]_{\mathfrak R}}(i)}=&\left\{
\begin{array}{lll}
x_{\mu_{\mathbf i\to\mathbf j}(i)}&\mbox{if }i\in I\backslash J\ ;\\
x_i&\mbox{if }i\in J\ .
\end{array}
\right.
\end{array}
\end{equation}
\end{lemma}
It implies in particular that the cluster transformation
$\mu_{[\mathbf i]_{\mathfrak R}\to[\mathbf j]_{\mathfrak R}}$ sends
${\mathcal X}_{[\mathbf{i}]_{\mathfrak R}}(t)$ to ${\mathcal X}_{[\mathbf{j}]_{\mathfrak R}}(t)$ for
every $\mathbf j$ linked to $\mathbf i$ by composition of generalized
$d$-moves. Finally, the cluster $\mathcal X$-variety
${\mathcal X}_{|[\mathbf i]_{\mathfrak R}|}$, constructed from the
seed $\mathcal X$-torus ${\mathcal X}_{[\mathbf i]_{\mathfrak R}}$ associated
to a double word $\mathbf i$ is called a \emph{right truncated cluster
$\mathcal X$-variety}. By Lemma \ref{lemma:obstrunc}, there exists
a \emph{truncation map} $\mathfrak t_{\mathfrak R}$ associating to every
cluster $\mathcal X$-variety ${\mathcal X}_{|\mathbf i|}$ its right truncated
cluster $\mathcal X$-variety ${\mathcal X}_{|[\mathbf i]_{\mathfrak R}|}$:
$$\mathfrak t_{\mathfrak R}:{\mathcal X}_{|\mathbf i|}\mapsto
{\mathcal X}_{|[\mathbf i]_{\mathfrak R}|}\ .$$

\subsection{Dual evaluations and first cluster $\mathcal X$-varieties
related to $(G,\pi_*)$}
We use the previous truncated cluster $\mathcal X$-varieties to
start the geometrical combinatorics of the Poisson manifold $(G,\pi_*)$.
For every element $x\in B_-B$, let $x=[x]_-[x]_0[x]_+$ be
its {Gauss decomposition}, that is: $[x]_{\pm}$ belongs to the unipotent parts $N_{\pm}$
of the respective Borel subgroups $B_{\pm}$ and $[x]_0$ to the Cartan part $H$ of $G$.
We will also consider the following notations.
\begin{equation}\label{equ:Gaussdec}
[x]_{-}[x]_{\geq 0}=[x]_-[x]_0[x]_+=[x]_{\leq 0}[x]_+\ .
\end{equation}

\begin{definition}
Let $v\in W$. For every $\mathbf i\in R(v,w_0)$, we define the \emph{dual evaluation map}
$\ev_{\mathbf i}^{\dual}:{\mathcal X}_{[\mathbf{i}]_{\mathfrak R}}\rightarrow G$
by the formula:
\begin{equation}\label{equ:evdual}
\begin{array}{lllll}
\ev_{\mathbf i}^{\dual}(\mathbf{x})=\ev_{\mathbf i}(\mathbf x)\widehat{w_0}
\ [\ev_{\mathbf i}^{\red}(\mathbf x)\widehat{w_0}]_{\leq 0}^{-1}\ .
\end{array}
\end{equation}
\end{definition}
Dual evaluations will be generalized into twisted evaluations
in Subsection \ref{section:twistedeval}. (We refer to Remark \ref{rem:evdualtwistev}
for more details.) For the moment, let us remember the Poisson stratification
(\ref{bruhatdec}) of $(G,\pi_*)$.

\begin{thm}\label{thm:evdual}For every $v\in W$, $t\in H$ and $\mathbf i\in R(v,w_0)$,
the map $\ev^{\dual}_{\mathbf i}:{\mathcal X}_{[\mathbf{i}]_{\mathfrak R}}(t)\rightarrow (F_{t,w_0v^{-1}},\pi_*)$
is a Poisson birational isomorphism onto a Zariski open set $F_{t,w_0v^{-1}}^0$
of $F_{t,w_0v^{-1}}$.
\end{thm}
Theorem \ref{thm:evdual} will be deduced from Theorem \ref{thm:ev*1}.
The synthesis diagram, given by Figure \ref{fig:clustervariety} and relating,
for every $u,v\in W$, the cluster $\mathcal X$-variety $\mathcal X^{u,v}$
to the double Bruhat cell $G^{u,v}$ can therefore be adapted to get a
cluster $\mathcal X$-variety ${\mathcal X}_{v\leq v}(t)$ associated
to $(F_{t,w_0v^{-1}},\pi_*)$, for every $v\in W$ and every $t\in H$.
It is illustrated in Figure \ref{fig:clustervarietydual1}.
(The weird terminology for "${\mathcal X}_{v\leq v}(t)$" and "$\ev_{v\leq v}$"
will be explained in Subsection \ref{subsection:taucombdual}.)
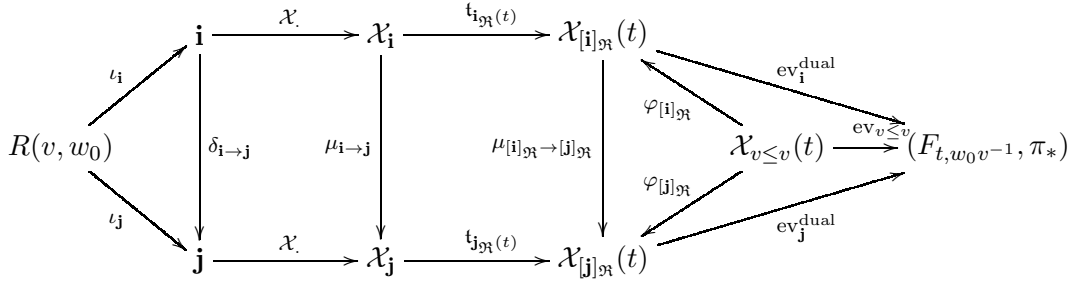
\begin{figure}[htbp]
\begin{center}
\setlength{\unitlength}{1.5pt}
\qquad\qquad
\xymatrix{
&\mathbf i\ar@/_0pc/[dd]^{\delta_{\mathbf i\to\mathbf j}}\ar@/_0pc/[rr]^{{\mathcal X}_.}&&
{\mathcal X}_{\mathbf i}\ar@/_0pc/[dd]_{\mu_{\mathbf i\to\mathbf j}}\ar@/_0pc/[rr]^{{\mathfrak t}_{\mathbf i_{\mathfrak R}(t)}}
&&{\mathcal X}_{[\mathbf i]_{\mathfrak R}}(t)\ar@/_0pc/[dd]_{\mu_{[\mathbf i]_{\mathfrak R}\to[\mathbf j]_{\mathfrak R}}}
\ar@/^0pc/[rrd]^{\ev_{\mathbf i}^{\dual}}&&\\
R(v,w_0)\ar@/^0pc/[ur]^{\iota_{\mathbf i}}\ar@/^0pc/[rd]_{\iota_{\mathbf j}}&&&&&&{{\mathcal X}_{v\leq v}(t)}
\ar@/^0pc/[lu]^{\varphi_{[\mathbf i]_{\mathfrak R}}}\ar@/^0pc/[ld]_{\varphi_{[\mathbf j]_{\mathfrak R}}}
\ar@/^0pc/[r]^{\ev_{v\leq v}}&(F_{t,w_0v^{-1}},\pi_*)\\
&\mathbf j\ar@/_0pc/[rr]^{{\mathcal X}_.}&&{\mathcal X}_{\mathbf {j}}\ar@/_0pc/[rr]^{{\mathfrak t}_{\mathbf j_{\mathfrak R}(t)}}
&&{\mathcal X}_{[\mathbf j]_{\mathfrak R}}(t)\ar@/^0pc/[rru]_{\ev_{\mathbf j}^{\dual}}&&
}
\end{center}
\vspace{-.1in}
\caption{The cluster $\mathcal X$-variety ${\mathcal X}_{v\leq v}(t)$ associated
to $(F_{t,w_0v^{-1}},\pi_*)$}
\label{fig:clustervarietydual1}
\end{figure}
Finally, for every double word $\mathbf i$, let ${\mathcal X}_{\mathbf i}^{\dual}
\subset{\mathcal X}_{[\mathbf{i}]_{\mathfrak R}}$ be {such that the elements of the
set of variables $\mathbf x(\mathfrak R)$ are pairwise disjoint. It is a Poisson
submanifold of ${\mathcal X}_{[\mathbf{i}]_{\mathfrak R}}$ because of (\ref{equ:defJ})}.
Thus, the following corollaries are respectively deduced from the second decomposition of (\ref{equ:decG^*})
with Theorem \ref{thm:evdual}, and Theorem \ref{fg} with Lemma \ref{lemma:muR}
and equation (\ref{equ:evdual}).

\begin{cor}For every $\mathbf i\in R(w_0,w_0)$,
the map $\ev^{\dual}_{\mathbf i}:{\mathcal X}_{\mathbf i}^{\dual}
\rightarrow (BB_-,\pi_*)$ is a Poisson
birational isomorphism on a Zariski open set of $BB_-$.
\end{cor}

\begin{cor}\label{cor:evdual} For any $v\in W$ and $\mathbf i,\mathbf j\in R(v,w_0)$,
we have $\ev_{\mathbf{i}}^{\dual}=\ev_{\mathbf{j}}^{\dual}
\circ{\mu}_{[\mathbf{i}]_{\mathfrak R}\rightarrow[\mathbf j]_{\mathfrak R}}$.
\end{cor}


\section{$\tau$-moves, tropical mutations, and twist maps}
\label{section:tropical}
We enlarge the combinatorics on double words and on their related
seed $\mathcal X$-tori by introducing respectively new moves
based on the involution $i\mapsto\overline{i}$ and new birational
Poisson isomorphisms on seed $\mathcal X$-tori
obtained by a tropicalization of the mutation formula.
This enables us to describe the Fomin-Zelevinsky twist maps and
their variations in terms on cluster transformations and tropical mutations.

\subsection{$\tau$-moves and tropical mutations}
We introduce \emph{$\tau$-moves} and \emph{tropical mutations}. From now on,
we suppose that for every seed $\mathbf I=(I,I_0,\varepsilon,d)$, the related
matrix $\varepsilon$ is such that $\varepsilon_{ij}$ is a rational number, for
every $i,j\in I$.

\subsubsection{The right and left $\tau$-moves}We now consider the map
$i\mapsto\overline{i}$ as an involution on $[1,l]\cup[\overline{1},\overline{l}]$.
Let us first enrich the combinatorics on double words.

\begin{definition}\label{def:tropicalmoves} For every double word
$\mathbf i=i_1\dots i_n$, let ${\mathfrak L}(\mathbf i)$ (resp. ${\mathfrak R}(\mathbf i)$)
be the double word obtained by changing the first letter
(resp. last letter) ${i}$ of $\mathbf i$ into $\overline i$:
$$\begin{array}{ccc}
{\mathfrak L}_{i_1}(\mathbf i)=\overline{i_1}i_2\dots i_n
&\mbox{and}&{\mathfrak R}_{i_n}(\mathbf i)=i_1\dots i_{n-1}\overline{i_n}\ .
\end{array}
$$
The map $\mathbf i\mapsto {\mathfrak L}_{i_1}(\mathbf i)$ (resp. $\mathbf i
\mapsto {\mathfrak R}_{j_n}(\mathbf i)$) is called a \emph{left} (resp.
\emph{right}) \emph{$\tau$-move} on $\mathbf i$. (We will simply
denote these maps $\mathbf i\mapsto {\mathfrak L}(\mathbf i)$ and $\mathbf i
\mapsto {\mathfrak R}(\mathbf i)$ when no confusion occurs.)
\end{definition}

\begin{rem}\label{rem:stabletauset}The set $R(u,v)$ of double reduced words associated to
the elements $u,v\in W$ is generally stable by neither left nor right $\tau$-moves. Even the set
of all the double reduced words associated to $W$ is stable by neither left nor right
$\tau$-moves. However it contains subsets that are stable. For example, consider the
following set $R^{\tau}(w)$ associated to any $w\in W$.
$$R^{\tau}(w)=\displaystyle\bigcup_{w'\leq w\in W}R({w'}^{-1},w{w'}^{-1})\ .$$
It is stable by right $\tau$-moves and will be used to describe the combinatorics
associated to the dual Poisson-Lie group $(G^*,\pi_{G^*})$.
\end{rem}

\subsubsection{Tropical mutations}
The $\tau$-moves on double words lead to a new type of
mutations on the associated seed $\mathcal X$-tori, called
{tropical mutations} and defined in the following way.
We first identify $\mathbb Q$ with the Cartesian product
$\mathbb Z\times\mathbb N\backslash\{0\}$, by decomposing every
element into its numerator $p$ and its positive denominator $q$ with
$(p,q)=1$, and associate to every value $\varepsilon_{ij}$ of
$\varepsilon$ its numerator $b_{ij}$. So we have $b_{ij}=
\varepsilon_{ij}$ unless $i,j\in I_0$; let us, from now on, suppose that
the denominator $q$ is the same for every $i,j\in I_0$.
In particular, we recall that for every double
word $\mathbf i$ we have
\begin{equation}\label{equ:b}
b(\mathbf i)_{kl}=\left\{
\begin{array}{cc}
2\varepsilon(\mathbf i)_{kl}& \mbox{for } k,l\in I_0\ ;\\
\varepsilon(\mathbf i)_{kl}& \mbox{otherwise}\ .
\end{array}
\right.
\end{equation}
Then we need an additional data of the set of outlets of a seed.

\begin{definition}Let $\mathbf I=(I_0,I,\varepsilon,d)$ be a seed
such that $I_0$ is not empty.
A \emph{cover} $\mathfrak{C}$ on $\mathbf I$ is a family of sets
$I_1,\dots,I_n\subset I_0$ such that $I_0=\cup_{i=1}^{n}I_i$.
(The union is not necessary disjoint.)
For every $k\in I_0$, we denote $I_0(k)$ the union
$$I_0(k):=\bigcup_{\{i\mid k\in I_i\}}I_i\ .$$
\end{definition}

\begin{rem}Every seed $\mathbf I=(I_0,I,\varepsilon,d)$ with a non empty set $I_0$
comes from two trivial covers, where the first one is obtained by setting $I_1:=I_0$,
and, on the opposite way, where the second one is given by the union $I_0=\cup_{k\in I_0}\{k\}$.
\end{rem}

\begin{definition}\label{def:trop}
Let $\mathbf{I}=(I,I_0,\varepsilon,d)$ and $\mathbf {I'}=(I',{I'}_0,{\varepsilon}',d')$
be two seeds with covers, and $k\in I_0$. A $\emph{tropical mutation}$ $\emph{in the direction k}$
is an involution $\mu_k:\mathbf I\longrightarrow\mathbf{I'}$
satisfying the following conditions:

(i)$\mu_k(I_0(i))={I'}_0(i)$;

(ii) ${d'}_{\mu_k(i)}=d_i$;

{(iii)$$
\varepsilon'_{\mu_k(i)\mu_k(j)} =
\begin{cases}
-\varepsilon_{ij} & \text{if $i=k$ or $j=k$;} \\[.05in]
\varepsilon_{ij}
& \text{if $i,j\in I_0(k)\backslash\{k\}$};\\
\varepsilon_{ij} -\varepsilon_{ik}b_{kj} &\text{otherwise.}
\end{cases}
$$
Tropical mutations induce maps between the corresponding seed $\mathcal X$-tori,
which are denoted by the same symbols $\mu_k$ and given by
\begin{equation}\label{equ:formuletropmut}
x_{\mu_k(i)}=\left\{ \begin{array}{lll}
{x_k}^{-1}& \mbox{if}\ i=k;\\
x_ix_k^{b_{ki}}& \mbox{if}\ i\in I_0(k)\backslash\{k\};\\
x_i& \mbox{otherwise}.
\end{array}\right.
\end{equation}}
\end{definition}

\begin{rem}
It is easy to see that, as mutations, tropical mutations are involutions.
\end{rem}


\begin{rem} Here is the reason underlying the terminology for \emph{tropical mutation}.
Following \cite[Example 5.6]{FZ1},
let us consider the abelian group (written multiplicatively) freely generated
by the cluster variables $x_i \, (i \in I)$, given with the addition
$$
\prod_i x_i^{a_i} \boxplus \prod_i x_i^{b_i} :=
\prod_i x_i^{\min (a_i, b_i)}.
$$
This tropical addition $\boxplus$ has not to be
confused with the tropical addition $\oplus$ in the usual tropical setting:
$
a\odot b = a + b$ and
$a \oplus b ={\rm min} (a,b)
$.
But is easy to see that they are strongly related:
$
\prod_i x_i^{a_i} \boxplus \prod_i x_i^{b_i} =
\prod_i x_i^{a_i\oplus b_i}$.

It turns out that the left and right tropical mutations, associated
respectively to left and right $\tau$-moves $\mathfrak{L}_{\overline{i}}$
and $\mathfrak{R}_j$ (but not $\mathfrak{L}_{{i}}$
and $\mathfrak{R}_{\overline j}$ !),
with $i,j\in[1,l]$ and described in the next subsection,
can alternatively be defined by the following formula, obtained by tropicalizing
the mutation formula in Definition \ref{def:mutation}:
\begin{equation}\label{equ:tropmut}
x_{\mu_k(i)}=
\left\{ \begin{array}{lll}
{x_k}^{-1}& \mbox{if}\ i=k;\\
x_ix_k^{(b_{ik})_+}(1\boxplus x_k)^{-b_{ik}}& \mbox{otherwise}.
\end{array}\right.
\end{equation}
The formula (\ref{equ:tropmut}) has relatives in the cluster algebra literature:
\begin{itemize}
\item
The formula covers the mutation combinatorics \cite{FZ4} attached
to a labeled $Y$-seed $(\mathbf y, B)$
defined by $\mathbf y=\{x_i| i\in I_0\}$ and $B=(b_{ij})_{i,j\in  I_0}$.
\item
Up to the rescaling $\varepsilon\rightarrow b$ given by
(\ref{equ:b}) and the choice of direction, the formula (\ref{equ:tropmut}) is equal
to the monomial part $\mu'_k$ of mutations
respecting to the decomposition given in \cite[Section 2.3]{FGdilog2}.
\item
Compositions of mutations and tropical mutations can be used to describe the
modified octohedron recurrence of Henriques and Kamnitzer \cite{HK}. The related Poisson
dynamics on double Bruhat cells involves twist maps and will be given in a separated paper \cite{RB3}.
\end{itemize}
\end{rem}

\begin{rem}\label{rem:tropamal2}Like the amalgamation product, tropical mutations
act on cluster variables $x_j$ associated to outlets, therefore the commutation
between them is not always satisfied.
In fact a tropical mutation, acting non trivially on a set of cluster variables
$\{x_j\}_{j\in J}$ associated to a set of outlets $J$, commutes with the amalgamation
product if and only we have $L\cap J=\emptyset$, where the set $L$ denotes the set
of amalgamation, as given in Definition~\ref{def:amalg}.
\end{rem}

Finally, in the same way that mutations and symmetries were respecting the set of outlets in the
definitions of Subsection \ref{section:backcluster}, we suppose, from now on, that they respect
covers on seeds.

\begin{definition}
A \emph{generalized cluster transformation} linking two seeds
(and two seed $\mathcal X$-tori) is a composition of symmetries, mutations, and
tropical mutations.
\end{definition}

\subsubsection{Left and right tropical mutations}
We are now ready to describe the way to relate left and $\tau$-moves to particular
tropical mutations, respectively called left and right tropical mutations.

\begin{definition}Let $\mathbf i$ be a double word and recall the subsets
$I_0^{\mathfrak{L}}(\mathbf i)$ and $I_0^{\mathfrak{R}}(\mathbf i)$ of $I_0(\mathbf i)$
defined in Subsection \ref{section:seeddoublewords}. From now on, we will denote
$\mathbf{I(i)}$ the seed $(I(\mathbf i),I_0(\mathbf i),\varepsilon(\mathbf i),d(\mathbf i))$
given with the cover
$$I_0(\mathbf i)=I_0^{\mathfrak{L}}(\mathbf i)\cup I_0^{\mathfrak{R}}(\mathbf i)\ .$$
\end{definition}

\begin{prop}\label{prop:muttrop}The following tropical mutations
are Poisson birational isomorphisms for every double word $\mathbf i=i_1\dots i_n$.
\begin{equation}\label{equprop:muttrop}
\begin{array}{lcr}
\mu_{\binom {i_1}{0}}:{\mathcal X}_{\mathbf i}\rightarrow {\mathcal X}_{{\mathfrak L}(\mathbf i)}
&\mbox{and}&
\mu_{\binom {i_n}{N^{i_n}(\mathbf i)}}:{\mathcal X}_{\mathbf i}\rightarrow {\mathcal X}_{{\mathfrak R}(\mathbf i)}\ .
\end{array}
\end{equation}
\end{prop}
\begin{proof}The birational part is clear, so we focus on the Poisson part.
Let us notice that, because right (resp. left) tropical mutations commute with an
amalgamated product done on the left side (resp. right side), as detailed in Remark
\ref{rem:tropamal}, it suffices to show that the proposition
is true for $\mathbf i\in\{i,\overline{i}\}$. Consider the case
$\mathbf i={i}$. So we have ${\mathfrak L}(\mathbf i)={\mathfrak R}(\mathbf i)=\overline{i}$.
and the equality $b(\mathbf i)=2\varepsilon(\mathbf i)$.
It is then straightforward to check that the matrices
$\varepsilon({\mathfrak L}(\mathbf i))$ and $\varepsilon({\mathfrak R}(\mathbf i))$
are equal to $\varepsilon(\overline{i})$.
The case $\mathbf i=\overline{i}$ is proved in the same way.
\end{proof}

\begin{definition}\label{def:RLtropmut}Let $\mathbf i=i_1\dots i_n$ be a double word. The tropical
mutations given by equation (\ref{equprop:muttrop}) are
respectively called \emph{left} and \emph{right tropical mutations}
and we denote the associated directions respectively by
$\lozenge^{\mathfrak L}_{i_1}$ and $\lozenge^{\mathfrak R}_{i_n}$:
$$\begin{array}{ccc}
\lozenge^{\mathfrak L}_{i_1}=\binom {i_1}{0}&\mbox{and}&
\lozenge^{\mathfrak R}_{i_n}=\binom {i_n}{N^{i_n}(\mathbf i)}\ .
\end{array}$$
\end{definition}

\begin{rem}\label{rem:tropamal}Remark \ref{rem:tropamal2} can be
refined in the following way: tropical mutation in a left (resp.
right) outlet direction
applied to a seed $\mathcal X$-torus ${\mathcal X}_{\mathbf I}$ (resp.
${\mathcal X}_{\mathbf J}$) is equal to the tropical mutation in the same
direction applied to the amalgamated seed $\mathcal X$-torus
${\mathcal X}_{\mathfrak{m}(\mathbf I,\mathbf J)}$.
\end{rem}

Left and right tropical mutations are easily described on Dynkin quivers.
Let us take a Dynkin quiver $\Gamma$ and let $\mathbf i=i_1\dots i_n$ be
the associated double word, so we have $\Gamma=\Gamma_{\mathfrak g}(\mathbf i)$.
In particular, $\Gamma$ is obtained by the amalgamation of the elementary Dynkin
quivers $\Gamma_{\mathfrak g}(i_1),\dots,\Gamma_{\mathfrak g}(i_n)$. Let us remember
the behavior between tropical mutations and the amalgamated product given by
Remark \ref{rem:tropamal}; and change the orientation of the arrows of
$\Gamma_{\mathfrak g}(i_1)$ (resp. $\Gamma_{\mathfrak g}(i_m)$) if we have
a left (resp. right) tropical mutation. We get the elementary Dynkin quiver
$\Gamma_{\mathfrak g}(\overline{i_1})$, (resp. $\Gamma_{\mathfrak g}(\overline{i_m})$).
Now, let us perform the amalgamation
$$\begin{array}{ccc}
\mathfrak m:\Gamma_{\mathfrak g}
(\overline{i_1})\times\dots\times\Gamma_{\mathfrak g}(i_n)\to\Gamma_{\mathfrak g}
({\mathfrak L}(\mathbf i))&
(\mbox{resp.}&  \mathfrak m:\Gamma_{\mathfrak g}(i_1)\times\dots
\times\Gamma_{\mathfrak g}(\overline{i_n})\to\Gamma_{\mathfrak g}({\mathfrak R}(\mathbf i)))\ .
\end{array}
$$
The resulting quiver $\Gamma'$ is therefore the image of $\Gamma$ by the left (resp. right)
tropical mutation. This procedure is illustrated in Figure \ref{fig:tropmut}.

\begin{figure}[htbp]
\begin{center}
\setlength{\unitlength}{1.5pt}
\begin{picture}(20,48)(0,-27)
\thicklines
\put(1,1.7){\line(2,3){8}}
\put(1,1.7){\vector(2,3){6}}
\put(21,1.7){\line(2,3){8}}
\put(21,1.7){\vector(2,3){6}}
\put(41,1.7){\line(2,3){7.5}}
\put(41,1.7){\vector(2,3){6}}
\put(11,-13.3){\line(2,3){7.5}}
\put(11,-13.3){\vector(2,3){6}}
\put(31,-13.3){\line(2,3){7.5}}
\put(31,-13.3){\vector(2,3){6}}
\put(-9,13.3){\line(2,-3){7.5}}
\put(-9,13.3){\vector(2,-3){6}}
\put(11,13.3){\line(2,-3){7.5}}
\put(11,13.3){\vector(2,-3){6}}
\put(31,13.3){\line(2,-3){7.5}}
\put(31,13.3){\vector(2,-3){6}}
\put(1,-1.7){\line(2,-3){7.5}}
\put(1,-1.7){\vector(2,-3){6}}
\put(21,-1.7){\line(2,-3){7.5}}
\put(21,-1.7){\vector(2,-3){6}}
\put(10,15){\circle*{4}}
\put(30,15){\circle*{4}}
\put(-10,15){\circle{4}}
\put(50,15){\circle{4}}
\put(10,-15){\circle{4}}
\put(30,-15){\circle{4}}
\put(0,0){\circle{4}}
\put(20,0){\circle*{4}}
\put(40,0){\circle{4}}
\put(8,15){\vector(-1,0){11}}
\put(8,15){\line(-1,0){16}}
\put(28,15){\vector(-1,0){11}}
\put(28,15){\line(-1,0){16}}
\put(48,15){\vector(-1,0){11}}
\put(48,15){\line(-1,0){16}}
\put(18,0){\vector(-1,0){11}}
\put(18,0){\line(-1,0){16}}
\put(38,0){\vector(-1,0){11}}
\put(38,0){\line(-1,0){16}}
\put(28,-15){\vector(-1,0){11}}
\put(28,-15){\line(-1,0){16}}
\put(45,-7.7){\line(2,-3){15}}
\put(45,-7.7){\vector(2,-3){15}}
\end{picture}
\qquad\qquad\qquad\qquad
\begin{picture}(20,48)(0,-27)
\thicklines
\put(0,0){\vector(1,0){40}}
\put(16,8){$\mu_{\binom{1}{3}}$}
\end{picture}
\qquad\qquad\qquad
\begin{picture}(20,48)(0,-27)
\thicklines
\put(1,1.7){\line(2,3){8}}
\put(1,1.7){\vector(2,3){6}}
\put(21,1.7){\line(2,3){8}}
\put(21,1.7){\vector(2,3){6}}
\put(49,13.3){\line(-2,-3){7.5}}
\put(49,13.3){\vector(-2,-3){6}}
\put(11,-13.3){\line(2,3){7.5}}
\put(11,-13.3){\vector(2,3){6}}
\put(31,-13.3){\line(2,3){7.5}}
\put(31,-13.3){\vector(2,3){6}}
\put(-9,13.3){\line(2,-3){7.5}}
\put(-9,13.3){\vector(2,-3){6}}
\put(11,13.3){\line(2,-3){7.5}}
\put(11,13.3){\vector(2,-3){6}}
\put(1,-1.7){\line(2,-3){7.5}}
\put(1,-1.7){\vector(2,-3){6}}
\put(21,-1.7){\line(2,-3){7.5}}
\put(21,-1.7){\vector(2,-3){6}}
\put(10,15){\circle*{4}}
\put(30,15){\circle*{4}}
\put(-10,15){\circle{4}}
\put(50,15){\circle{4}}
\put(10,-15){\circle{4}}
\put(30,-15){\circle{4}}
\put(0,0){\circle{4}}
\put(20,0){\circle*{4}}
\put(40,0){\circle{4}}
\put(8,15){\vector(-1,0){11}}
\put(8,15){\line(-1,0){16}}
\put(28,15){\vector(-1,0){11}}
\put(28,15){\line(-1,0){16}}
\put(32,15){\vector(1,0){11}}
\put(32,15){\line(1,0){16}}
\put(18,0){\vector(-1,0){11}}
\put(18,0){\line(-1,0){16}}
\put(38,0){\vector(-1,0){11}}
\put(38,0){\line(-1,0){16}}
\put(28,-15){\vector(-1,0){11}}
\put(28,-15){\line(-1,0){16}}
\put(-20,-30.7){\line(2,3){15}}
\put(-20,-30.7){\vector(2,3){15}}
\end{picture}
$$ $$
\begin{picture}(20,48)(0,-27)
\thicklines
\put(55,17){\circle{4}}
\put(35,17){\circle{4}}
\put(45,2){\circle{4}}
\put(46,3.7){{\line(2,3){7.5}}}
\put(46,3.7){{\vector(2,3){6}}}
\put(36,15.3){\line(2,-3){7.5}}
\put(36,15.3){\vector(2,-3){6}}
\put(53,17){{\vector(-1,0){11}}}
\put(53,17){{\line(-1,0){16}}}
\put(1,1.7){\line(2,3){8}}
\put(1,1.7){\vector(2,3){6}}
\put(21,1.7){{\line(2,3){8}}}
\put(21,1.7){{\vector(2,3){6}}}
\put(11,-13.3){\line(2,3){7.5}}
\put(11,-13.3){\vector(2,3){6}}
\put(31,-13.3){\line(2,3){7.5}}
\put(31,-13.3){\vector(2,3){6}}
\put(-9,13.3){\line(2,-3){7.5}}
\put(-9,13.3){\vector(2,-3){6}}
\put(11,13.3){\line(2,-3){7.5}}
\put(11,13.3){\vector(2,-3){6}}
\put(31,13.3){\line(2,-3){7.5}}
\put(31,13.3){\vector(2,-3){6}}
\put(1,-1.7){\line(2,-3){7.5}}
\put(1,-1.7){\vector(2,-3){6}}
\put(21,-1.7){\line(2,-3){7.5}}
\put(21,-1.7){\vector(2,-3){6}}
\put(10,15){\circle*{4}}
\put(30,15){\circle{4}}
\put(-10,15){\circle{4}}
\put(10,-15){\circle{4}}
\put(30,-15){\circle{4}}
\put(0,0){\circle{4}}
\put(20,0){\circle*{4}}
\put(40,0){\circle{4}}
\put(8,15){\vector(-1,0){11}}
\put(8,15){\line(-1,0){16}}
\put(28,15){\vector(-1,0){11}}
\put(28,15){\line(-1,0){16}}
\put(18,0){\vector(-1,0){11}}
\put(18,0){\line(-1,0){16}}
\put(38,0){{\vector(-1,0){11}}}
\put(38,0){{\line(-1,0){16}}}
\put(28,-15){\vector(-1,0){11}}
\put(28,-15){\line(-1,0){16}}
\end{picture}
\qquad\qquad
\begin{picture}(20,48)(0,-27)
\thicklines
\put(0,0){\vector(1,0){20}}
\end{picture}
\quad
\begin{picture}(20,48)(0,-27)
\thicklines
\put(55,17){\circle{4}}
\put(35,17){\circle{4}}
\put(45,2){\circle{4}}
\put(54,15.3){\line(-2,-3){7.5}}
\put(54,15.3){\vector(-2,-3){6}}
\put(44,4){{\line(-2,3){7.5}}}
\put(44,4){{\vector(-2,3){6}}}
\put(1,1.7){\line(2,3){8}}
\put(1,1.7){\vector(2,3){6}}
\put(21,1.7){{\line(2,3){8}}}
\put(21,1.7){{\vector(2,3){6}}}
\put(11,-13.3){\line(2,3){7.5}}
\put(11,-13.3){\vector(2,3){6}}
\put(31,-13.3){\line(2,3){7.5}}
\put(31,-13.3){\vector(2,3){6}}
\put(-9,13.3){\line(2,-3){7.5}}
\put(-9,13.3){\vector(2,-3){6}}
\put(11,13.3){\line(2,-3){7.5}}
\put(11,13.3){\vector(2,-3){6}}
\put(31,13.3){\line(2,-3){7.5}}
\put(31,13.3){\vector(2,-3){6}}
\put(1,-1.7){\line(2,-3){7.5}}
\put(1,-1.7){\vector(2,-3){6}}
\put(21,-1.7){\line(2,-3){7.5}}
\put(21,-1.7){\vector(2,-3){6}}
\put(10,15){\circle*{4}}
\put(30,15){\circle{4}}
\put(-10,15){\circle{4}}
\put(10,-15){\circle{4}}
\put(30,-15){\circle{4}}
\put(0,0){\circle{4}}
\put(20,0){\circle*{4}}
\put(40,0){\circle{4}}
\put(8,15){\vector(-1,0){11}}
\put(8,15){\line(-1,0){16}}
\put(28,15){\vector(-1,0){11}}
\put(28,15){\line(-1,0){16}}
\put(37,17){{\vector(1,0){11}}}
\put(37,17){{\line(1,0){16}}}
\put(18,0){\vector(-1,0){11}}
\put(18,0){\line(-1,0){16}}
\put(38,0){{\vector(-1,0){11}}}
\put(38,0){{\line(-1,0){16}}}
\put(28,-15){\vector(-1,0){11}}
\put(28,-15){\line(-1,0){16}}
\end{picture}
\end{center}
\vspace{-.1in}
\caption{The tropical mutation $\mu_{\binom{1}{3}}:
\Gamma_{A_3}(123121)\mapsto\Gamma_{A_3}
(12312\overline 1)$}
\label{fig:tropmut}
\end{figure}
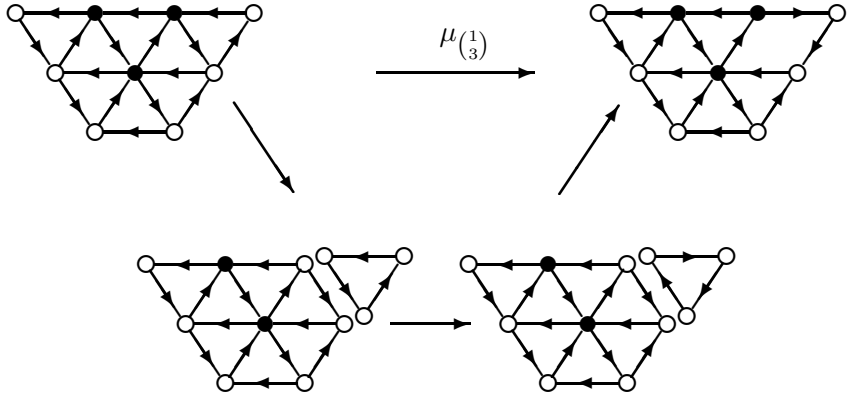

\subsection{Symmetries on seed $\mathcal X$-tori}
\label{section:involution}In this subsection and in the next one,
we define involutions on double words, and the related symmetries on seed
$\mathcal X$-tori. They will be use to describe various automorphisms
and anti-automorphisms on the group $G$.
Let us recall that the \emph{fundamental weights} $\omega_i\in{\mathfrak h}^*$,
given by $\omega_i(h_j)=\delta_{ij}$ for every $i,j\in[1,l]$,
are permuted by the transformation $(-w_0)$.
We denote $i\mapsto i^{\star}$ the induced permutation of the indices of these
weights, that is $\omega_{i^{\star}}=-w_0(\omega_i)$.
This automorphism on the Dynkin diagram leads to symmetries
on seed $\mathcal X$-tori described by the use of Dynkin quivers:
for every $u,v\in W$ and $\mathbf i\in R(u,v)$
we define the involutions
$$
\begin{array}{cccccccccc}
\star: & R(u,v)  & \rightarrow & R({v^\star},{u^\star})&&&\op: & R(u,v)
& \rightarrow & R(u^{-1},v^{-1})\\
                                & \mathbf i & \mapsto & \mathbf{i^\star}
                                &&& &\mathbf i &\mapsto &{\mathbf {i^{\op}}}
\end{array}
$$
in the following way: the double reduced word $\mathbf{i^{\star}}\in R({v^\star},{u^\star})$
is obtained by transforming each letter $i$ of
$[1,l]\cup[\overline{1},\overline{l}]$ into $\overline{i^\star}$, so if $\mathbf i=i_1\dots i_n$
then $\mathbf i^{\star}=\overline{i_1}^{\star}\dots\overline{i_n}^{\star}$, and the
double reduced word ${\mathbf{i^{\op}}}\in  R(u^{-1},v^{-1})$ is
obtained by reading $\mathbf i$ backwards: $\mathbf i^{\op}=i_n\dots i_1$.
Now,~let
$$\circlearrowright:
R(u,v)\rightarrow R({v^\star}^{-1},{u^\star}^{-1})$$
be the map such that the double reduced word $\mathbf{i^{\circlearrowright}}
\in  R({v^\star}^{-1},{u^\star}^{-1})$ is
obtained by converting each $i$ into $\overline{i^\star}$ and then
read the result backwards. Stated otherwise, this transformation is
defined by the equality $\circlearrowright=\op\circ \star$.
It induces the following symmetry on seed $\mathcal X$-tori.

\begin{equation}\label{equ:circlearrow}\begin{array}{rcccl}
\circlearrowright:&{\mathcal X}_{\mathbf {i}}&
\longrightarrow& {\mathcal X}_{\mathbf {i^{\circlearrowright}}}\\
&x_{\binom{i}{j}}&\longmapsto & x_{\binom{i^\star}{N^{i^\star}(\mathbf{i^{\circlearrowright}})-j}}
\end{array}.
\end{equation}

\begin{rem}\label{rem:rotation} The involutions $\star$ and $\op$ were
already defined on reduced words $\mathbf i\in R(w_0)$ in
\cite[Equation (3.1)]{BZtensor}.
Moreover, the notation $\circlearrowright$ for the last involution
on double reduced words symbolizes a rotation of 180 degrees performed on the
quiver $\Gamma_{\mathfrak{g}}(\mathbf i)$, as illustrated in
Figure \ref{fig:involution2}.
\end{rem}

\begin{figure}[htbp]
\begin{center}
\setlength{\unitlength}{1.5pt}
\begin{picture}(20,48)(0,-27)
\thicklines
\oval(30,30)[l]
\put(-21,0){$\star$}
\put(4,15){\vector(1,0){0}}
\put(6,13){$1$}
\put(6,-2){$2$}
\put(6,-17){$3$}
\put(12,15){\circle*{4}}
\put(12,13){\line(0,-1){12}}
\put(12,0){\circle*{4}}
\put(12,-2){\line(0,-1){12}}
\put(12,-15){\circle*{4}}
\end{picture}
\quad
\begin{picture}(20,48)(0,-27)
\thicklines
\put(0,0){\vector(1,0){20}}
\end{picture}
\quad
\begin{picture}(20,48)(0,-27)
\thicklines
\put(6,13){$3$}
\put(6,-2){$2$}
\put(6,-17){$1$}
\put(12,15){\circle*{4}}
\put(12,13){\line(0,-1){12}}
\put(12,0){\circle*{4}}
\put(12,-2){\line(0,-1){12}}
\put(12,-15){\circle*{4}}
\end{picture}
\end{center}
\vspace{-.1in}
\caption{The $\star$ involution on $\Gamma_{A_3}$}
\label{fig:involution}
\end{figure}
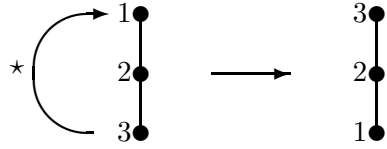

\begin{figure}[htbp]
\begin{center}
\setlength{\unitlength}{1.5pt}
\begin{picture}(20,48)(0,-27)
\thicklines
\put(-43,0){$\circlearrowright$}
\put(-20,0){\oval(30,30)[l]}
\put(-16,15){\vector(1,0){0}}
\put(1,1.7){\line(2,3){8}}
\put(1,1.7){\vector(2,3){6}}
\put(21,1.7){\line(2,3){8}}
\put(21,1.7){\vector(2,3){6}}
\put(41,1.7){\line(2,3){7.5}}
\put(41,1.7){\vector(2,3){6}}
\put(11,-13.3){\line(2,3){7.5}}
\put(11,-13.3){\vector(2,3){6}}
\put(31,-13.3){\line(2,3){7.5}}
\put(31,-13.3){\vector(2,3){6}}
\put(-9,13.3){\line(2,-3){7.5}}
\put(-9,13.3){\vector(2,-3){6}}
\put(11,13.3){\line(2,-3){7.5}}
\put(11,13.3){\vector(2,-3){6}}
\put(31,13.3){\line(2,-3){7.5}}
\put(31,13.3){\vector(2,-3){6}}
\put(1,-1.7){\line(2,-3){7.5}}
\put(1,-1.7){\vector(2,-3){6}}
\put(21,-1.7){\line(2,-3){7.5}}
\put(21,-1.7){\vector(2,-3){6}}
\put(10,15){\circle*{4}}
\put(30,15){\circle*{4}}
\put(-10,15){\circle{4}}
\put(50,15){\circle{4}}
\put(10,-15){\circle{4}}
\put(30,-15){\circle{4}}
\put(0,0){\circle{4}}
\put(20,0){\circle*{4}}
\put(40,0){\circle{4}}
\put(8,15){\vector(-1,0){11}}
\put(8,15){\line(-1,0){16}}
\put(28,15){\vector(-1,0){11}}
\put(28,15){\line(-1,0){16}}
\put(48,15){\vector(-1,0){11}}
\put(48,15){\line(-1,0){16}}
\put(18,0){\vector(-1,0){11}}
\put(18,0){\line(-1,0){16}}
\put(38,0){\vector(-1,0){11}}
\put(38,0){\line(-1,0){16}}
\put(28,-15){\vector(-1,0){11}}
\put(28,-15){\line(-1,0){16}}
\end{picture}
\qquad\qquad\quad
\begin{picture}(20,48)(0,-27)
\thicklines
\put(0,0){\vector(1,0){20}}
\end{picture}
\quad\qquad
\begin{picture}(20,48)(0,-27)
\thicklines
\put(29,13.3){\line(-2,-3){7.5}}
\put(29,13.3){\vector(-2,-3){6}}
\put(9,13.3){\line(-2,-3){7.5}}
\put(9,13.3){\vector(-2,-3){6}}
\put(39,-1.7){\line(-2,-3){7.5}}
\put(39,-1.7){\vector(-2,-3){6}}
\put(19,-1.7){\line(-2,-3){7.5}}
\put(19,-1.7){\vector(-2,-3){6}}
\put(-1,-1.7){\line(-2,-3){7.5}}
\put(-1,-1.7){\vector(-2,-3){6}}
\put(39,2){{\line(-2,3){7.5}}}
\put(39,2){{\vector(-2,3){6}}}
\put(19,2){{\line(-2,3){7.5}}}
\put(19,2){{\vector(-2,3){6}}}
\put(49,-13){{\line(-2,3){7.5}}}
\put(49,-13){{\vector(-2,3){6}}}
\put(29,-13){{\line(-2,3){7.5}}}
\put(29,-13){{\vector(-2,3){6}}}
\put(9,-13){{\line(-2,3){7.5}}}
\put(9,-13){{\vector(-2,3){6}}}
\put(10,15){\circle{4}}
\put(30,15){\circle{4}}
\put(-10,-15){\circle{4}}
\put(50,-15){\circle{4}}
\put(10,-15){\circle*{4}}
\put(30,-15){\circle*{4}}
\put(0,0){\circle{4}}
\put(20,0){\circle*{4}}
\put(40,0){\circle{4}}
\put(-8,-15){\vector(1,0){11}}
\put(-8,-15){\line(1,0){16}}
\put(12,15){\vector(1,0){11}}
\put(12,15){\line(1,0){16}}
\put(32,-15){\vector(1,0){11}}
\put(32,-15){\line(1,0){16}}
\put(2,0){\vector(1,0){11}}
\put(2,0){\line(1,0){16}}
\put(22,0){\vector(1,0){11}}
\put(22,0){\line(1,0){16}}
\put(12,-15){\vector(1,0){11}}
\put(12,-15){\line(1,0){16}}
\end{picture}
\end{center}
\vspace{-.1in}
\caption{The involution $\circlearrowright:\Gamma_{A_3}(123121)
\mapsto\Gamma_{A_3}(\overline{3}\overline 2\overline 3\overline 1\overline 2\overline 3)$}
\label{fig:involution2}
\end{figure}
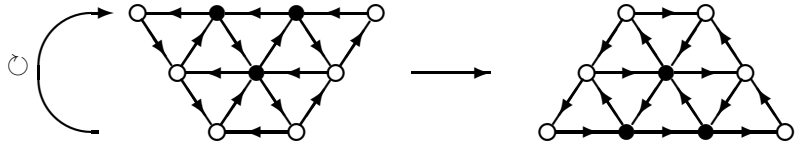

In the same way, to any double word $\mathbf i$, we associate the double word
$\mathbf{i^{\square}}$ obtained by
reading $\mathbf{i}$ backwards and
applying the involution $i\mapsto\overline{i}$ on every letter of the result.
In particular, for every $u,v\in W$ we get an involution
$$\square: R(u,v)\to R(v^{-1},u^{-1})\ .$$
This involution first induces a symmetry on the related seed: it corresponds
to a rotation along the vertical axis passing by the center of
the Dynkin quiver $\Gamma_{\mathfrak g}(\mathbf i)$, as illustrated in Figure
\ref{fig:involution4}. It also induces a symmetry on seed $\mathcal X$-tori,
also denoted $\square$ and given by:
$$\begin{array}{ccccc}
\square:&{\mathcal X}_{\mathbf {i}}&
\longrightarrow& {\mathcal X}_{\mathbf {i^{\square}}}\\
&x_{\binom{i}{j}}&\longmapsto & x_{\binom{i}{N^{i}(\mathbf{i})-j}}
\end{array}
$$

\begin{figure}[htbp]
\begin{center}
\setlength{\unitlength}{1.5pt}
\begin{picture}(20,58)(0,-37)
\put(20,15){\line(0,1){16}}
\put(20,-15){\line(0,-1){16}}
\thicklines
\put(17, -38){$\square$}
\put(20,-25){\oval(30,10)[b]}
\put(35,-22){\vector(0,1){0}}
\put(1,1.7){\line(2,3){8}}
\put(1,1.7){\vector(2,3){6}}
\put(21,1.7){\line(2,3){8}}
\put(21,1.7){\vector(2,3){6}}
\put(41,1.7){\line(2,3){7.5}}
\put(41,1.7){\vector(2,3){6}}
\put(11,-13.3){\line(2,3){7.5}}
\put(11,-13.3){\vector(2,3){6}}
\put(31,-13.3){\line(2,3){7.5}}
\put(31,-13.3){\vector(2,3){6}}
\put(-9,13.3){\line(2,-3){7.5}}
\put(-9,13.3){\vector(2,-3){6}}
\put(11,13.3){\line(2,-3){7.5}}
\put(11,13.3){\vector(2,-3){6}}
\put(31,13.3){\line(2,-3){7.5}}
\put(31,13.3){\vector(2,-3){6}}
\put(1,-1.7){\line(2,-3){7.5}}
\put(1,-1.7){\vector(2,-3){6}}
\put(21,-1.7){\line(2,-3){7.5}}
\put(21,-1.7){\vector(2,-3){6}}
\put(10,15){\circle*{4}}
\put(30,15){\circle*{4}}
\put(-10,15){\circle{4}}
\put(50,15){\circle{4}}
\put(10,-15){\circle{4}}
\put(30,-15){\circle{4}}
\put(0,0){\circle{4}}
\put(20,0){\circle*{4}}
\put(40,0){\circle{4}}
\put(8,15){\vector(-1,0){11}}
\put(8,15){\line(-1,0){16}}
\put(28,15){\vector(-1,0){11}}
\put(28,15){\line(-1,0){16}}
\put(48,15){\vector(-1,0){11}}
\put(48,15){\line(-1,0){16}}
\put(18,0){\vector(-1,0){11}}
\put(18,0){\line(-1,0){16}}
\put(38,0){\vector(-1,0){11}}
\put(38,0){\line(-1,0){16}}
\put(28,-15){\vector(-1,0){11}}
\put(28,-15){\line(-1,0){16}}
\end{picture}
\qquad\qquad\quad
\begin{picture}(20,58)(0,-37)
\thicklines
\put(0,0){\vector(1,0){20}}
\end{picture}
\quad\qquad
\begin{picture}(20,58)(0,-37)
\thicklines
\put(49,13.3){\line(-2,-3){7.5}}
\put(49,13.3){\vector(-2,-3){6}}
\put(29,13.3){\line(-2,-3){7.5}}
\put(29,13.3){\vector(-2,-3){6}}
\put(9,13.3){\line(-2,-3){7.5}}
\put(9,13.3){\vector(-2,-3){6}}
\put(39,-1.7){\line(-2,-3){7.5}}
\put(39,-1.7){\vector(-2,-3){6}}
\put(19,-1.7){\line(-2,-3){7.5}}
\put(19,-1.7){\vector(-2,-3){6}}
\put(29,-13){{\line(-2,3){7.5}}}
\put(29,-13){{\vector(-2,3){6}}}
\put(9,-13){{\line(-2,3){7.5}}}
\put(9,-13){{\vector(-2,3){6}}}
\put(-1,2){{\line(-2,3){7.5}}}
\put(-1,2){{\vector(-2,3){6}}}
\put(19,2){{\line(-2,3){7.5}}}
\put(19,2){{\vector(-2,3){6}}}
\put(39,2){{\line(-2,3){7.5}}}
\put(39,2){{\vector(-2,3){6}}}
\put(10,15){\circle*{4}}
\put(30,15){\circle*{4}}
\put(-10,15){\circle{4}}
\put(50,15){\circle{4}}
\put(10,-15){\circle{4}}
\put(30,-15){\circle{4}}
\put(0,0){\circle{4}}
\put(20,0){\circle*{4}}
\put(40,0){\circle{4}}
\put(-8,15){\vector(1,0){11}}
\put(-8,15){\line(1,0){16}}
\put(12,15){\vector(1,0){11}}
\put(12,15){\line(1,0){16}}
\put(32,15){\vector(1,0){11}}
\put(32,15){\line(1,0){16}}
\put(2,0){\vector(1,0){11}}
\put(2,0){\line(1,0){16}}
\put(22,0){\vector(1,0){11}}
\put(22,0){\line(1,0){16}}
\put(12,-15){\vector(1,0){11}}
\put(12,-15){\line(1,0){16}}
\end{picture}
\end{center}
\vspace{-.1in}
\caption{The involution $\square:\Gamma_{A_3}(123121)
\mapsto\Gamma_{A_3}(\overline{1}\overline 2\overline 1
\overline 3\overline 2\overline 1)$}
\label{fig:involution4}
\end{figure}
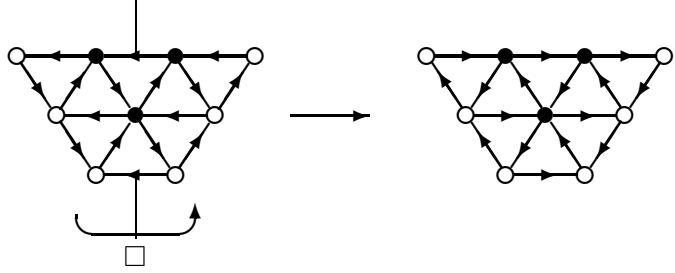

Here are now some related isomorphisms on $(G,\pi_G)$.
Starting with an elementary double word $\mathbf i\in\{\mathbf 1,i,\overline i\}$,
where $i\in[1,l]$, and then applying the properties of the amalgamation product,
we easily prove the following results.

\begin{lemma}\label{lemma:w_0} Let $u,v\in W$ and $\mathbf i\in R(u,v)$.
For every cluster $\mathbf x\in{\mathcal X}_{\mathbf i}$, let $\mathbf{x^{\star}}
\in{\mathcal X}_{\mathbf{i^\star}}$ be the cluster such that the following equality
is satisfied.
$$\widehat{w_0}\ev_{\mathbf i}(\mathbf x)\widehat{w_0}^{-1}=\ev_{\mathbf{i^\star}}
(\mathbf{x^{\star}})\ .$$
Then we have
\begin{equation}\label{equ:w_0}x^{\star}_{\binom i{j}}=\left\{
\begin{array}{rl}
-x_{\binom {i^\star}{j}}^{-1}&\mbox{if } 0=j\neq N^{i^\star}(\mathbf i)
\mbox{ or } 0\neq j= N^{i^\star}(\mathbf i);\\
x_{\binom {i^\star}{j}}^{-1}&\mbox{otherwise}.
\end{array}
\right.
\end{equation}
\end{lemma}

\begin{lemma}\label{lemma:inverse} Let $u,v\in W$ and $\mathbf i\in R(u,v)$.
For every cluster $\mathbf x\in{\mathcal X}_{\mathbf i}$, let $\mathbf{x^{\op}}
\in{\mathcal X}_{\mathbf{i^{\op}}}$ be the cluster such that the following equality
is satisfied.
$$\ev_{\mathbf i}(\mathbf x)^{-1}=\ev_{{\mathbf{i^{\op}}}}(\mathbf{x^{\op}})\ .$$
Then we have
$$x^{\op}_{\binom i{j}}=\left\{
\begin{array}{rl}
-x_{\binom {i}{N^{i}(\mathbf i)-j}}^{-1}&\mbox{if } 0=j\neq N^{i}(\mathbf i)
\mbox{ or } 0\neq j= N^{i}(\mathbf i);\\
x_{\binom {i}{N^{i}(\mathbf i)-j}}^{-1}&\mbox{otherwise}.
\end{array}
\right. $$
\end{lemma}

Let us denote $\kappa:G\times G\to G$ the composition of the conjugacy
map associated to the first argument $g\in G$ and the inverse map
$x\mapsto x^{-1}$ for the second argument $x\in G$, that is:
\begin{equation}\label{equ:kappa}
\kappa:(g,x)\mapsto gx^{-1}g^{-1}\ .
\end{equation}

\begin{prop}\label{prop:w0}For every $u,v\in W$, $\mathbf i\in R(u,v)$,
and every $\mathbf x\in{\mathcal X}_{\mathbf i}$, the following equality
is satisfied:
$$\kappa (\widehat{w_0},\ev_{\mathbf i}(\mathbf x))
=\ev_{\mathbf{i^\circlearrowright}}(\mathbf{x^\circlearrowright})\ .$$
\end{prop}
\begin{proof}The involutions defined by Lemma \ref{lemma:w_0},
Lemma \ref{lemma:inverse} and equation (\ref{equ:circlearrow})
on seed $\mathcal X$-tori clearly imply that
the relation $\circlearrowright=\op\circ \star$ remains valid on seed
${\mathcal X}$-tori. Therefore Lemma \ref{lemma:w_0}
and Lemma \ref{lemma:inverse} lead to the researched equality.
\end{proof}

\subsection{Chiral dual and other involutions}We introduce other involutions
on double words. These ones will be directly useful to describe the combinatorics
associated to the twist maps.

\begin{definition}\cite{FGdilog2}\label{def:chiraldual} The \emph{chiral dual} of a seed
$\mathbf I=(I,I_0,\varepsilon,d)$ is the seed
$\mathbf I^{\bigcirc}=(I,I_0,-\varepsilon,d)$. Chiral dual
induces an involutive map between the corresponding seed $\mathcal X$-tori,
which is denoted in the same way and given by
$x_{\bigcirc(i)}={x_i^{-1}}$.
\end{definition}


\begin{rem}It can be checked that the chiral dual commutes with cluster
transformations but not with tropical mutations.
\end{rem}

The chiral dual leads to the following involution on double reduced words,
also denoted~$\bigcirc$.
For every $u,v\in W$ and every double word $\mathbf i\in D(u,v)$,
let $\mathbf{i^{\bigcirc}}\in D(v,u)$ be the double word obtained by
applying the map $i\mapsto\overline{i}$ as a homomorphism on the double word $\mathbf i$.
It is clear that the double word $\mathbf{i^{\bigcirc}}$ is reduced
if and only if $\mathbf i$ is reduced. The following result is immediate.

\begin{lemma}For every double word $\mathbf i$, the chiral dual
$\mathbf{I(i)}^{\bigcirc}$ of the seed $\mathbf{I(i)}$ associated
to the double word $\mathbf i$ is the seed $\mathbf{I(i^{\bigcirc})}$
associated to the double word $\mathbf{i^{\bigcirc}}$.
\end{lemma}

Here is now the link with the group $G$.
Let us recall that the \emph{involutive Cartan group automorphism}
$\theta:G \rightarrow G:x \mapsto  x^{\theta}$ is
the map given by
\begin{equation}\label{equ:Cartan}
a^{\theta}=a^{-1}\in H,\ {E^i}^{\theta}=F^i,\ {F^i}^{\theta}=E^i.
\end{equation}

\begin{prop}\label{prop:chiral}For every $u,v\in W$ and every double reduced
word $\mathbf i\in R(u,v)$, the following diagram commutes. Its vertical edges
are labeled by birational Poisson isomorphisms whereas its horizontal edges
are labeled by birational anti-Poisson isomorphisms.
$$\xymatrix{
{{\mathcal X}_{\mathbf i}}\ar@/^1pc/[r]^{\bigcirc}\ar@/_1pc/[d]_{\ev_{\mathbf i}}
&{{\mathcal X}_{\mathbf {i^{\bigcirc}}}}\ar@/^1pc/[d]^{\ev_{\mathbf{i^{\bigcirc}}}}\\
 (G^{u,v},\pi_G)\ar@/_1pc/[r]_{\theta}&(G^{v,u},\pi_G)
}$$
\end{prop}
\begin{proof}Because the involution $\bigcirc$ commutes with the
amalgamated product, which intertwines the product on $G$ via the
evaluation map by equation (\ref{equ:defev}), and
because $\theta$ is a group automorphism for this product on $G$,
we just have to focus on the case $\mathbf i\in\{i,\overline{i}\}$.
The result is then easily derived from the definition of the involutions $\bigcirc$
and $\theta$, and the formula (\ref{equ:eps}).
\end{proof}

\subsection{Generalized cluster transformations and twist maps
on $(G,\pi_G)$}\label{section:Factorization}\label{section:twisted}
In this subsection, we use generalized cluster transformations to give the
cluster combinatorics underlying the Fomin-Zelevinsky twist maps.

\subsubsection{Tropical mutations and twist maps}
To $i\in[1,l]\cup[\overline{1},\overline{l}]$, we associated
the positive letter $| i|\in[1,l]$ given by the formula
\begin{equation}\label{equ:midi}
| i|=\left\{\begin{array}{ccc}
i&\mbox{if } i\in[1,l]\ ;\\
\overline{i}&\mbox{otherwise}\ .
\end{array}
\right.
\end{equation}
Let $w, w'\in W$. We denote $w\rightarrow w'$ if and only if we can find
a letter $i\in[1,l]$ such that $w=s_iw'$ and $\ell(w)=\ell(w')+1$, and denote
$\leq$ the \emph{right weak order} on $W$, i.e. $w'\leq w$ if there exists
a chain $w\rightarrow \dots\rightarrow w'$.

Let us recall that for every
$w'\leq w$, a reduced word $\mathbf i=i_1\dots i_{\ell(w)}\in  R(w)$ is
said to be \emph{adapted to $w'$} if we have the equality
$s_{i_1}\dots s_{i_{\ell(w')}}=w'$.
We extend this notation to every double word $\mathbf i=i_1\dots i_n$
by setting $|\mathbf i|:=| i_1|\dots| i_n|$.

\begin{definition}Let $e<w_1\leq u,e<w_2\leq v\in W$. A double reduced word
$\mathbf i=i_1\dots i_n\in R(u,v)$ is said to be \emph{$\mathfrak{L}$-adapted
to $w_1$} (resp. \emph{$\mathfrak{R}$-adapted to $w_2$}) if the reduced word
$$\begin{array}{ccc}
| i_1\dots i_{\ell(w_1)}| \in  R(w_1)&(\mbox{resp.}&| i_{\ell(u)+\ell(w_2^{-1}v)+1}
\dots i_{\ell(u)+\ell(v)}|\in  R(w_2))
\end{array}$$
is {adapted to $w_1$} (resp. adapted to $w_2$). And the double word $\mathbf i$
is \emph{$(w_1,w_2)$-adapted} if it is $\mathfrak{L}$-adapted to $w_1$ and
$\mathfrak{R}$-adapted to $w_2$. In particular, a $(w_1,w_2)$-adapted double
reduced word is $(w'_1,w'_2)$-adapted for every $w'_1\leq w_1,w'_2\leq w_2$.
\end{definition}

For example, the double reduced word $\mathbf i=\overline{2}\overline{1}2$
is $\mathfrak{L}$-adapted for the elements $s_2,s_2s_1\in W$, $\mathfrak{R}$-adapted for $s_2$,
$(s_2s_1,s_2)$-adapted, and $(s_2,s_2)$-adapted; whereas the double
reduced word $\mathbf j=2\overline{2}\overline{1}$ is neither $\mathfrak{L}$-adapted,
nor $\mathfrak{R}$-adapted, hence nor $(w_1,w_2)$-adapted, for any $w_1,w_2\in W$.

\begin{rem}For every $u,v\in W$, if the double reduced word $\mathbf i\in R(u,v)$
is $(u,v)$-adapted, then its first $\ell(u)$ letters (resp. $\ell(v)$ last letters)
are negative (resp. positive) and give a reduced expression for $u$ (resp. $v$).
Moreover, the following assertions are equivalent for every double reduced word
$\mathbf i\in R(u,v)$:
\begin{itemize}
\item
the double reduced word $\mathbf i$ is $(u,v)$-adapted;
\item
the double reduced word $\mathbf i$ is $\mathfrak{L}$-adapted to $u$;
\item
the double reduced word $\mathbf i$ is $\mathfrak{R}$-adapted to $v$.
\end{itemize}
\end{rem}

\begin{prop}\label{equ:trop}The following equalities are satisfied for every
$u,v\in W$, every $(u,v)$-adapted double word $\mathbf i=i_1\dots i_n\in R(u,v)$,
and every $\mathbf x\in{\mathcal X}_{\mathbf i}$.
$$\begin{array}{ccc}
[\ev_{\mathbf i}(\mathbf x)\widehat{v^{-1}}]_{\leq 0}&=&
[\ev_{{\mathfrak R}(\mathbf i)}\circ\mu_{\binom {i_n}{N^{i_n}(\mathbf i)}}
(\mathbf x)\widehat{s_{i_n}v^{-1}}]_{\leq 0}\\
{[\widehat{u}^{-1}\ev_{\mathbf i}(\mathbf x)]_{\geq 0}}&=&
[\widehat{s_{i_1}u}^{-1}\ev_{{\mathfrak L}(\mathbf i)}\circ\mu_{\binom {i_1}{0}}
(\mathbf x)]_{\geq 0}\ .
\end{array}$$
\end{prop}
\begin{proof}Let us remember the map $\varphi_j:\SL(2,\mathbb C)\hookrightarrow G$
defined in Subsection~\ref{section:Preliminaries1}. For any nonzero $t \in \mathbb C$
and any $i\in [1,l]$, let us denote
\begin{equation}
\label{eq:xnegative}
\begin{array}{cccccc}
x_{i} (t) =\varphi_i\left(
\begin{array}{cc}
1 & t\\
0 & 1
\end{array}
\right)
&,&
x_{\overline{i}} (t)=\varphi_i\left(
\begin{array}{cc}
1 & 0\\
t & 1
\end{array}
\right).
\end{array}
\end{equation}
For every $j\in[1,l]$, the following equality, easily checked on $\SL(2,\mathbb C)$ by
elementary matrix calculus, can be extended on $G$ using the map $\varphi_j$ of
Subsection~\ref{section:Preliminaries1}.
\begin{equation}\label{equ:Wsl2}{\widehat{s_j}}^{-1}x_{\overline{j}}(t)
=x_{\overline{j}}(-t^{-1})t^{h_j}x_j(t^{-1})\ .
\end{equation}
Using the definition of tropical mutation,
and the fact that tropical mutations commute with
the amalgamated product according to Remark \ref{rem:tropamal}, we deduce:
\begin{equation}\label{equ:si}{\widehat{s_{i_1}}}^{-1}\ev_{\mathbf{i}}(\mathbf {x})
=x_{\overline{i_1}}(-x^{-1}_{\binom{i_1}{0}})\ \ \ev_{{\mathfrak L}(\mathbf{i})}
\circ\mu_{\binom {i_1}{0}}(\mathbf {x})\ .
\end{equation}
Moreover, we have the inequality $\ell(s_{i_1}u)<\ell(u)$
because the double word $\mathbf i$ is $\mathfrak{L}$-adapted to $u$. This inequality
implies that $\widehat{s_{i_1}u}^{-1}x_{\overline{j}}(t)
\widehat{s_{i_1}u}\in N_-$ for every $t\in\mathbb C$.
Therefore the second equation is proved.
The first equation is proved in the same way, using the following equality
instead of (\ref{equ:Wsl2}).
\begin{equation}\label{equ:si2}x_{{i}}(t){\widehat{s_i}}=x_{\overline{i}}
(t^{-1})t^{h_i}x_i(-t^{-1})\ .
\end{equation}
\end{proof}

\begin{definition}\cite[Definition 1.5]{FZtotal}\label{def:twistmap} Let $u,v\in W$.
The $\emph{twist map}$ $\zeta_{\theta}^{u,v}:x\mapsto x'$ is the map defined by
\begin{equation}\label{equ:FZtwist}
x'=([\widehat{u}^{-1}x]_-^{-1}\widehat{u}^{-1}x\widehat{v^{-1}}
[x\widehat{v^{-1}}]_+^{-1})^{\theta}\ .
\end{equation}
Because of \cite[Theorem 1.6]{FZtotal}, the right side of (\ref{equ:FZtwist})
is well defined for every $x\in G^{u,v}$ and the twist map
$\zeta_{\theta}^{u,v}$ establishes a biregular isomorphism between
$G^{u,v}$ and $G^{u^{-1},v^{-1}}$. Let us define the related map
\begin{equation}\label{equ:defxi}
\begin{array}{cccc}
\zeta^{u,v}: & G^{u,v} & \longrightarrow & G^{v^{-1},u^{-1}}\\
                    & x & \longmapsto & [\widehat{u}^{-1}x]_-^{-1}
                    \widehat{u}^{-1}x\widehat{v^{-1}}[x\widehat{v^{-1}}]_+^{-1}.
\end{array}
\end{equation}
So we get the equality $\zeta_{\theta}^{u,v}=\theta\circ\zeta^{u,v}$.
Let us notice that for every $x\in G^{u,1}$ and $y\in G^{1,v}$ we
have the relations
$\zeta^{u,1}(x)=[\widehat{u}^{-1}x]_{\geq 0}$ {and} $\zeta^{1,v}(y)=[y\widehat{v}]_{\leq 0}$.
\end{definition}

\begin{rem}
Because the map (\ref{equ:defxi}) will appear a lot in what
follow, it will be useful to denote it also by the expression "twist map".
When we will need to avoid confusion, we will refer to the
map (\ref{equ:FZtwist}) as the "Fomin-Zelevinsky twist map".
\end{rem}

\subsubsection{Generalized cluster transformations and twist maps
on $(B_{\pm},\pi_G)$}
We are going to describe twist maps at the level of seed $\mathcal X$-tori.
Let us start to associate a generalized cluster transformation to any
twist map on $(B_{\pm},\pi_G)$.
To do that, we first need to sharpen the preceding involution $\square$
on double words. For every positive reduced word
$\mathbf i=i_1\dots i_n$, every  negative reduced word
$\mathbf j=j_1\dots j_n$, and every $k\in[1,n+1]$, we introduce the double
words
\begin{equation}\label{equ:xiword}
\begin{array}{ccllll}
\mathbf i(k)=\mathbf i(k)_-\mathbf i(k)_+& \mbox{where}&\mathbf i(k)_+
=i_1\dots i_{k-1} &\mbox{and}& \mathbf i(k)_-=\overline{i_n}\dots\overline{i_{k}}\\
\mathbf j(k)=\mathbf j(k)_-\mathbf j(k)_+& \mbox{where}&\mathbf j(k)_+
=j_{k-1}\dots j_{1} &\mbox{and}& \mathbf j(k)_-=\overline{j_{k}}\dots\overline{j_{n}}
\end{array}.
\end{equation}
The following insight on these double reduced words will be developed in
the next section, by considering the $W$-permutohedron. It is derived from
an easy induction on the number $k\in[1,\ell(w)]$ that appear in the statement.

\begin{lemma}\label{lemma:p65}Let $w\in W$, $\mathbf i=i_1\dots i_{\ell(w)}\in R(1,w)$
be a positive reduced word, and $w_{\geq k}=s_{i_k}\dots s_{i_{\ell(w)}}$
be the element of $w$ associated to any $k\in[1,k]$. The double reduced
word $\mathbf i(k)$ belongs to the set $R(w_{\geq k}^{-1},ww_{\geq k}^{-1})$.
\end{lemma}

\begin{rem}
The involution $\square$ on positive or negative words
is rediscovered from equation (\ref{equ:xiword}) because of
the following equalities
$\mathbf{i}=\mathbf i(n+1)$, $\mathbf{j}=\mathbf j(1)$
{and} $\mathbf{i^{\square}}=\mathbf i(1)$,
$\mathbf{j^{\square}}=\mathbf j(n)$.
\end{rem}

\begin{ex}Let us choose the positive reduced word $\mathbf i=121\in R(1,w_0)$,
when $\mathfrak g=A_2$. We then get the following double reduced words
$\mathbf i(4)=121$,
$\mathbf i(3)=\overline{1}12$, $\mathbf i(2)=\overline{1}\overline{2}1$,
and $\mathbf i(1)=\overline{1}\overline{2}\overline{1}$.
In the same way, if we consider the negative double word
$\mathbf j=\overline{1}\overline{2}\overline{1}\in R(w_0,1)$, we get
the following double reduced words
$\mathbf j(1)=\overline{1}\overline{2}\overline{1}$,
$\mathbf j(2)=\overline{1}\overline{2}1$, $\mathbf j(3)=\overline{1}12$
and $\mathbf j(4)=121$.
\end{ex}

For every positive reduced word
$\mathbf i=i_1\dots i_n$, every  negative reduced word
$\mathbf j=j_1\dots j_n$ and every $k\in[1,n]$, we define
the generalized cluster transformations
$\zeta_{\mathbf i(k)}:{\mathcal X}_{\mathbf i(k)}
\to{\mathcal X}_{\mathbf i(k-1)}$ and $\zeta_{\mathbf j(k)}:{\mathcal X}_{\mathbf j(k)}
\to{\mathcal X}_{\mathbf j(k-1)}$  by the following formulas
\begin{equation}\label{equ:defxi+2}
\begin{array}{clc}
&\zeta_{\mathbf i(k)}=\mu_{\binom{i_k}{N^{i_k}(\mathbf i(k)_-)}}\circ
\mu_{\binom{i_k}{N^{i_k}(\mathbf i(k)_-)+1}}\circ\dots
\circ\mu_{\binom{i_k}{N^{i_k}(\mathbf i)}}\\
\mbox{and}\\
&\zeta_{\mathbf j(k)}=\mu_{\binom{j_k}{0}}\circ
\mu_{\binom{j_k}{1}}\circ\dots
\circ\mu_{\binom{j_k}{N^{j_k}(\mathbf j(k)_-)-1}}\ .
\end{array}
\end{equation}
\begin{ex}
Here are the generalized cluster transformations related
to the double reduced words of the previous example.
$$\begin{array}{clc}
&\zeta_{\mathbf i(4)}=\mu_{\binom{1}{1}}\circ\mu_{\binom{1}{2}},\
\zeta_{\mathbf i(3)}=\mu_{\binom{2}{1}},\ \zeta_{\mathbf i(2)}=\mu_{\binom{1}{2}}\ ,\\
\mbox{and}\\
&\zeta_{\mathbf j(1)}=\mu_{\binom{1}{1}}\circ\mu_{\binom{1}{0}},\
\zeta_{\mathbf j(2)}=\mu_{\binom{2}{0}},\ \zeta_{\mathbf j(3)}=\mu_{\binom{1}{0}}\ .
\end{array}$$
\end{ex}

\begin{cor}\label{cor:tor} For every $u,v\in W$, $\mathbf i\in R(1,v)$
and $\mathbf j\in R(u,1)$, the following maps are Poisson birational
isomorphisms.
\begin{equation}\label{equ:tor}
\begin{array}{rccl}
\zeta_{\mathbf i(\geq k)}:&{\mathcal X}_{\mathbf i}&\longrightarrow &{\mathcal X}_{\mathbf i(k-1)}\\
           & \mathbf x          &\longmapsto     &\zeta_{\mathbf i(k)}\circ\dots\circ\zeta_{\mathbf i(n)}(\mathbf x)
\end{array}
\ \
\begin{array}{rccl}
\zeta_{\mathbf j(\leq k)}:&{\mathcal X}_{\mathbf j}&\longrightarrow &{\mathcal X}_{\mathbf j(k-1)}\\
           & \mathbf x          &\longmapsto     &\zeta_{\mathbf j(k)}\circ\dots\circ\zeta_{\mathbf j(1)}(\mathbf x)
\end{array} .
\end{equation}
\end{cor}

\begin{proof}We use Proposition \ref{prop:muttrop}, equation
(\ref{equ:defxi+2}), and the fact that mutations are Poisson
birational isomorphisms.
\end{proof}

\begin{rem} The generalized cluster transformations (\ref{equ:tor})
will be used in Section \ref{section:evaluationdual} to describe
the unipotent parts of the dual Poisson-Lie group $(G^*,\pi_{G^*})$.
\end{rem}

\begin{ex}\label{ex:twistA2}Let us keep the positive reduced word $\mathbf i=121\in R(1,w_0)$,
when $\mathfrak g=A_2$. The generalized cluster transformations associated
to the double reduced words of the example above are then
the following. (We denote $\mathbf x\in{\mathcal X}_{\mathbf i}$ the cluster
$(x_{\binom{1}{0}},x_{\binom{1}{1}},x_{\binom{1}{2}},x_{\binom{2}{0}},x_{\binom{2}{1}})$.)
$$\begin{array}{rrr}
\zeta_{\mathbf i(\geq 3)}(\mathbf x)=\zeta_{\mathbf i(3)}(\mathbf x)=\left(
\begin{array}{c}
x_{\binom{1}{0}}({1+x_{\binom{1}{1}}})
,\ x_{\binom{1}{1}}^{-1},\ x_{\binom{1}{2}}^{-1}({1+x_{\binom{1}{1}}}),\\
\\
x_{\binom{2}{0}}({1+x_{\binom{1}{1}}^{-1}})^{-1},\ x_{\binom{2}{1}}x_{\binom{1}{2}}
\end{array}
\right)\ ;\\
\\
\zeta_{\mathbf i(\geq 2)}(\mathbf x)=\zeta_{\mathbf i(2)}
\circ\zeta_{\mathbf i(3)}(\mathbf x)=\left(
\begin{array}{c}
x_{\binom{1}{0}}({1+x_{\binom{1}{1}}})
,\ x_{\binom{1}{1}}^{-1},\ x_{\binom{2}{1}}({1+x_{\binom{1}{1}}}),\\
\\
x_{\binom{2}{0}}({1+x_{\binom{1}{1}}^{-1}})^{-1},\ x_{\binom{2}{1}}^{-1}
x_{\binom{1}{2}}^{-1}
\end{array}
\right)\ ;\\
\\
\zeta_{\mathbf i(\geq 1)}(\mathbf x)=\zeta_{\mathbf i(1)}
\circ\zeta_{\mathbf i(2)}
\circ\zeta_{\mathbf i(3)}(\mathbf x)=\left(
\begin{array}{c}
x_{\binom{1}{0}}({1+x_{\binom{1}{1}}})
,\ x_{\binom{1}{1}}^{-1},\ x_{\binom{2}{1}}^{-1}({1+x_{\binom{1}{1}}})^{-1},\\
\\
x_{\binom{2}{0}}({1+x_{\binom{1}{1}}^{-1}})^{-1},
\ x_{\binom{1}{2}}^{-1}({1+x_{\binom{1}{1}}})
\end{array}
\right)\ .
\end{array}$$
\end{ex}

Special cases of the generalized cluster transformations (\ref{equ:tor})
are given by the following birational Poisson isomorphisms. These are the
ones that we are going to associate to twist maps on $(G,\pi_G)$.
\begin{equation}\label{equ:defxi+}
\begin{array}{lcl}
\zeta_{\mathbf i}:{\mathcal X}_{\mathbf{i}}
\to{\mathcal X}_{\mathbf{i^{\square}}}:\mathbf x\mapsto\zeta_{\mathbf i(\geq 1)}(\mathbf x)&
\mbox{and}&\zeta_{\mathbf j}:{\mathcal X}_{\mathbf{j}}\to{\mathcal X}_{\mathbf{j^{\square}}}:
\mathbf x\mapsto\zeta_{\mathbf j(\leq \ell(u))}(\mathbf x)\ .
\end{array}
\end{equation}

\begin{prop}\label{prop:twist-} Let $w'\leq u,w\leq v\in W$,
$\mathbf i\in R({w'}^{-1}u,{w'}^{-1})$ and $\mathbf j\in R(w^{-1},vw^{-1})$
be double reduced words respectively $({w'}^{-1}u,{w'}^{-1})$-adapted and
$(w^{-1},vw^{-1})$-adapted. The following equalities are satisfied
for every cluster $\mathbf x\in {\mathcal X}_{\mathbf i}$ and
$\mathbf y\in {\mathcal X}_{\mathbf j}$:
$$\begin{array}{ccc}
[\widehat{w'^{-1}u}^{-1}\ev_{\mathbf{i}}(\mathbf x)]_{\geq 0}
=\ev_{\mathbf{i_-^{\square}}\mathbf{i_+}}\circ\zeta_{\mathbf{i_-}}(\mathbf x)
&\mbox{and}&[\ev_{\mathbf{j}}(\mathbf y)\widehat{wv^{-1}}]_{\leq 0}
=\ev_{\mathbf{j_-}\mathbf{j_+^{\square}}}\circ\zeta_{\mathbf{j_+}}(\mathbf y)\ .
\end{array}
$$
\end{prop}
\begin{proof}Let $w\in W$ and $j\in[1,l]$ be such that $w<ws_j\leq v$ for
the right weak order. The following equality is satisfied for every $g\in G^{v,1}$.
$$[\widehat{s_j}^{-1}[\widehat{w}^{-1}g]_{\geq 0}]_{\geq 0}
=[\widehat{s_j}^{-1}\widehat{w}^{-1}g]_{\geq 0}\ .$$
An induction on the length of $v$, involving at each step the equation
(\ref{equ:si}), Theorem \ref{fg} and the definition (\ref{equ:defxi+})
leads to the second equality. The first equality is proved
in the same way.
\end{proof}

\begin{cor}\label{cor:twistborel} For every $u,v\in W$, every (double) reduced words $\mathbf i\in R(1,u)$
and $\mathbf j\in R(v,1)$, and every cluster $\mathbf x\in {\mathcal X}_{\mathbf i}$ and
$\mathbf y\in {\mathcal X}_{\mathbf j}$, the following equalities are
satisfied:
\begin{equation}\label{equ:torev}\begin{array}{ccc}
[\widehat{u}^{-1}\ev_{\mathbf{i}}(\mathbf x)]_{\geq 0}
=\ev_{\mathbf{i^{\square}}}\circ\zeta_{\mathbf i}(\mathbf x)
&\mbox{and}&[\ev_{\mathbf{j}}(\mathbf y)\widehat{v^{-1}}]_{\leq 0}
=\ev_{\mathbf{j^{\square}}}\circ\zeta_{\mathbf j}(\mathbf y)\ .
\end{array}
\end{equation}
\end{cor}
\begin{proof}We apply Proposition \ref{prop:twist-} with $w=w'=e$.
\end{proof}

\begin{cor}\label{lemma:tormut} The following
equality is satisfied for every $v\in W$ and every reduced words
$\mathbf i,\mathbf j\in R(1,v)$, or $\mathbf i,\mathbf j\in R(v,1)$
$$\mu_{\mathbf{i^{\square}}\rightarrow
\mathbf{j^{\square}}}\circ\zeta_{\mathbf i}=\zeta_{\mathbf j}\circ
\mu_{\mathbf i\rightarrow\mathbf j}\ .$$
\end{cor}
\begin{proof}We suppose that $\mathbf i,\mathbf j\in R(1,v)$.
Let us recall that the involution $\square$ maps double reduced
words to double reduced words, and that the evaluation map $\ev_{\mathbf j}$ associated
to any double reduced word $\mathbf j$ is birational because of Theorem
\ref{thm:evG}. Therefore an equality $\mathbf y=\mathbf z$ between cluster variables on
${\mathcal X}_{\mathbf{i^{\square}}}$ is satisfied if and only if
the equality $\ev_{\mathbf{i^{\square}}}(\mathbf y)
=\ev_{\mathbf{i^{\square}}}(\mathbf z)$ is satisfied on $G$. Now, it suffices
to apply Theorem \ref{fg} and the second equation of (\ref{equ:torev}) to
obtain the following equality for every $\mathbf x\in {\mathcal X}_{\mathbf i}$.
The case $\mathbf i,\mathbf j\in R(v,1)$ is proved in the same way.
$$\ev_{\mathbf{i^{\square}}}\circ\zeta_{\mathbf i}(\mathbf x)
=[\ev_{\mathbf{j}}\circ\mu_{\mathbf i\rightarrow\mathbf j}
(\mathbf x)\widehat{v^{-1}}]_{\leq 0}=\ev_{\mathbf{j^{\square}}}
\circ\zeta_{\mathbf j}\circ\mu_{\mathbf i\to \mathbf j}(\mathbf x)
=\ev_{\mathbf{i^{\square}}}\circ\mu_{\mathbf{j^{\square}}\rightarrow
\mathbf{i^{\square}}}\circ\zeta_{\mathbf j}\circ
\mu_{\mathbf i\rightarrow\mathbf j}(\mathbf x)\ .$$
\end{proof}

\subsubsection{From twist maps on $(B_{\pm},\pi_G)$ to twist maps on $(G,\pi_G)$}
\label{def:posnegpart}
We are now ready to give the generalized cluster transformations associated to any
twist maps on $(G,\pi_G)$.
For every double word $\mathbf i$, let $\mathbf{i_+}$ and $\mathbf{i_-}$ be
respectively the words obtained by erasing all the negative and positive letters
of $\mathbf{i}$, without changing the order of the remaining letters. The word $\mathbf{i_+}$ (resp.
$\mathbf{i_-}$) is called the \emph{positive part} (resp. \emph{negative
part}) of $\mathbf i$. In particular, the double word $\mathbf{i}$
is linked by compositions of mixed $2$-moves to the double words
$\mathbf{i_+}\mathbf{i_-}$ and $\mathbf{i_-}\mathbf{i_+}$.
(This definition is compatible with the notation used in
equation (\ref{equ:xiword}).)
Following Corollary \ref{lemma:tormut}, we then introduce for every
$\mathbf i\in R(u,v)$ the maps $\zeta_{\mathbf{i}}:
{\mathcal X}_{\mathbf i}\rightarrow{\mathcal X}_{\mathbf{i^{\square}}}$
and $\zeta_{\mathbf{i_{\bigcirc}}}:{\mathcal X}_{\mathbf i}
\rightarrow{\mathcal X}_{\mathbf{i^{\op}}}$ by the following formulas
and state the main result of this subsection.
\begin{equation}\label{equ:extzeta}
\begin{array}{ccccc}
{\zeta}_{\mathbf i}=\mu_{\mathbf{i_-^{\square}}\mathbf{i_+^{\square}}
\rightarrow\mathbf {i^{\square}}}\circ\zeta_{\mathbf{i_-}}\circ
\zeta_{\mathbf{i_+}}\circ\mu_{\mathbf i\rightarrow\mathbf{i_-}\mathbf{i_+}}
&&\mbox{    and    }&&
{{\zeta}_{\mathbf i}}_{\bigcirc}=\bigcirc\circ\zeta_{\mathbf{i}}\ .
\end{array}
\end{equation}

\begin{thm}\label{thm:twist}For every $u,v\in W$ and every double
reduced word $\mathbf i\in R(u,v)$, the following diagrams are commutative.
All edges of the left diagram are Poisson birational isomorphisms,
whereas vertical edges of the right diagram are Poisson birational
isomorphisms and horizontal edges are anti-Poisson birational
isomorphisms.

$$\xymatrix{
{{\mathcal X}_{\mathbf i}}\ar@/^1pc/[r]^{\zeta_{\mathbf i}}
\ar@/_1pc/[d]_{\ev_{\mathbf i}}&{{\mathcal X}_{\mathbf {i^{\square}}}}
\ar@/^1pc/[d]^{\ev_{\mathbf{i^{\square}}}}\\
 (G^{u,v},\pi_G)\ar@/_1pc/[r]_{\zeta^{u,v}}&(G^{{v}^{-1},{u}^{-1}},\pi_G)
}
\ \ \ \ \ \ \ \ \ \ \ \ \
\xymatrix{
{{\mathcal X}_{\mathbf i}}\ar@/^1pc/[r]^{{\zeta_{\mathbf i}}_{\bigcirc}}
\ar@/_1pc/[d]_{\ev_{\mathbf i}}&{{\mathcal X}_{\mathbf {i^{\op}}}}
\ar@/^1pc/[d]^{\ev_{\mathbf{i^{\op}}}}\\
 (G^{u,v},\pi_G)\ar@/_1pc/[r]_{\zeta^{u,v}_{\theta}}&(G^{{u}^{-1},{v}^{-1}},\pi_G)
}$$
\end{thm}
\begin{proof}Let first notice that Proposition \ref{prop:chiral} and the definition
of $\zeta^{u,v}$ and $\zeta^{u,v}_{\theta}$ imply that the commutativity
of the right diagram can be derived from the commutativity of the
left one. So, let us focus on the left one. The cases $(u,v)=(u,1)$
and $(u,v)=(1,v)$ are proved by Proposition \ref{prop:twist-}. Moreover,
among all the remaining cases, it suffices to prove the case $\mathbf i=\mathbf{i_-}\mathbf{i_+}$
(with $\mathbf{i_-}\in R(u,1)$ and $\mathbf{i_+}\in R(1,v)$),
because of the definition (\ref{equ:extzeta}) of $\zeta_{\mathbf i}$ and Theorem \ref{fg}.
The demonstration relies on the following equality (\ref{equ:gsv}), borrowed to \cite[Theorem 3.1]{GSV}. For every
$x\in G^{u,v}$, the definition (\ref{equ:defxi}) leads to:

\begin{equation}\label{equ:gsv}
\begin{array}{ccl}
\zeta^{u,v}(x)& =& [\widehat u^{-1}x]_-^{-1}\widehat u^{-1}x
                \widehat{v^{-1}}[x\widehat{v^{-1}}]_+^{-1}\\
                  & =& {[\widehat u^{-1}x]}_0{[\widehat u^{-1}x]}_+v^{-1}
                  {{[x\widehat{v^{-1}}]}_+}^{-1}\\
                  & =& {[\widehat u^{-1}[x]_-[x]_0[x]_+]}_0{[\widehat u^{-1}[x]_-[x]_0[x]_+]}_+
                  (x^{-1}x)\widehat{v^{-1}}{{[x\widehat{v^{-1}}]}_+}^{-1}\\
                  & =& {[\widehat u^{-1}[x]_{\leq 0}]}_0{[\widehat u^{-1}[x]_-]}_+
                  [x]_+x^{-1}{[x\widehat{v^{-1}}]}_{\leq 0}\\
                  & =& {[\widehat u^{-1}[x]_{\leq 0}]}_{\geq 0}[x]_+x^{-1}[x]_-
                  {[[x]_{\geq 0}\widehat{v^{-1}}]}_{\leq 0}\\
                  & =& \zeta^{u,1}([x]_{\leq 0})[x]_0^{-1}\zeta^{1,v}([x]_{\geq 0})\ .
\end{array}
\end{equation}
Let $\mathbf i=\mathbf {i_-}\mathbf {i_+}\in R(u,v)$ and
$\mathbf x\in{\mathcal X}_{\mathbf i}$, $\mathbf{x_-}\in{\mathcal X}_{\mathbf{i_-}}$,
$\mathbf{x_+}\in{\mathcal X}_{\mathbf{i_+}}$ be the cluster variables such
that the equalities $x=\ev_{\mathbf i}(\mathbf x)$
and $\mathfrak{m}(\mathbf{x_-},\mathbf{x_+})=\mathbf x$ are satisfied.
We start by introducing the following maps $\pi_{\mathbf j}:{\mathcal X}_{\mathbf j}
\to{\mathcal X}_{\mathbf 1}$, for every double word $\mathbf j$, given by:
\begin{equation}\label{equ:foldpi}
\begin{array}{cccc}
\pi_{\mathbf j}:&{\mathcal X}_{\mathbf j}&\to&{\mathcal X}_{\mathbf 1}\\
&x_{\pi_{\mathbf i}(\binom{i}{j})}&=&x_{\binom{i}{0}}x_{\binom{i}{1}}
\dots x_{\binom{i}{N^i(\mathbf i)}}\ .
\end{array}
\end{equation}
We use these maps to define the elements
$\mathbf{x_{\leq 0}}\in{\mathcal X}_{\mathbf {i_-}}$,
$\mathbf{x_0}\in{\mathcal X}_{\mathbf 1}$ and
$\mathbf{x_{\geq 0}}\in{\mathcal X}_{\mathbf {i_+}}$ related
to $\mathbf x$ in the following way. (The associated equalities are
easily proved.)

\begin{equation}\label{equ:evgauss}
\begin{array}{lllll}
\mathbf{x_{\leq 0}}=\mathfrak{m}(\mathbf{x_-},\pi_{\mathbf{i_+}}(\mathbf {x_+}))&,&
\mathbf{x_0}=\pi_{\mathbf i}(\mathbf x)&\mbox{and} &
\mathbf{x_{\geq 0}}=\mathfrak{m}(\pi_{\mathbf{i_-}}(\mathbf {x_-}),\mathbf{x_+})\ .\\
\\
{[x]}_{\leq 0}=\ev_{\mathbf{i_-}}(\mathbf{x_{\leq 0}})&, &
{[x]}_0^{-1}=\ev_{1}(\mathbf{x_0^{\bigcirc}})&\mbox{and} &
{[x]}_{\geq 0}=\ev_{\mathbf{i_+}}(\mathbf{x_{\geq 0}})\ .
\end{array}
\end{equation}
From Remark \ref{rem:tropamal} we now get the following relations:
$$
\begin{array}{lllll}
\zeta_{\mathbf{i_-}}(\mathbf{x_{\leq 0}})=\mathfrak{m}(\zeta_{\mathbf{i_-}}
(\mathbf{x_-}),\pi_{\mathbf{i_+}}(\mathbf {x_+}))&\mbox{and} &
\zeta_{\mathbf{i_+}}(\mathbf{x_{\geq 0}})=\mathfrak{m}(\pi_{\mathbf{i_-}}(\mathbf {x_-})
,\zeta_{\mathbf{i_+}}(\mathbf{x_+}))\ ;
\end{array}
$$
$$
\begin{array}{c}
\zeta_{\mathbf{i}}(\mathbf x)=\zeta_{\mathbf{i_-}}\circ\zeta_{\mathbf{i_+}}(\mathbf {x})
=\mathfrak{m}(\zeta_{\mathbf{i_-}}(\mathbf{x_{\leq 0}}),\mathbf{x_0^{\bigcirc}}
,\zeta_{\mathbf{i_+}}(\mathbf{x_{\geq 0}}))
=\zeta_{\mathbf{i_+}}\circ\zeta_{\mathbf{i_-}}(\mathbf {x})\ .\\
\\
\end{array}$$
Proposition \ref{prop:twist-}
and equalities (\ref{equ:gsv}), (\ref{equ:evgauss}) then lead to
$\zeta^{u,v}(\ev_{\mathbf i}(\mathbf x))=\ev_{\mathbf{i^{\square}}}
(\zeta_{\mathbf{i}}(\mathbf{x}))$.
Finally, the Poisson and birational statements are clear from
Theorem \ref{fg}, Proposition \ref{cor:twistborel} and Proposition
\ref{prop:chiral}.
\end{proof}

\section{$\tau$-combinatorics, $W$-permutohedron, and evaluations
on $(G,\pi_G)$}
\label{section:taucombi}
We will refer to \emph{$\tau$-combinatorics} as the combinatorics on
double reduced words generated by generalized $d$-moves and enriched
with right tropical moves. The idea is to prepare the ground for the
cluster combinatorics related to twisted evaluations and dual Poisson-Lie
groups, developed in Section \ref{section:Loop}. Here, we associate a
family of cluster $\mathcal X$-varieties to every double Bruhat cell
$G^{w,1}$ by linking cluster $\mathcal{X}$-varieties with tropical
mutations via the $W$-permutehedron associated to the Lie algebra
$\mathfrak g$.

\subsection{The $W$-permutohedron, $\uparrow$-paths and $\downarrow$-paths}
We recall here that any reduced expression of any element of $W$ can be described
as a monotone paths on a particular polytope: the $W$-permutohedron.
(or moment polytope, or weight polytope \cite{PPermutohedra}).
Let us recall that $\Lambda$ denotes the integer weight lattice
associated to $\mathfrak g$ and denote $\Lambda_{\mathbb R}= \Lambda\otimes{\mathbb R}$
the weight space.  The roots in $\Pi$ span the root lattice $L\subseteq \Lambda$.
The associated  Weyl group $W$ acts on the weight space $\Lambda_{{\mathbb R}}$.
For $x\in \Lambda_{{\mathbb R}}$, we can define the {\it $W$-permutohedron}
$P_W(x)$ as the convex hull of a Weyl group orbit:
$$
P_W(x) :=\mathrm{ConvexHull}(w(x)| w\in W)\subset \Lambda_{{\mathbb R}}.
$$
For the Lie type $A_n$, the $W$-permutohedron
$P_W(x)$ is the permutohedron $P_{n+1}(x)$  defined as the convex hull
of all vectors obtained from $(x_1,\dots,x_{n+1})$
by permutations of the coordinates:
$$
P_{n+1}(x_1,\dots,x_{n+1}):=\mathrm{ConvexHull}((x_{w(1)},\dots,x_{w(n+1)})
| w\in S_{n+1}).
$$
From now on, we fix a generic $x\in \Lambda_{{\mathbb R}}$ such that the
associated $W$-permutohedron $P_W(x)$ has maximal dimension. This
$W$-permutohedron will be (abusively) denoted $P_W$.
%
It is well-known that we can label vertices and edges of $P_W$ respectively
by the set of elements of $W$ and the set of elementary reflections $s_i\in S$
that generates $W$, in such a way that
\begin{itemize}
\item
every vertex has a different label;
\item
every labeled vertices $w_1$ and $w_2$ of $P_W$ are
related by a labeled edge $s_i$ if and only if the equality $w_2=w_1s_i$ is satisfied.
\end{itemize}
In particular, the number of vertices of $P_W$ is given by the cardinal of $W$.
When we draw a picture of the $W$-permutohedron $P_W$, the bottom vertex can be
associated with the identity element $1\in W$, so that the top vertex is
the longest element $w_0\in W$. As remarked in \cite{FR}, a reduced word
for $w$ then corresponds to a path along edges from $1$ to $w$ which moves
up in a monotone fashion.
Let us call  \emph{$\uparrow$-path} a path along edges of $P_W$ which moves up in a monotone
fashion on $P_W$.
In the same way, a path which moves down in a monotone fashion will be called
a \emph{$\downarrow$-path}. A $\uparrow$-path relating the vertex $w$ to the
vertex $w'$ is called a \emph{$w\nearrow w'$-path} and the corresponding $\downarrow$-path
is is called a \emph{$w'\searrow w$-path}. In particular, a $\uparrow$-path along edges from
$1$ to $w$ is called a \emph{$\uparrow_w$-path}. The following result is clear.

\begin{prop}Let $u\leq v\in W$. The $u\nearrow v$-paths (resp. $v\searrow u$-paths)
are in bijection with reduced
expressions of the element $vu^{-1}\in W$ (resp. $uv^{-1}\in W$).
In particular, for a given $w\in W$, the number of $\uparrow_w$-paths
is equal to the number of reduced expressions of $w$.
\end{prop}

Some $\uparrow$-path and $\downarrow$-path on the permutohedron
$P_3$ are given by Figure \ref{fig:pathWpermutohedre}. Let us notice that
the $\uparrow$-path at the left of Figure \ref{fig:pathWpermutohedre} is
a $\uparrow_{s_1s_2}$-path and that it is the only one.
In fact the only $w\in W$ such that the related $\uparrow_w$-path is not unique
is $w=w_0$ and there are then two $\uparrow_{w_0}$-paths.

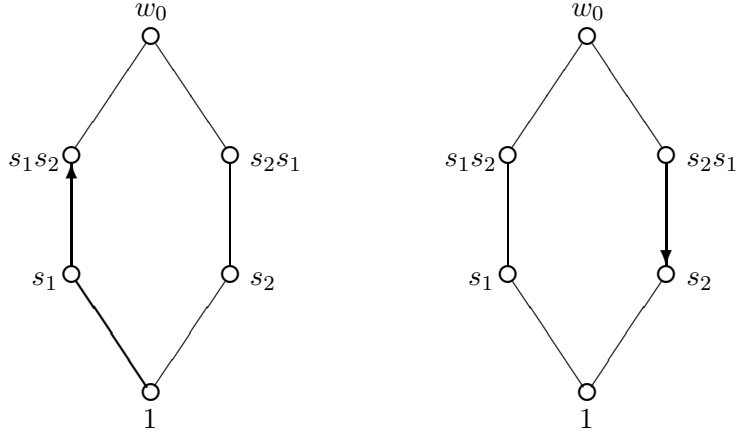
\begin{figure}[htbp]
\begin{center}
\setlength{\unitlength}{1.5pt}
\qquad\qquad\qquad\qquad\qquad
\begin{picture}(20,110)(0,-80)
\put(-29,-63.5){\line(2,3){17.7}}
\put(-10,-33){\line(0,1){26}}
\put(-11,-3.5){\line(-2,3){17.7}}
\put(-49,-3.5){\line(2,3){17.7}}
\thicklines
\put(-31,-63.5){\line(-2,3){17.7}}
\put(-50,-33){\vector(0,1){26}}
\put(-30,25){\circle{4}}
\put(-34,30){$w_0$}
\put(-50,-5){\circle{4}}
\put(-66,-8){$s_1s_2$}
\put(-10,-5){\circle{4}}
\put(-5,-8){$s_2s_1$}
\put(-50,-35){\circle{4}}
\put(-60,-38){$s_1$}
\put(-10,-35){\circle{4}}
\put(-5,-38){$s_2$}
\put(-30,-65){\circle{4}}
\put(-32,-74){$1$}
\end{picture}
\qquad\qquad\qquad\qquad\qquad
\qquad
\begin{picture}(20,110)(0,-80)
\put(-29,-63.5){\line(2,3){17.7}}
\put(-31,-63.5){\line(-2,3){17.7}}
\put(-50,-33){\line(0,1){26}}
\put(-49,-3.5){\line(2,3){17.7}}
\put(-11,-3.5){\line(-2,3){17.7}}
\thicklines
\put(-10,-7){\vector(0,-1){26}}
\put(-30,25){\circle{4}}
\put(-34,30){$w_0$}
\put(-50,-5){\circle{4}}
\put(-66,-8){$s_1s_2$}
\put(-10,-5){\circle{4}}
\put(-5,-8){$s_2s_1$}
\put(-50,-35){\circle{4}}
\put(-60,-38){$s_1$}
\put(-10,-35){\circle{4}}
\put(-5,-38){$s_2$}
\put(-30,-65){\circle{4}}
\put(-32,-74){$1$}
\end{picture}
\end{center}
\vspace{-.1in}
\caption{A $\uparrow$-path and a $\downarrow$-path on the $W$-permutohedron
$P_3$}
\label{fig:pathWpermutohedre}
\end{figure}

\subsection{The set $R^{\tau}(w)$}
It turns out, as seen in Remark \ref{rem:stabletauset}, that we can defined
stable subsets $R^{\tau}(w)$ of double reduced words. In fact, their combinatorics,
involving right $\tau$-moves and generalized $d$-moves (or left $\tau$-moves and
generalized $d$-moves) is given by the $W$-permutohedron associated
to $\mathfrak{g}$. In this subsection, we choose to focus on the right $\tau$-moves,
but the same combinatorics could be developed by considering left $\tau$-moves.

\begin{definition}
Let $\mathbf i$ be a double word. A \emph{right $d^\tau$-move}
(or simply a \emph{$d^\tau$-move}) on $\mathbf i$ is given by one
of these transformations:
\begin{itemize}
\item
a generalized $d$-move;
\item
a right $\tau$-move.
\end{itemize}
For every $w\in W$, let $R^{\tau}(w)$ be the set of
all the double words obtained from a word $\mathbf i\in R(1,w)$ by
composition of $d^\tau$-moves. (The choice of the double word
$\mathbf i$ doesn't matter, because of Theorem \ref{thm:Tits}.)
\end{definition}

It is easy to see that this definition coincides with the one given in
Remark \ref{rem:stabletauset}: the set $R^{\tau}(w)$ is the union of the disjoint sets
$R({w'}^{-1},w{w'}^{-1})$ for every $w'\leq w\in W$, that is
\begin{equation}\label{equ:Rwdoublered}
R^{\tau}(w)=\displaystyle\bigcup_{w'\leq w\in W}R({w'}^{-1},w{w'}^{-1})\ .
\end{equation}
(In particular, to every $\mathbf i\in R^{\tau}(w)$ there exists $w'\in W$
such that $\mathbf i\in R({w'}^{-1},w{w'}^{-1})$.)

In the same way Lemma \ref{lemma:transitivedoubleword} relates generalized
$d$-moves to the set $R(u,v)$ of double reduced words associated to the elements
$u,v\in W$, the link between $d^{\tau}$-moves to the set $R^{\tau}(w)$ associated to
$w\in W$ is the following, whose proof is immediate.

\begin{lemma}A double reduced word $\mathbf j\in R^{\tau}(w)$ can be
obtained from a double reduced word $\mathbf i\in R^{\tau}(w')$ by a
sequence of $d^{\tau}$-moves $\delta_{\mathbf i\to\mathbf j}^{\tau}$ if
and only if the equality $w'=w$ on $W$ is satisfied.
\end{lemma}

The link between double reduced words and the set $R^{\tau}(w)$ given
by equation (\ref{equ:Rwdoublered}) can be strengthened by relating
$R^{\tau}(w)$ to the vertices of the $W$-permutohedron in the
following way. Let us recall that to $i\in[1,l]\cup[\overline{1},\overline{l}]$,
we can associate the positive letter $| i|\in[1,l]$ given by the formula
(\ref{equ:midi}). Let us denote, for every double word $\mathbf i=i_1\dots i_n$,
$\mathfrak{R}_{|i_n|}(\mathbf i)$ (resp. $\mathfrak{L}_{|i_1|}(\mathbf i)$)
the result of a right (resp. left) $\tau$-move on~$\mathbf i$.

\begin{lemma}\label{lemma:PW}
We have the following statements for every $w\in W$.
\begin{itemize}
\item
The set $R^{\tau}(w)$ is the disjoint union of the
labels $R({w'}^{-1},w{w'}^{-1})$ associated to the vertices $w'$
of the $W$-permutohedron $P_W$ that are crossed or reached by a
$\uparrow_w$-path.
\item
Two labeled vertices $R(u,v)\subset R^{\tau}(w)$ and $R(u',v')
\subset R^{\tau}(w)$ of $P_W$ are related by the edge $s_j$ if and only
if there exist double reduced words $\mathbf i\in R(u,v)$ and
$\mathbf j\in R(u',v')$ such that $\mathbf j={\mathfrak R}_j(\mathbf i)$.
\end{itemize}
\end{lemma}
\begin{proof}Let $W_w\subset~W$ be the set of elements $w'$ such that $w'\leq w$.
Noticing that the map $W_w\to R^{\tau}(w):w'\mapsto R({w'}^{-1},w{w'}^{-1})$
is a bijection for every $w\in W$, the first statement is just a translation
of equation (\ref{equ:Rwdoublered}). So let us consider the second
statement. It is easy to see that if double reduced words $\mathbf i\in R(u,v)$
and $\mathbf j\in R(u',v')$ are related by a right $\tau$-move ${\mathfrak R}_j$,
then there exist $w_1$ and $w_2$ such that $\mathbf i\in R({w_1}^{-1},w_0{w_1}^{-1})$
and $\mathbf j\in R({w_2}^{-1},w_0{w_2}^{-1})$ and $w_2=w_1s_j$ by using
Lemma \ref{lemma:p65}. Now, if the vertices $w'_1$ and $w'_2$ are linked by the edge $s_j$,
we have the equality $w'_2=w'_1s_j$. The associated set of double reduced words are
therefore $R({w'_1}^{-1},w_0{w'_1}^{-1})$ and $R(s_j{w'_1}^{-1},w_0s_j{w'_1}^{-1})$,
and we still use Lemma \ref{lemma:p65} to get some double reduced words $\mathbf i$
and $\mathbf j$ such that $\mathbf j=\mathfrak{R}_j(\mathbf i)$.
\end{proof}

\begin{ex}\label{ex:WRstable}
Let $\mathfrak g=A_2$. Figure \ref{fig:Wpermutohedre} gives
examples of sets $R^{\tau}(w)$ and their link with the permutohedron $P_3$.
\end{ex}

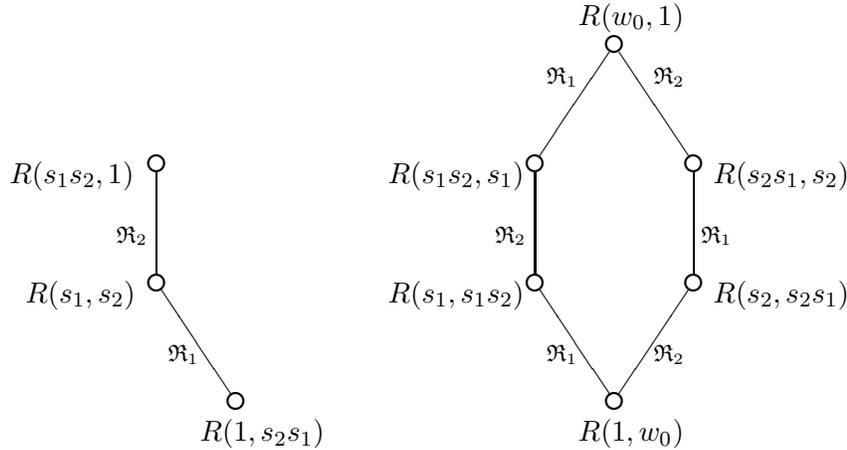
\begin{figure}[htbp]
\begin{center}
\setlength{\unitlength}{1.5pt}
\qquad\qquad\qquad\qquad\qquad
\begin{picture}(20,105)(0,-80)
\put(-31,-63.5){\line(-2,3){17.7}}
\put(-50,-33){\line(0,1){26}}
\thicklines
\put(-50,-5){\circle{4}}
\put(-87,-10){$R({s_1s_2,1})$}
\put(-50,-35){\circle{4}}
\put(-83,-40){$R({s_1,s_2})$}
\put(-30,-65){\circle{4}}
\put(-39,-75){$R(1,s_2s_1)$}
\put(-60,-25){\footnotesize ${\mathfrak R}_2$}
\put(-47,-55){\footnotesize ${\mathfrak R}_1$}
\end{picture}
\qquad\qquad\qquad\qquad\qquad
\begin{picture}(20,105)(0,-80)
\put(-29,-63.5){\line(2,3){17.7}}
\put(-10,-33){\line(0,1){26}}
\put(-11,-3.5){\line(-2,3){17.7}}
\put(-31,-63.5){\line(-2,3){17.7}}
\put(-50,-33){\line(0,1){26}}
\put(-49,-3.5){\line(2,3){17.7}}
\thicklines
\put(-30,25){\circle{4}}
\put(-39,30){$R({w_0,1})$}
\put(-50,-5){\circle{4}}
\put(-87,-10){$R({s_1s_2,s_1})$}
\put(-10,-5){\circle{4}}
\put(-5,-10){$R({s_2s_1,s_2})$}
\put(-50,-35){\circle{4}}
\put(-87,-40){$R({s_1,s_1s_2})$}
\put(-10,-35){\circle{4}}
\put(-5,-40){$R({s_2,s_2s_1})$}
\put(-30,-65){\circle{4}}
\put(-39,-75){$R(1,w_0)$}
\put(-20,15){\footnotesize ${\mathfrak R}_2$}
\put(-47,15){\footnotesize ${\mathfrak R}_1$}
\put(-60,-25){\footnotesize ${\mathfrak R}_2$}
\put(-8,-25){\footnotesize ${\mathfrak R}_1$}
\put(-47,-55){\footnotesize ${\mathfrak R}_1$}
\put(-20,-55){\footnotesize ${\mathfrak R}_2$}
\end{picture}
\end{center}
\vspace{-.1in}
\caption{The sets $R^{\tau}(s_2s_1)$ and $R^{\tau}(w_0)$ when $\mathfrak{g}=A_2$}
\label{fig:Wpermutohedre}
\end{figure}

\subsection{The family $\mathcal{X}^{\tau(w)}$ of cluster $\mathcal X$-varieties
related to $G^{1,w}$}
\label{section:mutau}
The same ideas can be applied at the level of cluster $\mathcal X$-varieties.
Let us remember the notations of Definition \ref{def:RLtropmut}.
To any double reduced words $\mathbf i,\mathbf{i'}\in R^{\tau}(w)$ such that
there exists a $d^{\tau}$-move $\delta:\mathbf{i}\rightarrow\mathbf{i'}$ we associate the
generalized cluster transformation $\mu^{\tau}_{\mathbf i\rightarrow \mathbf{i'}}:
{\mathcal X}_{\mathbf i}\rightarrow{\mathcal X}_{\mathbf{i'}}$ given by
\begin{itemize}
\item
the cluster transformation $\mu_{\mathbf i\rightarrow \mathbf{i'}}$ if $\delta$ is
a generalized $d$-move;
\item
the tropical mutation $\mu_{\lozenge^{\mathfrak R}_i}$ if
$\delta$ is the $\tau$-move ${\mathfrak R}_i$.
\end{itemize}
We extend this definition to every $\mathbf i,\mathbf{j}\in  R^{\tau}(w)$
in the following way. If $\mathbf i,\mathbf j$ are double words linked by a
sequence $\delta_{\mathbf i\to\mathbf j}^{\tau}$ of $d^{\tau}$-moves and
$\mathbf{i}\to\mathbf{i_1}\rightarrow\dots\rightarrow\mathbf{i_{n-1}}
\rightarrow\mathbf{j}$ is the associated chain of elements, we define the map
${\mu}^{\tau}_{\mathbf{i}\rightarrow\mathbf{j}}$ as the composition
${\mu}^{\tau}_{\mathbf{i_{n-1}}\rightarrow\mathbf j}\circ \dots\circ
{\mu}^{\tau}_{\mathbf{i}\rightarrow\mathbf{i_1}}$.

Because the birational Poisson isomorphism $\mu_{\mathbf i\to\mathbf j}^{\tau}:
{\mathcal X}_{\mathbf i}\to{\mathcal X}_{\mathbf j}$ associated to such a sequence
$\delta_{\mathbf i\to\mathbf j}^{\tau}$ is a generalized cluster transformation for every
$\mathbf i,\mathbf j\in R^{\tau}(w)$, we get a family of cluster
$\mathcal X$-varieties $\mathcal{X}_{|\mathbf i|}$ associated to the set $R^{\tau}(w)$
and related by tropical mutations, that we denote ${\mathcal X}^{\tau}_w$.
The combinatorics is in fact encoded
by the $W$-permutohedron $P_W$, as stated by the following
result, straightforwardly deduced from
Lemma \ref{lemma:PW}.

\begin{lemma}\label{lemma:Wpermtrop}Let $w\in W$. Replace each label
$w'\in W$ of a vertex of the $W$-permutohedron $P_W$
by the cluster $\mathcal X$-variety ${\mathcal X}^{{w'}^{-1},w{w'}^{-1}}$.
We then have the following properties.
\begin{itemize}
\item
The family ${\mathcal X}_{w}^\tau$ of cluster $\mathcal X$-varieties
contains the cluster $\mathcal X$-variety $\mathcal X^{{w'}^{-1},w{w'}^{-1}}$
associated to any $w'\in W$ of the $W$-permutohedron that can be crossed
or reached by a $\uparrow_w$ path.
\item
For every $i\in[1,l]$, if two vertices respectively related to the labels
${\mathcal X}^{u,v},{\mathcal X}^{u',v'}\subset{\mathcal X}_{w}^\tau$ of $P_W$
are related by the edge $s_i\in W$, then there exist two double reduced words
$\mathbf i,\mathbf j\in R^{\tau}(w_0)$ such that the associated seed $\mathcal X$-tori
${\mathcal X}_{\mathbf i}$ and ${\mathcal X}_{\mathbf j}$ are related by the
right tropical mutation $\mu_{\lozenge^{\mathfrak R}_i}$ associated to $i$.
\end{itemize}
\end{lemma}

We finally define new evaluation maps to associate the family $\mathcal{X}^{\tau}_w$
to every double Bruhat cell $G^{w,1}$.
To any $w\in W$ and any double word $\mathbf i\in R^{\tau}(w)$,
we associate the evaluation map
$\ev^{\tau}_{\mathbf i}:{\mathcal X}_{\mathbf i}\rightarrow (G^{w,1},\pi_G)$
by the formula
\begin{equation}\label{equ:tropev}
\begin{array}{ccc}
\ev_{\mathbf i}^{\tau}(\mathbf x)=[\ev_{\mathbf i}(\mathbf x)\widehat{w'w^{-1}}]_{\leq 0},
&\text{for every}&\mathbf x\in{\mathcal X}_{\mathbf i}\ .
\end{array}
\end{equation}

\begin{lemma}\label{lemma:taumut'} For every $w\in W$ and every $\mathbf i\in R^{\tau}(w)$
the map $\ev^{\tau}_{\mathbf i}$ is a birational Poisson isomorphism
on a Zarisky open set of $G^{w,1}$.
\end{lemma}
\begin{proof}Let us recall from \cite[Theorem 1.6]{FZtotal} that
the map $g\mapsto[g\widehat{w}]_{\leq 0}$ is biregular on $G^{1,w}$.
Therefore the statement is implied by Theorem \ref{thm:evG} and
Proposition \ref{prop:muttrop}.
\end{proof}

\begin{lemma}\label{lemma:taumut} For every $w\in W$ and every
$\mathbf i,\mathbf{j}\in R^{\tau}(w)$, the  equality
$\ev^{\tau}_{\mathbf i}=\ev^{\tau}_{\mathbf{j}}\circ
\mu_{\mathbf i\rightarrow\mathbf{j}}^{\tau}$
is satisfied.
\end{lemma}
\begin{proof}Derived from Theorem \ref{fg} and Proposition \ref{equ:trop},
involving respectively cluster transformations and tropical mutations.
\end{proof}

When $\mathfrak{g}=A_2$, a synthesis of this section is provided
by Figure \ref{fig:commutativeweakordertau}.

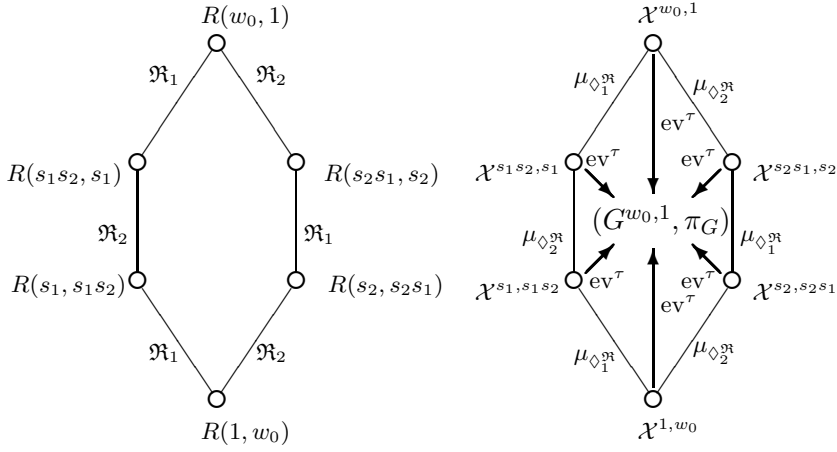
\begin{figure}[htbp]
\begin{center}
\setlength{\unitlength}{1.5pt}
\qquad\qquad\qquad\qquad
\begin{picture}(20,115)(0,-80)
\put(-29,-63.5){\line(2,3){17.7}}
\put(-10,-33){\line(0,1){26}}
\put(-11,-3.5){\line(-2,3){17.7}}
\put(-31,-63.5){\line(-2,3){17.7}}
\put(-50,-33){\line(0,1){26}}
\put(-49,-3.5){\line(2,3){17.7}}
\thicklines
\put(-30,25){\circle{4}}
\put(-34,30){\footnotesize $R({w_0,1})$}
\put(-50,-5){\circle{4}}
\put(-83,-10){\footnotesize $R({s_1s_2,s_1})$}
\put(-10,-5){\circle{4}}
\put(-3,-10){\footnotesize $R({s_2s_1,s_2})$}
\put(-50,-35){\circle{4}}
\put(-82,-38){\footnotesize $R({s_1,s_1s_2})$}
\put(-10,-35){\circle{4}}
\put(-2,-38){\footnotesize $R({s_2,s_2s_1})$}
\put(-30,-65){\circle{4}}
\put(-34,-75){\footnotesize $R(1,w_0)$}
\put(-20,15){\footnotesize ${\mathfrak R}_2$}
\put(-47,15){\footnotesize ${\mathfrak R}_1$}
\put(-60,-25){\footnotesize ${\mathfrak R}_2$}
\put(-8,-25){\footnotesize ${\mathfrak R}_1$}
\put(-47,-55){\footnotesize ${\mathfrak R}_1$}
\put(-20,-55){\footnotesize ${\mathfrak R}_2$}
\end{picture}
\qquad\qquad\qquad\qquad
\qquad\qquad
\begin{picture}(20,115)(0,-80)
\put(-29,-63.5){\line(2,3){17.7}}
\put(-10,-33){\line(0,1){26}}
\put(-11,-3.5){\line(-2,3){17.7}}
\put(-31,-63.5){\line(-2,3){17.7}}
\put(-50,-33){\line(0,1){26}}
\put(-49,-3.5){\line(2,3){17.7}}
\thicklines
\put(-30,25){\circle{4}}
\put(-34,30){\footnotesize ${\mathcal X}^{w_0,1}$}
\put(-50,-5){\circle{4}}
\put(-75,-10){\footnotesize ${\mathcal X}^{s_1s_2,s_1}$}
\put(-10,-5){\circle{4}}
\put(-5,-10){\footnotesize ${\mathcal X}^{s_2s_1,s_2}$}
\put(-50,-35){\circle{4}}
\put(-75,-40){\footnotesize ${\mathcal X}^{s_1,s_1s_2}$}
\put(-10,-35){\circle{4}}
\put(-5,-40){\footnotesize ${\mathcal X}^{s_2,s_2s_1}$}
\put(-30,-65){\circle{4}}
\put(-34,-75){\footnotesize ${\mathcal X}^{1,w_0}$}
\put(-30,22){\line(0,-1){35}}
\put(-30,22){\vector(0,-1){35}}
\put(-28,3){\footnotesize $\ev^{\tau}$}
\put(-47,-7){\line(1,-1){7}}
\put(-47,-7){\vector(1,-1){7}}
\put(-47,-6){\footnotesize $\ev^{\tau}$}
\put(-13,-7){\line(-1,-1){7}}
\put(-13,-7){\vector(-1,-1){7}}
\put(-23,-6){\footnotesize $\ev^{\tau}$}
\put(-47,-33){\line(1,1){7}}
\put(-47,-33){\vector(1,1){7}}
\put(-46,-37){\footnotesize $\ev^{\tau}$}
\put(-13,-33){\line(-1,1){7}}
\put(-13,-33){\vector(-1,1){7}}
\put(-23,-37){\footnotesize $\ev^{\tau}$}
\put(-30,-62){\line(0,1){35}}
\put(-30,-62){\vector(0,1){35}}
\put(-28,-43){\footnotesize $\ev^{\tau}$}
\put(-20,13){\footnotesize $\mu_{\lozenge^{\mathfrak R}_2}$}
\put(-50,15){\footnotesize $\mu_{\lozenge^{\mathfrak R}_1}$}
\put(-63,-25){\footnotesize $\mu_{\lozenge^{\mathfrak R}_2}$}
\put(-8,-25){\footnotesize $\mu_{\lozenge^{\mathfrak R}_1}$}
\put(-50,-55){\footnotesize $\mu_{\lozenge^{\mathfrak R}_1}$}
\put(-20,-53){\footnotesize $\mu_{\lozenge^{\mathfrak R}_2}$}
\put(-48,-22){ $(G^{w_0,1},\pi_G)$}
\end{picture}
\end{center}
\vspace{-.1in}
\caption{The set $R^{\tau}(w_0)$ and related cluster $\mathcal X$-varieties
when $\mathfrak{g}=A_2$}
\label{fig:commutativeweakordertau}
\end{figure}

\section{Twisted evaluations, $(w_1,w_2)_v$-maps, and cluster varieties
related to $(G,\pi_*)$}
\label{section:twistedev}
We continue the evaluation procedure for dual Poisson-Lie groups started in Section
\ref{subsection:evdual}.
For every $v\in W$, we introduce a new set $D(v)$ of double words,
containing the set $R(v,w_0)$. To each double word $\mathbf i$ of this
set, we associate a \emph{twisted evaluation} $\widehat{\ev}_{\mathbf i}:
{\mathcal X}_{[\mathbf{i}]_{\mathfrak R}}\rightarrow G$, which then generalized the
dual evaluation of Section \ref{subsection:evdual}.
These twisted evaluations are obtained by composing the
Fock-Goncharov evaluation maps of Section \ref{section:ClusterG} with
new maps called \emph{$(w_1,w_2)_v$-maps}, where
for $w_1=w_2=e$, we rediscover the Evens-Lu morphisms \cite[Section 5]{EL}.
We then use the $\tau$-combinatorics of Section \ref{section:taucombi}
with the truncation maps of Section \ref{subsection:evdual} to get a family of
cluster $\mathcal X$-varieties ${\mathcal X}_{w}$, associated to each element $w$ of
$W$ parameterizing $(BB_-,\pi_*)$. In particular, setting $w=e$, we rediscover the
result of Section \ref{subsection:evdual}. (The way to relate these cluster
$\mathcal X$-varieties will be given in Section \ref{section:Loop} by introducing
the birational Poisson isomorphisms on seed $\mathcal X$-tori called saltations.)

\subsection{The $(w_1,w_2)_v$-maps}\label{section:twistedeval}
The following $(w_1,w_2)_v$-maps generalize the Poisson birational isomorphisms
studied in \cite[Section 5]{EL}, which link direct products of double Bruhat cells
and double reduced Bruhat cells to Steinberg fibers. Let us remember the involution
$i\mapsto i^{\star}$ on double word
Now, let
$w\in W$ and $s_{i_1}\dots s_{i_n}$ be a reduced decomposition of $w$, then $w^{\star}\in W$
is the element given by $w^{\star}=s_{i_1^{\star}}\dots s_{i_n^{\star}}$. (Using the
Tits theorem, it is easy to see that the result doesn't depend on the choice of the
decomposition of $w$ into simple reflections.)
Remember the notation (\ref{equ:double reduced})
for double reduced Bruhat cells. We denote $\pi_{G\times G}$
the Poisson product structure on the manifold $G\times G$ induced by
the Poisson manifold $(G,\pi_G)$. For every $w_1\leq v,w_2\in W$,
let $({(G,L)}^{(w_1,w_2)_v},\pi_{(w_1,w_2)_v})$ be the quotient of
the direct product $(G^{{w_1^{\star}}^{-1},vw_1^{-1}}{\times}
G^{{w_2}^{-1},w_0w_2^{-1}},\pi_{G\times G})$ by the
$H\times H$-right action given by
\begin{equation}\label{equ:HxHaction}
\begin{array}{cll}
&(g_1,g_2).(h_1,h_2)=(g_1h_1,h_1^{-1}g_2h_2)\\
\mbox{with }\\
&g_1\in G^{{w_1^{\star}}^{-1},vw_1^{-1}},\
g_2\in G^{{w_2}^{-1},w_0w_2^{-1}},\ \mbox{and } h_1,h_2\in H.
\end{array}
\end{equation}
In particular, when  $w_1=v$ and $w_2=e$, the quotient set ${(G,L)}^{(w_1,w_2)_v}$
is the set $L^{v,w_0}$.

\begin{definition}Let $w_1\leq v,w_2\in W$. The \emph{$(w_1,w_2)_v$-right map}
$\varrho_{(w_1,w_2)_v}$ and the \emph{$(w_1,w_2)_v$-left map} $\lambda_{(w_1,w_2)_v}$
are defined by the following formulas
$$\begin{array}{clll}
&\varrho_{(w_1,w_2)_v}:&{(G,L)}^{(w_1,w_2)_v}\to L^{w_0,vw_1^{-1}}:&
(b_1,bH)\mapsto b_1[b\widehat{w_2w_0}]_{\leq 0}H\ ,\\
\mbox{and}\\
&\lambda_{(w_1,w_2)_v}:&{(G,L)}^{(w_1,w_2)_v}\to L^{{w_1^{\star}}^{-1},w_0}:&
(b_1,bH)\mapsto b_1[[b\widehat{w_2w_0}]_{\leq 0}^{\theta}\widehat{w_0}]_{\leq 0}^{\theta}H\ ,
\end{array}$$
where $\theta$ denotes the Cartan involution on $G$ given by (\ref{equ:Cartan}).
For every $t\in H$, the \emph{$(w_1,w_2)_v$-maps} (or simply
\emph{$(w_1,w_2)$-maps} when no confusion occurs) are then the maps given by
$$\begin{array}{rcl}
\rho_{t,(w_1,w_2)_v}:&{(G,L)}^{(w_1,w_2)_v}\to G:
&(b_1,bH)\mapsto\lambda_{(w_1,w_2)_v}(b_1,bH)\ \widehat{w_0} t
\ \varrho_{(w_1,w_2)_v}(b_1,bH)^{-1}\ .
\end{array}$$
\end{definition}

\begin{ex}The $(w_1,w_2)_v$-maps associated to $\mathfrak{g}=A_2$
are obtained in the following way. Let us first describe, for every $v\in W$,
the set $W_{\leq v}$ of elements $w_1\in W$ such that $w_1\leq v$.
They are the following: $W_{\leq e}=\{e\}$, $W_{\leq s_i}=\{e,s_i\}$,
$W_{\leq s_is_j}=\{e,s_i,s_is_j\}$ and $W_{\leq s_is_js_i}=W$,
$i,j\in[1,2]$ being different numbers.
Then, let $g_1\in G^{{w_1^{\star}}^{-1},vw_1^{-1}}$,
$b\in G^{1,w_0}$, $g\in G^{s_i,s_is_j}$, $g'\in G^{s_is_j,s_i}$ and
$c\in G^{w_0,1}$; the different $(w_1,w_2)_v$-maps are given by
$$\begin{array}{lcl}
\rho_{t,(w_1,e)_v}(g_1,b)&=&g_1b\ t\widehat{w_0}\ (g_1[b\widehat{w_0}]_{\leq 0})^{-1}\ ;\\
\\
\rho_{t,(w_1,s_i)_v}(g_1,g)&=&g_1[[g\widehat{s_js_i}]_{\leq 0}^{\theta}\widehat{w_0}]_{\leq 0}^{\theta}
\ t\widehat{w_0}\ (g_1[g\widehat{s_js_i}]_{\leq 0})^{-1}\ ;\\
\\
\rho_{t,(w_1,s_is_j)_v}(g_1,g')&=&g_1[[g'\widehat{s_i}]_{\leq 0}^{\theta}\widehat{w_0}]_{\leq 0}^{\theta}
\ t\widehat{w_0}\ (g_1[g'\widehat{s_i}]_{\leq 0})^{-1}\ ;\\
\\
\rho_{t,(w_1,w_0)_v}(g_1,c)&=&g_1[c^{\theta}\widehat{w_0}]_{\leq 0}^{\theta}
\ t\widehat{w_0}\ (g_1c)^{-1}\ .
\end{array}$$
\end{ex}

The way to relate the geometries of $(G,\pi_G)$ and $(G,\pi_*)$ is given by the
following result, which is directly deduced from Theorem \ref{thm:evG}, Lemma
\ref{lemma:w_1w_2} and the forthcoming Theorem~\ref{thm:ev*1}.

\begin{prop}\label{prop:w_1w_2}For every $t\in H$ and $w_1\leq v,w_2\in W$,
the $(w_1,w_2)_v$-map $\rho_{t,(w_1,w_2)_v}$ is a birational Poisson isomorphism of
$({(G,L)}^{(w_1,w_2)_v},\pi_{(w_1,w_2)_v})$
on a Zariski open set of $(F_{t,w_0v^{-1}},\pi_*)$.
\end{prop}

\begin{rem}\label{rem:evdualtwistev}For every $v\in W$, the $(e,e)_v$-maps
are the maps denoted $\mu\rho$ in
\cite[Section 5]{EL} and Proposition \ref{prop:w_1w_2} below then rephrases
\cite[Corollary~5.11]{EL}. Another interesting case is given by $w_1=v$ and $w_2=e$;
in this case, the related $(w_1,w_2)_v$-maps are the following maps, strongly related
to the dual evaluations of Section \ref{subsection:evdual}.
$$\begin{array}{cccc}
{\rho}_{t,(v,e)_v}: & (L^{v,w_0},\pi_{G\backslash H})
& \rightarrow & (F_{t,w_0v^{-1}},\pi_*)\\
        & gH & \mapsto & g\widehat{w_0}\ t\ [g\widehat{w_0}]_{\leq 0}^{-1}
\end{array}\ .$$
Indeed, the equality $\ev_{\mathbf i}^{\dual}={\rho}_{t,(v,e)_v}\circ\ev_{\mathbf i}^{\red}$
is clearly satisfied on ${\mathcal X}_{[\mathbf i]_{\mathfrak{R}}}(t)$
for every $t\in H$ and every $\mathbf i\in R(v,w_0)$.
\end{rem}

\subsection{Twisted evaluations on $(G,\pi_*)$}\label{section:twistedeval}
As it will be stated by Lemma \ref{lemma:w_1w_2}, a composition of evaluations and
reduced evaluations with $(w_1,w_2)_v$-maps leads to a generalization
of dual evaluations, called twisted evaluations and given by the
formula (\ref{equ:evhat}).
But, before, we introduce new sets
of double words, denoted $W(w_1,w_2)_v$, and a new operation
on seed $\mathcal X$-tori called $\mathcal X$-split, as preliminaries
to the definition of twisted evaluations.

\begin{definition}\label{def:Wwords}
Let $w_1\leq v,w_2\in W$. A \emph{$(w_1,w_2)_v$-word} $\mathbf i$
is a double word linked to a product $\mathbf{i_1}\mathbf{i_2}$, with
$\mathbf {i_1}\in R({w_1^{\star}}^{-1},vw_1^{-1})$ and
$\mathbf{i_2}\in R({w_2}^{-1},w_0w_2^{-1})$, by a sequence of mixed $2$-moves.
The product $\mathbf{i_1}\mathbf{i_2}$ is called a \emph{trivial $(w_1,w_2)_v$-word}
and the decomposition $(\mathbf{i_1},\mathbf{i_2})$ associated to $\mathbf i$
is called the \emph{$(w_1,w_2)_v$-decomposition} of $\mathbf i$. (For example, every
$(e,e)_v$-word is a trivial $(e,e)_v$-word.) The set of $(w_1,w_2)_v$-words will
be denoted $W(w_1,w_2)_v$.

In particular, the set $R(v,w_0)$ is the set of $(v,e)_v$-words. Let $D(v)$
be the (disjoint) union over $w_1\leq v,w_2\in W$ of all the $(w_1,w_2)_v$-words.
$$D(v)=\displaystyle\bigsqcup_{w_1\leq v,w_2\in W}W(w_1,w_2)_v\ .$$
Therefore, we have the inclusion $R(v,w_0)\subset D(v)$ for every $v\in W$.
\end{definition}

A complete description of the sets $D(v)$ will be given in Subsection
\ref{subsection:dualmove} via the $W$-permutohedron.
Here is, for the moment, an example with ${\mathfrak g}=A_2$ and $v=s_1$. The associated
sets $R({w_1^{\star}}^{-1},vw_1^{-1})$ and $R({w_2}^{-1},w_0w_2^{-1})$ are
then given by the following list.
$$\begin{array}{ccc}
R(1,s_1),\ R(s_2,1)&\mbox{and}&R(1,w_0),\ R(s_1,s_1s_2),\ R(s_1s_2,s_1),\\
&&R(s_2,s_2s_1),\ R(s_2s_1,s_2),\ R(w_0,1)\ .
\end{array}
$$Trivial
$(w_1,w_2)_v$-words are therefore products $\mathbf{i_1}\mathbf{i_2}$ of
some double words
$$\begin{array}{ccc}
\mathbf{i_1}\in\{1,\overline{2}\}&\mbox{and}& \mathbf{i_2}\in\{121,212,
\overline{1}12,1\overline{1}2,12\overline{1},\overline{1}\overline{2}1,
\overline{1}1\overline{2},1\overline{1}\overline{2},
\overline{2}21,2\overline{2}1,21\overline{2},\overline{2}\overline{1}2,\\
&&\overline{2}2\overline{1},2\overline{2}\overline{1},
\overline{1}\overline{2}\overline{1},\overline{2}\overline{1}\overline{2}\}\ .
\end{array}$$
We then introduce a new operation on seed $\mathcal X$-tori that will be useful
to describe the oncoming combinatorics related to the following twisted
evaluations.

\begin{definition}A \emph{split} of a seed $\mathbf I$ is a pair
of seeds $(\mathbf I_1,\mathbf I_2)$ such that $\mathbf I$ is their
amalgamated product, that is $\mathbf I=\mathfrak{m}(\mathbf I_1,\mathbf I_2)$.
An associated \emph{$\mathcal X$-split}
is a section of the amalgamation map $\mathfrak{m}:{\mathcal X}_{\mathbf I_1}
\times{\mathcal X}_{\mathbf I_2}\rightarrow{\mathcal X}_{\mathbf I}$, i.e.
a map $\mathfrak{s}:{\mathcal X}_{\mathbf I}\rightarrow{\mathcal X}_{\mathbf I_1}
\times{\mathcal X}_{\mathbf I_2}$ such that the product $\mathfrak{m}\circ\mathfrak{s}$
gives the identity map on ${\mathcal X}_{\mathbf I}$.
For every $\mathcal X$-split $\mathfrak{s}$ associated to the decomposition
$\mathbf I\to(\mathbf I_1,\mathbf I_2)$, we will associate to any
$\mathbf x\in{\mathcal X}_{\mathbf I}$, the elements $\mathbf x_{(1)}
\in{\mathcal X}_{\mathbf{I_1}}$ and $\mathbf x_{(2)}\in{\mathcal X}_{\mathbf{I_2}}$
given by
$$
\mathfrak{s}(\mathbf{x})=(\mathbf x_{(1)},\mathbf x_{(2)})\ .
$$
\end{definition}

Figure \ref{fig:amalgblock} describes the split associated to the decomposition
$2\overline{2}\to(2,\overline{2})$ when $\mathfrak{g}=A_3$.
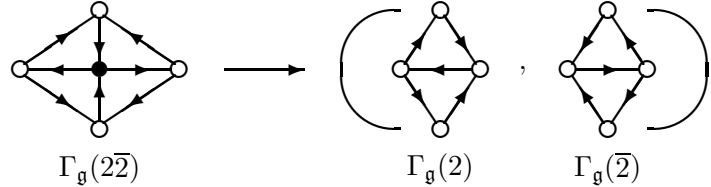
\begin{figure}[htbp]
\begin{center}
\setlength{\unitlength}{1.5pt}
\begin{picture}(20,48)(0,-27)
\thicklines
\put(39,-1.5){\line(-3,-2){18}}
\put(39,-1.5){\vector(-3,-2){12}}
\put(39,1.5){\line(-3,2){18}}
\put(39,1.5){\vector(-3,2){12}}
\put(1,1.5){\line(3,2){18}}
\put(1,1.5){\vector(3,2){12}}
\put(1,-1.5){\line(3,-2){18}}
\put(1,-1.5){\vector(3,-2){12}}
\put(20,-13){\line(0,1){11}}
\put(20,-13){\vector(0,1){10}}
\put(20,2){\line(0,1){11}}
\put(20,13){\vector(0,-1){10}}
\put(20,15){\circle{4}}
\put(20,-15){\circle{4}}
\put(0,0){\circle{4}}
\put(20,0){\circle*{4}}
\put(40,0){\circle{4}}
\put(2,0){\line(1,0){16}}
\put(18,0){\vector(-1,0){11}}
\put(22,0){\line(1,0){16}}
\put(22,0){\vector(1,0){11}}
\put(20,-25){\makebox(0,0){$\Gamma_{\mathfrak{g}}(2\overline 2)$}}
\end{picture}
\qquad\qquad
\begin{picture}(20,48)(0,-27)
\thicklines
\put(0,0){\vector(1,0){20}}
\end{picture}
\quad\qquad
\begin{picture}(20,48)(0,-27)
\thicklines
{\oval(30,30)[l]}
\multiput(1,1.5)(10,-15){2}{\line(2,3){8}}
\put(1,1.5){\vector(2,3){5}}
\put(11,13.5){\vector(2,-3){6}}
\multiput(1,-1.5)(10,15){2}{\line(2,-3){8}}
\put(1,-1.5){\vector(2,-3){5}}
\put(11,-13.5){\vector(2,3){6}}
\multiput(10,15)(0,-30){2}{\circle{4}}
\multiput(0,0)(20,0){2}{\circle{4}}
\put(2,0){\line(1,0){16}}
\put(18,0){\vector(-1,0){11}}
\put(10,-25){\makebox(0,0){$\Gamma_{\mathfrak{g}}(2)$}}
\end{picture}
\quad
\begin{picture}(10,48)(0,-27)
\thicklines
\put(0,0){$,$}
\end{picture}
\begin{picture}(20,48)(0,-27)
\thicklines
\put(20,0){\oval(30,30)[r]}
\multiput(1,1.5)(10,-15){2}{\line(2,3){8}}
\put(19,1.5){\vector(-2,3){5}}
\put(9,13.5){\vector(-2,-3){6}}
\multiput(1,-1.5)(10,15){2}{\line(2,-3){8}}
\put(19,-1.5){\vector(-2,-3){5}}
\put(9,-13.5){\vector(-2,3){6}}
\multiput(10,15)(0,-30){2}{\circle{4}}
\multiput(0,0)(20,0){2}{\circle{4}}
\put(2,0){\line(1,0){16}}
\put(2,0){\vector(1,0){11}}
\put(10,-25){\makebox(0,0){$\Gamma_{\mathfrak{g}}(\overline 2)$}}
\end{picture}
\end{center}
\vspace{-.1in}
\caption{The split $\mathfrak s:\Gamma_{\mathfrak{g}}(2\overline 2)\mapsto(\Gamma_{\mathfrak{g}}(2)
,\Gamma_{\mathfrak{g}}(\overline 2))$
for $\mathfrak{g}=A_3$}
\label{fig:amalgblock}
\end{figure}
Let us stress that, because the amalgamated product is not an isomorphism
of seed $\mathcal X$-tori, different $\mathcal X$-splits
can be associated to a given split of seed. Indeed,
let $\mathbf i,\mathbf{i_1},\mathbf{i_2}$ be double words
such that $\mathbf i=\mathbf{i_1}\mathbf{i_2}$, if $\mathfrak s$ is a
$\mathcal X$-split associated to the decomposition
$\mathbf i\to(\mathbf{i_1},\mathbf{i_2})$, then for every
$\mathbf t\in{\mathcal X}_{\mathbf 1}$ the following map ${\mathfrak s}_{\mathbf t}$
is also a $\mathcal X$-split.
$$\begin{array}{rcl}
{\mathfrak{s}}_{\mathbf t}(\mathbf{x})=({\mathfrak m}(\mathbf x_{(1)},\mathbf t),
{\mathfrak m}({\mathbf t}^{-1},\mathbf x_{(2)}))&\mbox{where}&\mathfrak{s}(\mathbf{x})
=(\mathbf x_{(1)},\mathbf x_{(2)})\\
\\
\mbox{and } {\mathbf t}^{-1}=(t_1^{-1},\dots,t_l^{-1})&\mbox{when}&
{\mathbf t}=(t_1,\dots,t_l)\ .
\end{array}
$$
However, in what follows, nearly each time we will need a
$\mathcal X$-split, this freedom of choice will not affect
the related result.
We are now ready to define twisted evaluations on $(G,\pi_*)$.

\begin{definition}
Let $w_1\leq v,w_2\in W$, $\mathbf i=\mathbf{i_1}\mathbf{i_2}$ be a trivial
$(w_1,w_2)_v$-word and $\mathfrak s$ be a $\mathcal X$-split associated
to the decomposition $\mathbf i\to(\mathbf{i_1},\mathbf{i_2})$.
We define the \emph{left} and \emph{right evaluations} $\ev^{\mathfrak L}_{\mathbf i},
\ev^{\mathfrak R}_{\mathbf i}:{\mathcal X}_{\mathbf i}\rightarrow G$ by
the formulas:
\begin{equation}\label{equ:evRL}
\begin{array}{cl}
&\ev^{\mathfrak R}_{\mathbf i}(\mathbf x)=\ev_{\mathbf{i_1}}(\mathbf{x}_{(1)})
[\ev_{\mathbf{i_2}}^{\red}(\mathbf{x}_{(2)})\widehat{w_2w_0}]_{\leq 0}\\
\mbox{and}\\
&\ev^{\mathfrak L}_{\mathbf i}(\mathbf x)=\ev_{\mathbf{i_1}}(\mathbf{x}_{(1)})
[(\ev^{\mathfrak R}_{\mathbf{i_2}}(\mathbf{x}_{(2)}))^{\theta}\widehat{w_0}]_{\leq 0}^{\theta}\ ,
\end{array}
\end{equation}
where the reduced evaluations $\ev^{\red}$ are the evaluation maps given
in Section \ref{subsection:evdual}. (Let us notice that these left and right
evaluations don't depend on the choice of the $\mathcal X$-split
$\mathfrak{s}$.)
These maps are extended to every
$\mathbf i\in D(v)$ by setting
\begin{equation}\label{equ:mutsidedev}
\begin{array}{ccc}\ev^{\mathfrak L}_{\mathbf i}=\ev^{\mathfrak L}_{\mathbf{i_1}\mathbf{i_2}}
\circ\mu_{\mathbf i\to\mathbf{i_1}\mathbf{i_2}}& \mbox{and}&
\ev^{\mathfrak R}_{\mathbf i}=\ev^{\mathfrak R}_{\mathbf{i_1}\mathbf{i_2}}
\circ\mu_{\mathbf i\to\mathbf{i_1}\mathbf{i_2}}\ .
\end{array}
\end{equation}
Finally, for every $v\in W$ and $\mathbf i\in D(v)$, we define
the \emph{twisted evaluation}
\begin{equation}\label{equ:evhat}
\begin{array}{c}
\widehat{\ev}_{\mathbf i}:{\mathcal X}_{[\mathbf{i}]_{\mathfrak R}}
\rightarrow G:\mathbf{x}\mapsto
\ev_{\mathbf{i}}^{\mathfrak L}(\mathbf{x})\ x_{\mathbf 1}({\mathbf x(\mathfrak R)})\widehat{w_0}\
{\ev_{\mathbf{i}}^{\mathfrak R}(\mathbf{x})}^{-1}\ ,\\
\mbox{where }\ev_{\mathbf 1}({\mathbf x(\mathfrak R)})=\displaystyle
\prod_{j=1}^lH^j(x_{\binom{j}{N^j(\mathbf i)}})\ .
\end{array}
\end{equation}
\end{definition}

The relations between twisted evaluations, $(w_1,w_2)$-maps and the Fock-Goncharov
evaluation maps is given by the following lemma, straightforwardly checked.

\begin{lemma}\label{lemma:w_1w_2}Let $w_1\leq v,w_2\in W$, $\mathbf i$ be a trivial
$(w_1,w_2)_v$-word, and $\mathfrak s$ be a $\mathcal X$-split associated
to the $(w_1,w_2)_v$-decomposition $\mathbf i\to(\mathbf{i_1},\mathbf{i_2})$.
The following equalities are satisfied.
$$\begin{array}{ll}
\ev^{\mathfrak R}_{\mathbf i}(\mathbf x)=\varrho_{(w_1,w_2)_v}(\ev_{\mathbf{i_1}}(\mathbf{x}_{(1)}),
\ev_{\mathbf{i_2}}^{\red}(\mathbf{x}_{(2)}))\ ,\\
\\
\ev^{\mathfrak L}_{\mathbf i}(\mathbf x)=\lambda_{(w_1,w_2)_v}(\ev_{\mathbf{i_1}}(\mathbf{x}_{(1)}),
\ev^{\red}_{\mathbf{i_2}}(\mathbf{x}_{(2)}))\ ,\\
\\
\widehat{\ev}_{\mathbf i}(\mathbf x)=\rho_{\ev_{\mathbf 1}(\mathbf x(\mathfrak R)),(w_1,w_2)_v}
(\ev_{\mathbf{i_1}}(\mathbf{x}_{(1)}),
\ev_{\mathbf{i_2}}^{\red}(\mathbf{x}_{(2)}))\ .
\end{array}$$
\end{lemma}

Here is finally the analog of Theorem \ref{thm:evG} for $(G,\pi_*)$.

\begin{thm}\label{thm:ev*1}For every $v\in W$, $t\in H$ and $\mathbf i\in D(v)$,
the restriction $\widehat{\ev}_{\mathbf i}:{\mathcal X}_{[\mathbf i]_{\mathfrak{R}}}(t)\rightarrow (F_{t,w_0v^{-1}},\pi_*)$
is a Poisson birational isomorphism onto a Zariski open set $F_{t,w_0v^{-1}}^0$ of $F_{t,w_0v^{-1}}$.
\end{thm}
\begin{proof}We are going to use the three following lemmas, the
first one being crucial to our construction.
\begin{lemma}\cite[Corollary~5.11]{EL}\label{lemme:lu} Let $v\in W$.
The $(e,e)_{v}$-map associated to every $t\in H$
is a birational Poisson isomorphism of $((G,L)^{(e,e)_{v}},\pi_{(e,e)_{v}})$
on a Zariski open set of~$(F_{t,w_0v^{-1}},\pi_*)$.
\end{lemma}
\begin{lemma}\label{lemme:lu2}Let $t\in H$, $v\in W$, $\mathbf i$ be a
$(e,e)_v$-word and  $\mathfrak s$ be a $\mathcal X$-split associated
to the $(e,e)_v$-decomposition $\mathbf i\to(\mathbf{i_1},\mathbf{i_2})$.
The following evaluation map is a birational Poisson isomorphism onto
a Zariski open set $F_{t,w_0v^{-1}}^{0}\subset F_{t,w_0v^{-1}}$.
$$\begin{array}{crcl}
\widehat{\ev}_{t,\mathbf i}:&\mathcal{X}_{\mathbf i}^{\red}\to(F_{t,w_0v^{-1}},\pi_*):
&\mathbf x\mapsto{\ev}_{\mathbf i}^{\red}(\mathbf x)
\widehat{w_0}t(\ev_{\mathbf{i_1}}(\mathbf{x}_{(1)})
[\ev_{\mathbf{i_2}}^{\red}(\mathbf{x}_{(2)})\widehat{w_0}]_{\leq 0})^{-1}\ .
\end{array}$$
\end{lemma}

Lemma \ref{lemme:lu2} is easily deduced from Lemma \ref{lemme:lu}
and Theorem \ref{fg}.

\begin{lemma}\label{lemma:twistinvol}The map $b\mapsto[b\widehat{w_0}]_{\leq 0}^{\theta}$
is an involution on $G^{e,w_0}$.
\end{lemma}

Lemma \ref{lemma:twistinvol} is well-known; to prove it, we have to remember that the map
$\theta$ is an involution, and use successively the facts that
$\widehat{w_0}^{\theta}=\widehat{w_0}^{-1}$,
$b^{\theta}\in G^{w_0,e}$, and $\widehat{w_0}^{-1}[b\widehat{w_0}]_+^{\theta}\widehat{w_0}\in N$ to get
$$b^{\theta}=(b\widehat{w_0})^{\theta}\widehat{w_0}=[(b\widehat{w_0})^{\theta}
\widehat{w_0}]_{\leq 0}=[[b\widehat{w_0}]_{\leq 0}^{\theta}
\widehat{w_0}]_{\leq 0}\ .$$
Let us now choose $\mathbf x(\mathfrak R)\in\mathcal{X}_{\mathbf 1}$
such that $\ev_{\mathbf 1}(\mathbf x(\mathfrak R))=t$.
From Lemma \ref{lemme:lu2} and Lemma \ref{lemma:twistinvol}, we
deduce that when the $(w_1,w_2)_v$-word $\mathbf i\in D(v)$ is such
that $w_1=w_2=e$ the evaluation $\widehat{\ev}_{t,\mathbf i}$ gives the restriction of
the map $\widehat{\ev}_{\mathbf i}$ on the set ${\mathcal X}_{[\mathbf i]_{\mathfrak{R}}}(t)$.
The general case will be deduced from Theorem \ref{thm:ev*} by noticing that
the associated map ${\widehat{\mu}}_{\mathbf{i}\rightarrow\mathbf j}$, defined
in Section \ref{section:Loop}, is a birational Poisson
isomorphism for every double words $\mathbf{i},\mathbf j\in D(v)$.
\end{proof}

In particular, Theorem \ref{thm:evdual} is deduced from Theorem
\ref{thm:ev*1} because $\widehat{\ev}_{\mathbf i}$ and $\ev^{\dual}_{\mathbf i}$
coincide for every $\mathbf i\in R(v,w_0)$, that is for every $(w_1,w_2)_v$-word
$\mathbf i$ such that $w_1=v$ and $w_2=e$.
Finally, let us recall that for every double word $\mathbf i$, ${\mathcal X}_{\mathbf i}^{\dual}
\subset{\mathcal X}_{[\mathbf{i}]_{\mathfrak R}}$ is, as a set, such that the elements of the
set of variables $\mathbf x(\mathfrak R)$ are pairwise disjoint.
Thus, we get the following corollary from the second decomposition of (\ref{equ:decG^*})
with Theorem \ref{thm:ev*1}.

\begin{cor}For every $\mathbf i\in D(w_0)$, the map $\widehat{\ev}_{\mathbf i}:
{\mathcal X}_{\mathbf i}^{\dual}\rightarrow (BB_-,\pi_*)$ is a Poisson
birational isomorphism on a Zariski open set of $BB_-$.
\end{cor}

\begin{rem}\label{rem:twitevlaconnect}
The same construction can be done (with the same results) if $G$ is no more
of adjoint type but simply connected. In this case, the twisted
evaluations we have to consider are the following
$$\begin{array}{c}
\widehat{\underline{\ev}}_{\mathbf i}:{\mathcal X}_{[\mathbf{i}]_{\mathfrak R}}
\rightarrow G:\mathbf{x}\mapsto
\ev_{\mathbf{i}}^{\mathfrak L}(\mathbf{x})\ x_{\mathbf 1}({\mathbf x(\mathfrak R)})\widehat{w_0}\
{\ev_{\mathbf{i}}^{\mathfrak R}(\mathbf{x})}^{-1}\ ,\\
\mbox{where }x_{\mathbf 1}({\mathbf x(\mathfrak R)})=\displaystyle
\prod_{j=1}^lH_j(x_{\binom{j}{N^j(\mathbf i)}})\ ,
\end{array}$$
the generators $H_j(.)$ being given by (\ref{equ:defphi2}).
\end{rem}

\subsection{$\tau$-combinatorics and cluster $\mathcal X$-varieties related to $(G,\pi_*)$}
\label{subsection:taucombdual}
We now relate twisted evaluations by cluster transformations to get Poisson
parameterizations related to $(G,\pi_*)$ by cluster $\mathcal X$-varieties.
This is achieved by mixing the truncation maps of Section \ref{subsection:evdual} with the
$\tau$-combinatorics developed in Section \ref{section:taucombi}. We then
get a family of $\mathcal X$-varieties $\mathcal X_w$, indexed by the
Weyl group $W$ of $G$, evaluating the dual Poisson Lie-group $(BB_-,\pi_*)$.

\subsubsection{Double reduced words, the set $D_{w_1}(v)$, and the
$W$-permutohedron associated to $\mathfrak{g}$}
As it was done in Section \ref{section:ClusterG}, we start by the study
of the combinatorics on double words before to consider the related birational
Poisson isomorphisms on seed $\mathcal X$-tori.
We fix $w_1\leq v\in W$,
and denote $D_{w_1}(v)$ the set of $(w_1,w_2)_v$-words
for every $w_2\in W$. Therefore, the set $D(v)$ is the union
over all $w_1\leq v\in W$ of the sets $D_{w_1}(v)$.
\begin{equation}\label{equ:Dsets}
\begin{array}{rrcccl}
&D_{w_1}(v)=\displaystyle\bigcup_{w_2\in W}W(w_1,w_2)_v
&\mbox{and}
&D(v)=\displaystyle\bigcup_{w_1\leq v}D_{w_1}(v)\ .
\end{array}
\end{equation}

The following result is clear from Definition \ref{def:Wwords}. It uses amalgamation
to relate these sets $D_{w_1}(v)$ to the set $R^{\tau}(w_0)$ already studied in
the Section \ref{section:taucombi}.

\begin{lemma}\label{lemma:amalRtauD}The set $D_{w_1}(v)$ is the set of double words that
can be obtained by a composition of mixed $2$-moves from an amalgamation
of a double reduced word $\mathbf{i_1}\in R({w_1^{\star}}^{-1},vw_1^{-1})$ with any double
reduced word $\mathbf{i_2}\in R^{\tau}(w_0)$.
\end{lemma}

We can therefore use the fact that
the set $R^{\tau}(w)$ has been described in Section \ref{section:taucombi}
via the $W$-permutohedron to relate the sets $W(w_1,w_2)_v$ to $D_{w_1}(v)$,
for every $w_1\leq v\in W$.

\begin{lemma}\label{lemma:relPWR}We have the following statements for every
$v,w_1\in W$ such that $w_1\leq v$.
\begin{itemize}
\item
The set  $D_{w_1}(v)$ is the disjoint union of the
labels $W(w_1,w'^{-1})_v$ associated to the vertices $w'$
of the $W$-permutohedron $P_W$ that are crossed or reached by a
$\uparrow_v$-path.
\item
Two labeled vertices $W(w_1,w^{-1})_v,W(w_1,w'^{-1})_v\subset D_{w_1}(v)$
of $P_W$ are related by the edge $s_j$ if and only if there exist double reduced
words $\mathbf i\in W(w_1,w^{-1})_v$ and $\mathbf j\in W(w_1,w'^{-1})_v$ such
that $\mathbf j={\mathfrak R}_j(\mathbf i)$.
\end{itemize}
\end{lemma}
\begin{proof}It suffices to apply Lemma \ref{lemma:PW} and the fact that
the map $R(w'^{-1},vw'^{-1})\mapsto W(w_1,w'^{-1})_v$
gives a bijection from $R^{\tau}(v)$ to $D_{w_1}(v)$.
\end{proof}

\begin{ex}\label{ex:permutoWD}We still take $\mathfrak{g}=A_2$.
From Figure \ref{fig:Wpermutohedre}, we get the $W$-permutohedron
at the right of Figure \ref{fig:commutativeweakordertau+2}.
Let us stress that this picture is valid for \emph{any} $w_1\in W$.
\end{ex}

\begin{figure}[htbp]
\begin{center}
\setlength{\unitlength}{1.5pt}
\qquad\qquad\qquad\qquad
\begin{picture}(20,110)(0,-80)
\put(-29,-63.5){\line(2,3){17.7}}
\put(-10,-33){\line(0,1){26}}
\put(-11,-3.5){\line(-2,3){17.7}}
\put(-31,-63.5){\line(-2,3){17.7}}
\put(-50,-33){\line(0,1){26}}
\put(-49,-3.5){\line(2,3){17.7}}
\thicklines
\put(-30,25){\circle{4}}
\put(-39,30){$R({w_0,1})$}
\put(-50,-5){\circle{4}}
\put(-87,-10){$R({s_1s_2,s_1})$}
\put(-10,-5){\circle{4}}
\put(-5,-10){$R({s_2s_1,s_2})$}
\put(-50,-35){\circle{4}}
\put(-87,-40){$R({s_1,s_1s_2})$}
\put(-10,-35){\circle{4}}
\put(-5,-40){$R({s_2,s_2s_1})$}
\put(-30,-65){\circle{4}}
\put(-39,-75){$R(1,w_0)$}
\put(-20,15){\footnotesize ${\mathfrak R}_2$}
\put(-47,15){\footnotesize ${\mathfrak R}_1$}
\put(-60,-25){\footnotesize ${\mathfrak R}_2$}
\put(-8,-25){\footnotesize ${\mathfrak R}_1$}
\put(-47,-55){\footnotesize ${\mathfrak R}_1$}
\put(-20,-55){\footnotesize ${\mathfrak R}_2$}
\end{picture}
\qquad\qquad\qquad\qquad\qquad
\qquad\qquad\qquad
\begin{picture}(20,110)(0,-80)
\put(-29,-63.5){\line(2,3){17.7}}
\put(-10,-33){\line(0,1){26}}
\put(-11,-3.5){\line(-2,3){17.7}}
\put(-31,-63.5){\line(-2,3){17.7}}
\put(-50,-33){\line(0,1){26}}
\put(-49,-3.5){\line(2,3){17.7}}
\thicklines
\put(-30,25){\circle{4}}
\put(-39,30){$W({w_1,w_0})_{w_0}$}
\put(-50,-5){\circle{4}}
\put(-95,-10){$W({w_1,s_1s_2})_{w_0}$}
\put(-10,-5){\circle{4}}
\put(-5,-10){$W({w_1,s_2s_1})_{w_0}$}
\put(-50,-35){\circle{4}}
\put(-85,-40){$W({w_1,s_1})_{w_0}$}
\put(-10,-35){\circle{4}}
\put(-5,-40){$W({w_1,s_2})_{w_0}$}
\put(-30,-65){\circle{4}}
\put(-39,-75){$W(w_1,1)_{w_0}$}
\put(-20,15){\footnotesize ${\mathfrak R}_2$}
\put(-47,15){\footnotesize ${\mathfrak R}_1$}
\put(-60,-25){\footnotesize ${\mathfrak R}_2$}
\put(-8,-25){\footnotesize ${\mathfrak R}_1$}
\put(-47,-55){\footnotesize ${\mathfrak R}_1$}
\put(-20,-55){\footnotesize ${\mathfrak R}_2$}
\end{picture}
\end{center}
\vspace{-.1in}
\caption{From the set $R^{\tau}(w_0)$ to the set $D_{w_1}(w_0)$ when $\mathfrak{g}=A_2$}
\label{fig:commutativeweakordertau+2}
\end{figure}
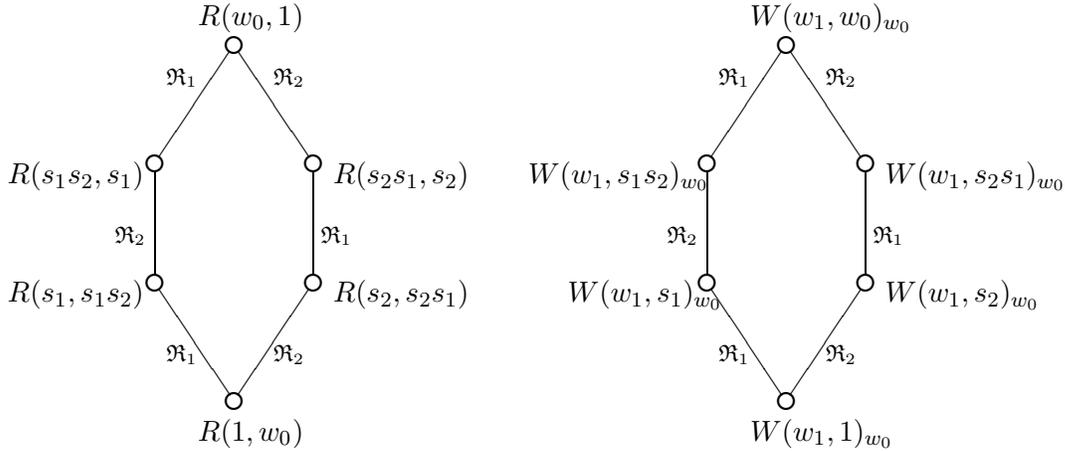

\subsubsection{The cluster $\mathcal X$-varieties ${\mathcal X}_{w_1\leq v}(t)$
associated to the set $D_{w_1}(v)$}
We now consider the related cluster transformations and associate a right
truncated cluster $\mathcal X$-variety ${\mathcal X}_{w_1\leq v}$ to each
set $D_{w_1}(v)$.
Let $w_1\leq v\in W$. To any $d^{\tau}$-move $\delta$
linking two double words $\mathbf i,\mathbf i'\in D_{w_1}(v)$ we associate the
cluster transformation $\mu_{[\mathbf i]_{\mathfrak R}\to[\mathbf i']_{\mathfrak R}}:
{\mathcal X}_{[\mathbf i]_{\mathfrak R}}\to{\mathcal X}_{[\mathbf i']_{\mathfrak R}}$
such that
$$\mu_{[\mathbf i]_{\mathfrak R}\to[\mathbf i']_{\mathfrak R}}\circ
\mathfrak{t}_{\mathbf i_{\mathfrak R}(t)}=\mathfrak{t}_{\mathbf i'_{\mathfrak R}(t)}
\circ{\mu}^{\tau}_{\mathbf{i}\rightarrow\mathbf{i'}}\ ,$$
where, as in Subsection \ref{section:mutau}, the generalized cluster
${\mu}^{\tau}_{\mathbf{i}\rightarrow\mathbf{j}}$ denotes
\begin{itemize}
\item
the cluster transformation $\mu_{\mathbf i\rightarrow \mathbf{i'}}$ if $\delta$ is
a generalized $d$-move;
\item
the tropical mutation $\mu_{\binom{i_{\ell(w)}}{\lozenge^{\mathfrak{R}}_j}}$ if
$\delta$ is the $\tau$-move ${\mathfrak R}_j$.
\end{itemize}
As usual, we extend this definition to every $\mathbf i,\mathbf{j}\in D_{w_1}(v)$:
if $\mathbf i,\mathbf j$ are double words linked by a
sequence $\delta_{\mathbf i\to\mathbf j}^{\tau}$ of $d^{\tau}$-moves and
$\mathbf{i}\to\mathbf{i_1}\rightarrow\dots\rightarrow\mathbf{i_{n-1}}
\rightarrow\mathbf{j}$ is the associated chain of elements, we define the map
${\mu}_{[\mathbf{i}]_{\mathfrak R}\rightarrow[\mathbf{j}]_{\mathfrak R}}:
{\mathcal X}_{[\mathbf i]_{\mathfrak R}}\to{\mathcal X}_{[\mathbf j]_{\mathfrak R}}$ as the composition
$${\mu}_{[\mathbf{i_{n-1}}]_{\mathfrak R}\rightarrow[\mathbf j]_{\mathfrak R}}\circ \dots\circ
{\mu}_{[\mathbf{i}]_{\mathfrak R}\rightarrow[\mathbf{i_1}]_{\mathfrak R}}\ .$$
We will denote ${\mathcal X}_{w_1\leq v}$, or simply ${\mathcal X}_{w}$ if the
equality $v=w_0$ is satisfied, the cluster
$\mathcal X$-variety associated to the set $D_{w_1}(v)$.
Moreover, because of equation (\ref{equ:defJ}), the cluster $\mathcal X$-variety
${\mathcal X}_{w_1\leq v}$ can be Poisson stratified into the disjoint union
over $H$ of cluster $\mathcal X$-varieties ${\mathcal X}_{w_1\leq v}(t)$:
$${\mathcal X}_{w_1\leq v}=\displaystyle\bigsqcup_{t\in H}{\mathcal X}_{w_1\leq v}(t)\ .$$
The way to relate the cluster $\mathcal X$-variety ${\mathcal X}_{w_1\leq v}
(t)$ to the $W$-permutohedron is the following.
As a preliminary, we attach to every set $W(w_1,w_2)_v$ a cluster $\mathcal X$-variety,
denoted ${\mathcal X}_{(w_1,w_2)_v}$ and resulting of the following amalgamation.
$$\mathfrak{m}:{\mathcal X}^{{w_1^{\star}}^{-1},vw_1^{-1}}\times
{\mathcal X}^{w_2^{-1},vw_2^{-1}}\to{\mathcal X}_{(w_1,w_2)_v}\ .$$
The first lemma is directly deduced from the definition of the amalgamated
product, and the definition of the set  $W(w_1,w_2)_v$.

\begin{lemma}Let $w_1\leq v,w_2\in W$. The cluster $\mathcal X$-variety
${\mathcal X}_{(w_1,w_2)_v}$ contains the seed $\mathcal X$-torus
${\mathcal X}_{[\mathbf{i}]_{\mathfrak R}}$ associated to any double
word $\mathbf i\in W(w_1,w_2)_v$.
\end{lemma}

\begin{lemma}\label{lemma:PWtruncX}
Replace each label
$w'\in W$ of a vertex of the $W$-permutohedron $P_W$
by the cluster $\mathcal X$-variety ${\mathcal X}^{{w'}^{-1},w_0{w'}^{-1}}$.
We then have the following properties.
\begin{itemize}
\item
The cluster $\mathcal X$-variety ${\mathcal X}_{w}$, indexed by any element
$w\in W$, contains the seed $\mathcal X$-torus
${\mathcal X}_{[\mathbf{i}]_{\mathfrak R}}$ associated to any double
word $\mathbf i\in D_w(w_0)$.
\item
The cluster $\mathcal X$-variety ${\mathcal X}_{w_1\leq v}$
is the image of the cluster $\mathcal X$-variety ${\mathcal X}_{(w_1,w')_v}$
by the right truncation map ${\mathbf t}_{\mathfrak R}$ for every $w'\in W$.
\item
For every $i\in[1,l]$, if two vertices respectively labeled by the cluster
$\mathcal X$-varieties
${\mathcal X}^{u,v},{\mathcal X}^{u',v'}\subset{\mathcal X}_{w_0}$ of $P_W$
are related by the edge $s_i\in W$, then there exist two double reduced words
$\mathbf i,\mathbf j\in D_{w_1}(v)$ such that the associated seed $\mathcal X$-tori
${\mathcal X}_{\mathbf i}$ and ${\mathcal X}_{\mathbf j}$ are related by the
right tropical mutation associated to $i$ and denoted $\mu_{\lozenge^{\mathfrak R}_i}$.
\end{itemize}
\end{lemma}
\begin{proof}It is clear that the seed
$\mathcal X$-tori obtained by right truncation from seed $\mathcal X$-tori
related by a right tropical mutation are the same. Therefore, the cluster
$\mathcal X$-varieties included in the family of cluster $\mathcal X$-varieties
and described by the $W$-permutohedron in Section \ref{section:taucombi} are sent
to the same truncated cluster $\mathcal X$-variety. Therefore, it suffices
to apply Lemma \ref{lemma:Wpermtrop}, because right tropical
mutations and right truncations commute with a left amalgamation; indeed, we change
each cluster $\mathcal X$-variety label ${\mathcal X}^{u,v}$ of Lemma \ref{lemma:Wpermtrop}
into the amalgamated cluster $\mathcal X$-variety ${\mathcal X}_{(w_1,w_2)_v}$.
\end{proof}

This result is illustrated by Figure \ref{fig:commutativeweakordertau+} in the case
$\mathfrak{g}=A_2$.
\begin{figure}[htbp]
\begin{center}
\setlength{\unitlength}{1.5pt}
\qquad\qquad\qquad\qquad
\begin{picture}(20,110)(0,-80)
\put(-29,-63.5){\line(2,3){17.7}}
\put(-10,-33){\line(0,1){26}}
\put(-11,-3.5){\line(-2,3){17.7}}
\put(-31,-63.5){\line(-2,3){17.7}}
\put(-50,-33){\line(0,1){26}}
\put(-49,-3.5){\line(2,3){17.7}}
\thicklines
\put(-30,25){\circle{4}}
\put(-34,30){${\mathcal X}_{w_1,1}$}
\put(-50,-5){\circle{4}}
\put(-70,-10){${\mathcal X}_{w_1,s_1}$}
\put(-10,-5){\circle{4}}
\put(-5,-10){${\mathcal X}_{w_1,s_2}$}
\put(-50,-35){\circle{4}}
\put(-75,-40){${\mathcal X}_{w_1,s_1s_2}$}
\put(-10,-35){\circle{4}}
\put(-5,-40){${\mathcal X}_{w_1,s_2s_1}$}
\put(-30,-65){\circle{4}}
\put(-34,-75){${\mathcal X}_{w_1,w_0}$}
\put(-40,-22){${\mathcal X}_{w_1}(t)$}
\put(-30,22){\line(0,-1){35}}
\put(-30,22){\vector(0,-1){35}}
\put(-28,3){\footnotesize $\mathfrak{t}_{{\mathfrak R}(t)}$}
\put(-47,-7){\line(1,-1){7}}
\put(-47,-7){\vector(1,-1){7}}
\put(-47,-6){\footnotesize $\mathfrak{t}_{{\mathfrak R}(t)}$}
\put(-13,-7){\line(-1,-1){7}}
\put(-13,-7){\vector(-1,-1){7}}
\put(-23,-6){\footnotesize $\mathfrak{t}_{{\mathfrak R}(t)}$}
\put(-47,-33){\line(1,1){7}}
\put(-47,-33){\vector(1,1){7}}
\put(-46,-37){\footnotesize $\mathfrak{t}_{{\mathfrak R}(t)}$}
\put(-13,-33){\line(-1,1){7}}
\put(-13,-33){\vector(-1,1){7}}
\put(-23,-37){\footnotesize $\mathfrak{t}_{{\mathfrak R}(t)}$}
\put(-30,-62){\line(0,1){35}}
\put(-30,-62){\vector(0,1){35}}
\put(-28,-43){\footnotesize $\mathfrak{t}_{{\mathfrak R}(t)}$}
\put(-20,13){\footnotesize $\mu_{\lozenge^{\mathfrak R}_2}$}
\put(-50,15){\footnotesize $\mu_{\lozenge^{\mathfrak R}_1}$}
\put(-63,-25){\footnotesize $\mu_{\lozenge^{\mathfrak R}_2}$}
\put(-8,-25){\footnotesize $\mu_{\lozenge^{\mathfrak R}_1}$}
\put(-50,-55){\footnotesize $\mu_{\lozenge^{\mathfrak R}_1}$}
\put(-20,-53){\footnotesize $\mu_{\lozenge^{\mathfrak R}_2}$}
\end{picture}
\end{center}
\vspace{-.1in}
\caption{Truncation maps and the cluster $\mathcal X$-variety ${\mathcal X}_{w_1}(t)$
when $\mathfrak{g}=A_2$}
\label{fig:commutativeweakordertau+}
\end{figure}
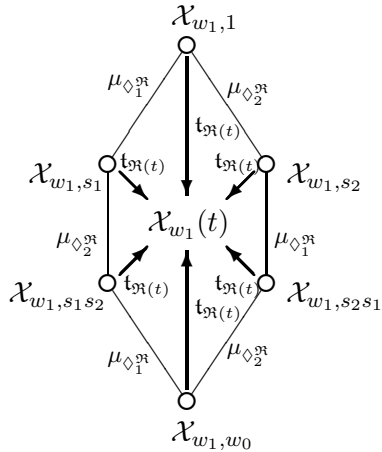

\subsubsection{Twisted evaluation maps relating ${\mathcal X}_{w_1\leq v}(t)$
to $(F_{t,v},\pi_*)$} We finally use twisted evaluation maps to get, for every
$v\in W$, a family of truncated cluster $\mathcal X$-varieties
${\mathcal X}_{w_1\leq v}(t)$, $w_1\leq v\in W$,
parameterizing the Poisson submanifold $(F_{t,w_0v^{-1}},\pi_*)$.
In particular, the family of truncated cluster $\mathcal X$-varieties
${\mathcal X}_{w_1\leq w_0}$ (also denoted ${\mathcal X}_{w_1}$), $w_1\in W$,
will parameterize the dual Poisson-Lie group $(BB_-,\pi_*)$.


\begin{prop}\label{lemma:dtaumoves} For every $t\in H$ and $w_1\leq v\in W$ and
every double words $\mathbf i,\mathbf{i'}\in D_{w_1}(v)$, the equality
$\widehat{\ev}_{\mathbf i}=\widehat{\ev}_{\mathbf{i'}}\circ\mu_{[\mathbf i]_{\mathfrak R}
\rightarrow[\mathbf{i'}]_{\mathfrak R}}$ is satisfied.
\end{prop}
\begin{proof}We introduce relatives of the left and right evaluations of (\ref{equ:evRL})
 in the following way. Let $w_1\leq v,w_2\in W$ and $\mathfrak s$ be a $\mathcal X$-split associated
to the $(w_1,w_2)_v$-decomposition $\mathbf i\mapsto(\mathbf{i_1},\mathbf{i_2})$.
We define the maps $\overline{\ev}^{\mathfrak L}_{\mathbf i},\overline{\ev}^{\mathfrak R}_{\mathbf i}:
{\mathcal X}_{\mathbf i}\rightarrow G$ and $\overline{\ev}_{\mathbf i}:
{\mathcal X}_{\mathbf{i}}\times {\mathcal X}_{\mathbf{1}}\rightarrow G$ by
the following formulas: we start by a definition on trivial double words
$$
\overline{\ev}^{\mathfrak R}_{\mathbf i}(\mathbf x)=\ev_{\mathbf{i_1}}(\mathbf{x}_{(1)})
[\ev_{\mathbf{i_2}}(\mathbf{x}_{(2)})\widehat{w_2w_0}]_{\leq 0}\ ,\mbox{and }\
\overline{\ev}^{\mathfrak L}_{\mathbf i}(\mathbf x)=\ev_{\mathbf{i_1}}(\mathbf{x}_{(1)})
[(\overline{\ev}^{\mathfrak R}_{\mathbf{i_2}}(\mathbf{x}_{(2)}))^{\theta}\widehat{w_0}]_{\leq 0}^{\theta}\ ,$$
before to use the same idea as in (\ref{equ:mutsidedev}) to extend this definition for every $\mathbf i\in D_{w_1}(v)$ by setting
$\overline{\ev}^{\mathfrak L}_{\mathbf i}=\overline{\ev}^{\mathfrak L}_{\mathbf{i_1}\mathbf{i_2}}
\circ\mu_{\mathbf i\to\mathbf{i_1}\mathbf{i_2}}$ {and}
$\overline{\ev}^{\mathfrak R}_{\mathbf i}=\overline{\ev}^{\mathfrak R}_{\mathbf{i_1}\mathbf{i_2}}
\circ\mu_{\mathbf i\to\mathbf{i_1}\mathbf{i_2}}$, and we introduce finally
$$
\overline{\ev}_{\mathbf i}(\mathbf{x},\mathbf{t})
\mapsto\overline{\ev}_{\mathbf{i}}^{\mathfrak L}(\mathbf{x})\ev_{\mathbf 1}
({\mathbf t})\widehat{w_0}{\overline{\ev}_{\mathbf{i}}^{\mathfrak R}(\mathbf{x})}^{-1}\ .
$$
Lemma \ref{lemma:taumut} with the extension of the kind (\ref{equ:mutsidedev}) just described
then implies that the following equalities are satisfied for every double words
$\mathbf i,\mathbf{i'}\in D_{w_1}(v)$.
$$\begin{array}{ccccc}
\overline{\ev}^{\mathfrak R}_{\mathbf i}=\overline{\ev}^{\mathfrak{R}}_{\mathbf{i'}}\circ
\mu_{\mathbf i\rightarrow\mathbf{i'}}^{\tau}&,&
\overline{\ev}^{\mathfrak L}_{\mathbf i}=\overline{\ev}^{\mathfrak{L}}_{\mathbf{i'}}\circ
\mu_{\mathbf i\rightarrow\mathbf{i'}}^{\tau}&\mbox{and}&
\overline{\ev}_{\mathbf i}=\overline{\ev}_{\mathbf{i'}}
(\mu_{\mathbf i\rightarrow\mathbf{i'}}^{\tau}(\mathbf{x}),\mathbf{t})\ .
\end{array}$$
Moreover, it is easy to see that the element $\overline{\ev}_{\mathbf i}
(\mathbf{x},\mathbf{t})$ doesn't depend on $x_j$ when $j\in I_0^{\mathfrak R}(\mathbf i)$.
Now, let us notice that tropical mutations associated to
right $\tau$-moves affect only the variables associated to right outlets.
Therefore, we get the equality
$$\overline{\ev}_{\mathbf i}(\mathbf x,\mathbf t)=\overline{\ev}_{\mathbf{i'}}
(\mu_{[\mathbf i]_{\mathfrak R}\rightarrow[\mathbf{i'}]_{\mathfrak R}}(\mathbf x),\mathbf t)\ .$$
Finally, it is clear that the relation $\widehat{\ev}_{\mathbf i}(\mathbf x)
=\overline{\ev}_{\mathbf i}(\mathbf x,\mathbf x(\mathfrak R))$ is satisfied for every
$\mathbf i\in D_{w_1}(v)$. The proposition is then proved, because, by equation (\ref{equ:defJ}),
the cluster variables belonging to $\mathbf x(\mathfrak R)$ are invariant by mutations.
\end{proof}

For every $t\in H$ and every $w_1\leq v\in W$, the cluster $\mathcal X$-variety
${\mathcal X}_{w_1\leq v}(t)$, has been therefore attached to the Poisson submanifold
$(F_{t,v},\pi_*)$.
Let us denote
${\mathcal X}_{[.]_{\mathfrak R}}(t)$ the application which associates to any double
word $\mathbf i$ the corresponding seed $\mathcal X$-torus
${\mathcal X}_{[\mathbf i]_{\mathfrak R}}(t)$. Applied to a seed $\mathbf i$, the
application  ${\mathcal X}_{[.]_{\mathfrak R}}(t)$
is therefore the composition of the map ${\mathcal X}_{.}$ and the truncation map
$\mathfrak{t}_{\mathbf i_{\mathfrak R}(t)}$. We then sum-up Theorem \ref{thm:ev*1}
and Proposition \ref{lemma:dtaumoves} by (abusively) saying that there
exists a Poisson map
$\widehat{\ev}_{w_1\leq v}:{\mathcal X}_{w_1\leq v}\to(F_{t,v},\pi_{*})\ .$
We thus get the following commutative diagram, which generalize Figure
\ref{fig:clustervarietydual1}. (The link between these truncated cluster varieties
${\mathcal X}_{w_1\leq v}(t)$, as well as the way to link the different
twisted evaluations, is given in the next section via the
introduction of saltation maps.)

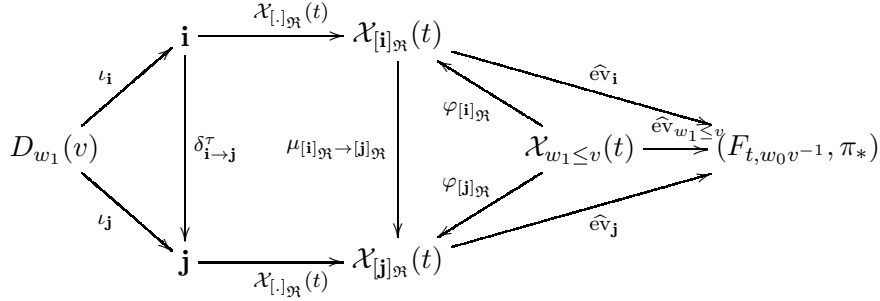
\begin{figure}[htbp]
\begin{center}
\setlength{\unitlength}{1.5pt}
\qquad\qquad
\xymatrix{
&\mathbf i\ar@/_0pc/[dd]^{\delta_{\mathbf i\to\mathbf j}^{\tau}}
\ar@/_0pc/[rr]^{\mathcal{X}_{[.]_{\mathfrak R}}(t)}&&{\mathcal X}_{[\mathbf i]_{\mathfrak R}}(t)
\ar@/_0pc/[dd]_{\mu_{[\mathbf i]_{\mathfrak R}\to[\mathbf j]_{\mathfrak R}}}
\ar@/^0pc/[rrd]^{\widehat{\ev}_{\mathbf i}}&&\\
D_{w_1}(v)\ar@/^0pc/[ru]^{\iota_{\mathbf i}}\ar@/^0pc/[rd]_{\iota_{\mathbf j}}
&&&&{{\mathcal X}_{w_1\leq v}(t)}\ar@/^0pc/[lu]^{\varphi_{[\mathbf i]_{\mathfrak R}}}\ar@/^0pc/[ld]_{\varphi_{[\mathbf j]_{\mathfrak R}}}
\ar@/^0pc/[r]^{\widehat{\ev}_{w_1\leq v}}&(F_{t,w_0v^{-1}},\pi_*)\\
&\mathbf j\ar@/_0pc/[rr]_{\mathcal{X}_{[.]_{\mathfrak R}}(t)}&&{\mathcal X}_{[\mathbf {j}]_{\mathfrak R}}(t)
\ar@/^0pc/[rru]_{\widehat{\ev}_{\mathbf j}}&&
}
\end{center}
\vspace{-.1in}
\caption{The cluster $\mathcal X$-variety ${\mathcal X}_{w_1\leq v}(t)$ associated
to $(F_{t,w_0v^{-1}},\pi_*)$}
\label{fig:clustervariety2}
\end{figure}

\section{Saltations and cluster $\mathcal X$-varieties for $(BB_-,\pi_*)$}
\label{section:Loop}
We relate the twisted evaluations
$\widehat{\ev}_{\mathbf i}:{\mathcal X}_{[\mathbf i]_{\mathfrak R}}\to (G,\pi_*)$ of
Section \ref{section:twistedev} by composition of cluster transformations with new birational
Poisson isomorphisms called saltations. As a corollary, we get a parametrization of
the dual Poisson Lie-group $(BB_-,\pi_*)$ by a family of cluster
$\mathcal X$-varieties; moreover, the cluster $\mathcal X$-varieties of this
family are related by saltations indexed by the $1$-skeleton
of the $W$-permutohedron $P_W$.

\subsection{Various moves on the set $D(v)$}\label{subsection:dualmove}
We sharpen the description of the set $D(v)$ for every $v\in W$. To do that,
we enlarge the combinatorics on double words by introducing \emph{dual moves}
and mix them with the $d^{\tau}$-moves described in Section \ref{section:taucombi}.
We start by adding a
variation of $\tau$-moves, involving the involution $i\mapsto i^{\star}$
on the set $[1,l]$
.

\begin{definition} Let $\mathbf i=i_1\dots i_n$ be a double word. We denote
${\mathfrak R}^{\star}_{|i_n|}(\mathbf i)$, or simply ${\mathfrak R}^{\star}(\mathbf i)$
when no confusion occurs, the double word obtained by changing the last letter
${i}$ of $\mathbf i$ into~$\overline i^{\star}$:
$$\begin{array}{rcl}
{\mathfrak R}^{\star}_{i_n}(\mathbf i)=i_1\dots i_{n-1}\overline{i_n}^{\star}
&\mbox{if}&i_n\in[1,l]\ ;\\
\\
{\mathfrak R}^{\star}_{\overline{i_n^{\star}}}(\mathbf i)=i_1\dots i_{n-1}
\overline{i_n}^{\star}&\mbox{if}&i_n\in[\overline{1},\overline{l}]\ .
\end{array}
$$
The map $\mathbf i\mapsto {\mathfrak R}^{\star}(\mathbf i)$ is called a
\emph{right} \emph{$\tau^{\star}$-move} on $\mathbf i$.
Because the maps $i\mapsto i^{\star}$ and $i\mapsto\overline{i}$ are
involutions, it is clear that every map $\mathfrak{R}_j^{\star}$ is an involution
on the set of double words.
A \emph{$d^{\tau^{\star}}$-move} is then given by one
of these transformations:
\begin{itemize}
\item
a generalized $d$-move;
\item
a $\tau^{\star}$-move.
\end{itemize}
For every $w\in W$, let $R^{\tau^{\star}}(w)$ be the set of
all the double words obtained from a double reduced word $\mathbf i\in R(1,w)$ by
composition of $d^{\tau^{\star}}$-moves. (The choice of the double word
$\mathbf i$ doesn't matter, because of Theorem \ref{thm:Tits}.)
\end{definition}

\begin{ex}\label{ex:A2taustar}
When $\mathfrak{g}=A_2$, the action of
$\tau^{\star}$-move on the double reduced words $\mathbf{i_1}=121$,
$\mathbf{i_2}=\overline{2}12$ and $\mathbf{i_3}=\overline{2}\overline{1}1$
is given by:
$$\begin{array}{cccc}
{\mathfrak R}^{\star}_1(\mathbf{i_1})=12\overline{2},&
{\mathfrak R}^{\star}_2(\mathbf{i_2})=\overline{2}1\overline{1},
&\mbox{and}&{\mathfrak R}^{\star}_1(\mathbf{i_3})=\overline{2}
\overline{1}\overline{2}\ .
\end{array}$$
\end{ex}

As every set $R^{\tau}(w)$, the set $R^{\tau^{\star}}(w)$, $w\in W$, is
easily described via the $W$-permutohedron. Indeed, we have the following
analog of Lemma \ref{lemma:PW} for the set $R^{\tau^{\star}}(w)$.
It is proved in the same way.

\begin{lemma}\label{lemma:PWRstar}The following statements are satisfied
for every $w\in W$.
\begin{itemize}
\item
The set $R^{\tau^{\star}}(w)$ is the disjoint union of the
labels $R({w'^{\star}}^{-1},w{w'}^{-1})$ associated to the vertices $w'$
of the $W$-permutohedron $P_W$ that are crossed or reached by a
$\uparrow_w$-path.
\item
Two labeled vertices $R(u,v)\subset R^{\tau^{\star}}(w)$ and $R(u',v')
\subset R^{\tau^{\star}}(w)$ of $P_W$ are related by the edge $s_i$ if and only
if there exist double reduced words $\mathbf i\in R(u,v)$ and
$\mathbf j\in R(u',v')$ such that $\mathbf j={\mathfrak R}^{\star}(\mathbf i)$.
\end{itemize}
\end{lemma}

Here is the variation of Example \ref{ex:WRstable} which has described the
sets $R^{\tau}(s_2s_1)$ and $R^{\tau}(w_0)$ when $\mathfrak g=A_2$:
Figure \ref{fig:Wpermutohedre} is now replaced by Figure \ref{fig:Wpermutohedre*}.
\begin{figure}[htbp]
\begin{center}
\setlength{\unitlength}{1.5pt}
\qquad\qquad\qquad\qquad\qquad
\begin{picture}(20,110)(0,-80)
\put(-31,-63.5){\line(-2,3){17.7}}
\put(-50,-33){\line(0,1){26}}
\thicklines
\put(-50,-5){\circle{4}}
\put(-87,-10){$R({s_2s_1,1})$}
\put(-50,-35){\circle{4}}
\put(-83,-40){$R({s_2,s_2})$}
\put(-30,-65){\circle{4}}
\put(-39,-75){$R(1,s_2s_1)$}
\put(-47,-55){\footnotesize ${\mathfrak R}^{\star}_1$}
\put(-60,-25){\footnotesize ${\mathfrak R}^{\star}_2$}
\end{picture}
\qquad\qquad\qquad\qquad\qquad
\begin{picture}(20,110)(0,-80)
\put(-29,-63.5){\line(2,3){17.7}}
\put(-10,-33){\line(0,1){26}}
\put(-11,-3.5){\line(-2,3){17.7}}
\put(-31,-63.5){\line(-2,3){17.7}}
\put(-50,-33){\line(0,1){26}}
\put(-49,-3.5){\line(2,3){17.7}}
\thicklines
\put(-30,25){\circle{4}}
\put(-39,30){$R({w_0,1})$}
\put(-50,-5){\circle{4}}
\put(-87,-10){$R({s_2s_1,s_1})$}
\put(-10,-5){\circle{4}}
\put(-5,-10){$R({s_1s_2,s_2})$}
\put(-50,-35){\circle{4}}
\put(-87,-40){$R({s_2,s_1s_2})$}
\put(-10,-35){\circle{4}}
\put(-5,-40){$R({s_1,s_2s_1})$}
\put(-30,-65){\circle{4}}
\put(-39,-75){$R(1,w_0)$}
\put(-20,15){\footnotesize ${\mathfrak R}^{\star}_2$}
\put(-47,15){\footnotesize ${\mathfrak R}^{\star}_1$}
\put(-60,-25){\footnotesize ${\mathfrak R}^{\star}_2$}
\put(-8,-25){\footnotesize ${\mathfrak R}^{\star}_1$}
\put(-47,-55){\footnotesize ${\mathfrak R}^{\star}_1$}
\put(-20,-55){\footnotesize ${\mathfrak R}^{\star}_2$}
\end{picture}
\end{center}
\vspace{-.1in}
\caption{The sets $R^{\tau^{\star}}(s_2s_1)$ and $R^{\tau^{\star}}(w_0)$ when $\mathfrak{g}=A_2$}
\label{fig:Wpermutohedre*}
\end{figure}
We now use the previous $\tau^{\star}$-moves and the involution
$\square$ given in Subsection \ref{section:involution} to define
new moves on double words, besides generalized $dn$-moves and $\tau$-moves;
they are called {dual-moves}.

\begin{definition} Let $\mathbf i\in R(1,w_0)\cup R(w_0,1)$ be a positive or
negative reduced word associated to $w_0$ and $\mathbf j$ be a
double word. The following \emph{dual-move} $\Delta_{j}$ associated to the
last letter of $\mathbf j$ transforms the product $\mathbf j\mathbf i$ into
the following double word:
\begin{equation}\label{equ:defdual moves}
\Delta_{j}:\mathbf j\mathbf i\mapsto {\mathfrak R}^{\star}_j(\mathbf j)\
{\mathbf i}^{\square}\ .
\end{equation}
Right $\tau^{\star}$-moves and the map $\square$ being involutions,
dual-moves are in fact involutions on the set of such products
$\mathbf j\mathbf i$. Let $v\in W$ and $\mathbf i\in D(v)$ be a double word.
A \emph{$\widehat{d}$-move} on $\mathbf i$ is one of the following transformations.
(In particular, every $\widehat{d}$-move is a $d^{\tau}$-move when $v=1$.)
\begin{itemize}
\item
a $d^{\tau}$-move;
\item
a dual-move $\Delta_{i}$.
\end{itemize}
\end{definition}

As an example, let us keep the notation of Example \ref{ex:A2taustar} and consider
the double words $\mathbf{j_1}=\mathbf{i_1}\mathbf j$, $\mathbf{j_2}=
\mathbf{i_2}\mathbf j$ and $\mathbf{j_3}=\mathbf{i_3}\mathbf j$, with
$\mathbf j=\overline{1}\overline{2}\overline{1}$. Because the equality
$\mathbf j^{\square}=121$ is satisfied, the action of dual
moves on these double words give:
$$\begin{array}{cccc}
\Delta_{2}(\mathbf{j_1})=12\overline{2}\ 121,&
\Delta_{1}(\mathbf{j_2})=\overline{2}1\overline{1}\ 121,
&\mbox{and}&\Delta_{2}(\mathbf{j_3})=\overline{2}
\overline{1}\overline{2}\ 121\ .
\end{array}$$

\begin{lemma}\label{lemma:Whatdmoves}
The following statements are satisfied for every $v\in W$.
\begin{itemize}
\item
The set $D(v)$ is the disjoint union of the
labels $D_{w_1}(v)$ associated to the vertices $w_1$
of the $W$-permutohedron $P_W$ that are crossed or reached by a
$\uparrow_v$-path.
\item
Two labeled vertices $D_{w_1}(v)\subset D(v)$ and $D_{w'_1}(v)
\subset D(v)$ of $P_W$ are related by the edge $s_i$
if and only if there exist trivial double reduced words $\mathbf i\in D_{w_1}(v)$ and
$\mathbf j\in D_{w'_1}(v)$ such that $\mathbf j=\Delta_i(\mathbf i)$.
\end{itemize}
\end{lemma}
\begin{proof}The first statement is just a reformulation of the second
relation given by (\ref{equ:Dsets}). The second statement is obtained
by applying Lemma \ref{lemma:relPWR} and Lemma \ref{lemma:PWRstar},
because of the formula (\ref{equ:defdual moves}) describing dual moves.
\end{proof}

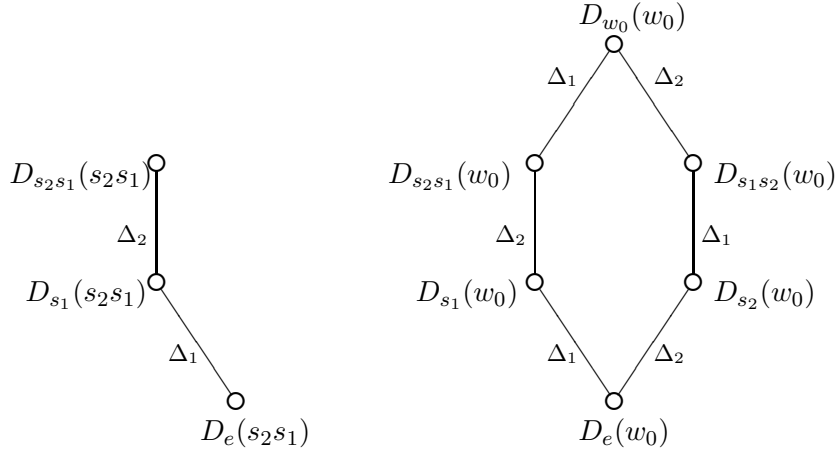
\begin{figure}[htbp]
\begin{center}
\setlength{\unitlength}{1.5pt}
\qquad\qquad\qquad\qquad\qquad
\begin{picture}(20,110)(0,-80)
\put(-31,-63.5){\line(-2,3){17.7}}
\put(-50,-33){\line(0,1){26}}
\thicklines
\put(-50,-5){\circle{4}}
\put(-87,-10){$D_{s_2s_1}(s_2s_1)$}
\put(-50,-35){\circle{4}}
\put(-83,-40){$D_{s_1}(s_2s_1)$}
\put(-30,-65){\circle{4}}
\put(-39,-75){$D_e(s_2s_1)$}
\put(-47,-55){\footnotesize $\Delta_1$}
\put(-60,-25){\footnotesize $\Delta_2$}
\end{picture}
\qquad\qquad\qquad\qquad\qquad
\begin{picture}(20,110)(0,-80)
\put(-29,-63.5){\line(2,3){17.7}}
\put(-10,-33){\line(0,1){26}}
\put(-11,-3.5){\line(-2,3){17.7}}
\put(-31,-63.5){\line(-2,3){17.7}}
\put(-50,-33){\line(0,1){26}}
\put(-49,-3.5){\line(2,3){17.7}}
\thicklines
\put(-30,25){\circle{4}}
\put(-39,30){$D_{w_0}(w_0)$}
\put(-20,15){\footnotesize $\Delta_2$}
\put(-47,15){\footnotesize $\Delta_1$}
\put(-50,-5){\circle{4}}
\put(-87,-10){$D_{s_2s_1}(w_0)$}
\put(-60,-25){\footnotesize $\Delta_2$}
\put(-10,-5){\circle{4}}
\put(-5,-10){$D_{s_1s_2}(w_0)$}
\put(-8,-25){\footnotesize $\Delta_1$}
\put(-50,-35){\circle{4}}
\put(-80,-40){$D_{s_1}(w_0)$}
\put(-47,-55){\footnotesize $\Delta_1$}
\put(-10,-35){\circle{4}}
\put(-5,-40){$D_{s_2}(w_0)$}
\put(-20,-55){\footnotesize $\Delta_2$}
\put(-30,-65){\circle{4}}
\put(-39,-75){$D_e(w_0)$}
\end{picture}
\end{center}
\vspace{-.1in}
\caption{The sets $D(s_2s_1)$ and $D(w_0)$ when $\mathfrak{g}=A_2$}
\label{fig:WpermutohedreD*}
\end{figure}

We continue our running example with the case $\mathfrak g=A_2$. Using Figure
\ref{fig:Wpermutohedre*}, we give in Figure \ref{fig:WpermutohedreD*}
the description of the sets $D(s_2s_1)$ and $D(w_0)$ in terms of subsets
$D_{w_1}(s_2s_1)$ and $D_{w_1}(w_0)$ related by dual moves. Now, we have seen
in Example \ref{ex:permutoWD} that each subset $D_{w_1}(w_0)$ can be decomposed
into sets $W(u,v)$, for appropriate $u,v\in W$, related by right $\tau^{\star}$-moves;
this is described by Figure \ref{fig:commutativeweakordertau+2}.
Therefore, mixing Figure \ref{fig:commutativeweakordertau+2} and Figure
\ref{fig:WpermutohedreD*}, we get in Figure \ref{fig:Dpermutosynthese} a description
of the set $D(w_0)$ as unions of sets $W(u,v)$ related by $\widehat{d}$-moves.
Let us notice the double occurring of the permutohedron $P_3$ in this picture.

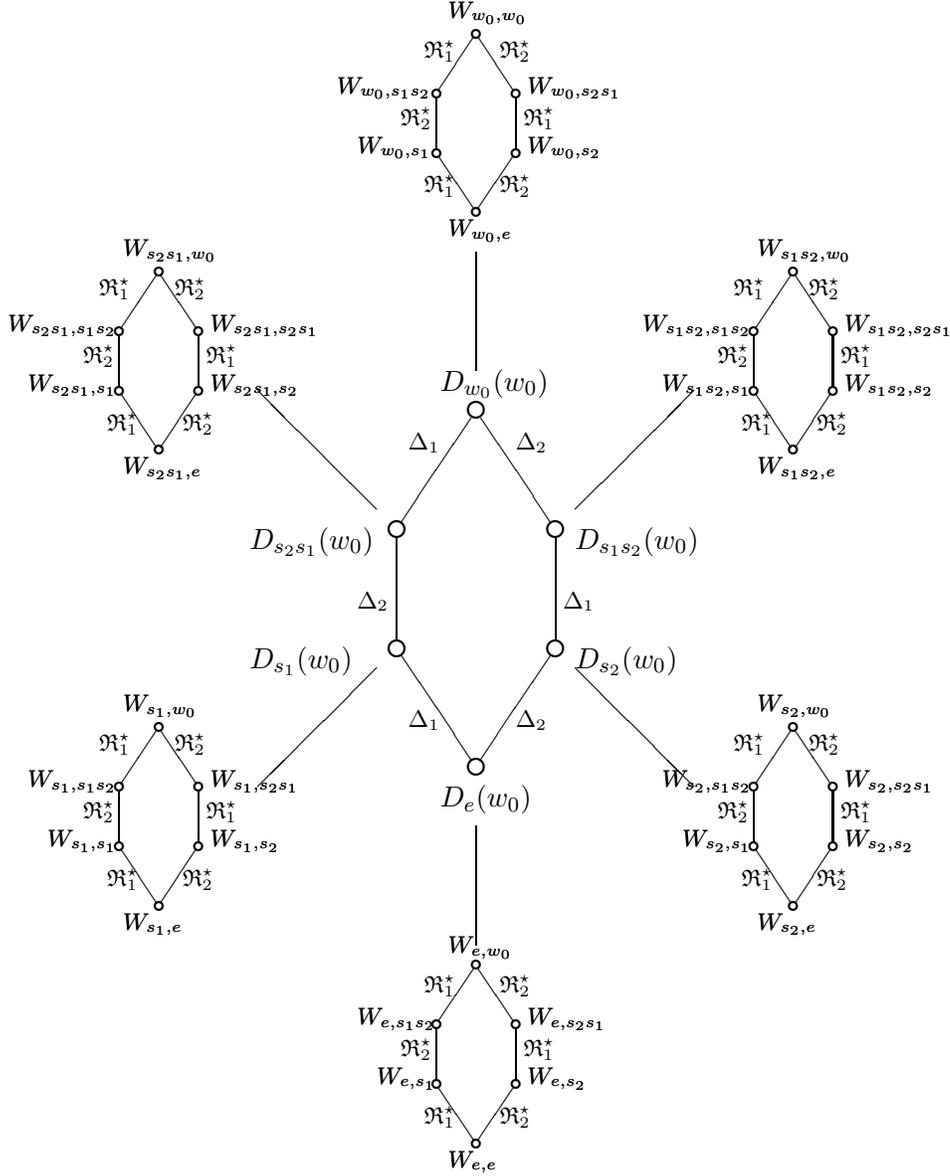
\begin{figure}[htbp]
\begin{center}
\setlength{\unitlength}{1.5pt}
\qquad\qquad\qquad\qquad\qquad
\begin{picture}(20,300)(0,-170)
\put(-29,-63.5){\line(2,3){17.7}}
\put(-10,-33){\line(0,1){26}}
\put(-11,-3.5){\line(-2,3){17.7}}
\put(-31,-63.5){\line(-2,3){17.7}}
\put(-50,-33){\line(0,1){26}}
\put(-49,-3.5){\line(2,3){17.7}}
\put(-30,35){\line(0,1){30}}
\put(-30,-80){\line(0,-1){30}}
\put(-55,0){\line(-1,1){30}}
\put(-5,0){\line(1,1){30}}
\put(-55,-40){\line(-1,-1){30}}
\put(-5,-40){\line(1,-1){30}}
\put(-29.5,75.75){\line(2,3){8.85}}
\put(-20,91){\line(0,1){13}}
\put(-30.5,75.75){\line(-2,3){8.85}}
\put(-20.5,105.75){\line(-2,3){8.85}}
\put(-40,91){\line(0,1){13}}
\put(-39.5,105.75){\line(2,3){8.85}}
\put(-29.5,-159.25){\line(2,3){8.85}}
\put(-20,-144){\line(0,1){13}}
\put(-30.5,-159.25){\line(-2,3){8.85}}
\put(-20.5,-129.25){\line(-2,3){8.85}}
\put(-40,-144){\line(0,1){13}}
\put(-39.5,-129.25){\line(2,3){8.85}}
\put(50.5,-99.25){\line(2,3){8.85}}
\put(60,-84){\line(0,1){13}}
\put(49.5,-99.25){\line(-2,3){8.85}}
\put(59.5,-69.25){\line(-2,3){8.85}}
\put(40,-84){\line(0,1){13}}
\put(40.5,-69.25){\line(2,3){8.85}}
\put(50.5,15.75){\line(2,3){8.85}}
\put(60,31){\line(0,1){13}}
\put(49.5,15.75){\line(-2,3){8.85}}
\put(59.5,45.75){\line(-2,3){8.85}}
\put(40,31){\line(0,1){13}}
\put(40.5,45.75){\line(2,3){8.85}}
\put(-109.5,-99.25){\line(2,3){8.85}}
\put(-120,-84){\line(0,1){13}}
\put(-110.5,-99.25){\line(-2,3){8.85}}
\put(-119.5,-69.25){\line(2,3){8.85}}
\put(-100,-84){\line(0,1){13}}
\put(-100.5,-69.25){\line(-2,3){8.85}}
\put(-109.5,15.75){\line(2,3){8.85}}
\put(-120,31){\line(0,1){13}}
\put(-110.5,15.75){\line(-2,3){8.85}}
\put(-119.5,45.75){\line(2,3){8.85}}
\put(-100,31){\line(0,1){13}}
\put(-100.5,45.75){\line(-2,3){8.85}}
\thicklines
\put(-30,25){\circle{4}}
\put(-39,30){$D_{w_0}(w_0)$}
\put(-50,-5){\circle{4}}
\put(-87,-10){$D_{s_2s_1}(w_0)$}
\put(-10,-5){\circle{4}}
\put(-5,-10){$D_{s_1s_2}(w_0)$}
\put(-50,-35){\circle{4}}
\put(-87,-40){$D_{s_1}(w_0)$}
\put(-10,-35){\circle{4}}
\put(-5,-40){$D_{s_2}(w_0)$}
\put(-30,-65){\circle{4}}
\put(-39,-75){$D_e(w_0)$}
\put(-30,120){\circle{2}}
\put(-37,124){\footnotesize $W_{w_0,w_0}$}
\put(-20,105){\circle{2}}
\put(-17,105){\footnotesize $W_{w_0,s_2s_1}$}
\put(-20,90){\circle{2}}
\put(-17,90){\footnotesize $W_{w_0,s_2}$}
\put(-40,105){\circle{2}}
\put(-65,105){\footnotesize $W_{w_0,s_1s_2}$}
\put(-40,90){\circle{2}}
\put(-60,90){\footnotesize $W_{w_0,s_1}$}
\put(-30,75){\circle{2}}
\put(-37,69){\footnotesize $W_{w_0,e}$}
\put(-110,60){\circle{2}}
\put(-119,64){\footnotesize $W_{s_2s_1,w_0}$}
\put(-100,45){\circle{2}}
\put(-97,45){\footnotesize $W_{s_2s_1,s_2s_1}$}
\put(-100,30){\circle{2}}
\put(-97,30){\footnotesize $W_{s_2s_1,s_2}$}
\put(-120,45){\circle{2}}
\put(-148,45){\footnotesize $W_{s_2s_1,s_1s_2}$}
\put(-120,30){\circle{2}}
\put(-143,30){\footnotesize $W_{s_2s_1,s_1}$}
\put(-110,15){\circle{2}}
\put(-119,9){\footnotesize $W_{s_2s_1,e}$}
\put(50,60){\circle{2}}
\put(41,64){\footnotesize $W_{s_1s_2,w_0}$}
\put(40,45){\circle{2}}
\put(12,45){\footnotesize $W_{s_1s_2,s_1s_2}$}
\put(40,30){\circle{2}}
\put(17,30){\footnotesize $W_{s_1s_2,s_1}$}
\put(60,45){\circle{2}}
\put(63,45){\footnotesize $W_{s_1s_2,s_2s_1}$}
\put(60,30){\circle{2}}
\put(63,30){\footnotesize $W_{s_1s_2,s_2}$}
\put(50,15){\circle{2}}
\put(41,9){\footnotesize $W_{s_1s_2,e}$}
\put(-110,-55){\circle{2}}
\put(-119,-51){\footnotesize $W_{s_1,w_0}$}
\put(-100,-70){\circle{2}}
\put(-97,-70){\footnotesize $W_{s_1,s_2s_1}$}
\put(-100,-85){\circle{2}}
\put(-97,-85){\footnotesize $W_{s_1,s_2}$}
\put(-120,-70){\circle{2}}
\put(-143,-70){\footnotesize $W_{s_1,s_1s_2}$}
\put(-120,-85){\circle{2}}
\put(-138,-85){\footnotesize $W_{s_1,s_1}$}
\put(-110,-100){\circle{2}}
\put(-119,-106){\footnotesize $W_{s_1,e}$}
\put(50,-55){\circle{2}}
\put(41,-51){\footnotesize $W_{s_2,w_0}$}
\put(40,-70){\circle{2}}
\put(17,-70){\footnotesize $W_{s_2,s_1s_2}$}
\put(40,-85){\circle{2}}
\put(22,-85){\footnotesize $W_{s_2,s_1}$}
\put(60,-70){\circle{2}}
\put(63,-70){\footnotesize $W_{s_2,s_2s_1}$}
\put(60,-85){\circle{2}}
\put(63,-85){\footnotesize $W_{s_2,s_2}$}
\put(50,-100){\circle{2}}
\put(41,-106){\footnotesize $W_{s_2,e}$}
\put(-30,-160){\circle{2}}
\put(-37,-166){\footnotesize $W_{e,e}$}
\put(-20,-145){\circle{2}}
\put(-17,-145){\footnotesize $W_{e,s_2}$}
\put(-20,-130){\circle{2}}
\put(-17,-130){\footnotesize $W_{e,s_2s_1}$}
\put(-40,-145){\circle{2}}
\put(-55,-145){\footnotesize $W_{e,s_1}$}
\put(-40,-130){\circle{2}}
\put(-60,-130){\footnotesize $W_{e,s_1s_2}$}
\put(-30,-115){\circle{2}}
\put(-37,-112){\footnotesize $W_{e,w_0}$}
\put(-20,15){\footnotesize $\Delta_2$}
\put(-47,15){\footnotesize $\Delta_1$}
\put(-60,-25){\footnotesize $\Delta_2$}
\put(-8,-25){\footnotesize $\Delta_1$}
\put(-47,-55){\footnotesize $\Delta_1$}
\put(-20,-55){\footnotesize $\Delta_2$}
\put(-37,124){\footnotesize $W_{w_0,w_0}$}
\put(-43,114){\footnotesize ${\mathfrak R}^{\star}_1$}
\put(-24,114){\footnotesize ${\mathfrak R}^{\star}_2$}
\put(-17,105){\footnotesize $W_{w_0,s_2s_1}$}
\put(-49,97){\footnotesize ${\mathfrak R}^{\star}_2$}
\put(-18,97){\footnotesize ${\mathfrak R}^{\star}_1$}
\put(-17,90){\footnotesize $W_{w_0,s_2}$}
\put(-43,80){\footnotesize ${\mathfrak R}^{\star}_1$}
\put(-24,80){\footnotesize ${\mathfrak R}^{\star}_2$}
\put(-65,105){\footnotesize $W_{w_0,s_1s_2}$}
\put(-60,90){\footnotesize $W_{w_0,s_1}$}
\put(-37,69){\footnotesize $W_{w_0,e}$}
\put(-119,64){\footnotesize $W_{s_2s_1,w_0}$}
\put(-125,54){\footnotesize ${\mathfrak R}^{\star}_1$}
\put(-106,54){\footnotesize ${\mathfrak R}^{\star}_2$}
\put(-97,45){\footnotesize $W_{s_2s_1,s_2s_1}$}
\put(-129,37){\footnotesize ${\mathfrak R}^{\star}_2$}
\put(-98,37){\footnotesize ${\mathfrak R}^{\star}_1$}
\put(-97,30){\footnotesize $W_{s_2s_1,s_2}$}
\put(-123,20){\footnotesize ${\mathfrak R}^{\star}_1$}
\put(-104,20){\footnotesize ${\mathfrak R}^{\star}_2$}
\put(-148,45){\footnotesize $W_{s_2s_1,s_1s_2}$}
\put(-143,30){\footnotesize $W_{s_2s_1,s_1}$}
\put(-119,9){\footnotesize $W_{s_2s_1,e}$}
\put(41,64){\footnotesize $W_{s_1s_2,w_0}$}
\put(35,54){\footnotesize ${\mathfrak R}^{\star}_1$}
\put(54,54){\footnotesize ${\mathfrak R}^{\star}_2$}
\put(12,45){\footnotesize $W_{s_1s_2,s_1s_2}$}
\put(17,30){\footnotesize $W_{s_1s_2,s_1}$}
\put(63,45){\footnotesize $W_{s_1s_2,s_2s_1}$}
\put(31,37){\footnotesize ${\mathfrak R}^{\star}_2$}
\put(62,37){\footnotesize ${\mathfrak R}^{\star}_1$}
\put(63,30){\footnotesize $W_{s_1s_2,s_2}$}
\put(37,20){\footnotesize ${\mathfrak R}^{\star}_1$}
\put(56,20){\footnotesize ${\mathfrak R}^{\star}_2$}
\put(41,9){\footnotesize $W_{s_1s_2,e}$}
\put(-119,-51){\footnotesize $W_{s_1,w_0}$}
\put(-125,-61){\footnotesize ${\mathfrak R}^{\star}_1$}
\put(-106,-61){\footnotesize ${\mathfrak R}^{\star}_2$}
\put(-97,-70){\footnotesize $W_{s_1,s_2s_1}$}
\put(-129,-78){\footnotesize ${\mathfrak R}^{\star}_2$}
\put(-98,-78){\footnotesize ${\mathfrak R}^{\star}_1$}
\put(-97,-85){\footnotesize $W_{s_1,s_2}$}
\put(-123,-95){\footnotesize ${\mathfrak R}^{\star}_1$}
\put(-104,-95){\footnotesize ${\mathfrak R}^{\star}_2$}
\put(-143,-70){\footnotesize $W_{s_1,s_1s_2}$}
\put(-138,-85){\footnotesize $W_{s_1,s_1}$}
\put(-119,-106){\footnotesize $W_{s_1,e}$}
\put(41,-51){\footnotesize $W_{s_2,w_0}$}
\put(35,-61){\footnotesize ${\mathfrak R}^{\star}_1$}
\put(54,-61){\footnotesize ${\mathfrak R}^{\star}_2$}
\put(17,-70){\footnotesize $W_{s_2,s_1s_2}$}
\put(22,-85){\footnotesize $W_{s_2,s_1}$}
\put(63,-70){\footnotesize $W_{s_2,s_2s_1}$}
\put(31,-78){\footnotesize ${\mathfrak R}^{\star}_2$}
\put(62,-78){\footnotesize ${\mathfrak R}^{\star}_1$}
\put(63,-85){\footnotesize $W_{s_2,s_2}$}
\put(37,-95){\footnotesize ${\mathfrak R}^{\star}_1$}
\put(56,-95){\footnotesize ${\mathfrak R}^{\star}_2$}
\put(41,-106){\footnotesize $W_{s_2,e}$}
\put(-37,-166){\footnotesize $W_{e,e}$}
\put(-17,-145){\footnotesize $W_{e,s_2}$}
\put(-43,-155){\footnotesize ${\mathfrak R}^{\star}_1$}
\put(-24,-155){\footnotesize ${\mathfrak R}^{\star}_2$}
\put(-17,-130){\footnotesize $W_{e,s_2s_1}$}
\put(-49,-138){\footnotesize ${\mathfrak R}^{\star}_2$}
\put(-18,-138){\footnotesize ${\mathfrak R}^{\star}_1$}
\put(-55,-145){\footnotesize $W_{e,s_1}$}
\put(-60,-130){\footnotesize $W_{e,s_1s_2}$}
\put(-37,-112){\footnotesize $W_{e,w_0}$}
\put(-43,-122){\footnotesize ${\mathfrak R}^{\star}_1$}
\put(-24,-122){\footnotesize ${\mathfrak R}^{\star}_2$}
\end{picture}
\end{center}
\vspace{-.1in}
\caption{The set $D(w_0)$ when $\mathfrak{g}=A_2$}
\label{fig:Dpermutosynthese}
\end{figure}

\subsection{Saltations}\label{subsection:saltation}
We are now ready to introduce the {saltations}, and use them to
describe the cluster combinatorics associated to double words differing
from a dual move. Roughly speaking saltations are a generalization of
generalized cluster transformations involving truncation maps. When we deal
with generalized cluster transformations, the combinatorics giving the formulas
is described by the Poisson bivector of the seed $\mathcal X$-torus (i.e. the seed
matrix usually denoted $\varepsilon$), which, in turn, is transformed
by these generalized cluster transformations. The idea underlying the definition of
the saltations is simple: we allow a little more freedom between the combinatorics on
seed $\mathcal X$-tori and their Poisson geometry.
Let us remember the truncated torus ${\mathcal X}_J^0$ associated to a
truncation map ${\mathfrak t}_J$ and given by Definition \ref{def:truncatedtorus}.

\begin{definition}\label{def:saltation}Let $\mathbf I=(I,I_0,\varepsilon,d)$,
and $\mathbf{I'}=(I',I'_0,\varepsilon',d')$ be two seeds, $J\subset I$,
$J'\subset I'$ be two isomorphic subsets, and $\mathfrak{t}_{J}(\mathbf I)$,
$\mathfrak{t}_{J'}(\mathbf{I'})$ the related truncation maps. A birational
Poisson isomorphism
$\Xi:\mathcal{X}_{\mathfrak{t}_{J}(\mathbf I)}
\to\mathcal{X}_{\mathfrak{t}_{J'}(\mathbf{I'})}$
is said to be a \emph{saltation} (relatively to the subsets $J,J'$) if there exists
a generalized cluster transformation
$\phi_{\mathbf I\to\mathbf{I'}}^{\tau}:\mathcal{X}_{\mathbf I}
\to\mathcal{X}_{\mathbf{I'}}$ which makes the following diagram commutative for every
$\mathbf t\in{\mathcal X}_J^0$.
\begin{equation}\label{equ:saltacom}
\xymatrix{
{{\mathcal X}_{\mathbf{I}}}\ar@/^1pc/[r]^{\phi_{\mathbf I\to\mathbf{I'}}^{\tau}}
\ar@/_1pc/[d]_{\mathfrak{t}_{J}(\mathbf{t})}&{{\mathcal X}_{\mathbf{I'}}}
\ar@/^1pc/[d]^{\mathfrak{t}_{J'}(\mathbf{t})}\\
{\mathcal{X}_{\mathfrak{t}_{J}(\mathbf I)}(\mathbf t)}\ar@/_1pc/[r]_{\Xi}
&{\mathcal{X}_{\mathfrak{t}_{J'}(\mathbf{I'})}(\mathbf t)}
}
\end{equation}
Saltation are easily composed: if $\Xi_1$ is a saltation relatively to the
sets $J,J'$ and to a generalized cluster transformation $\phi_1$ and $\Xi_2$
a saltation relatively to the sets $J',J''$ and to the generalized cluster
transformation $\phi_2$, then the composition $\Xi_2\circ\Xi_1$ is a saltation
relatively to the sets $J,J''$ and to the generalized cluster transformation
$\phi_2\circ\phi_1$.
\end{definition}

A few saltations have already been encountered before.
\begin{itemize}
\item
Every generalized cluster transformation is a saltation:
just take the sets $J$ and $J'$ equal to the empty set $\emptyset$.
\item
For every $u,v\in W$ and every double reduced words $\mathbf i,\mathbf j\in R(u,v)$,
the cluster transformation $\mu_{[\mathbf i]_{\mathfrak R}\to[\mathbf j]_{\mathfrak R}}$ is the saltation
relative to the cluster transformation $\mu_{\mathbf i\to\mathbf j}$, $J$ being the
set of right outlets relative to the seed $\mathbf I(\mathbf i)$ and $J'$ the set
of right outlets relative to the seed $\mathbf I(\mathbf j)$.
\item
For every $w\in W$ and every $\mathbf i\in R^{\tau}(w)$, the cluster transformation
$\zeta_{[\mathbf i]_{\mathfrak R}}$ is the saltation relative to the generalized cluster transformation
$\zeta_{\mathbf i}$, where the set $J$ is the set of right outlets relative to the seed
$\mathbf I(\mathbf i)$ and $J'$ is the set of right outlets relative to the seed
$\mathbf I(\mathbf i^{\square})$.
\item
More generally, the saltation associated to any generalized cluster transformation $\phi$
is always a cluster transformation if the set $J$ contains the directions relative to all the
tropical mutations that are used to factorize $\phi$, and $\phi(J)\subset J'$.
\end{itemize}

\begin{rem}Every generalized cluster transformation $\phi_{\mathbf I
\to\mathbf{I'}}^{\tau}$ is a product of symmetries, mutations and
tropical mutations, and every symmetry of a finite set $J$ can be
decomposed into a product of transpositions. Therefore it is tempting
to decompose every saltation $\Xi$ as a product of elementary saltations
of the following forms.
$$\xymatrix{
{\mathcal X}_{\mathbf I}\ar@/_0pc/[d]_{\mathfrak{t}_J}
\ar@/_0pc/[rr]^{\id}&&{\mathcal X}_{\mathbf I}
\ar@/_0pc/[d]^{\mathfrak{t}_{J'}}&{\mathcal X}_{\mathbf I}\ar@/_0pc/[d]_{\mathfrak{t}_J}
\ar@/_0pc/[rr]^{\mu_k}&&{\mathcal X}_{\mu_k(\mathbf I)}
\ar@/_0pc/[d]^{\mathfrak{t}_{J}}\\
{\mathcal X}_{\mathbf I_J}(t)\ar@/_0pc/[rr]_{{\Xi}_1}&&{\mathcal X}_{\mathbf I_{J'}}(t)
&{\mathcal X}_{\mathbf I_J}(t)\ar@/_0pc/[rr]_{{\Xi}_2}&&{\mathcal X}_{\mu_k(\mathbf I)_J}(t)
}
$$
$$\xymatrix{
{\mathcal X}_{\mathbf I}\ar@/_0pc/[d]_{\mathfrak{t}_J}
\ar@/_0pc/[rr]^{\mu_k^{\tau}}&&{\mathcal X}_{\mu_k^{\tau}(\mathbf I)}
\ar@/_0pc/[d]^{\mathfrak{t}_{J}}\\
{\mathcal X}_{\mathbf I_J}(t)\ar@/_0pc/[rr]_{{\Xi}_3}&&{\mathcal X}_{\mu_k^{\tau}(\mathbf I)_J}(t)}
$$
The problem is that the first map $\Xi_1$ is clearly undefined if the
symmetry $s:J\mapsto J'$ is not the identity!
\end{rem}

Here is the main reason for introducing saltations.
We define for every double reduced word $\mathbf j$, every positive word
$\mathbf i_+\in R(1,w_0)$ and every $i\in[1,l]$, the
map $\Xi_{k}:{\mathcal X}_{[\mathbf j\mathbf{i_+}
\overline{k}]_{\mathfrak R}}\to{\mathcal X}_{[\mathbf j\mathbf i_+^{\square}
k^{\star}]_{\mathfrak R}}$ given by
\begin{equation}\label{equ:defXiI}x_{\Xi_{k}\binom{i}{j}}=\left\{
\begin{array}{lll}
x_{\zeta_{\mathbf i_+}\binom{i}{j}}&\mbox{if }j<N^i(\mathbf j\mathbf i_+)\ ;\\
x_{\zeta_{\mathbf i_+}\binom{i}{j}}x_{\binom{i^{\star}}{N^i(\mathbf j\mathbf i_+\overline{k})}}^{-1}
&\mbox{if }j=N^i(\mathbf j\mathbf i_+)<N^i(\mathbf j\mathbf i_+\overline{k})\ ;\\
x_{\binom{i}{j}}&\mbox{otherwise}\ .
\end{array}
\right.
\end{equation}

\begin{prop}\label{prop:obstruction} The
map $\Xi_{i}:{\mathcal X}_{[\mathbf j\mathbf{i_+}
\overline{i}]_{\mathfrak R}}\to{\mathcal X}_{[\mathbf j\mathbf i_+^{\square}
i^{\star}]_{\mathfrak R}}$ is a saltation, but not a generalized cluster transformation.
\end{prop}

Proposition \ref{prop:obstruction} will be proved in Subsection \ref{subsection:proofProp}.
For the moment, we focus on the link between saltations and dual moves,
given by the following result.

\begin{cor}\label{cor:dualsalt}
Let $i\in[1,l]$ and $\mathbf i$ be a double word such that we can
apply the dual move $\Delta_i$ on it. Then the following product
is a birational Poisson isomorphism.
\begin{equation}\label{equ:Xisi}
\begin{array}{cccl}
\Xi_{s_i}:&{\mathcal X}_{[\mathbf i]_{\mathfrak R}}
&\longrightarrow&{\mathcal X}_{[\Delta_i(\mathbf i)]_{\mathfrak R}}\\
&\mathbf x&\longmapsto&\mu_{[\mathbf i_+^{\square}i^{\star}]_{\mathfrak R}\to[\Delta_{i}
(\mathbf i)]_{\mathfrak R}}\circ\Xi_{i}\circ\mu_{[\mathbf i]_{\mathfrak R}\to
[\mathbf i_+\overline{i}]_{\mathfrak R}}(\mathbf x)\ .
\end{array}
\end{equation}
\end{cor}
\begin{proof}We use Proposition \ref{prop:obstruction} and the fact that,
for every double words $\mathbf j$ and $\mathbf k$, the cluster transformation
$\mu_{[\mathbf j]_{\mathfrak R}\to[\mathbf k]_{\mathfrak R}}$, when it exists,
is a birational Poisson isomorphism between the seed $\mathcal X$-tori
${\mathcal X}_{[\mathbf j]_{\mathfrak R}}$ and
${\mathcal X}_{[\mathbf k]_{\mathfrak R}}$.
\end{proof}

We finally prove that the relations between the various cluster
$\mathcal X$-varieties $\mathcal X_w$ associated to the different
double words $\mathbf i\in D(w_0)$ involve saltations and are
described by the $W$-permutohedron $P_W$.

\begin{lemma}\label{lemme:saltation} Let us replace every label $w'\in W$
of the vertices of the $W$-permutohedron $P_W$ by the cluster
$\mathcal X$-variety $\mathcal X_w$  associated to a seed $\mathcal X$-torus
${\mathcal X}_{[\mathbf{i}]_{\mathfrak R}}$
related to a double word $\mathbf i\in D_w(w_0)$.
We have the following statements.
\begin{itemize}
\item
The cluster $\mathcal X$-variety ${\mathcal X}_{w}$ related to any
$w\in W$ contains the seed $\mathcal X$-torus
${\mathcal X}_{[\mathbf{i}]_{\mathfrak R}}$ associated to any double
word $\mathbf i\in D_w(w_0)$.
\item
For any $i\in[1,l]$, if two vertices are respectively labeled by
${\mathcal X}_w$ and ${\mathcal X}_{w'}$ of $P_W$ are related by
the edge $s_i\in W$, there exist two trivial double words
$\mathbf i,\mathbf j\in D(w_0)$ such that the seed $\mathcal X$-tori
associated ${\mathcal X}_{[\mathbf i]_{\mathfrak R}}$ and
${\mathcal X}_{[\mathbf j]_{\mathfrak R}}$ are related by the saltation
$\Xi_{s_i}$.
\end{itemize}
\end{lemma}

\begin{proof}The first statement is given by Lemma \ref{lemma:PWtruncX}.
And the second statement comes from Lemma \ref{lemma:Whatdmoves} and
Corollary \ref{cor:dualsalt}.
\end{proof}

\subsection{Cluster ${\mathcal X}$-varieties for $(G,\pi_*)$, saltations and the $W$-permutohedron}
\label{subsection:thmWcluster}
We obtain finally, in Theorem \ref{thm:ev*}, the cluster combinatorics relating the
twisted evaluations of Section \ref{section:twistedev}. This cluster combinatorics
involves cluster transformations and saltations.

To any trivial double words $\mathbf i,\mathbf{i'}\in D(v)$ such that
there exists a $\widehat{d}$-move $\delta:\mathbf{i}\rightarrow\mathbf{i'}$ we associate a
birational Poisson isomorphism $\widehat{\mu}_{\mathbf i\rightarrow \mathbf{i'}}:
{\mathcal X}_{\mathbf i}\rightarrow{\mathcal X}_{\mathbf{i'}}$ given by
\begin{itemize}
\item
the cluster transformation $\mu_{[\mathbf i]_{\mathfrak R}\rightarrow [\mathbf{i'}]_{\mathfrak R}}$
if $\delta$ is a $d^{\tau}$-move;
\item
the birational Poisson isomorphism $\Xi_{s_i}$ if $\delta$ is the dual-move $\Delta_{i}$.
\end{itemize}
From Lemma \ref{lemma:Whatdmoves}, there exists a sequence of $\widehat{d}$-moves
relating any two trivial double words $\mathbf i,\mathbf{i'}\in D(v)$.
We therefore extend this definition to every $\mathbf i,\mathbf{j}\in  D(v)$ in the usual way:
If  $\mathbf i,\mathbf j$ are trivial double words linked be a sequence
of $\widehat{d}$-moves and  $\mathbf{i}\to\mathbf{i_1}\rightarrow
\dots\rightarrow\mathbf{i_{n-1}}\rightarrow\mathbf{j}$ is the associated chain of elements,
we define the map
$\widehat{{\mu}}_{\mathbf{i}\rightarrow\mathbf{j}}$ as the composition
$\widehat{{\mu}}_{\mathbf{i_{n-1}}\rightarrow\mathbf j}\circ \dots\circ
\widehat{{\mu}}_{\mathbf{i}\rightarrow\mathbf{i_1}}$.
Finally, because every double word $\mathbf k\in D(v)$ is related (at least) to
a trivial double word of $D(v)$ by a sequence of generalized $d$-moves, we associate
the cluster transformation $\mu_{\mathbf k\to\mathbf i}$ to complete the picture. Finally,
we get a birational Poisson isomorphism $\widehat{\mu}_{\mathbf i\rightarrow \mathbf{i'}}:
{\mathcal X}_{\mathbf i}\rightarrow{\mathcal X}_{\mathbf{i'}}$ associated to any
double words $\mathbf i,\mathbf j\in D(v)$.
We can now relate the Poisson birational isomorphism of Corollary \ref{cor:dualsalt}
with the twisted evaluations of Section \ref{section:twistedev}.

\begin{prop}\label{prop:salt}For every $i\in[\overline{1},\overline{l}]$
and every double reduced word $\mathbf i\in R(s_i,w_0)$ starting with
the letter $i$, we have the following equality.
$$\widehat{\ev}_{\mathbf i}=\widehat{\ev}_{\Delta_{i}(\mathbf i)}
\circ\Xi_{s_i}\ .$$
\end{prop}

Proposition \ref{prop:salt} is proved in Subsection \ref{subsection:(UB)invol}.
Now, because of Lemma \ref{lemma:Whatdmoves}, there exists a composition of
$\widehat{d}$-moves $\delta$ that satisfy the relation $\delta:\mathbf i\to\mathbf{j}$.
Therefore, it suffices to apply Proposition \ref{lemma:dtaumoves} and Proposition
\ref{prop:salt} to prove the following result, which was the missing argument
to prove Theorem \ref{thm:ev*1}.

\begin{thm}\label{thm:ev*} For every $v\in W$ and $\mathbf i,\mathbf{j}\in D(v)$ the maps
$\widehat{\ev}_{\mathbf i}$ and $\widehat{\ev}_{\mathbf{j}}$ satisfy the equality~$\widehat{\ev}_{\mathbf i}=
\widehat{\ev}_{\mathbf{j}}\circ\widehat{\mu}_{\mathbf i\rightarrow\mathbf{j}}$.
\end{thm}

Figure \ref{fig:commutativeweakorder}
describes, in the case $\mathfrak g=A_2$, the full picture of cluster combinatorics we obtain for $(BB_-,\pi_*)$
from the cluster $\mathcal X$-varieties related to $(G,\pi_G)$ and described in
Section \ref{section:ClusterG}.

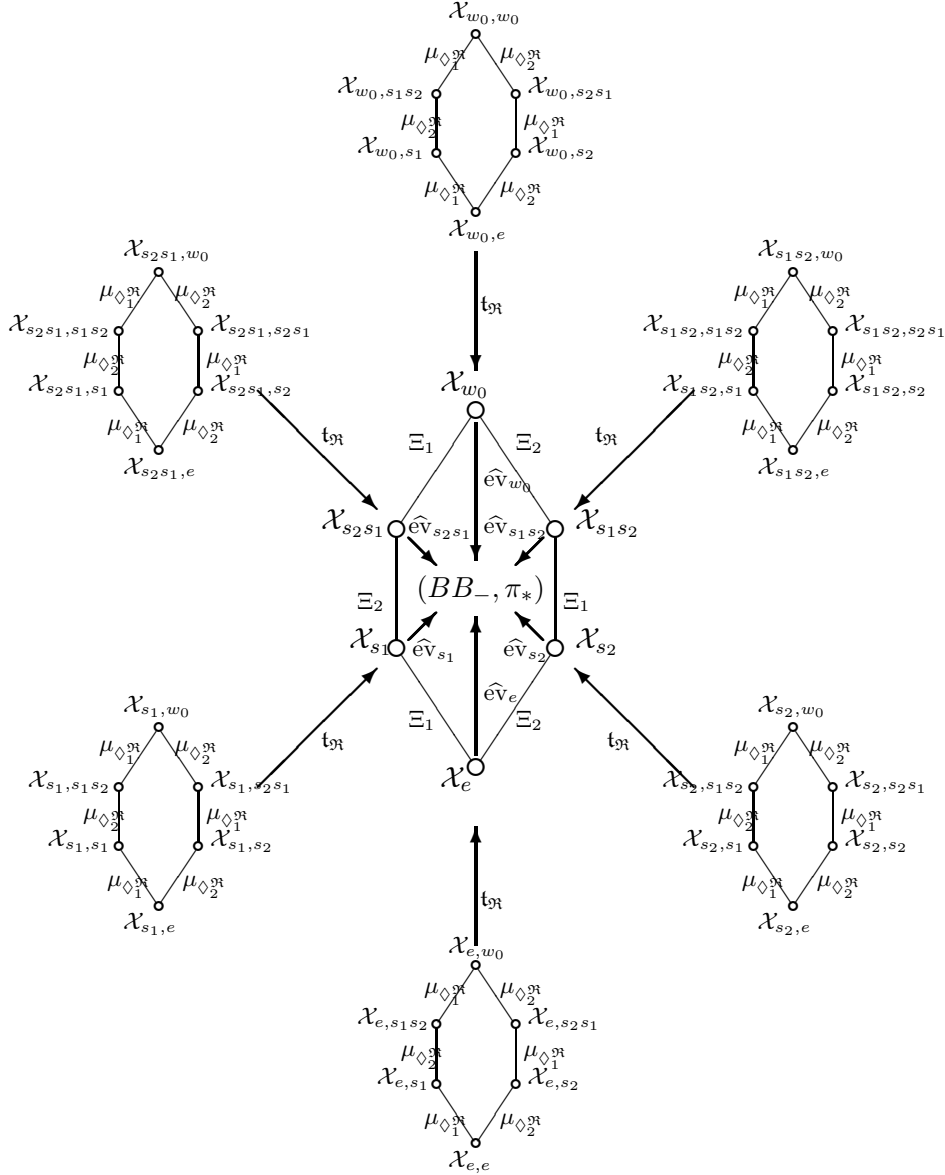
\begin{figure}[htbp]
\begin{center}
\setlength{\unitlength}{1.5pt}
\qquad\qquad\qquad\qquad\qquad
\begin{picture}(20,300)(0,-170)
\put(-29,-63.5){\line(2,3){17.7}}
\put(-10,-33){\line(0,1){26}}
\put(-11,-3.5){\line(-2,3){17.7}}
\put(-31,-63.5){\line(-2,3){17.7}}
\put(-50,-33){\line(0,1){26}}
\put(-49,-3.5){\line(2,3){17.7}}
\put(-29.5,75.75){\line(2,3){8.85}}
\put(-20,91){\line(0,1){13}}
\put(-30.5,75.75){\line(-2,3){8.85}}
\put(-20.5,105.75){\line(-2,3){8.85}}
\put(-40,91){\line(0,1){13}}
\put(-39.5,105.75){\line(2,3){8.85}}
\put(-29.5,-159.25){\line(2,3){8.85}}
\put(-20,-144){\line(0,1){13}}
\put(-30.5,-159.25){\line(-2,3){8.85}}
\put(-20.5,-129.25){\line(-2,3){8.85}}
\put(-40,-144){\line(0,1){13}}
\put(-39.5,-129.25){\line(2,3){8.85}}
\put(50.5,-99.25){\line(2,3){8.85}}
\put(60,-84){\line(0,1){13}}
\put(49.5,-99.25){\line(-2,3){8.85}}
\put(59.5,-69.25){\line(-2,3){8.85}}
\put(40,-84){\line(0,1){13}}
\put(40.5,-69.25){\line(2,3){8.85}}
\put(50.5,15.75){\line(2,3){8.85}}
\put(60,31){\line(0,1){13}}
\put(49.5,15.75){\line(-2,3){8.85}}
\put(59.5,45.75){\line(-2,3){8.85}}
\put(40,31){\line(0,1){13}}
\put(40.5,45.75){\line(2,3){8.85}}
\put(-109.5,-99.25){\line(2,3){8.85}}
\put(-120,-84){\line(0,1){13}}
\put(-110.5,-99.25){\line(-2,3){8.85}}
\put(-119.5,-69.25){\line(2,3){8.85}}
\put(-100,-84){\line(0,1){13}}
\put(-100.5,-69.25){\line(-2,3){8.85}}
\put(-109.5,15.75){\line(2,3){8.85}}
\put(-120,31){\line(0,1){13}}
\put(-110.5,15.75){\line(-2,3){8.85}}
\put(-119.5,45.75){\line(2,3){8.85}}
\put(-100,31){\line(0,1){13}}
\put(-100.5,45.75){\line(-2,3){8.85}}
\thicklines
\put(-30,35){\line(0,1){30}}
\put(-30,35){\vector(0,-1){0}}
\put(-30,-80){\line(0,-1){30}}
\put(-30,-80){\vector(0,1){0}}
\put(-55,0){\line(-1,1){30}}
\put(-55,0){\vector(1,-1){0}}
\put(-5,0){\line(1,1){30}}
\put(-5,0){\vector(-1,-1){0}}
\put(-55,-40){\line(-1,-1){30}}
\put(-55,-40){\vector(1,1){0}}
\put(-5,-40){\line(1,-1){30}}
\put(-5,-40){\vector(-1,1){0}}
\put(-30,25){\circle{4}}
\put(-30,22){\line(0,-1){35}}
\put(-30,22){\vector(0,-1){35}}
\put(-28,6){\footnotesize $\widehat{\ev}_{w_0}$}
\put(-39,30){${\mathcal X}_{w_0}$}
\put(-50,-5){\circle{4}}
\put(-47,-7){\line(1,-1){7}}
\put(-47,-7){\vector(1,-1){7}}
\put(-47,-6){\footnotesize $\widehat{\ev}_{s_2s_1}$}
\put(-69,-5){${\mathcal X}_{s_2s_1}$}
\put(-10,-5){\circle{4}}
\put(-13,-7){\line(-1,-1){7}}
\put(-13,-7){\vector(-1,-1){7}}
\put(-28,-6){\footnotesize $\widehat{\ev}_{s_1s_2}$}
\put(-5,-5){${\mathcal X}_{s_1s_2}$}
\put(-50,-35){\circle{4}}
\put(-47,-33){\line(1,1){7}}
\put(-47,-33){\vector(1,1){7}}
\put(-46,-37){\footnotesize $\widehat{\ev}_{s_1}$}
\put(-62,-35){${\mathcal X}_{s_1}$}
\put(-10,-35){\circle{4}}
\put(-13,-33){\line(-1,1){7}}
\put(-13,-33){\vector(-1,1){7}}
\put(-23,-37){\footnotesize $\widehat{\ev}_{s_2}$}
\put(-5,-35){${\mathcal X}_{s_2}$}
\put(-30,-65){\circle{4}}
\put(-30,-62){\line(0,1){35}}
\put(-30,-62){\vector(0,1){35}}
\put(-28,-48){\footnotesize $\widehat{\ev}_{e}$}
\put(-39,-70){${\mathcal X}_e$}
\put(-30,120){\circle{2}}
\put(-20,105){\circle{2}}
\put(-20,90){\circle{2}}
\put(-40,105){\circle{2}}
\put(-40,90){\circle{2}}
\put(-30,75){\circle{2}}
\put(-110,60){\circle{2}}
\put(-100,45){\circle{2}}
\put(-100,30){\circle{2}}
\put(-120,45){\circle{2}}
\put(-120,30){\circle{2}}
\put(-110,15){\circle{2}}
\put(50,60){\circle{2}}
\put(40,45){\circle{2}}
\put(40,30){\circle{2}}
\put(60,45){\circle{2}}
\put(60,30){\circle{2}}
\put(50,15){\circle{2}}
\put(-110,-55){\circle{2}}
\put(-100,-70){\circle{2}}
\put(-100,-85){\circle{2}}
\put(-120,-70){\circle{2}}
\put(-120,-85){\circle{2}}
\put(-110,-100){\circle{2}}
\put(50,-55){\circle{2}}
\put(40,-70){\circle{2}}
\put(40,-85){\circle{2}}
\put(60,-70){\circle{2}}
\put(60,-85){\circle{2}}
\put(50,-100){\circle{2}}
\put(-30,-160){\circle{2}}
\put(-20,-145){\circle{2}}
\put(-20,-130){\circle{2}}
\put(-40,-145){\circle{2}}
\put(-40,-130){\circle{2}}
\put(-30,-115){\circle{2}}
\put(-45,-22){$(BB_-,\pi_*)$}
\put(-20,15){\footnotesize $\Xi_2$}
\put(-47,15){\footnotesize $\Xi_1$}
\put(-60,-25){\footnotesize $\Xi_2$}
\put(-8,-25){\footnotesize $\Xi_1$}
\put(-47,-55){\footnotesize $\Xi_1$}
\put(-20,-55){\footnotesize $\Xi_2$}
\put(-29,-100){\footnotesize ${\mathfrak t}_{{\mathfrak R}}$}
\put(-29,50){\footnotesize ${\mathfrak t}_{{\mathfrak R}}$}
\put(-69,17){\footnotesize ${\mathfrak t}_{{\mathfrak R}}$}
\put(0,17){\footnotesize ${\mathfrak t}_{{\mathfrak R}}$}
\put(-69,-60){\footnotesize ${\mathfrak t}_{{\mathfrak R}}$}
\put(3,-60){\footnotesize ${\mathfrak t}_{{\mathfrak R}}$}
\put(-37,124){\footnotesize ${\mathcal X}_{w_0,w_0}$}
\put(-43,114){\footnotesize $\mu_{{\lozenge}^{\mathfrak R}_1}$}
\put(-24,114){\footnotesize $\mu_{{\lozenge}^{\mathfrak R}_2}$}
\put(-17,105){\footnotesize ${\mathcal X}_{w_0,s_2s_1}$}
\put(-49,97){\footnotesize $\mu_{{\lozenge}^{\mathfrak R}_2}$}
\put(-18,97){\footnotesize $\mu_{{\lozenge}^{\mathfrak R}_1}$}
\put(-17,90){\footnotesize ${\mathcal X}_{w_0,s_2}$}
\put(-43,80){\footnotesize $\mu_{{\lozenge}^{\mathfrak R}_1}$}
\put(-24,80){\footnotesize $\mu_{{\lozenge}^{\mathfrak R}_2}$}
\put(-65,105){\footnotesize ${\mathcal X}_{w_0,s_1s_2}$}
\put(-60,90){\footnotesize ${\mathcal X}_{w_0,s_1}$}
\put(-37,69){\footnotesize ${\mathcal X}_{w_0,e}$}
\put(-119,64){\footnotesize ${\mathcal X}_{s_2s_1,w_0}$}
\put(-125,54){\footnotesize $\mu_{{\lozenge}^{\mathfrak R}_1}$}
\put(-106,54){\footnotesize $\mu_{{\lozenge}^{\mathfrak R}_2}$}
\put(-97,45){\footnotesize ${\mathcal X}_{s_2s_1,s_2s_1}$}
\put(-129,37){\footnotesize $\mu_{{\lozenge}^{\mathfrak R}_2}$}
\put(-98,37){\footnotesize $\mu_{{\lozenge}^{\mathfrak R}_1}$}
\put(-97,30){\footnotesize ${\mathcal X}_{s_2s_1,s_2}$}
\put(-123,20){\footnotesize $\mu_{{\lozenge}^{\mathfrak R}_1}$}
\put(-104,20){\footnotesize $\mu_{{\lozenge}^{\mathfrak R}_2}$}
\put(-148,45){\footnotesize ${\mathcal X}_{s_2s_1,s_1s_2}$}
\put(-143,30){\footnotesize ${\mathcal X}_{s_2s_1,s_1}$}
\put(-119,9){\footnotesize ${\mathcal X}_{s_2s_1,e}$}
\put(41,64){\footnotesize ${\mathcal X}_{s_1s_2,w_0}$}
\put(35,54){\footnotesize $\mu_{{\lozenge}^{\mathfrak R}_1}$}
\put(54,54){\footnotesize $\mu_{{\lozenge}^{\mathfrak R}_2}$}
\put(12,45){\footnotesize ${\mathcal X}_{s_1s_2,s_1s_2}$}
\put(17,30){\footnotesize ${\mathcal X}_{s_1s_2,s_1}$}
\put(63,45){\footnotesize ${\mathcal X}_{s_1s_2,s_2s_1}$}
\put(31,37){\footnotesize $\mu_{{\lozenge}^{\mathfrak R}_2}$}
\put(62,37){\footnotesize $\mu_{{\lozenge}^{\mathfrak R}_1}$}
\put(63,30){\footnotesize ${\mathcal X}_{s_1s_2,s_2}$}
\put(37,20){\footnotesize $\mu_{{\lozenge}^{\mathfrak R}_1}$}
\put(56,20){\footnotesize $\mu_{{\lozenge}^{\mathfrak R}_2}$}
\put(41,9){\footnotesize ${\mathcal X}_{s_1s_2,e}$}
\put(-119,-51){\footnotesize ${\mathcal X}_{s_1,w_0}$}
\put(-125,-61){\footnotesize $\mu_{{\lozenge}^{\mathfrak R}_1}$}
\put(-106,-61){\footnotesize $\mu_{{\lozenge}^{\mathfrak R}_2}$}
\put(-97,-70){\footnotesize ${\mathcal X}_{s_1,s_2s_1}$}
\put(-129,-78){\footnotesize $\mu_{{\lozenge}^{\mathfrak R}_2}$}
\put(-98,-78){\footnotesize $\mu_{{\lozenge}^{\mathfrak R}_1}$}
\put(-97,-85){\footnotesize ${\mathcal X}_{s_1,s_2}$}
\put(-123,-95){\footnotesize $\mu_{{\lozenge}^{\mathfrak R}_1}$}
\put(-104,-95){\footnotesize $\mu_{{\lozenge}^{\mathfrak R}_2}$}
\put(-143,-70){\footnotesize ${\mathcal X}_{s_1,s_1s_2}$}
\put(-138,-85){\footnotesize ${\mathcal X}_{s_1,s_1}$}
\put(-119,-106){\footnotesize ${\mathcal X}_{s_1,e}$}
\put(41,-51){\footnotesize ${\mathcal X}_{s_2,w_0}$}
\put(35,-61){\footnotesize $\mu_{{\lozenge}^{\mathfrak R}_1}$}
\put(54,-61){\footnotesize $\mu_{{\lozenge}^{\mathfrak R}_2}$}
\put(17,-70){\footnotesize ${\mathcal X}_{s_2,s_1s_2}$}
\put(22,-85){\footnotesize ${\mathcal X}_{s_2,s_1}$}
\put(63,-70){\footnotesize ${\mathcal X}_{s_2,s_2s_1}$}
\put(31,-78){\footnotesize $\mu_{{\lozenge}^{\mathfrak R}_2}$}
\put(62,-78){\footnotesize $\mu_{{\lozenge}^{\mathfrak R}_1}$}
\put(63,-85){\footnotesize ${\mathcal X}_{s_2,s_2}$}
\put(37,-95){\footnotesize $\mu_{{\lozenge}^{\mathfrak R}_1}$}
\put(56,-95){\footnotesize $\mu_{{\lozenge}^{\mathfrak R}_2}$}
\put(41,-106){\footnotesize ${\mathcal X}_{s_2,e}$}
\put(-37,-166){\footnotesize ${\mathcal X}_{e,e}$}
\put(-17,-145){\footnotesize ${\mathcal X}_{e,s_2}$}
\put(-43,-155){\footnotesize $\mu_{{\lozenge}^{\mathfrak R}_1}$}
\put(-24,-155){\footnotesize $\mu_{{\lozenge}^{\mathfrak R}_2}$}
\put(-17,-130){\footnotesize ${\mathcal X}_{e,s_2s_1}$}
\put(-49,-138){\footnotesize $\mu_{{\lozenge}^{\mathfrak R}_2}$}
\put(-18,-138){\footnotesize $\mu_{{\lozenge}^{\mathfrak R}_1}$}
\put(-55,-145){\footnotesize ${\mathcal X}_{e,s_1}$}
\put(-60,-130){\footnotesize ${\mathcal X}_{e,s_1s_2}$}
\put(-37,-112){\footnotesize ${\mathcal X}_{e,w_0}$}
\put(-43,-122){\footnotesize $\mu_{{\lozenge}^{\mathfrak R}_1}$}
\put(-24,-122){\footnotesize $\mu_{{\lozenge}^{\mathfrak R}_2}$}
\end{picture}
\end{center}
\vspace{-.1in}
\caption{Cluster $\mathcal X$-varieties evaluating $(BB_-,\pi_*)$
when $\mathfrak{g}=A_2$}
\label{fig:commutativeweakorder}
\end{figure}

\subsection{Proof of Proposition \ref{prop:obstruction}}
\label{subsection:proofProp}
We need a few preliminaries to prove Proposition \ref{prop:obstruction}.
Let us remember the map $\pi_{\mathbf i}$ associated to any double
$\mathbf i$ given by (\ref{equ:foldpi}). Here is a generalization.
Let $\mathfrak{s}$ be a $\mathcal{X}$-split associated to the decomposition
$\mathbf i\mathbf j\to(\mathbf i,\mathbf j)$: for
every $\mathbf x\in\mathcal{X}_{\mathbf i\mathbf j}$, we define
\begin{equation}\label{equ:pi2}{equ:pi}
\begin{array}{ccc}
\pi_{\mathbf i\mathbf j\to\mathbf i}(\mathbf x)=\mathfrak{m}(\mathbf x_{(1)},
\pi_{\mathbf j}(\mathbf x_{(2)}))&\mbox{and}&\pi_{\mathbf i\mathbf j\to\mathbf j}
(\mathbf x)=\mathfrak{m}(\pi_{\mathbf i}(\mathbf x_{(1)}),\mathbf x_{(2)})\ .
\end{array}
\end{equation}
(It is clear that these equalities don't depend on the choice of $\mathfrak s$ and that
they are also satisfied for every $\mathbf x\in\mathcal{X}_{[\mathbf i\mathbf j]_{\mathfrak R}}$.)
Figure \ref{fig:folding} and  Figure \ref{fig:folding2} give examples
of these maps in the case $\mathfrak g=A_3$, whereas the following result is straightforward.

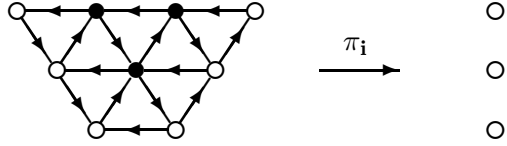
\begin{figure}[htbp]
\begin{center}
\setlength{\unitlength}{1.5pt}
\begin{picture}(20,48)(0,-27)
\thicklines
\put(1,1.7){\line(2,3){8}}
\put(1,1.7){\vector(2,3){6}}
\put(21,1.7){\line(2,3){8}}
\put(21,1.7){\vector(2,3){6}}
\put(41,1.7){\line(2,3){7.5}}
\put(41,1.7){\vector(2,3){6}}
\put(11,-13.3){\line(2,3){7.5}}
\put(11,-13.3){\vector(2,3){6}}
\put(31,-13.3){\line(2,3){7.5}}
\put(31,-13.3){\vector(2,3){6}}
\put(-9,13.3){\line(2,-3){7.5}}
\put(-9,13.3){\vector(2,-3){6}}
\put(11,13.3){\line(2,-3){7.5}}
\put(11,13.3){\vector(2,-3){6}}
\put(31,13.3){\line(2,-3){7.5}}
\put(31,13.3){\vector(2,-3){6}}
\put(1,-1.7){\line(2,-3){7.5}}
\put(1,-1.7){\vector(2,-3){6}}
\put(21,-1.7){\line(2,-3){7.5}}
\put(21,-1.7){\vector(2,-3){6}}
\put(10,15){\circle*{4}}
\put(30,15){\circle*{4}}
\put(-10,15){\circle{4}}
\put(50,15){\circle{4}}
\put(10,-15){\circle{4}}
\put(30,-15){\circle{4}}
\put(0,0){\circle{4}}
\put(20,0){\circle*{4}}
\put(40,0){\circle{4}}
\put(8,15){\vector(-1,0){11}}
\put(8,15){\line(-1,0){16}}
\put(28,15){\vector(-1,0){11}}
\put(28,15){\line(-1,0){16}}
\put(48,15){\vector(-1,0){11}}
\put(48,15){\line(-1,0){16}}
\put(18,0){\vector(-1,0){11}}
\put(18,0){\line(-1,0){16}}
\put(38,0){\vector(-1,0){11}}
\put(38,0){\line(-1,0){16}}
\put(28,-15){\vector(-1,0){11}}
\put(28,-15){\line(-1,0){16}}
\end{picture}
\qquad\qquad\qquad
\begin{picture}(20,48)(0,-27)
\thicklines
\put(0,0){\vector(1,0){20}}
\put(6,5){$\pi_{\mathbf i}$}
\end{picture}
\quad\qquad
\begin{picture}(20,48)(0,-27)
\thicklines
\put(0,15){\circle{4}}
\put(0,-15){\circle{4}}
\put(0,0){\circle{4}}
\end{picture}
\end{center}
\vspace{-.1in}
\caption{The generalized folding $\pi_{\mathbf i}:
\Gamma_{A_3}(\mathbf i)\to\Gamma_{A_3}(\mathbf 1)$
for $\mathbf i=123121$}
\label{fig:folding}
\end{figure}
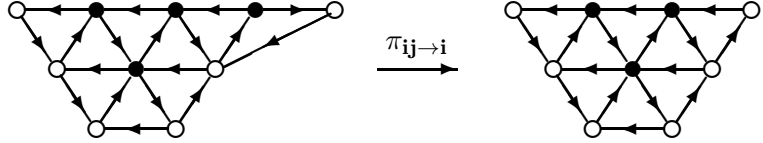
\begin{figure}[htbp]
\begin{center}
\setlength{\unitlength}{1.5pt}
\begin{picture}(20,48)(0,-27)
\thicklines
\put(1,1.7){\line(2,3){8}}
\put(1,1.7){\vector(2,3){6}}
\put(21,1.7){\line(2,3){8}}
\put(21,1.7){\vector(2,3){6}}
\put(41,1.7){\line(2,3){7.5}}
\put(41,1.7){\vector(2,3){6}}
\put(11,-13.3){\line(2,3){7.5}}
\put(11,-13.3){\vector(2,3){6}}
\put(31,-13.3){\line(2,3){7.5}}
\put(31,-13.3){\vector(2,3){6}}
\put(-9,13.3){\line(2,-3){7.5}}
\put(-9,13.3){\vector(2,-3){6}}
\put(11,13.3){\line(2,-3){7.5}}
\put(11,13.3){\vector(2,-3){6}}
\put(31,13.3){\line(2,-3){7.5}}
\put(31,13.3){\vector(2,-3){6}}
\put(1,-1.7){\line(2,-3){7.5}}
\put(1,-1.7){\vector(2,-3){6}}
\put(21,-1.7){\line(2,-3){7.5}}
\put(21,-1.7){\vector(2,-3){6}}
\put(10,15){\circle*{4}}
\put(30,15){\circle*{4}}
\put(-10,15){\circle{4}}
\put(50,15){\circle*{4}}
\put(70,15){\circle{4}}
\put(10,-15){\circle{4}}
\put(30,-15){\circle{4}}
\put(0,0){\circle{4}}
\put(20,0){\circle*{4}}
\put(40,0){\circle{4}}
\put(8,15){\vector(-1,0){11}}
\put(8,15){\line(-1,0){16}}
\put(28,15){\vector(-1,0){11}}
\put(28,15){\line(-1,0){16}}
\put(48,15){\vector(-1,0){11}}
\put(48,15){\line(-1,0){16}}
\put(52,15){\vector(1,0){11}}
\put(52,15){\line(1,0){16}}
\put(69,14){\vector(-2,-1){17}}
\put(69,14){\line(-2,-1){27}}
\put(18,0){\vector(-1,0){11}}
\put(18,0){\line(-1,0){16}}
\put(38,0){\vector(-1,0){11}}
\put(38,0){\line(-1,0){16}}
\put(28,-15){\vector(-1,0){11}}
\put(28,-15){\line(-1,0){16}}
\end{picture}
\qquad\qquad\qquad\qquad
\begin{picture}(20,48)(0,-27)
\thicklines
\put(0,0){\vector(1,0){20}}
\put(2,5){$\pi_{\mathbf i\mathbf j\to\mathbf i}$}
\end{picture}
\quad\qquad
\begin{picture}(20,48)(0,-27)
\thicklines
\put(1,1.7){\line(2,3){8}}
\put(1,1.7){\vector(2,3){6}}
\put(21,1.7){\line(2,3){8}}
\put(21,1.7){\vector(2,3){6}}
\put(41,1.7){\line(2,3){7.5}}
\put(41,1.7){\vector(2,3){6}}
\put(11,-13.3){\line(2,3){7.5}}
\put(11,-13.3){\vector(2,3){6}}
\put(31,-13.3){\line(2,3){7.5}}
\put(31,-13.3){\vector(2,3){6}}
\put(-9,13.3){\line(2,-3){7.5}}
\put(-9,13.3){\vector(2,-3){6}}
\put(11,13.3){\line(2,-3){7.5}}
\put(11,13.3){\vector(2,-3){6}}
\put(31,13.3){\line(2,-3){7.5}}
\put(31,13.3){\vector(2,-3){6}}
\put(1,-1.7){\line(2,-3){7.5}}
\put(1,-1.7){\vector(2,-3){6}}
\put(21,-1.7){\line(2,-3){7.5}}
\put(21,-1.7){\vector(2,-3){6}}
\put(10,15){\circle*{4}}
\put(30,15){\circle*{4}}
\put(-10,15){\circle{4}}
\put(50,15){\circle{4}}
\put(10,-15){\circle{4}}
\put(30,-15){\circle{4}}
\put(0,0){\circle{4}}
\put(20,0){\circle*{4}}
\put(40,0){\circle{4}}
\put(8,15){\vector(-1,0){11}}
\put(8,15){\line(-1,0){16}}
\put(28,15){\vector(-1,0){11}}
\put(28,15){\line(-1,0){16}}
\put(48,15){\vector(-1,0){11}}
\put(48,15){\line(-1,0){16}}
\put(18,0){\vector(-1,0){11}}
\put(18,0){\line(-1,0){16}}
\put(38,0){\vector(-1,0){11}}
\put(38,0){\line(-1,0){16}}
\put(28,-15){\vector(-1,0){11}}
\put(28,-15){\line(-1,0){16}}
\end{picture}
\end{center}
\vspace{-.1in}
\caption{The generalized folding $\pi_{\mathbf i\mathbf j
\to\mathbf i}:\Gamma_{A_3}(\mathbf i\mathbf j)\to\Gamma_{A_3}(\mathbf i)$
for $\mathbf i=123121$ and $\mathbf j=\overline 1$}
\label{fig:folding2}
\end{figure}

\begin{lemma}\label{lemma:proj2}Let $u,v\in W$ and $\mathbf i\in D(u,1)\cup D(1,v)$
be a positive or negative double word. The following equality is satisfied for every
$\mathbf x\in{\mathcal X}_{\mathbf{i}}$.
$$
[\ev_{\mathbf{i}}(\mathbf x)]_{0}=\ev_{\mathbf{1}}\circ\pi_{\mathbf{i}}(\mathbf x)\ .
$$
\end{lemma}

\begin{lemma}\label{lemma:proj}Let $w_1\leq u,w_2\leq v\in W$
and $\mathbf i\in R(w_2,v)$, $\mathbf j\in R(u,w_1)$ be such that
$\mathbf i=\mathbf{i_+}\mathbf{i_-}$ and $\mathbf j=\mathbf{j_+}\mathbf{j_-}$.
The following equalities are satisfied for every $\mathbf x\in{\mathcal X}_{\mathbf{i}}$
and $\mathbf y\in{\mathcal X}_{\mathbf{j}}$.
$$\begin{array}{lcr}
[\ev_{\mathbf{i}}(\mathbf x)\widehat{v^{-1}}]_{\leq 0}=[\ev_{\mathbf{i_+}}
\circ\pi_{\mathbf{i}\to\mathbf{i_+}}(\mathbf x)\widehat{v^{-1}}]_{\leq 0}&\mbox{and}&
{[\widehat{u}^{-1}\ev_{\mathbf j}(\mathbf y)]_{\geq 0}}=[\widehat{u}^{-1}
\ev_{\mathbf{j_-}}\circ\pi_{\mathbf j\to\mathbf{j_-}}(\mathbf y)]_{\leq 0}\ .
\end{array}$$
\end{lemma}
\begin{proof}We prove the first equality. Let ${\mathfrak s}$ be a
$\mathcal X$-split associated to the decomposition $\mathbf i\to(\mathbf{i_+},
\mathbf{i_-})$. Because $\mathbf{i_-}\in R(w_2,1)$ and $w_2\leq v$, the
conjugation of $\ev_{\mathbf i_-}(\mathbf x_{(2)})$ by $\widehat v$ belongs
to the Borel subgroup $B$. But it is clear that for every $b\in B$, the equality
$[b]_{\leq 0}=[b]_0$ is satisfied. Therefore, we get the result by applying the
definition (\ref{equ:pi2}) for $\pi_{\mathbf{i}\to\mathbf{i_+}}$ and Lemma
\ref{lemma:proj2}. The second equality is proved in the same way.
\end{proof}

\begin{lemma}\label{prop:conjtwist}Let $w_1\leq u,w_2\leq v\in W$
and $\mathbf i\in R(w_2,v)$, $\mathbf j\in R(u,w_1)$ be double reduced words
such that $\mathbf i=\mathbf{i_+}\mathbf{i_-}$ and $\mathbf j=\mathbf{j_+}\mathbf{j_-}$.
The following equalities are satisfied:
\begin{equation}\label{equ:conjtwist}
\begin{array}{ccc}
\pi_{\mathbf i\to\mathbf{i_+}}=\zeta_{\mathbf{i_+}}^{-1}\circ
\mu_{\mathbf{i_-}\mathbf{i_+^{\square}}\to\mathbf{i_+^{\square}}}\circ\zeta_{\mathbf{i_+}}\circ
\mu_{\mathbf i\to\mathbf{i_-}\mathbf{i_+}};\\
\\
\pi_{\mathbf j\to\mathbf{j_-}}=\zeta_{\mathbf{j_-}}^{-1}\circ
\mu_{\mathbf{j_-^{\square}}\mathbf{j_+}\to\mathbf{j_-^{\square}}}
\circ\zeta_{\mathbf{j_-}}\circ\mu_{\mathbf j\to\mathbf{j_-}\mathbf{j_+}}.
\end{array}
\end{equation}
\end{lemma}
\begin{proof}Let $\mathbf z\in{\mathcal X}_{\mathbf{i_-}\mathbf{i_+}}$ and
$b_-:=\zeta^{1,v}(\ev_{\mathbf{i_-}\mathbf{i_+}}(\mathbf z))$. Using
Proposition \ref{prop:twist-}, equation (\ref{equ:defev}), Remark \ref{rem:tropamal}
and Theorem \ref{fg}, it is clear that the evaluation map $\ev_{{\mathbf{i_+^{\square}}}}$ sends
the following element on $b_-$
$$\mu_{\mathbf{i_-}\mathbf{i_+^{\square}}\to\mathbf{i_+^{\square}}}\circ\zeta_{\mathbf{i_+}}
(\mathbf z)\ .$$
But this evaluation also sends the element $\zeta_{\mathbf{i_+}}\circ\pi_{\mathbf i\to\mathbf{i_+}}\circ
\mu_{\mathbf{i_-}\mathbf{i_+}\to\mathbf i}(\mathbf z)$
on $b_-$, using the first equality of Lemma \ref{lemma:proj}, Theorem \ref{fg} and equation (\ref{equ:defev}).
Now, the double word $\mathbf {i_+^{\square}}$ is reduced because $\mathbf{i_+}$ is a double
reduced word. Therefore the maps $\ev_{\mathbf {i_+^{\square}}}$ and $\zeta_{\mathbf{i_+}}$
are birational isomorphisms and the first equality of (\ref{equ:conjtwist}) is proved.
The second equality is proved in the same way, using the second equality of Lemma \ref{lemma:proj}.
\end{proof}

Here is a last preliminary.
Let us recall that, because of the erasing map in the first line of equation
(\ref{equ:elemmut}), the Poisson map $\mu_{\mathbf i\to\mathbf j}$
associated to a $\nil$-move $\delta:\mathbf i\mapsto\mathbf j$ is
not a birational isomorphism. We then define the related cluster
transformation $\overline{\mu}_{\mathbf i\to\mathbf j}$ such that
$$\overline{\mu}_{\mathbf i\to\mathbf j}=\varsigma_{\binom{i}{1}}
\circ\mu_{\mathbf i\to\mathbf j}\ .$$

\begin{lemma}\label{lemma:Sinvoletsalt2}Let $i\in[{1},{l}]$ and
$\mathbf i\in R(s_i,w_0)$ be a double reduced word starting with
the letter $\overline{i}$. The following equality is satisfied for
every $\mathbf t\in{\mathcal X}_{\mathbf 1}.$
$${\mathfrak{t}'}_{\mathbf i_{\mathfrak R}(\mathbf t)}\circ\overline{\mu}_{\overline{i}
\mathbf i_+^{\square}\to\mathbf i_+^{\square}}\circ\zeta_{\mathbf i_+}
\circ\mu_{\mathbf i\to\overline{i}{\mathbf i}_+}=\Xi_{i}
\circ\mathfrak{t}_{{\mathbf i}_{\mathfrak R}(\mathbf t)}\ ,$$
where ${\mathfrak{t}'}_{\mathbf i_{\mathfrak R}}$ denotes the truncation map
associated to the set $I_0^{\mathfrak R}(\mathbf i)'$ given by
\begin{equation}\label{equ:I0set}
I_0^{\mathfrak R}(\mathbf i)'=I_0^{\mathfrak R}(\mathbf i_+^{\square}i^{\star})
\cup\{(^i_1)\}\backslash\{(^{i^{\star}}_{N^{i^{\star}}
(\mathbf i_+^{\square}i^{\star})})\}\ .
\end{equation}
\end{lemma}
\begin{proof}Let us define, for every $i\in[1,l]$ and every double word $\mathbf i$,
the erasing map $\varsigma_{\mathfrak{R}^{i},\mathbf i}$ as the product of $j$-erasing
maps $\varsigma_j$ for every right outlet $j$ which is not a $i$-vertex:
$$\varsigma_{\mathfrak{R}^{i},\mathbf i}=\displaystyle\prod_{j\in I_0^{\mathfrak R}(\mathbf i)
\backslash\{\binom{i}{N^i(\mathbf i)}\}}\varsigma_j\ .$$
Let us then remark that the map $\Xi_i$ can also be defined as
\begin{equation}\label{equ:elemsalt}
\begin{array}{cclll}
\Xi_{i}:&{\mathcal X}_{[\mathbf{i_+}
\overline{i}]_{\mathfrak R}}\longrightarrow{\mathcal X}_{[\mathbf i_+^{\square}
i^{\star}]_{\mathfrak R}}:
&\mathbf x\longmapsto(\zeta_{\mathbf i_+}
\circ\pi_{\mathbf i_+\overline{i}\to\mathbf i_+}(\mathbf x),\mathbf x(\mathfrak R))\ .
\end{array}
\end{equation}
We then use the formulas (\ref{equ:conjtwist})
and (\ref{equ:elemsalt}) to deduce the following equality, which
is satisfied for every $\mathbf x\in{\mathcal{X}}_{[\mathbf i_+
\overline{i}]_{\mathfrak R}}$.
$$(\varsigma_{\mathfrak{R}^{i^{\star}},\mathbf i_+^{\square}}\circ\varsigma_{\binom{i}{1}}
\circ\overline{\mu}_{\overline{i}\mathbf i_+^{\square}
\to\mathbf i_+^{\square}}\circ\zeta_{\mathbf i_+}\circ\mu_{\mathbf i\to\overline{i}{\mathbf i}_+}(\mathbf x),
\mathbf x(\mathfrak R))=\Xi_{i}\circ\mathfrak{t}_{
{\mathbf i}_{\mathfrak R}(\mathbf t)}(\mathbf x)\ .$$
It suffices then to use the definition of truncation maps to end
the proof of the lemma.
\end{proof}

Lemma \ref{lemma:Sinvoletsalt2} therefore implies that
the map $\Xi_i$ is a saltation associated to the generalized
cluster transformation
$\overline{\mu}_{\overline{i}\mathbf i_+^{\square}\to\mathbf i_+^{\square}}
\circ\zeta_{\mathbf i_+}\circ\mu_{\mathbf i\to\overline{i}{\mathbf i}_+}$,
so the first part of Proposition \ref{prop:obstruction} is proved.

\begin{lemma}\label{lemma:saltcasimir}Let $\mathbf I$ be a seed such that there exist a cluster
variable $x_i$ of the seed $\mathcal X$-torus ${\mathcal X}_{\mathbf I}$
which is a Casimir function. Then, for every cluster $\mathbf x$ of
${\mathcal X}_{\mathbf I}$ and every generalized cluster transformation $\phi$,
we have the equality $x_{\phi(i)}=x_i$ and, for every $j\neq i$, the cluster
variable $x_{\phi(j)}$ doesn't depend on the cluster variable $x_i$.
\end{lemma}
\begin{proof}It suffices to factorize the generalized cluster transformation
$\phi$ into a product of tropical mutations, symmetries, and mutations $\phi_k$.
The properties are then easily checked for each $\phi_k$, using formulas for
mutations and tropical mutations.
\end{proof}

To prove that the saltation $\Xi_i$ is not a generalized cluster transformation,
it suffices now to  apply Lemma \ref{lemma:saltcasimir}, the second
line of formula (\ref{equ:defXiI}), and the fact
that every cluster variable $x_i$, $i\in I_0^{\mathfrak R}(\mathbf i)$ is a
Casimir function for every right truncated seed $[\mathbf i]_{\mathfrak R}$.

\subsection{Proof of Proposition \ref{prop:salt}}\label{subsection:(UB)invol}
Let us recall that every element $x\in B_-B$ can be decomposed
into $x=[x]_-[x]_0[x]_+$, where $[x]_-\in N_-$, $[x]_0\in H$,
and $[x]_+\in N$. If $x\in B_-B\cap BB_-$ it can also be decomposed into
$x=[[x]]_+[[x]]_0[[x]]_-$, with $[[x]]_-\in N_-$, $[[x]]_0\in H$,
and $[[x]]_+\in N$. The two decompositions are easily related: using
the fact that the map $x\mapsto x^{-1}$ is an involution, we get
\begin{equation}\label{equ:gdecomposition}
[[x]]_+[[x]]_0[[x]]_-=[x^{-1}]_+^{-1}[x^{-1}]_0^{-1}[x^{-1}]_-^{-1}.
\end{equation}
Let us now recall
the application $\kappa: G\times G\to G$ given by equation (\ref{equ:kappa}).
Let $t\in H$, $v\in W$ and let us first introduce the map
$\Xi_t: L^{v,w_0}\to L^{w_0,{v^{\star}}^{-1}}$ given by:
\begin{equation}\label{equ:Xidef}
\Xi_t(gH)=[[[gH]]_{\geq 0}\widehat{w_0}]_{\leq 0}\
[\kappa((t\widehat{w_0})^{-1},[[gH]]_{<0})]_+H\ .
\end{equation}

\begin{lemma}\label{prop:tau*moves}Let $t\in H$ and $v\in W$.
We have for every $g\in G^{v,w_0}$:
$$\rho_{t,(v,e)_v}(gH)=[\Xi_t(gH)^{\theta}\widehat{w_0}]_{\leq 0}^{\theta}\ t\widehat{w_0}\
\Xi_t(gH)^{-1}.$$
\end{lemma}
\begin{proof}Let us set $b=[[g]]_{\geq 0}$ and $n_-=[[g]]_{<0}$. We use
successively the fact that $[\kappa (\widehat{w_0}^{-1},n_-)]_0$ and
$[\kappa((t\widehat{w_0})^{-1},n_-)]_0$ are equal (to the unit $1_G$
of $G$), the fact that the map $\kappa(g,.)$ commutes with the inverse
map $x\mapsto x^{-1}$ for every $g\in G$, and Lemma \ref{lemma:twistinvol}
to obtain:
$$\begin{array}{ll}
\rho_{t,(v,e)_v}(gH)&=bn_-t\widehat{w_0}[bn_-\widehat{w_0}]_{\leq 0}^{-1}\\
\\
&=bt\widehat{w_0}\ \kappa((t\widehat{w_0})^{-1},n_-^{-1})\ ([b\widehat{w_0}]_{\leq 0}
[\kappa((\widehat{w_0})^{-1},n_-^{-1})]_0)^{-1}\\
\\
&=bt\widehat{w_0}\
([b\widehat{w_0}]_{\leq 0}[\kappa((t\widehat{w_0})^{-1},n_-)]_+)^{-1}\\
\\
&=[([b\widehat{w_0}]_{\leq 0}[\kappa((t\widehat{w_0})^{-1},
n_-)]_+)^{\theta}\widehat{w_0}]_{\leq 0}^{\theta}\ t\widehat{w_0}\
([b\widehat{w_0}]_{\leq 0}[\kappa((t\widehat{w_0})^{-1},n_-)]_+)^{-1}\\
\\
&=[\Xi_t(gH)^{\theta}\widehat{w_0}]_{\leq 0}^{\theta}\
t\widehat{w_0}\ \Xi_t(gH)^{-1}.
\end{array}$$
\end{proof}

As an immediate corollary, we get, for every $t\in H$ and every $v\in W$:
\begin{equation}\label{lemma:Lambda}
\begin{array}{rll}
\rho_{t,(e,w_0)_v}=\rho_{t,(v,e)_v}\circ\Lambda_t&\mbox{where}&
\Lambda_t(b_1,cH)=\Xi_t(b_1cH)\ .
\end{array}
\end{equation}

Here is now the related cluster combinatorics. We will focus on the
case $v=s_i$ because it is all what we need, but similar statements,
although more technical, can in fact be obtained for a general
$v\in W$. Let us remember the maps $\pi_{\mathbf i\mathbf j\to\mathbf i}$
given by (\ref{equ:pi2}).

\begin{lemma}\label{lemma:Stwistcluster1}Let $i\in [1,l]$, $t\in H$, and
$\mathbf i\in R(s_i,w_0)$ be a double reduced word satisfying the equality
$\mathbf i=\mathbf{i}_+\mathbf i_-$. The following equality is satisfied
for every $\mathbf x\in{\mathcal X}_{[\mathbf i]_{\mathfrak R}}(t)$.
$$\begin{array}{cccl}
[\ \Xi_{t}\circ{\ev}_{\mathbf i}(\mathbf x)\ ]_{\leq 0}={\ev}_{\mathbf i_+^{\square}}
\circ\zeta_{\mathbf i_+}\circ\pi_{\mathbf i\to\mathbf i_+}(\mathbf x)\ .
\end{array}$$
\end{lemma}
\begin{proof}The double reduced word $\mathbf i$ being
$(s_i,w_0)$-adapted, it is easy to see that the
equality~$[[\ev_{\mathbf i}(\mathbf x)]]_{\geq 0}=\ev_{\mathbf i_+}\circ
\pi_{\mathbf i\to\mathbf i_+}(\mathbf x)$ is satisfied for every
$\mathbf x\in{\mathcal X}_{[\mathbf i]_{\mathfrak R}}(t)$.
Therefore, it suffices to apply Theorem \ref{thm:twist}
and (\ref{equ:Xidef}) to prove the lemma.
\end{proof}

\begin{lemma}\label{lemma:Stwistcluster2}Let $i\in[1,l]$, $g\in G^{s_i,w_0}$
and $t\in H$. Denoting $t^{\star}$ the conjugation
of $t$ by $\widehat{w_0}$, we obtain
$[\Xi_{t}(gH)]_+H=t^{\star}E^{i^{\star}}H$.
\end{lemma}
\begin{proof}From  $g\in G^{s_i,w_0}$, we get the equality $[[gH]]_{<0}H=F^iH$.
So, up to $H$, we have $[\kappa((t\widehat{w_0})^{-1},[[gH]]_{<0})]_+$
equal to $t^{\star}E^{i^{\star}}{t^{\star}}^{-1}$. It suffices then to
apply (\ref{equ:Xidef}).
\end{proof}

Let us remember the reduced evaluations of Subsection
\ref{subsection:reducedevaluation}. Because the equality $t=\ev_{\mathbf 1}
(\mathbf x(\mathfrak R))$ is clear for every $\mathbf x
\in{\mathcal X}_{[\mathbf i]_{\mathfrak R}}(t)$, the following
proposition is directly implied by using the definitions of $\Xi_{i}$ and
$\Xi_{t}$, with Lemma\ref{lemma:Stwistcluster1} and Lemma \ref{lemma:Stwistcluster2},
via an amalgamation procedure.

\begin{prop}\label{prop:Stwistcluster}Let $v\in W$, $t\in H$, and
$\mathbf i\in R(v,w_0)$ be a double reduced word satisfying the equality
$\mathbf i=\mathbf{i}_+\overline{i}$. The following relation is satisfied
for every $\mathbf x\in{\mathcal X}_{[\mathbf i]_{\mathfrak R}}(t)$.
$$\begin{array}{ccc}
\Xi_{t}({\ev}_{\mathbf i}^{\red}(\mathbf x^{\red}))={\ev}_{\mathbf i_+^{\square}
i^{\star}}^{\red}(\Xi_{i}(\mathbf x)^{\red})\ .
\end{array}$$
\end{prop}

Finally, Proposition \ref{prop:salt} is deduced from the definition of
$(w_1,w_2)$-maps given in Section~\ref{section:twistedev}, equation
(\ref{lemma:Lambda}) and Proposition \ref{prop:Stwistcluster}, the properties
of the amalgamated product and the definition (\ref{equ:Xisi}) of the birational
Poisson isomorphism $\Xi_{s_i}$.

\section{Evaluations and cluster ${\mathcal X}$-varieties for $(G^*,\pi_{G^*})$}
\label{section:evaluationdual}

We start by giving an alternative way to describe twist maps with mutations
and tropical mutations. We then describe the dual Poisson-Lie group $(G^*,\pi_{G^*})$
via $(w_1,w_2)$-maps and provide, in Theorem \ref{thm:G^*}, evaluations for $(G^*,\pi_{G^*})$
in the spirit of the Kirillov-Reshetikhin multiplicative formula for the quantum $R$-matrix
associated to ${\mathcal U}_q(\mathfrak{g})$.
Moreover, birational Poisson isomorphisms using to pass from the positive part
to the negative part of $(G^*,\pi_{G^*})$ (and vice-versa) can be read on the
$W$-permutohedron: they are described by the $\uparrow$-paths linking the cluster
$\mathcal X$-varieties corresponding to the identity and the longest element $w_0$
of $W$.

\subsection{Twist maps and coordinates in Schubert cells}
\label{subsection:Schubert}
We introduce parameterizations of unipotent subgroups of $G$ that will be
used to evaluate the Poisson-Lie group $(G^*,\pi_{G^*})$. They involve
the generalized cluster transformations of Subsection \ref{section:twisted}.
We start by recalling a few facts from \cite[Section 2.4]{FZtotal}.
For every $w \in W$, the corresponding \emph{Schubert cell}
$(BwB)/B \subset G/B$ is the image of the Bruhat cell $B w B$
under the natural projection of $G$ onto the flag variety~$G/B$.
Let us recall the subgroups $N_+(w) \subset N$ and
$N_-(w) \subset N_-$ given by
$$
N_+(w) = N \cap {\widehat w} N_- {\widehat w}^{-1}\ ,\quad 
N_-(w) = N_- \cap {\widehat w}^{-1} N {\widehat w} \ .
$$
The following proposition is essentially well known
(cf.~\cite[Corollary~23.60]{fulton-harris}).

\begin{prop}
\label{pr:Bruhat cell}
An element $x \in G$ lies in the Bruhat cell $B w B$ if and only if
we have ${\widehat w}^{-1} x \in G_0$ and $[{\widehat w}^{-1} x]_- \in N_- (w)$.
Furthermore, the correspondence $\tau_+ :x\mapsto y_+$ given by
$$
y_+ = \tau_+ (x) =\widehat w [{\widehat w}^{-1} x]_- {\widehat w}^{-1} \in N_+ (w)
$$
induces a biregular isomorphism between the Schubert cell $(BwB)/B$ and $N_+ (w)$.
\end{prop}

Let $\T:G\to G:x\mapsto x^{\T}$ be the involutive anti-automorphism
of $G$ defined in \cite{FZtotal} and given, for every $i\in[1,l]$ and
every complex number $t$, by:
$$
\begin{array}{lll}
a^{\T}=a\ (a\in H),& x_i(t)^{\T}=x_{\overline i}(t),&
x_{\overline i}(t)^{\T}=x_{ i}(t)\ .
\end{array}
$$
Using the transpose map $T$, one obtains a counterpart of Proposition~\ref{pr:Bruhat cell}
for the opposite Bruhat cell $B_- w B_-$.

\begin{prop}
\label{pr:opposite Bruhat cell}
An element $x \in G$ lies in $B_- w B_-$ if and only if
we have $x {\widehat w}^{-1} \in G_0$ and $[x {\widehat w}^{-1}]_+ \in N_+ (w)$.
Furthermore, the correspondence $\tau_- : x \mapsto y_-$ given by
$$
y_- = \tau_- (x) = {\widehat w}^{-1} [x {\widehat w}^{-1}]_+ {\widehat w}
\in N_- (w)
$$
induces a biregular isomorphism between the ``opposite Schubert cell"
$B_- \backslash (B_- w B_-)$ and~$N_-(w)$.
\end{prop}

The maps $\tau_+$ and $\tau_-$ are in fact easily described using mutations
and tropical mutations. Let us recall that
the group $N_- (w)$ is a unipotent Lie group of dimension $\ell(w)$,
hence it is isomorphic to the affine space $\mathbb C^{\ell(w)}$ as an algebraic variety.
We are going to associate with any negative reduced word $\mathbf i =i_1\dots i_{\ell(w)}$
and every positive reduced word $\mathbf j =j_1\dots j_{\ell(w)}$
the following system of coordinates on $N_{\pm} (w)$ which involves the
generalized cluster transformations of equation
(\ref{equ:tor}). For every $\mathbf x\in{\mathcal X}_{\mathbf i}$ and
$\mathbf y\in{\mathcal X}_{\mathbf j}$, we set
\begin{equation}\label{equ:evbeta}
\begin{array}{c}
\tau_{\mathbf i} (\mathbf x) = {\widehat w}^{-1} \cdot
\widehat s_{i_1} y_{{i_1}} (-{x}_{\binom {i_1}{0}}^{-1}) \cdots
\widehat s_{i_{\ell(w)}} y_{{i_{\ell(w)}}} (-{x}_{\zeta_{\mathbf i
(\leq \ell(w)-1)}\binom {i_{\ell(w)}}{0}}^{-1}) \ ;\\
\\
\tau_{\mathbf j} (\mathbf y) =
x_{{j_1}} (-{y}_{\zeta_{\mathbf j(\geq 2)}\binom {j_1}{N^{j_1}
(\mathbf j)}}^{-1})\widehat s_{j_1}  \cdots
x_{{j_{\ell(w)}}} (-{y}_{\binom {j_{\ell(w)}}{N^{j_{\ell(w)}}
(\mathbf j)}}^{-1})\widehat s_{j_{\ell(w)}} \cdot{\widehat w}^{-1} \ .
\end{array}
\end{equation}
Now, for every $w$, every reduced word $\mathbf i={i_1}\dots {i_{\ell(w)}}\in R(w)$,
and every $k\in[1,\ell(w)]$, let us set $w_{\mathbf i_{>k}}:=s_{i_{k+1}}\dots s_{i_n}$
and $w_{\mathbf i_{<k}}:=s_{i_{1}}\dots s_{i_{k-1}}$.
For every $u\in W$, $t\in\mathbb C$ and $i\in[1,l]$, we denote:
\begin{equation}\label{equ:xbeta}\begin{array}{ccc}
x_{u( i)}(t):=\widehat{u}^{-1}x_{i}(t)\widehat{u}&\mbox{and}&
y_{u( i)}(t):=\widehat{u}^{-1}y_{i}(t)\widehat{u}\ .
\end{array}
\end{equation}
It is well-known, and straightforward to check via an induction over
the length on $W$, that the following equalities are satisfied
for every $w$ and every complex numbers $t_1,\dots,t_{\ell(w)}$.
\begin{equation}\label{equ:prodw}
\begin{array}{cccc}
&&\displaystyle\prod_{k=1}^{\ell(w)} y_{w_{\mathbf i_{>k}}(i_k)}(t_k)=\widehat{w}^{-1}\cdot
\prod_{k=1}^{\ell(w)} \widehat{s_{i_k}}y_{i_k}(t_k)\\
\mbox{and}\\
&&\displaystyle\prod_{k=1}^{\ell(w)} x_{w_{\mathbf i_{<k}}(i_k)}(t_k)=
\prod_{k=1}^{\ell(w)}x_{i_k}(t_k) \widehat{s_{i_k}}\cdot\widehat{w}^{-1}\ .
\end{array}
\end{equation}

\begin{lemma}\label{lemma:beta} For every $w\in W$,
$\mathbf{i}\in R(w,1)$, $\mathbf{j}\in R(1,w)$ and
$\mathbf x\in{\mathcal X}_{\mathbf{i}}$,
$\mathbf y\in{\mathcal X}_{\mathbf{j}}$ we have
$$\begin{array}{ccc}
[\widehat{w}^{-1}\ev_{\mathbf{i}}(\mathbf x)]_-
=\tau_{\mathbf i}(\mathbf x)&\mbox{and}&
[\ev_{\mathbf{j}}(\mathbf y)\widehat{w}]_+
=\tau_{\mathbf j}(\mathbf y)\ .
\end{array}$$
\end{lemma}
\begin{proof}Let us focus on the first relation. Using equation (\ref{equ:si})
and the negative projection $[\ .\ ]_-:B_-B\to N_-:x\mapsto[x]_-$ on the unipotent
subgroup $N_-\subset G$ associated
to the Gauss decomposition (\ref{equ:Gaussdec}), we are led to the equality
$$[\widehat{w_{\mathbf i_{>k-1}}}^{-1}\ev_{\mathbf i({k-1})}\circ
\zeta_{\mathbf i(\leq {k-1})}(\mathbf x)]_-=y_{w_{\mathbf i_{>k}}
({i_{k}})}(-{x}_{\zeta_{\mathbf i(\leq k-1)}\binom {i_{k}}{0}}^{-1})
\ [\widehat{w_{>k}}^{-1}\ev_{\mathbf i({k})}\circ\zeta_{\mathbf i(\leq {k})}
(\mathbf x)]_-\ .$$
The first relation is then obtained by iteration of this formula because of
the first equality of (\ref{equ:prodw}). The second relation is proved in the
same way, using the second equality of~(\ref{equ:prodw}).
\end{proof}

\begin{prop} For every $w\in W$,
$\mathbf{i}\in R(w,1)$, $\mathbf{j}\in R(1,w)$ and
$\mathbf x\in{\mathcal X}_{\mathbf{i}}$,
$\mathbf y\in{\mathcal X}_{\mathbf{j}}$, we have
the equalities
$$\begin{array}{ccc}
\tau_-(\ev_{\mathbf{i}}(\mathbf x))
=\widehat{w}\tau_{\mathbf i}(\mathbf x)\widehat{w}^{-1}&\mbox{and}&
\tau_+(\ev_{\mathbf{j}}(\mathbf y))
=\widehat{w}^{-1}\tau_{\mathbf j}(\mathbf y)\widehat{w}\ .
\end{array}$$
\end{prop}
\begin{proof}Simply deduced from the preceding lemma and the expression
of $\tau_-$ and $\tau_+$ given by Proposition \ref{pr:Bruhat cell} and
Proposition \ref{pr:opposite Bruhat cell}.
\end{proof}

\subsection{From $(G_0,\pi_*)$ to $(G^*,\pi_{G^*})$ via $(w_1,w_2)_{w_0}$-maps}
\label{subsection:G*rhomaps}
Here are a few preliminary maps to get the evaluation
maps related to $(G^*,\pi_{G^*})$ which are given in the next subsection.
Let us remember the variation of the Gauss decomposition given by the formula
(\ref{equ:gdecomposition}).

\begin{lemma}\label{lemma:beta1} The following equalities are satisfied
for every $w_2\in W$ and every $t\in H$.
$$\begin{array}{cccc}[[\rho_{t,(e,w_2)_{w_0}}(b_1,gH)]]_-=[b_1\widehat{w_0}^{-1}]_-&
\mbox{and}
&[[\rho_{t,(w_0,w_2)_{w_0}}(c_1,gH)]]_+=[(c_1\widehat{w_0})^{-1}]_+^{-1}\ .
\end{array}
$$
\end{lemma}
\begin{proof}We use definitions of $(w_1,w_2)_v$-maps, equation
(\ref{equ:gdecomposition}), and the fact that conjugating every
element of the Borel subgroup $B_{\pm}$ by $\widehat{w_0}$ gives
an element of the opposite Borel subgroup $B_{\mp}$ to deduce
\begin{equation}\label{equ:dec}
\begin{array}{ll}
&\rho_{t,(e,e)_{w_0}}(b_1,bH)=\underbrace{b_1b
\ t\widehat{w_0}[b\widehat{w_0}]_{\leq 0}^{-1}
\widehat{w_0}^{-1}\ [b_1\widehat{w_0}^{-1}]
_{\geq 0}^{-1}}_{\in B} \underbrace{[b_1
\widehat{w_0}^{-1}]_{-}^{-1}}_{\in N_-}\ ;
\\
\mbox{and}\\
&
\rho_{t,(w_0,e)_{w_0}}(c_1,bH)=\underbrace{[[c_1\widehat{w_0}]]_+}_{\in N_+}
\underbrace{[[c_1\widehat{w_0}]]_{\leq 0}\ \widehat{w_0}^{-1}bt\widehat{w_0}
\ (c_1[b\widehat{w_0}]_{\leq 0})^{-1}}_{\in B_-}\ .
\end{array}
\end{equation}
Moreover, it is clear that the unipotent part in the previous
equalities doesn't depend on the element $bH\in N$ relative to the
choice of $w_2$. Therefore the lemma is true for every $w_2\in W$.
\end{proof}

Let us recall the map $\Lambda_t$ given by equation (\ref{lemma:Lambda}).

\begin{prop}\label{prop:preval*}Let $b_1,b\in B$ and $t\in H$. The triplet
$(h,n,n_-)\in H\times N\times N_-$, such that the equality
$$nhn_-^{-1}=\rho_{t,(e,e)_{w_0}}(b_1,bH)$$
is satisfied, is given by the following formulas.
\begin{equation}\label{equ:doubledecomposition}\left\{
\begin{array}{lcl}
h&=&[b_1b]_0
\ t\widehat{w_0}[b\widehat{w_0}]_{0}^{-1}
\widehat{w_0}^{-1}\ [b_1\widehat{w_0}^{-1}]_{0}^{-1}\ ;\\
\\
n&=&[(\Lambda_t(b_1,\zeta^{1,w_0}(b)H)\ \widehat{w_0})^{-1}]_+^{-1}
\ ;\\
\\
n_-&=&[b_1\widehat{w_0}^{-1}]_-\ .
\end{array}
\right.
\end{equation}
\end{prop}
\begin{proof}The negative unipotent part of (\ref{equ:doubledecomposition})
is given by the left equality of Lemma \ref{lemma:beta1}, its diagonal
part is given by the first equation of (\ref{equ:dec}), and its positive
unipotent part is given by Theorem \ref{thm:twist} and Lemma
\ref{lemma:Lambda} and the right equality of
Lemma~\ref{lemma:beta1}.
\end{proof}

\subsection{Evaluations related to $(G^*,\pi_{G^*})$}We give the cluster
combinatorics on $(G^*,\pi_{G^*})$. To do that, we start by giving
evaluation maps for the elements $h\in H$ in (\ref{equ:doubledecomposition}).

\begin{prop}\label{prop:evaldiag}The equality
${[[\widehat{\ev}_{\mathbf{i}}(\mathbf x)]]_0}=\ev_{\mathbf 1}(\mathbf X)$
is satisfied for every double word $\mathbf i=\mathbf{i_1}\mathbf{i_2}\in D_e(w_0)$
and every cluster $\mathbf x\in{\mathcal X}_{[\mathbf i]_{\mathfrak{R}}}$ if and
only if $\mathbf X=(X_1,\dots, X_l)$ is the set of monomials given by
$$\begin{array}{cll}
X_i=x_{\binom{i}{N^i(\mathbf i)}}\displaystyle\prod_{k=1}^{l}\prod_{\ell=1}^{\ell(w_0)}
\displaystyle\prod_{j_1< N^{i_{\ell}}(\mathbf{i_1}_{\ell}),j_2< N^{i_{\ell}}(\mathbf{i_2}_{\ell})}{(-x_{\binom{i_{\ell}^{\star}}{j_1}}^{-1}
x_{\binom{i_{\ell}^{\star}}{N^{i_{\ell}}(\mathbf{i_1}_{\ell})+j_2}})}^{(A^{-1})_{ki}\langle \alpha^\vee_{i_\ell},
{w_0}^{-1}_{\geq\ell}\omega_{k}\rangle}\ .
\end{array}
$$
\end{prop}
Because of its length, the proof of Proposition \ref{prop:evaldiag}
is postponed to Subsection \ref{section:proofdiagonal}. Let us however stress
that the same kind of monomial formula would have been obtained by choosing a
trivial double word $\mathbf i=\mathbf{i_1}\mathbf{i_2}\in D_{w_0}(w_0)$.

\begin{ex}\label{ex:diageval}As usual let us set $\mathfrak g=A_2$ and take $t\in H$. We
consider a cluster $\mathbf x\in{\mathcal X}_{[121\ 121]_{\mathfrak R}}(t)$
and the related elements
$$\begin{array}{ccccc}
b_1=\ev^{\red}_{121}(x_{\binom{1}{0}},x_{\binom{1}{1}},x_{\binom{2}{0}} ),&
b=\ev^{\red}_{121}(x_{\binom{1}{2}},x_{\binom{1}{3}},x_{\binom{2}{1}})
&and& t=\ev_{\mathbf 1}(t_1,t_2)\ .
\end{array}$$
Using Example \ref{ex:twistA2}, Lemma \ref{lemma:w_0}, Lemma \ref{lemma:inverse},
Proposition \ref{prop:w0}, and equation (\ref{equ:doubledecomposition}), or simply
the formula above, we get
$$\begin{array}{rcl}
[[\widehat{\ev}_{121\ 121}(\mathbf x)]]_0&=&\ev_{\mathbf 1}\circ{\mathfrak m}(\ (
x_{\binom{1}{0}}^{-1}x_{\binom{1}{1}},
x_{\binom{2}{0}}^{-1}x_{\binom{1}{1}}^{-1}),\
(x_{\binom{2}{1}}x_{\binom{1}{3}},x_{\binom{1}{2}}x_{\binom{1}{3}}^{-1}),\
(t_1,t_2)\ )\\
\\
&=&\ev_{\mathbf 1}(x_{\binom{1}{0}}^{-1}x_{\binom{1}{1}}x_{\binom{1}{3}}x_{\binom{2}{1}}t_1,\
x_{\binom{2}{0}}^{-1}x_{\binom{1}{1}}^{-1}x_{\binom{1}{2}}x_{\binom{1}{3}}^{-1}t_2)\ .
\end{array}$$
\end{ex}

We then focus on the evaluation maps relative to the elements $n\in N$ and
$n_-\in N_-$ in equation~(\ref{equ:doubledecomposition}).
Let us remember the involutions $\star$ and $\circlearrowright$ on
double words and seed $\mathcal X$-tori given by Subsection
\ref{section:involution}.

\begin{lemma}\label{lemma:beta2}For every
$\mathbf{i}\in R(1,w_0)$, $\mathbf{j}\in R(w_0,1)$ and
$\mathbf x\in{\mathcal X}_{\mathbf{i}}$,
$\mathbf y\in{\mathcal X}_{\mathbf{j}}$ we have
$$\begin{array}{ccc}
[\ev_{\mathbf{i}}(\mathbf x)\widehat{w_0}^{-1}]_-
=\tau_\mathbf{i^{\star}}(\mathbf{x^{\star}})&\mbox{and}&
[[\ev_{\mathbf{j}}(\mathbf y)\widehat{w_0}]]_+^{-1}
=\tau_{\mathbf{j^{\circlearrowright}}}({\mathbf{y^{\circlearrowright}}})\ .
\end{array}$$
\end{lemma}
\begin{proof}The first relation comes from Lemma \ref{lemma:w_0} and Lemma \ref{lemma:beta}.
Equation (\ref{equ:gdecomposition}) and Corollary \ref{prop:w0} implies that the L.H.S.
of the second relation is equal to
$[\ev_{\mathbf{j^\circlearrowright}}(\mathbf{y^\circlearrowright})\widehat{w_0}]_+$.
Lemma \ref{lemma:beta} and Lemma \ref{lemma:inverse} then lead to the R.H.S. of the
second relation.
\end{proof}

\begin{lemma}\label{cor:beta} Let $\mathbf{i},\mathbf{j}\in D(w_0)$,
$\mathbf{i}$ be a $(e,e)_{w_0}$-word, $\mathbf{j}$ be a $(w_0,e)_{w_0}$-word,
and $\mathfrak{s},\mathfrak{s'}$ be $\mathcal X$-splits relative respectively
to the $(w_1,w_2)_{w_0}$-decompositions $\mathbf i\to(\mathbf{i_1},\mathbf{i_2})$,
and $\mathbf j\to(\mathbf{j_1},\mathbf{j_2})$. The following equalities are satisfied.
$$\begin{array}{ccc}
[[\widehat{\ev}_{\mathbf{i}}(\mathbf x)]]_-^{-1}=\tau_{{\mathbf{i_1^{\star}}}}
(\mathbf{x}_{(1)}^{\star})&\mbox{and}&[[\widehat{\ev}_{\mathbf{j}}(\mathbf y)]]_+^{-1}
=\tau_{{\mathbf{j_1^{\circlearrowright}}}}(\mathbf{y}_{(1)}^{\circlearrowright})\ .
\end{array}$$
\end{lemma}
\begin{proof}The relations are derived from Lemma
\ref{lemma:w_1w_2}, Lemma \ref{lemma:beta1} and Lemma
\ref{lemma:beta2}.
\end{proof}

We can then get the cluster combinatorics on $(G^*,\pi_{G^*})$.
Let us associate to any double word $\mathbf i\in D(w_0)$ some double
words $\mathbf i_e\in D_e(w_0)$ and $\mathbf i_{w_0}\in D_{w_0}(w_0)$
being respectively a $(e,w_0)$-trivial double word and a $(w_0,e)$-trivial
double word. Moreover, let us respectively denote $\mathbf i_{e+}\in R(1,w_0)$
and $\mathbf i_{w_0-}\in R(w_0,1)$, the positive part of $\mathbf i_e$
and the negative part of $\mathbf i_{w_0}$, accordingly the definition of Subsection
\ref{def:posnegpart}, and denote $\wp_{e}:{\mathcal X}_{[\mathbf i_{e}]_{\mathfrak R}}
\to{\mathcal X}_{\mathbf i_{e+}}^{\red}$ and $\wp_{w_0}:
{\mathcal X}_{[\mathbf i_{w_0}]_{\mathfrak R}}\to{\mathcal X}_{\mathbf i_{w_0-}}^{\red}$
the corresponding canonical projections on seed $\mathcal X$-tori. Let us finally
remember the birational Poisson isomorphisms $\widehat{\mu}_{\mathbf i\to\mathbf i_{e}}$
and $\widehat{\mu}_{\mathbf i\to\mathbf i_{w_0}}$ defined in Subsection \ref{subsection:thmWcluster}.
To any $\mathbf x\in{\mathcal X}_{[\mathbf i]_{\mathfrak R}}$, we associate the clusters
$$\begin{array}{ccc}
\mathbf x_e=\wp_{e}\circ\widehat{\mu}_{\mathbf i\to\mathbf i_{e}}(\mathbf x)
&\mbox{and}&\mathbf x_{w_0}=\wp_{w_0}
\circ\widehat{\mu}_{\mathbf i\to\mathbf i_{w_0}}(\mathbf x)\ ,
\end{array}
$$
and derive the following lemma from Proposition \ref{prop:evaldiag}, Lemma \ref{cor:beta}
and Theorem \ref{thm:ev*}.
\begin{lemma}\label{lemma:G*}
The following decomposition is satisfied
for every $w\in W$, every double word $\mathbf{i}\in D_{w}(w_0)$
and every $\mathbf x\in{\mathcal X}_{w}$.
$$
\begin{array}{lll}
&\left\{
\begin{array}{llll}
{[[\widehat{\ev}_{\mathbf{i}}(\mathbf x)]]_0}&=&\ev_{\mathbf 1}({\mathbf X}_e)\ ;\\
\\
{[[\widehat{\ev}_{\mathbf{i}}(\mathbf x)]]_+}&=&
\tau_{\mathbf i_{w_0}}(\mathbf x_{w_0}
^{\circlearrowright})^{-1}\ ;\\
\\
{[[\widehat{\ev}_{\mathbf{i}}(\mathbf x)]]_-^{-1}}&=&\tau_{\mathbf i_{e}^{\star}}
(\mathbf{x}_e^{\star})\ ;
\end{array}
\right.
\end{array}$$
where ${\mathbf X}_e=({X}_1,\dots {X}_l)$ is the set of
monomials given in Proposition \ref{prop:evaldiag}
applied to the cluster $\widehat{\mu}_{\mathbf i\to\mathbf i_{e}}(\mathbf x)$.
\end{lemma}


Now, let us recall that the map $\phi:(G^*,\pi_{G^*})\to(BB_-,\pi_*)$
given by the formula $(nh,n_-h^{-1})\mapsto nh^2n_-^{-1}$ is not an
isomorphism but a covering of degree $2^l$.
Example \ref{ex:diageval} in particular shows that in the general case we
cannot expect to directly obtain rational evaluations for the dual Poisson-Lie
group $(G^*,\pi_{G^*})$, because of this covering
$h\mapsto h^2$ on the Cartan subgroup $H$ of $G$.
The remedying idea is to take covers on cluster variables which
mimic $\phi$.
Let $\mathbf I=(I,I_0,\varepsilon,d)$ be a seed;
the seed $\mathcal X$-torus denoted $\mathcal{X}_{\mathbf I^{1/2}}$ is the
torus $(\mathbb C_{\neq 0})^{|I|}$ given with the Poisson bracket
$$\{{x}_i,{x}_j\}=\frac{\widehat{\varepsilon}_{ij}}{4}
{x}_i{x}_j\ ,$$
where $\{{x}_i\mid i\in I\}$ still denote the standard coordinates on the
factors. In particular, the following map is a Poisson covering of degree
$2^{|I|}$.
\begin{equation}\label{equ:covtorus}
\begin{array}{cccc}
{\mathfrak c}_{\mathcal{X}}:&\mathcal{X}_{\mathbf I^{1/2}}\longrightarrow
\mathcal{X}_{\mathbf I}:&(x_1,\dots,x_{|I|})\longmapsto(x_1^2,\dots,x_{|I|}^2)\ .
\end{array}
\end{equation}
Thus, Lemma \ref{lemma:G*} and the fact that the maps
${\mathfrak c}_{\mathcal{X}}$ and $\phi$ are Poisson covering whose degrees are
some powers of $2$ lead us to the following result.

\begin{thm}\label{thm:G^*}Let $\mathbf i\in D(w_0)$. The following evaluation map
$\Ev_{\mathbf i}$ is a Poisson covering of degree $2^n$, for some $n\leq\dim G$, onto
a Zarisky open set of $G^*$.
$$
\begin{array}{cccl}
\Ev_{\mathbf i}:&{\mathcal X}_{[\mathbf i]_{\mathfrak{R}}^{1/2}}\to(G^*,\pi_{G^*}):
\quad{\mathbf x}\mapsto({{\ev}_{\mathbf{i}}^+(\mathbf{{x}})},
{{\ev}_{\mathbf{i}}^-(\mathbf {{x}})})
\end{array}
$$
\begin{equation}\label{equ:Ev*}
\begin{array}{lll}
\mbox{where}&\left\{
\begin{array}{llll}
{{\ev}_{\mathbf{i}}^+(\mathbf {{x}})}&=&
\tau_{\mathbf i_{w_0}}({\mathfrak c}_{\mathcal X}(\mathbf {{x}})_{w_0}
^{\circlearrowright})^{-1}\ev_{\mathbf 1}({\mathbf X}_{e})\ ;\\
\\
{{\ev}_{\mathbf{i}}^-(\mathbf {{x}})}&=&\tau_{\mathbf i_{e}^{\star}}
({\mathfrak c}_{\mathcal X}(\mathbf{{x}})_e^{\star})\ev_{\mathbf 1}
({\mathbf X}_e)^{-1}\ ,
\end{array}
\right.
\end{array}
\end{equation}
and the set ${\mathbf X}_e=({X}_1,\dots {X}_l)$
is the same as in Lemma \ref{lemma:G*}.
\end{thm}

\begin{rem}\label{rem:computcovering}A careful study of the cluster variables
appearing in the monomial formulas describing ${\mathbf X}_e=({X}_1,\dots{X}_l)$
and given in Proposition \ref{prop:evaldiag} leads to a choice of a subcovering of
the covering ${\mathfrak c}_{\mathcal X}$ that minimizes the value $n$ of
the previous theorem.
\end{rem}


\begin{ex} When the equality $\mathfrak{g}=A_2$ is satisfied, the
heuristics of Theorem \ref{thm:G^*} is illustrated by Figure
\ref{fig:commutativeweakorder*}, where we have used the notation
$G^*=(G^*_+,G^*_-)$ to abbreviate the description (\ref{equ:defG*double}).
In particular, if we choose the double word $\mathbf i=121121$, then we
can take $\mathbf i_e=\mathbf i$ and $\mathbf i_{w_0}=\overline{2}
\overline{1}\overline{2}121$. Therefore, to any $\mathbf x
\in{\mathcal X}_{[\mathbf i]_{\mathfrak{R}}}$ are associated the elements:
$$\begin{array}{ccc}
\mathbf x_e=\wp_{e}(\mathbf x)&\mbox{and}&
\mathbf x_{w_0}=\wp_{w_0}\circ\Xi_{s_1}\circ\mu_{\overline{2}1\overline{1}121\to \overline{2}\overline{1}1121}
\circ\Xi_{s_2}\circ\mu_{12\overline{2}121\to \overline{2}12121}\circ\Xi_{s_1}(\mathbf x),
\end{array}
$$
whereas $\mathbf X_e$ has already been given in Example \ref{ex:diageval}.
\end{ex}

\begin{figure}[htbp]
\begin{center}
\setlength{\unitlength}{1.5pt}
\begin{picture}(20,110)(0,-80)
\put(-29,-63.5){\line(2,3){17.7}}
\put(-10,-33){\line(0,1){26}}
\put(-11,-3.5){\line(-2,3){17.7}}
\put(-31,-63.5){\line(-2,3){17.7}}
\put(-50,-33){\line(0,1){26}}
\put(-49,-3.5){\line(2,3){17.7}}
\thicklines
\put(-30,25){\circle{4}}
\put(-60,33){${\mathcal X}_{w_0}\supset{\mathcal X}_{[\mathbf i_{w_0}]_{\mathfrak R}}$}
\put(28,25){\vector(-1,0){56}}
\put(28,25){\line(-1,0){56}}
\put(0,30){${\mathfrak c}_{\mathcal{X}}$}
\put(30,25){\circle{4}}
\put(23,33){${\mathcal X}_{[\mathbf i_{w_0}]_{\mathfrak R}^{1/2}}$}
\put(-50,-5){\circle{4}}
\put(-70,-10){${\mathcal X}_{s_2s_1}$}
\put(-7,-10){${\mathcal X}_{s_1s_2}$}
\put(-10,-5){\circle{4}}
\put(-50,-35){\circle{4}}
\put(-62,-40){${\mathcal X}_{s_1}$}
\put(-7,-40){${\mathcal X}_{s_2}$}
\put(-10,-35){\circle{4}}
\put(-30,-65){\circle{4}}
\put(-55,-75){${\mathcal X}_e\supset{\mathcal X}_{[\mathbf i_e]_{\mathfrak R}}$}
\put(28,-65){\vector(-1,0){56}}
\put(28,-65){\line(-1,0){56}}
\put(0,-70){${\mathfrak c}_{\mathcal{X}}$}
\put(30,-65){\circle{4}}
\put(26,-75){${\mathcal X}_{[\mathbf i_e]_{\mathfrak R}^{1/2}}$}
\put(87,-15){$G^*_+$}
\put(87,-30){$G^*_-$}
\put(32,-65){\vector(3,2){55}}
\put(32,-65){\line(3,2){55}}
\put(61,-54){${\ev}^-$}
\put(32,25){\vector(3,-2){55}}
\put(32,25){\line(3,-2){55}}
\put(60,12){${\ev}^+$}
\end{picture}
\end{center}
\vspace{-.1in}
\caption{Evaluations related to $(G^*,\pi_{G^*})$
when $\mathfrak{g}=A_2$}
\label{fig:commutativeweakorder*}
\end{figure}
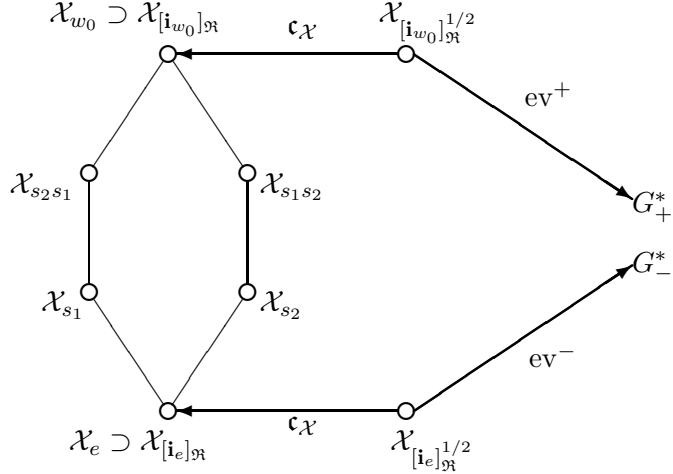

These evaluation maps on $(G^*,\pi_{G^*})$ are, of course, compatible with
the cluster combinatorics already developed.
The following theorem is then derived from Theorem \ref{thm:ev*} and the
definitions (\ref{equ:covtorus}) and (\ref{equ:Ev*}) of the covering
${\mathfrak c}_{\mathcal X}$ and the evaluation map $\Ev_{\mathbf i}$.

\begin{thm}The following diagram is commutative for any double words
$\mathbf i,\mathbf j\in D(w_0)$.
$$
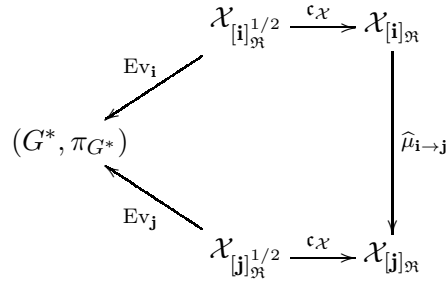
\xymatrix{
&{\mathcal X}_{[\mathbf i]_{\mathfrak{R}}^{1/2}}\ar@/^0pc/[ld]_{\Ev_{\mathbf i}}
\ar@/^0pc/[r]^{\mathfrak{c}_{\mathcal{X}}}&{\mathcal X}_{[\mathbf i]_{\mathfrak{R}}}
\ar@/^0pc/[dd]^{\widehat{\mu}_{\mathbf i\to\mathbf j}}\\
(G^*,\pi_{G^*})&\\
&{\mathcal X}_{[\mathbf j]_{\mathfrak{R}}^{1/2}}\ar@/^0pc/[ul]^{\Ev_{\mathbf j}}
\ar@/^0pc/[r]^{\mathfrak{c}_{\mathcal{X}}}&{\mathcal X}_{[\mathbf j]_{\mathfrak{R}}}
}
$$
\end{thm}

\subsection{Proof of Proposition \ref{prop:evaldiag}}
\label{section:proofdiagonal}
The main ingredient to prove Proposition \ref{prop:evaldiag} is the
factorization theorem \cite[Theorems~1.10
and formula (1.21)]{FZtotal} of Fomin and Zelevinsky.
Here, we mainly follow the exposition of \cite{KZ}.
Let $\widetilde{G}$ be the simply connected cover of $G$, and denote
$\widetilde{B}$, $\widetilde{B}_-$ the Borel subgroups of $\widetilde{G}$ such
that their images in $G$ are respectively $B$ and $B_-$, and the intersection $\widetilde{H}=
\widetilde{B}\cap\widetilde{B}_-$. In the same way, for every $u,v\in  W$,
denote $\widetilde{G}^{u,v}$ the double Bruhat cell in $\widetilde{G}$ whose image
in $G$ is $G^{u,v}$. For $x\in \widetilde{B}\widetilde{B}_-$ and a fundamental
weight $\omega_i$, define $\overline{\Delta}_i(x)=[x]_0^{\omega_i}$.
It is shown in \cite{FZtotal} that $\overline{\Delta}_i$ extends to a regular
function on $\widetilde{G}$. For type $A_n$ (when $\widetilde{G} = SL(n+1,\mathbb C)$), this is
just the principal $i \times i$ minor of a matrix $x$.
For any pair $u,v\in W$, the corresponding \emph{generalized
minor} is a regular function on $\widetilde{G}$ given by
$$
\Delta_{u\omega_i,v\omega_i} (x) =\overline{\Delta}_i({\widehat u}^{\ -1}x
\widehat v)\ .
$$
It is shown in \cite{FZtotal} that these functions are well defined,
that is they depend only on the weights $u\omega_i$ and
$v\omega_i$ and do not depend on the particular choice of $u$ and
$v$.
For $i=1,\dots,l$, we denote $\varepsilon(i)=+1$ and
$\varepsilon(\bar i)=-1$, and recall that $|i|=|\bar i|=i$.
In what follows, we fix $u,v \in W$ and a double reduced word
${\bf i}$ of $(u,v)$. We append $l$ entries
$i_{m+1},\dots,i_{m+l}$ to $\bf i$ by setting $i_{m+j}=\bar j$.
For $k = 1, \dots, m$, we set
$$
u_{\geq k}=\doublesubscript{\prod}{\ell=m,\dots,k}
{\varepsilon(i_\ell) = -1} s_{|i_\ell|}\ , \ \
v_{<k}=\doublesubscript{\prod}{\ell=1,\dots,k-1}
{\varepsilon(i_\ell) = +1} s_{|i_\ell|}\ ,
$$
where the notation implies that the index $\ell$ in the first
(resp.~second) product is decreasing (resp.~increasing). We also
set $u_{\geq k}=e$, $v_{<k}=v$ for $k = m+1, \dots, m+l$. For
example, if ${\bf i}=1\bar 2 2 \bar 3 3 2\bar 1$ then
$u_{\geq 4}=s_1s_3$, $v_{<4}=s_1s_2$.
For every $k = 1, \dots, m+l$, we set
$\gamma^k= u_{\geq k}\omega_{|i_k|}$, $\delta^k = v_{<k}\omega_{|i_k|}$,
and introduce a regular function $M_k$ on $\widetilde{G}^{u,v}$ by setting
\begin{equation}
\label{eq:twisted-minors-def} M_k (x) =\Delta_{\gamma^k\ ,
\delta^k}(x')\ ,
\end{equation}
where $x'$ is the twist of $x$ given by the formula (\ref{equ:FZtwist}).
We refer to the family $M_1, \dots, M_{m+l}$ as \emph{twisted
minors} associated with a reduced word $\bf i$. Their significance
stems from the following result (see \cite[Theorems~1.2, 1.9, 1.10
and formula (1.21)]{FZtotal}). Let us remember the notation given in
(\ref{eq:xnegative}) and let us define,for every double word
$\mathbf i=i_1 \ldots i_m$, the map $x_\mathbf i: \mathbb C_{\neq 0}^m \to G$ by
$$
\begin{array}{ccc}
x_\mathbf i (\mathbf t) =x_{i_1} (t_1) \cdots x_{i_m} (t_m) \,
&\mbox{where}&\mathbf t=(t_1, \ldots, t_m)\ .
\end{array}
$$

\begin{thm}\cite[Theorem 2.3]{KZ}
\label{biregular} The map $x_{\bf i}:\widetilde{H}\times \mathbb C^m \to \widetilde{G}$
given by
$$
x_{\bf i}(a;t_1,\dots,t_m)=ax_{i_1}(t_1)\dots x_{i_m}(t_m)
$$
restricts to a biregular isomorphism between a complex torus
$\widetilde{H}\times (\mathbb C - \{0\})^m$ and a Zariski open subset
$U_{\mathbf i}=\{x \in \widetilde{G}^{u,v} : M_k (x) \neq 0 \text { for } 1
\leq k \leq m+l\}$ of the double Bruhat cell $\widetilde{G}^{u,v}$. Furthermore,
for $k = 1, \dots, m+l$ and $x = x_{\mathbf i}(a;t_1,\dots,t_m) \in U_{\mathbf i}$,
we have
\begin{equation}
\label{eq:M-k-monomial} M_k(x)=a^{-u \gamma^k}
\doublesubscript{\prod}{1 \leq \ell<k}{\varepsilon(i_\ell) = -1}
t_{\ell}^{\langle \alpha^\vee_{|i_\ell|},
u^{-1}_{\geq\ell}\gamma^k\rangle} \doublesubscript{\prod}{k \leq
\ell \leq m}{\varepsilon(i_\ell) = +1} t_{\ell}^{\langle
\alpha^\vee_{|i_\ell|},v^{-1}_{<\ell+1}\delta^k \rangle}\ .
\end{equation}
\end{thm}

We are going to use this theorem to prove Proposition \ref{prop:evaldiag}.
Let us first remark that for $u=w_0$ and $v=e$, we have
the equalities $\gamma^k=\omega_{k}$ and $\delta^k=\omega_{k}$
for every $k>\ell(w_0)$. Now, let us notice that if $b_1$ and $b$ belong
to the double Bruhat cell $\widetilde{G}^{e,w_0}$, then the elements
$$\begin{array}{ccc}
b_1^{\star}=\widehat{w_0}^{-1}b_1\widehat{w_0}
&\mbox{and}&b^{\circlearrowright}=\widehat{w_0}^{-1}b^{-1}\widehat{w_0}
\end{array}$$
belong to the double Bruhat cell $\widetilde{G}^{w_0,e}$. Therefore, using the
definition of the twist map $x\mapsto x'$, the formula (\ref{eq:twisted-minors-def}),
and the fact that the relation $a^{\theta}=a^{-1}$ is satisfied for every $a\in \widetilde{H}$,
we get the following equalities for every $k>\ell(w_0)$.
$$\begin{array}{cccc}
M_k(b_1^{\star})=([\widehat{w_0}^{-1}b_1^{\star}]_{\geq 0}^{\theta})^{\omega_{i_k}}
=[b_1\widehat{w_0}^{-1}]_{0}^{-\omega_{k}}&\mbox{and}&
M_k(b^{\circlearrowright})^{-1}=(\widehat{w_0}[b\widehat{w_0}]_0^{-1}\widehat{w_0}^{-1})^{\omega_k}\ .
\end{array}$$
Taking any $\mathbf i\in R(w_0,1)$ to parameterize $b_1^{\star}=ax_{\mathbf i}(t^{\star}_1,\dots, t^{\star}_{\ell(w_0)})$ and
$b^{\circlearrowright}=a'x_{\mathbf i}(t^{\circlearrowright}_1,\dots, t^{\circlearrowright}_{\ell(w_0)})$ and applying the formula
(\ref{eq:M-k-monomial}) lead to:
$$\begin{array}{ccc}
[b_1\widehat{w_0}^{-1}]_{0}^{-\omega_{k}}=a^{\omega_{k^{\star}}}
\displaystyle\prod_{\ell}
{t^{\star}_{\ell}}^{\langle \alpha^\vee_{|i_\ell|},
{w_0}^{-1}_{\geq\ell}\omega_{k}\rangle}&\mbox{and}&
(\widehat{w_0}[b\widehat{w_0}]_0^{-1}\widehat{w_0}^{-1})^{\omega_k}=a'^{-\omega_{k^{\star}}}
\displaystyle\prod_{\ell}
{t^{\circlearrowright}_{\ell}}^{-\langle \alpha^\vee_{|i_\ell|},
{w_0}^{-1}_{\geq\ell}\omega_{k}\rangle}\ .
\end{array}
$$
So we have the following equality.
\begin{equation}\label{equ:Hpreformula}\begin{array}{ccc}
([b_1]_0[b_1\widehat{w_0}^{-1}]_{0}^{-1}[b]_0\widehat{w_0}
[b\widehat{w_0}]_0^{-1}\widehat{w_0}^{-1})^{\omega_k}=\displaystyle\prod_{\ell}
({t}^{\star}_{\ell}{t^{\circlearrowright}_{\ell}}^{-1})^{\langle \alpha^\vee_{|i_\ell|},
{w_0}^{-1}_{\geq\ell}\omega_{k}\rangle}\ .
\end{array}
\end{equation}
Now, let us fix some $\mathbf j=j_1\dots j_{\ell(w_0)}\in R(1,w_0)$ and evaluate
some $b_1,b\in N\cap G^{e,w_0}$ by $\mathbf z,\mathbf{z'}\in{\mathcal X}_{\mathbf i}$,
that is: $b_1=\ev_{\mathbf j}(\mathbf z)$ and $b=\ev_{\mathbf j}(\mathbf{z'})$. Let us
also denote $\mathbf j_{\ell}:=j_1\dots j_{\ell}$ for every $\ell\leq\ell(w_0)$.
We obtain from the relations ${\mathbf x}_{\mathbf j}(t_1,\dots, t_{\ell(w_0)})
=\ev_{\mathbf j}(\mathbf z)$ and ${\mathbf x}_{\mathbf j}(t'_1,\dots, t'_{\ell(w_0)})
=\ev_{\mathbf j}(\mathbf{z'})$ the equalities:
$$\begin{array}{ccc}
t_{\ell}=\displaystyle\prod_{j< N^{j_{\ell}}(\mathbf j_{\ell})}z_{\binom{j_{\ell}}{j}}
&\mbox{and}&t'_{\ell}=\displaystyle\prod_{j< N^{j_{\ell}}(\mathbf j_{\ell})}z'_{\binom{j_{\ell}}{j}}\ .
\end{array}$$
Moreover, because $b\in N$, we have $\pi_{\mathbf i}(\mathbf{z'})=1$.
Thus, adding Lemma \ref{lemma:w_0}, the formula (\ref{equ:circlearrow}) and Proposition
\ref{prop:w0} to the previous equalities gives:
$$\begin{array}{ccc}
t^{\star}_{\ell}{t^{\circlearrowright}_{\ell}}^{-1}=-\displaystyle\prod_{j< N^{j_{\ell}}
(\mathbf j_{\ell})}({z_{\binom{j_{\ell}^{\star}}{j}}z'_{\binom{j_{\ell}^{\star}}
{N^{j_{\ell}}(\mathbf j)-j}}})^{-1}
=-\displaystyle\prod_{j< N^{j_{\ell}}(\mathbf j_{\ell})}
z_{\binom{j_{\ell}^{\star}}{j}}^{-1}z'_{\binom{j_{\ell}^{\star}}{j}}\ .
\end{array}$$
In fact, it is easy to see that the same kind of formula can be obtained when
the evaluation of $b_1$ and $b$ are done respectively for any $\mathbf i_1$ and $\mathbf i_2$
in $R(1,w_0)$.
The setting $\mathbf i=\mathbf{i_1}\mathbf{i_2}\in D_e(w_0)$ and the equality
$b_1bH=\ev_{\mathbf i}^{\red}(\mathbf x)$
then imply that $\mathbf x^{\red}=\mathfrak{m}(\mathbf z,\mathbf{z'}^{\red})$.
We finally apply Proposition \ref{prop:preval*} and the equality (\ref{equ:defphi2}) on
the formula (\ref{equ:Hpreformula}) to end the proof of Proposition~\ref{prop:evaldiag},
that is:
$$\begin{array}{cll}
X_i&=x_{\binom{i}{N^i(\mathbf i)}}\displaystyle\prod_{k=1}^{l}\prod_{\ell=1}^{\ell(w_0)}
\displaystyle\prod_{j< N^{i_{\ell}}(\mathbf{i_1}_{\ell}),j'< N^{i_{\ell}}(\mathbf{i_2}_{\ell})}{(-z_{\binom{i_{\ell}^{\star}}{j}}^{-1}
z'_{\binom{i_{\ell}^{\star}}{j'}})}^{(A^{-1})_{ki}\langle \alpha^\vee_{i_\ell},
{w_0}^{-1}_{\geq\ell}\omega_{k}\rangle}\\
&=x_{\binom{i}{N^i(\mathbf i)}}\displaystyle\prod_{k=1}^{l}\prod_{\ell=1}^{\ell(w_0)}
\displaystyle\prod_{j_1< N^{i_{\ell}}(\mathbf{i_1}_{\ell}),j_2< N^{i_{\ell}}(\mathbf{i_2}_{\ell})}{(-x_{\binom{i_{\ell}^{\star}}{j_1}}^{-1}
x_{\binom{i_{\ell}^{\star}}{N^{i_{\ell}}(\mathbf{i_1}_{\ell})+j_2}})}^{(A^{-1})_{ki}\langle \alpha^\vee_{i_\ell},
{w_0}^{-1}_{\geq\ell}\omega_{k}\rangle}\ .
\end{array}$$

\section{An elementary approach for the case $G=\SL(2,\mathbb C)$.}\label{section:SL2}
To fix the ideas, we consider with full details all the evaluation maps met before,
and the related cluster combinatorics, in the simplest case: the case $G=SL(2,\mathbb C)$.
We thus start by recalling the construction of Fock and Goncharov for
$(\SL(2,\mathbb C),\pi_G)$ and successively consider the models $(\SL(2,\mathbb C),\pi_*)$
and $(\SL(2,\mathbb C)^*,\pi_{G^*})$ for dual Poisson-Lie groups.
And, as a conclusion, we give the quantization of this elementary construction by
considering the cluster combinatorics associated with the quantized universal enveloping
algebra $\mathcal{U}_q(\mathfrak{g})$ of the Lie algebra
$\mathfrak{g}=\sl(2,\mathbb C)$. This section is written to be as self-contained
as possible.
\subsection{Elementary Lie data}
Let us recall that the complex simple Lie group
\begin{equation}\label{equ:sl2}
SL(2, \mathbb{C})=\{\left(
\begin{array}{cc}
t_{11} & t_{12}\\
t_{21} & t_{22}
\end{array}
\right):t_{11}t_{22}-t_{12}t_{21}=1,\ \  t_{ij}\in \mathbb{C}\}\ .
\end{equation}
has its Lie algebra $\mathfrak{g}$ equal to the set $\sl(2,\mathbb C)$ of
$2$-squared complex matrices which have a zero trace. The Chevalley generators
$\{e_1,f_1,h_1\}$ and its related basis $\{e_1,f_1,h^1\}$ are then given by
the following matrices:
$$\begin{array}{llll}
e_1=\left(
\begin{array}{cc}
0 & 1\\
0 & 0
\end{array}
\right), &
f_1=\left(
\begin{array}{cc}
0 & 0\\
1 & 0
\end{array}
\right), &
h_1=\left(
\begin{array}{cc}
1 & 0\\
0 & -1
\end{array}
\right), &
h^1=\left(
\begin{array}{cc}
1/2 & 0\\
0 & -1/2
\end{array}
\right)
\end{array}\ .$$
Using the exponential map $\exp:\mathfrak{g}\rightarrow G$, which, in this
case, associates to a matrix $M\in\mathfrak g$ the usual matrix
$\sum_{n=0}^{\infty}\frac{M^n}{n!}\in G$,
we get the following generators of $G$, the two last ones being associated
to every non-zero complex number $x$.
$$\begin{array}{llll}
E^1=\left(
\begin{array}{cc}
1 & 1\\
0 & 1
\end{array}
\right), &
F^1=\left(
\begin{array}{cc}
1 & 0\\
1 & 1
\end{array}
\right), &
H_1(x)=\left(
\begin{array}{cc}
x & 0\\
0 & x^{-1}
\end{array}
\right),&
H^1(x)=\left(
\begin{array}{cc}
x^{1/2} & 0\\
0 & x^{-1/2}
\end{array}
\right).
\end{array}$$
In particular, these generators of the diagonal subgroup $H$ of $G$
satisfy the relation $H^1(x^2)=H_1(x)$ for every complex number
$x\in{\mathbb C}_{\neq 0}$, which agrees with the formula (\ref{equ:defphi2}),
because the Cartan matrix $A$ is simply here the number $2$.
Let us stress, however, that the generator $H^1(x)$ is generally ill-defined
on $\SL(2,\mathbb C)$. It is because $\SL(2,\mathbb C)$ is not of adjoint type,
but simply connected. The related adjoint group is $\PGL(2,\mathbb C)$, and
$H^1(x)$ is well-defined on $\PGL(2,\mathbb C)$, because of the following identity.
$$H^1(x)=\left(
\begin{array}{cc}
x^{1/2} & 0\\
0 & x^{-1/2}
\end{array}
\right)\stackrel{\PGL(2,\mathbb C)}{=}
\left(
\begin{array}{cc}
x & 0\\
0 & 1
\end{array}
\right)\ .$$
Now, because there is only one simple root $\alpha_1$, the Weyl group $W$ contains only two elements
$\{1,s_1\}$ and the different double reduced words are the double words $1$,
$\overline{1}$, $1\overline{1}$, $\overline{1}1$ without forgetting the trivial
double word $\mathbf 1$ associated to the unity element of the direct product
$W\times W$. Finally, the $r$-matrix $r\in\mathfrak g\wedge\mathfrak g$ associated
to $\sl(2,\mathbb C)$ and its related elements $r_{\pm}\in\mathfrak g\otimes\mathfrak g$
are given by the following formulas.
\begin{equation}\label{equ:elemrmatrix}
\begin{array}{cccc}
r=e_1\wedge f_1,&r_+=\displaystyle\frac{1}{4}h_1\otimes h_1+e_1\otimes f_1&\mbox{and}&
r_-=-\displaystyle\frac{1}{4}h_1\otimes h_1-f_1\otimes e_1\ .
\end{array}\end{equation}

\subsection{Cluster ${\mathcal X}$-varieties related to $(\SL(2,\mathbb C),\pi_G)$}
The evaluation maps of Fock and Goncharov associated to the previous
double reduced words are then the following:
$$\begin{array}{rcllll}
\ev_{\mathbf 1}(x_0)&=&H^1(x_0)\in G^{1,1}\\
&=&\left(
\begin{array}{cc}
x_0^{1/2} & 0\\
0 & x_0^{-1/2}
\end{array}
\right).\\
\\
\ev_{1}(y_0,y_1)&=&H^1(y_0)E^1H^1(y_1)\in G^{1,w_0}\\
&=&\left(\begin{array}{cc}
{y_0}^{1/2}{y_1}^{1/2} & {y_0}^{1/2}{y_1}^{-1/2}\\
0 & {y_0}^{-1/2}{y_1}^{-1/2}
\end{array}
\right)
=\left(\begin{array}{cc}
1 & {y_0}\\
0 & 1
\end{array}
\right)H^1(y_0y_1)\ .\\
\\
\ev_{\overline{1}}(z_0,z_1)&=&H^1(z_0)F^1H^1(z_1)\in G^{w_0,1}\\
&=&\left(\begin{array}{cc}
{z_0}^{1/2}{z_1}^{1/2} & 0\\
{z_0}^{-1/2}{z_1}^{1/2} & {z_0}^{-1/2}{z_1}^{-1/2}
\end{array}
\right)
=\left(\begin{array}{cc}
1 & 0\\
z_0^{-1} & 1
\end{array}
\right)H^1(z_0z_1)\ .
\end{array}$$
$$\begin{array}{rcllll}
\ev_{1\overline{1}}(u_0,u_1,u_2)&=&H^1(u_0)E^1H^1(u_1)F^1H^1(u_2)\in G^{w_0,w_0}\\
&=&\left(\begin{array}{cc}
u_0^{1/2}u_1^{1/2}u_2^{1/2}+u_0^{1/2}u_1^{-1/2}u_2^{1/2} & u_0^{1/2}u_1^{-1/2}u_2^{-1/2}\\
u_0^{-1/2}u_1^{-1/2}u_2^{1/2} & u_0^{-1/2}u_1^{-1/2}u_2^{-1/2}
\end{array}
\right)\\
&=&\left(\begin{array}{cc}
1+u_1^{-1} & u_0\\
u_0^{-1}u_1^{-1} & 1
\end{array}
\right)H^1(u_0u_1u_2)\ .\\
\\
\ev_{\overline{1}1}(v_0,v_1,v_2)&=&H^1(v_0)F^1H^1(v_1)E^1H^1(v_2)\in G^{w_0,w_0}\\
&=&\left(\begin{array}{cc}
v_0^{1/2}v_1^{1/2}v_2^{1/2} & v_0^{1/2}v_1^{1/2}v_2^{-1/2}\\
v_0^{-1/2}v_1^{1/2}v_2^{1/2} & v_0^{-1/2}v_1^{-1/2}v_2^{-1/2}+v_0^{-1/2}v_1^{1/2}v_2^{-1/2}
\end{array}
\right)\\
&=&\left(\begin{array}{cc}
1 & v_0v_1\\
v_0^{-1} & v_1+1
\end{array}
\right)H^1(v_0v_1v_2)\ .
\end{array}
$$
Again, the reader annoyed with the rational powers is free to replace
$\SL(2,\mathbb C)$ by $\PGL(2,\mathbb C)$. Let us remark, however, that
for every $u,v\in W$ and every double reduced word $\mathbf i\in R(u,v)$,
the associated reduced evaluation maps
$\ev_{\mathbf i}^{\red}:{\mathcal X}_{\mathbf i}^{\red}\to G^{u,v}/H$
described in Subsection \ref{subsection:reducedevaluation}
are well-defined birational isomorphisms both on $\SL(2,\mathbb C)$ and
$\PGL(2,\mathbb C)$.
Moreover, let us notice that it is also possible to construct the two last
evaluation maps from the others, using the amalgamated product. Indeed,
according to the formula (\ref{equ:amal}), we get the relations
$$\begin{array}{cccc}
&u_0:=y_0,& u_1:=y_1z_0,&u_2:=z_1\\
\mbox{and}\\
&v_0:=z_0,& v_1:=z_1y_0,& v_2:=y_1\ .
\end{array}$$
From the other hand, if $\ev_{1\overline{1}}(u_0,u_1,u_2)=\ev_{\overline{1}1}(v_0,v_1,v_2)$,
then we have the following relations between the $u_i$ and the $v_j$:
\begin{equation}\label{equ:ex}
\begin{array}{ccc}
\left\{
\begin{array}{l}
v_0=u_0(1+u_1)\\
v_1=u_1^{-1}\\
v_2=u_2(1+u_1)
\end{array}
\right.
&\mbox{and}&\left\{
\begin{array}{l}
u_0=v_0(1+v_1^{-1})^{-1}\\
u_1=v_1^{-1}\\
u_2=v_2(1+v_1^{-1})^{-1}
\end{array}.
\right.
\end{array}
\end{equation}

Now, for every $i,j\in[1,2]$, let $t_{ij}$ be the coordinate function associated to
(\ref{equ:sl2}). Applying the formula (\ref{equ:elemrmatrix}) to the Sklyanin bracket
(\ref{equ:Sbracket}), we can see that the standard Poisson bracket on the Poisson-Lie
group $G$ is given by the following equalities:
$$\left\{
\begin{array}{lll}
\{t_{11},t_{12}\}_G=\frac{1}{2}t_{11}t_{12}, & \{t_{11},t_{21}\}_G=\frac{1}{2}t_{11}t_{21}, \\
\{t_{11},t_{22}\}_G=t_{12}t_{21}, & \{t_{12},t_{21}\}_G=0, \\
\{t_{12},t_{22}\}_G=\frac{1}{2}t_{12}t_{22}, & \{t_{21},t_{22}\}_G=\frac{1}{2}t_{21}t_{22}.
\end{array}
\right.$$
We quickly check that the maps $\ev_{\mathbf 1}$, $\ev_{1}$, and $\ev_{\overline{1}}$
are Poisson when the matrices (resp. quivers) establishing the Poisson structure on the
seed $\mathcal X$-tori is given respectively by:
$$\begin{array}{ccc}
\varepsilon({\mathbf 1})=(0),&
\varepsilon({1})=\left(
\begin{array}{cc}
0 & -1\\
1 & 0
\end{array}
\right),&
\varepsilon({\overline 1})=\left(
\begin{array}{cc}
0 & 1\\
-1 & 0
\end{array}
\right).\\
\\
\begin{picture}(20,18)(0,-7)
\put(0,0){\circle{4}}
\end{picture}
&
\begin{picture}(20,18)(0,-7)
\put(-10,0){\circle{4}}
\put(10,0){\circle{4}}
\put(8,0){\vector(-1,0){11}}
\put(8,0){\line(-1,0){16}}
\end{picture}
&
\begin{picture}(20,18)(0,-7)
\put(-10,0){\circle{4}}
\put(10,0){\circle{4}}
\put(-8,0){\vector(1,0){11}}
\put(-8,0){\line(1,0){16}}
\end{picture}
\end{array}
$$
Then the amalgamation procedure leads to the following the matrices
(resp. quivers) establishing the Poisson structures on the associated
seed $\mathcal X$-tori for which the maps
$\ev_{1\overline{1}}$ and $\ev_{\overline{1}1}$ are Poisson:
\begin{equation}\label{equ:epsilonSL2+}
\begin{array}{cc}
\varepsilon({1\overline{1}})=\left(
\begin{array}{ccc}
0 & -1&0\\
1 & 0&1\\
0&-1&0
\end{array}
\right),&
\varepsilon({\overline 1 1})=\left(
\begin{array}{ccc}
0 & 1&0\\
-1 & 0&-1\\
0&1&0
\end{array}
\right). \\
\\
\begin{picture}(20,18)(0,-7)
\put(-20,0){\circle{4}}
\put(0,0){\circle*{4}}
\put(20,0){\circle{4}}
\put(-2,0){\vector(-1,0){11}}
\put(-2,0){\line(-1,0){16}}
\put(2,0){\vector(1,0){11}}
\put(2,0){\line(1,0){16}}
\end{picture}
&
\begin{picture}(20,18)(0,-7)
\put(-20,0){\circle{4}}
\put(0,0){\circle*{4}}
\put(20,0){\circle{4}}
\put(18,0){\vector(-1,0){11}}
\put(18,0){\line(-1,0){16}}
\put(-18,0){\vector(1,0){11}}
\put(-18,0){\line(1,0){16}}
\end{picture}
\end{array}
\end{equation}
Looking at them, it is clear that the expressions in (\ref{equ:ex}) describe respectively the
cluster transformation $\mu_{1\overline{1}\to\overline{1}1}:{\mathcal X}_{1\overline{1}}
\to{\mathcal X}_{\overline{1}1}$ associated to the variable $u_1$
and the cluster transformation $\mu_{\overline{1}1\to 1\overline{1}}:{\mathcal X}_{\overline{1}1}
\to{\mathcal X}_{1\overline{1}}$  associated to $v_1$:
$$\begin{array}{ccc}
(v_0,v_1,v_2)=\mu_{\binom{1}{1}}(u_0,u_1,u_2)
&\mbox{and}&
(u_0,u_1,u_2)=\mu_{\binom{1}{1}}(v_0,v_1,v_2)\ .
\end{array}$$
Let us also notice that
both are mutations and that there is no other direction of mutation.
We get therefore the following summary.
$$
\xymatrix{
{\mathcal X}_{\mathbf 1}\ar@/^0pc/[r]^{\ev_{\mathbf 1}}
&{(G^{1,1},\pi_G)}&
{\mathcal X}_{1}\ar@/^0pc/[r]^{\ev_{1}}
&{(G^{1,w_0},\pi_G)}\\
&&{\mathcal X}_{1\overline{1}}\ar@/^0pc/[dr]^{\ev_{1\overline{1}}}\\
{\mathcal X}_{\overline{1}}\ar@/^0pc/[r]^{\ev_{\overline{1}}}
&{(G^{w_0,1},\pi_G)}&&{(G^{w_0,w_0},\pi_G)}\\
&&{\mathcal X}_{\overline{1}1}\ar@/^0pc/[ur]_{\ev_{\overline{1}1}}
\ar@{<->}[uu]^{\mu_{\binom{1}{1}}}
}
$$

\subsection{Cluster ${\mathcal X}$-varieties related to $(\SL(2,\mathbb C),\pi_*)$}
Applying still the formula (\ref{equ:elemrmatrix}), but this time on the
Semenov-Tian-Shansky Poisson bracket given by Proposition \ref{prop:STSPoisson},
and still using the previous coordinate functions $t_{ij}$, it is easy to prove
that in the matricial case, the Poisson bracket on $(G,\pi_*)$ is given by the
following equalities:
$$\left\{\begin{array}{lll}
\{t_{11},t_{12}\}_*=t_{12}t_{22}, & \{t_{11},t_{21}\}_*=-t_{21}t_{22}, \\
\{t_{11},t_{22}\}_*=0, & \{t_{12},t_{21}\}_*=t_{11}t_{22}-t_{22}^2, \\
\{t_{12},t_{22}\}_*=t_{12}t_{22}, & \{t_{21},t_{22}\}_*=-t_{21}t_{22}.
\end{array}
\right.$$
\subsubsection{Evaluations maps for $(\SL(2,\mathbb C),\pi_*)$}It is
easy to check that the evaluation
${\ev}_1^{\dual}:{\mathcal X}_{[1]_{\mathfrak R}}\to (G,\pi_*)$,
parameterizing the union, denoted $F_{w_0}$, over $t\in H$ of the
varieties $F_{t,s_1}$, given by (\ref{equ:decG^*}), is Poisson. Indeed,
it is given by the following expression:

$$\begin{array}{rl}
{\ev}_1^{\dual}(x_0,t)&=H^1(x_0)E^1\widehat{w_0}
H_1(t)(F^1)^{-1}H^1(x_0^{-1})\in F_{w_0}\\
&=\left(\begin{array}{cc}
t+t^{-1} & -x_0t^{-1}\\
x_0^{-1}t & 0
\end{array}\right) .
\end{array}$$
The evaluations ${\ev}_{\overline 11}^{\dual}:{\mathcal X}_{[\overline 11]_{\mathfrak R}}\to BB_-$
and ${\ev}_{{1}\overline 1}^{\dual}:{\mathcal X}_{[1\overline 1]_{\mathfrak R}}\to BB_-$,
parameterizing the variety $BB_-$, are then obtained by the following computation:

$$\begin{array}{rl}
{\ev}_{\overline{1}1}^{\dual}(y_0,y_1,t)
&=H^1(y_0)F^1\ {\ev}_{1}^{\dual}(y_1,t)\ (F^1)^{-1}H^1(y_0^{-1})\\
&=\left(\begin{array}{cc}
t^{-1}(1+y_1)+t & -y_0y_1t^{-1}\\
y_0^{-1}(t(1+y_1^{-1})+t^{-1}(1+y_1)) & -y_1t^{-1}
\end{array}\right),\\
\\
{\ev}_{1\overline{1}}^{\dual}(\widetilde{y_0},\widetilde{y_1},t)
&=\left(\begin{array}{cc}
t^{-1/2}(1+\widetilde{y_1}^{-1})+t & -t^{-1}\widetilde{y_0}(1+\widetilde{y_1}^{-1})\\
\widetilde{y_0}^{-1}(t+t^{-1}\widetilde{y_1}^{-1}) & -\widetilde{y_1}^{-1}t
\end{array}\right).
\end{array}$$
And it is straightforward to check  that $\mu_{[\overline 1{1}]_{\mathfrak{R}}
\to[{1}\overline 1]_{\mathfrak{R}}}:(y_0,y_1,t)\mapsto(\widetilde{y_0},
\widetilde{y_1},t)$.

These evaluations are particular cases of the twisted evaluations described in
Subsection \ref{section:twisted}. The remaining twisted
evaluations $\widehat{\ev}_{11},\widehat{\ev}_{1\overline 1}:
{\mathcal X}_{[11]_{\mathfrak R}}\to BB_-$ and
$\widehat{\ev}_{\overline{1}\overline 1}:
{\mathcal X}_{[\overline 11]_{\mathfrak R}}\to BB_-$, which are described in
Subsection \ref{section:twistedeval}, also parameterize the variety $BB_-$. They
are given by the following formulas:

$$\begin{array}{rl}
\widehat{\ev}_{11}(z_0,z_1,t)&=\widehat{\ev}_{1\overline 1}(z_0,z_1,t)\\
&=H^1(z_0)E^1\ \widehat{\ev}_{1}(z_1,t)\ (E^1)^{-1}H^1(z_0^{-1})\\
&=\left(\begin{array}{cc}
(1+z_1^{-1})t+t^{-1} & -z_0((1+z_1^{-1})t+(1+z_1)t^{-1})\\
z_0^{-1}z_1^{-1}t & -z_1^{-1}t
\end{array}\right)\\
\\
\widehat{\ev}_{\overline{1}\overline 1}(y_0,y_1,t)&={\ev}_{\overline{1} 1}^{\dual}(y_0,y_1,t)\\
&=\left(\begin{array}{cc}
t^{-1}(1+y_1)+t & -t^{-1}y_0y_1\\
y_0^{-1}(t(1+y_1^{-1})+t^{-1}(1+y_1)) & -y_1t^{-1}
\end{array}\right)
\end{array}$$
It is easy to check that all these maps are Poisson when the
matrices (resp. quivers) establishing the Poisson structure
on the related seed ${\mathcal X}$-tori are given respectively
by the matrices (resp. quivers):
\begin{equation}\label{equ:epsilonSL2}
\begin{array}{ccc}
\eta({1{1}})=\eta({1\overline{1}})=\left(
\begin{array}{ccc}
0 & -1&0\\
1 & 0&0\\
0&0&0
\end{array}
\right),&
\eta({\overline 1 1})=\eta({\overline 1 \overline 1})=\left(
\begin{array}{ccc}
0 & 1&0\\
-1 & 0&0\\
0&0&0
\end{array}
\right). \\
\\
\begin{picture}(20,18)(0,-7)
\put(-20,0){\circle{4}}
\put(0,0){\circle*{4}}
\put(20,0){\circle{4}}
\put(-2,0){\vector(-1,0){11}}
\put(-2,0){\line(-1,0){16}}
\end{picture}
&
\begin{picture}(20,18)(0,-7)
\put(-20,0){\circle{4}}
\put(0,0){\circle*{4}}
\put(20,0){\circle{4}}
\put(-18,0){\vector(1,0){11}}
\put(-18,0){\line(1,0){16}}
\end{picture}
\end{array}
\end{equation}
Therefore, the truncation map (\ref{equ:defJ}) gives the way to pass
from the Poisson structures defined by (\ref{equ:epsilonSL2+}) to the
Poisson structures defined by (\ref{equ:epsilonSL2}). We thus get a cluster
$\mathcal X$-variety, denoted ${\mathcal X}_{e\leq e}$, for the variety $F_{w_0}$ and two
isomorphic cluster $\mathcal X$-varieties for the variety $BB_-$, denoted ${\mathcal X}_{e}$
and ${\mathcal X}_{w_0}$, and respectively associated to the cluster variables
$(y_0,y_1,t)$ and $(z_0,z_1,t)$.

\subsubsection{Remarks about evaluations maps for $(\PGL(2,\mathbb C),\pi_*)$}
The careful reader will have noticed that the evaluation maps related to
$(\SL(2,\mathbb C),\pi_*)$ we have just obtained slightly differ from the
twisted evaluation maps of Section \ref{section:twistedev}. Again, it is because
the Lie group $\SL(2,\mathbb C)$ is not of adjoint type. According to Remark
\ref{rem:twitevlaconnect}, the corresponding evaluation maps for
$G=\PGL(2,\mathbb C)$ are the following.
$$\begin{array}{rl}
{\underline{\ev}}_1^{\dual}(x_0,t)&=H^1(x_0)E^1\widehat{w_0}
H^1(t)(F^1)^{-1}H^1(x_0^{-1})\in F_{w_0}\\
&=\left(\begin{array}{cc}
t^{1/2}+t^{-1/2} & -x_0t^{-1/2}\\
x_0^{-1}t^{1/2} & 0
\end{array}\right) .
\end{array}$$

$$\begin{array}{rl}
{\underline{\ev}}_{\overline{1}1}^{\dual}(y_0,y_1,t)
&=H^1(y_0)F^1\ {\ev}_{1}^{\dual}(y_1,t)\ (F^1)^{-1}H^1(y_0^{-1})\\
&=\left(\begin{array}{cc}
t^{-1/2}(1+y_1)+t^{1/2} & -y_0y_1t^{-1/2}\\
y_0^{-1}(t^{1/2}(1+y_1^{-1})+t^{-1/2}(1+y_1)) & -y_1t^{-1/2}
\end{array}\right),\\
\\
{\underline{\ev}}_{1\overline{1}}^{\dual}(\widetilde{y_0},\widetilde{y_1},t)
&=\left(\begin{array}{cc}
t^{-1/2}(1+\widetilde{y_1}^{-1/2})+t^{1/2} & -t^{-1/2}\widetilde{y_0}(1+\widetilde{y_1}^{-1})\\
\widetilde{y_0}^{-1}(t^{1/2}+t^{-1/2}\widetilde{y_1}^{-1}) & -\widetilde{y_1}^{-1}t^{1/2}
\end{array}\right).
\end{array}$$

$$\begin{array}{rl}
\widehat{\underline{\ev}}_{11}(z_0,z_1,t)&=\widehat{\underline{\ev}}_{1\overline 1}(z_0,z_1,t)\\
&=H^1(z_0)E^1\ \widehat{\ev}_{1}(z_1,t)\ (E^1)^{-1}H^1(z_0^{-1})\\
&=\left(\begin{array}{cc}
(1+z_1^{-1})t^{1/2}+t^{-1/2} & -z_0((1+z_1^{-1})t^{1/2}+(1+z_1)t^{-1/2})\\
z_0^{-1}z_1^{-1}t^{1/2} & -z_1^{-1}t^{1/2}
\end{array}\right)\\
\\
\widehat{\underline{\ev}}_{\overline{1}\overline 1}(y_0,y_1,t)&=
{\underline{\ev}}_{\overline{1} 1}^{\dual}(y_0,y_1,t)\\
&=\left(\begin{array}{cc}
t^{-1/2}(1+y_1)+t^{1/2} & -t^{-1/2}y_0y_1\\
y_0^{-1}(t^{1/2}(1+y_1^{-1})+t^{-1/2}(1+y_1)) & -y_1t^{-1/2}
\end{array}\right)
\end{array}$$
\subsubsection{How to use the saltation map}If the
evaluations ${\ev}_{\overline{1}1}^{\dual}(y_0,y_1,t)$ and
$\widehat{\ev}_{11}(z_0,z_1,t)$  parameterize the same element, we quickly check
with the expressions above that the map $\varphi:(y_0,y_1,t)
\longmapsto (z_0,z_1,t)$ is given by:
\begin{equation}\label{equ:y->z}
\left\{\begin{array}{rcl}
z_0&=&y_0{(1+y_1^{-1})}^{-1}{(1+y_1^{-1}t^2)}^{-1}\\
z_1&=&t^2y_1^{-1}
\end{array}\right. .
\end{equation}
Before to link the map $\varphi$ with the cluster combinatorics
we have developed, let us stress (again) that saltations are really
needed in the story because, the variable $t$ being a Casimir function,
you cannot expect to obtain a formula such as $z_1=t^2y_1^{-1}$
by only cluster transformation. Now, let us remark that the cover
$\mathfrak{p}_{\mathcal X}
:{\mathcal X}_{[\overline{1}1]_{\mathfrak{R}}}\to
{\mathcal X}_{[\overline{1}1]_{\mathfrak{R}}}$
is of degree $2$ and given by the formula
$$\mathfrak{p}_{\mathcal X}:(y_1,y_2,t)\mapsto(y_1,y_2,t^2)\ .$$
We are going to prove that the equality $\mathfrak{p}_{\mathcal X}\circ\varphi=
\Xi_{s_1}\circ \mathfrak{p}_{\mathcal X}$
is satisfied, where $\Xi_{s_1}$ denotes the birational Poisson isomorphism
given by Corollary \ref{cor:dualsalt}. To get it, let us first describe
the saltation $\Xi_1$ given by (\ref{equ:elemsalt}). It is associated
to the following generalized cluster transformation, acting on every element
$(x_0,x_1,x_2)\in{\mathcal X}_{1\overline{1}}$.
$$\begin{array}{rll}
\overline{\mu}_{\overline{1}\overline{1}}\circ\zeta_1\circ
\mu_{1\overline{1}\to\overline{1}1}(x_0,x_1,x_2)
&=\mu_{\binom{1}{1}}\circ\mu_{\binom{1}{2}}\circ\mu_{\binom{1}{1}}(x_0,x_1,x_2)\\
\\
&=\mu_{\binom{1}{1}}\circ\mu_{\binom{1}{2}}(x_0(1+x_1),x_1^{-1},x_2(1+x_1))\\
\\
&=\mu_{\binom{1}{1}}(x_0(1+x_1),x_1^{-1},x_2^{-1}(1+x_1)^{-1})\\
\\
&=(x_0,x_1,x_2^{-1}x_1^{-1})\ .\\
\\
\mbox{Therefore\ \ \ \ \ \ \ \  }\Xi_1(x_0,x_1,t)&=(x_0,x_1^{-1}t^{-1},t)\ ,\\
\\
\mbox{because\ \ \ \ \ \ \ \ \ \ \ \  }\Xi_1\circ\mathfrak{t}_{\binom{1}{2}(t)}&
=\mathfrak{t}_{\binom{1}{1}(t)}\circ\mu_{\binom{1}{1}}\circ\mu_{\binom{1}{2}}
\circ\mu_{\binom{1}{1}}\ .
\end{array}$$
We then get the following formula for the birational isomorphism $\Xi_{s_1}$.
\begin{equation}\label{equ:compsalt}
\begin{array}{rll}
\Xi_{s_1}(y_0,y_1,t)&=\mu_{[{1}\overline 1]_{\mathfrak R}\to
[\overline 1{1}]_{\mathfrak R}}\circ\Xi_1\circ\mu_{[\overline{1}1]_{\mathfrak R}
\to [1\overline{1}]_{\mathfrak R}}(y_0,y_1,t)\\
\\
&=\mu_{[{1}\overline 1]_{\mathfrak R}\to
[\overline 1{1}]_{\mathfrak R}}\circ\Xi_1(y_0{(1+y_1^{-1})}^{-1},y_1^{-1},t)\\
\\
&=\mu_{[{1}\overline 1]_{\mathfrak R}\to
[\overline 1{1}]_{\mathfrak R}}(y_0{(1+y_1^{-1})}^{-1},y_1t^{-1},t)\\
\\
&=(y_0{(1+y_1^{-1})}^{-1}{(1+y_1^{-1}t)}^{-1},y_1^{-1}t,t)\ .
\end{array}
\end{equation}
And it is now clear that the equality $\mathfrak{p}_{\mathcal X}\circ\varphi=
\Xi_{s_1}\circ \mathfrak{p}_{\mathcal X}$ is satisfied. Moreover, the way the
cluster varieties $\mathcal X_e$ and $\mathcal X_{s_1}$ are related by the
saltation $\Xi_1$ is described by the elementary $W$-permutohedron pictured
in Figure \ref{fig:amalgblock2}.
\begin{figure}[htbp]
\begin{center}
\setlength{\unitlength}{1.5pt}
\begin{picture}(20,33)(0,-17)
\thicklines
\multiput(0,10)(0,-20){2}{\circle{4}}
\put(0,-8){\line(0,1){16}}
\put(-3,15){${\mathcal X}_{s_1}$}
\put(3,-2){$\Xi_1$}
\put(-3,-20){${\mathcal X}_{e}$}
\end{picture}
\end{center}
\vspace{-.1in}
\caption{Saltation and cluster $\mathcal X$-varieties for $\mathfrak{g}=A_1$}
\label{fig:amalgblock2}
\end{figure}
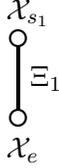
Let us notice, for the reader who prefers to deal with $\PGL(2,\mathbb C)$,
that things are much simpler for him because the covering ${\mathfrak p}_{\mathcal X}$
is then the identity map on seed $\mathcal X$-tori while the formulas for the saltation
and cluster transformations remain unchanged.

\subsection{Evaluation maps for $(\SL(2,\mathbb C)^*,\pi_{G^*})$}
We first recall that the set $\SL(2,\mathbb C)^*$ has the following
description
\begin{equation}\label{equ:sl2*}
\begin{array}{rl}
SL(2, \mathbb{C})^*&=\{\left(\left(
\begin{array}{cc}
t_{11}^+ & t_{12}^+\\
0 & t_{22}^+
\end{array}
\right),
\left(
\begin{array}{cc}
t_{11}^- & 0\\
t_{21}^- & t_{22}^-
\end{array}
\right)\right)\\
\\
&:t_{11}^-=(t_{11}^+)^{-1},\ \ \ t_{22}^-=(t_{22}^+)^{-1},
\ \ \ t_{11}^+t_{22}^+=1,\ \  t_{ij}^{\pm}\in \mathbb{C}\}\ .
\end{array}
\end{equation}
Now, let us recall its Poisson structure $\pi_{G^*}$. Again, we denote
$t^{\pm}_{ij}$ the corresponding coordinate functions. The Poisson bracket
we are looking for on $SL(2, \mathbb{C})^*$ is given by the following equalities:
$$\left\{
\begin{array}{lll}
\{t_{11}^{\pm},t_{12}^+\}_{G^*}=\pm t_{11}^{\pm}t_{12}^+,
& \{t_{11}^{\pm},t_{21}^-\}_{G^*}=\mp t_{11}^{\pm}t_{21}^-, \\
\{t_{11}^{\pm},t_{22}^{\pm}\}_{G^*}=0, & \{t_{12}^+,t_{21}^-\}_{G^*}=t_{22}^+t_{11}^-, \\
\{t_{12}^+,t_{22}^{\pm}\}_{G^*}=\pm t_{12}^+t_{22}^{\pm},
& \{t_{21}^-,t_{22}^{\pm}\}_{G^*}=\mp t_{21}^-t_{22}^{\pm}.
\end{array}
\right.$$
Because of the Poisson covering $\phi:(\SL(2,\mathbb C)^*,\pi_{G^*})
\to(\SL(2,\mathbb C),\pi_{*})$ of degree $2$, we use the previous
evaluation maps on $(\SL(2,\mathbb C),\pi_{*})$ to get the evaluation
maps on $(\SL(2,\mathbb C)^*,\pi_{G^*})$. To do that, let us introduce
the following notations for every non-zero complex
number $x$.
$$
\begin{array}{lll}
E^1(x)=H^1(x)E^1H^1(x^{-1})&
\mbox{and}
&F^1(x)=H^1(x^{-1})F^1H^1(x)\ .
\end{array}
$$
Now, if the evaluations ${\ev}_{\overline{1}1}^{\dual}(y_0,y_1,t)$
and $\widehat{\ev}_{11}(z_0,z_1,t)$ gives the same element $nh^2n_-^{-1}$
such that $n$ (resp. $n_-$) is an upper (resp. lower) triangular matrix with the
number $1$ on the diagonal entries and $h$ a diagonal matrix, then we can evaluate
$h$, $n$ and $n_-$ is the following way, using for example a computation in the
spirit of Lemma \ref{lemma:beta1}.
$$
\left\{
\begin{array}{l}
h=H^1(-t^{-1}z_1)=H^1(-ty_1^{-1})\\
n=[[F^1({y_0})\left(
\begin{array}{rr}
0 & -1\\
1 & 0
\end{array}
\right)]]_+=E^1({y_0})\\
n_-=[E^1(z_0)\left(
\begin{array}{rr}
0 & 1\\
-1 & 0
\end{array}
\right)]_-=F^1(z_0^{-1})\ .
\end{array}
\right.
$$
And because the map $\Xi_{s_1}$ is such that the equality
$\mathfrak{p}_{\mathcal X}\circ\varphi=\Xi_{s_1}\circ \mathfrak{p}_{\mathcal X}$
is satisfied, where $\varphi:(y_0,y_1,t)\mapsto (z_0,z_1,t)$, these formulas
are in agreement with  Theorem \ref{thm:G^*} which states that:
$$
\begin{array}{lll}
\left\{
\begin{array}{l}
{\ev_{\overline{1}1}^+(\mathbf y)}=E^1({y_0})H^1(-ty_1^{-1})\\
{\ev_{\overline{1}1}^-(\mathbf y)}=F^1(\Xi_{s_1}(\mathbf y)_0^{-1})H^1(-ty_1^{-1})^{-1}
\end{array}
\right.
&\mbox{and}&
\left\{
\begin{array}{l}
{\ev_{11}^+(\mathbf z)}=E^1(\Xi_{s_1}^{-1}(\mathbf z)_0)H^1(-t^{-1}z_1)\\
{\ev_{11}^-(\mathbf z)}=F^1(z_0^{-1})H^1(-t^{-1}z_1)^{-1}
\end{array}
\right. .
\end{array}
$$
We finally the following description of $G^*$ using matrices, analogous to the
one given by (\ref{equ:sl2*}), which involves our cluster variables
and their relation $\varphi:(y_0,y_1,t)\mapsto (z_0,z_1,t)$ given by
the relation (\ref{equ:y->z}).
$$\begin{array}{ccccc}
\left(\left(
\begin{array}{cc}
(-ty_1^{-1})^{1/2}&(-ty_1^{-1})^{-1/2}y_0\\
0&(-ty_1^{-1})^{-1/2}
\end{array}
\right)
,
\left(
\begin{array}{cc}
(-ty_1^{-1})^{-1/2}&0\\
(-ty_1^{-1})^{-1/2}z_0^{-1}&(-ty_1^{-1})^{1/2}
\end{array}
\right)\right)
\end{array}
\ .$$
We are thus in an optimal position for the quantization process. Indeed, the $R$-matrix
associated to the quantized universal enveloping algebra $\mathcal{U}_q(\mathfrak{g})$
associated to the Lie algebra $\mathfrak{g}=\sl(2,\mathbb C)$ and the related quantum
group $\mathcal{F}_q(\SL(2,\mathbb C)^{*})$ are respectively given by the following
formulas.
\begin{equation}\label{equ:Rsl2}
{\mathcal R}=q^{\frac{1}{2}H\otimes H}E\otimes F\ .
\end{equation}
$$\begin{array}{llllll}
L^+=\left(
\begin{array}{cc}
(qK)^{1/2}&(qK)^{1/2}(q-q^{-1})F\\
0&(qK)^{-1/2}
\end{array}
\right)
,\\
\\
L^-=\left(
\begin{array}{cc}
(qK)^{-1/2}&0\\
(qK)^{-1/2}(q-q^{-1})E&(qK)^{1/2}
\end{array}
\right)
\end{array}
\ .$$

\subsection{Quantum evaluation maps for ${\mathcal U}_q(\sl(2,\mathbb C))$}
The dual Poisson-Lie group
$(G^*,\pi_{G^*})$ is the semi-classical limit of the
quantum group $\mathcal{F}_q(G^*)$ which is isomorphic (as Hopf
algebra) to the very famous quantized universal enveloping algebra
$\mathcal{U}_q(\mathfrak{g})$.
As a conclusion to this work, we give the quantum picture
for $\mathfrak{g}=\sl(2,\mathbb C)$. For technical reasons,
we consider the quantized universal enveloping algebra
$\mathcal{U}_{q^{-1}}(\mathfrak{g})$ instead of $\mathcal{U}_q(\mathfrak{g})$.
It is the ${\mathbb C}(q)$-algebra generated by $E$, $F$, and $K$
with relations
$$\begin{array}{cccc}
KE=q^{-2}EK,&KF=q^2FK,
&\mbox{and}&EF-FE=\displaystyle\frac{K-K^{-1}}{q^{-1}-q}\ .
\end{array}$$
We now define the quantum tori
${\mathcal X}^q_{[\overline{1}\overline{1}]_{\mathfrak R}}$
and ${\mathcal X}^q_{[{1}{1}]_{\mathfrak R}}$ as the ${\mathbb C}(q)$-algebra
generated respectively by the elements $Y_0,Y_1,T$ and $Z_01,Z_1,T$
with the $q$-commutation relations:
$$\begin{array}{lcl}
Y_0Y_1=q^2Y_1Y_0&\mbox{and}& Z_0Z_1=q^{-2}Z_1Z_0\ ;\\
\\
Y_0T=TY_0&\mbox{and}& Z_0T=TZ_0\ ;\\
\\
TY_1=Y_1T&\mbox{and}& TZ_1=Z_1T\ .
\end{array}$$
In particular, it is clear that the seed $\mathcal X$-tori
${\mathcal X}_{[\overline{1}\overline{1}]_{\mathfrak R}}$
and ${\mathcal X}_{[{1}{1}]_{\mathfrak R}}$, whose Poisson structures are
given by (\ref{equ:epsilonSL2}), are respectively
the semi-classical limits of these quantum tori.
Luckily, the quantum evaluation maps for $\mathcal{U}_{q^{-1}}(\mathfrak{g})$
come without effort from the semi-classical evaluation maps
we have just obtained for $(\SL(2,\mathbb C)^*,\pi_{G^*})$.
(We don't forget to switch $q$ into $q^{-1}$ in the formula (\ref{equ:Rsl2})
according to the switch $\mathcal{U}_{q}(\mathfrak{g})\leftrightarrow
\mathcal{U}_{q^{-1}}(\mathfrak{g})$.) In particular, it is
straightforward to check that the following map is an algebra morphism.
$$\begin{array}{rlll}
\left\{
\begin{array}{l}
\ev^{q^{-1}}(K)=-qTY_1^{-1}=-qT^{-1}Z_1\ ;\\
\ev^{q^{-1}}(F)=-(q^{-1}-q)^{-1}T^{-1}Y_1Y_0\ ;\\
\ev^{q^{-1}}(E)=(q^{-1}-q)^{-1}Z_0^{-1}\ .
\end{array}
\right.
&\mbox{if}&
\left\{\begin{array}{rcl}
Z_0&=&Y_0{(1+qY_1^{-1})}^{-1}{(1+qY_1^{-1}T^2)}^{-1}\\
Z_1&=&T^2Y_1^{-1}
\end{array}\right.
\end{array}
$$
Moreover, the link between the $Y_i$ and the $Z_j$ is given by quantizing
the birational Poisson isomorphism $\varphi$. To see that, we use the
quantization formulas of \cite{FGdilog}: in the case $|\varepsilon_{ik}|=1$,
we get the following \emph{quantum mutations}:
\begin{equation}\label{equ:qmut}
X_{\mu_k^q(i)}=\left\{ \begin{array}{lll}
{X_k}^{-1}&  \mbox{if}\ i=k;\\
X_iX_k^{[\varepsilon_{ik}]_+}(1+qX_k)^{-\varepsilon_{ik}}& \mbox{if}\ i\not=k\ .
\end{array}\right.
\end{equation}
Therefore, the quantization of the computation (\ref{equ:compsalt}) gives us
the following equalities, using a still mysterious map which we denote $\Xi_1^q$.
$$
\begin{array}{rll}
\Xi_{s_1}^q(Y_0,Y_1,T)&=\mu_{[{1}\overline 1]_{\mathfrak R}\to
[\overline 1{1}]_{\mathfrak R}}^q\circ\Xi_1^q\circ\mu_{[\overline{1}1]_{\mathfrak R}^q
\to [1\overline{1}]_{\mathfrak R}}(Y_0,Y_1,T)\\
\\
&=\mu_{[{1}\overline 1]_{\mathfrak R}\to
[\overline 1{1}]_{\mathfrak R}}^q\circ\Xi_1^q(Y_0{(1+qY_1^{-1})}^{-1},Y_1^{-1},T)\\
\\
&=\mu_{[{1}\overline 1]_{\mathfrak R}\to
[\overline 1{1}]_{\mathfrak R}}^q(Y_0{(1+qY_1^{-1})}^{-1},Y_1T^{-1},T)\\
\\
&=(Y_0{(1+qY_1^{-1})}^{-1}{(1+qY_1^{-1}t)}^{-1},Y_1^{-1}T,T)\ .
\end{array}
$$
Moreover we have the equality $\mathfrak{p}_{\mathcal X}\circ\varphi^q=
\Xi_{s_1}^q\circ \mathfrak{p}_{\mathcal X}$, where $\varphi^q$ is the quantization of
the map $\varphi$, that is $\varphi^q:(Y_0,Y_1,T)\to(Z_0,Z_1,T)$.
Let us stress, however, that there is still something strange in this story.
Indeed, by keeping the tropicalization formula, we therefore also get tropical
quantum mutations from (\ref{equ:qmut});
we use them to introduce quantum saltation $\widetilde{\Xi}_{1}^q$, defined by
intertwining generalized quantum cluster transformations with truncation maps,
by mimicking the previous computation. We thus get:
$$\begin{array}{rll}
\mu_{\binom{1}{1}}^q\circ\mu_{\binom{1}{2}}^q\circ\mu_{\binom{1}{1}}^q(X_0,X_1,X_2)
&=\mu_{\binom{1}{1}}^q\circ\mu_{\binom{1}{2}}^q(X_0(1+qX_1),X_1^{-1},X_2(1+qX_1))\\
\\
&=\mu_{\binom{1}{1}}^q(X_0(1+qX_1),X_1^{-1},X_2^{-1}(1+qX_1)^{-1})\\
\\
&=(X_0,X_1,qX_2^{-1}X_1^{-1})\ .\\
\\
\mbox{Therefore\ \ \ \ \ \ \ \  }\widetilde{\Xi}_1^q(X_0,X_1,T)&=(X_0,qX_1^{-1}T^{-1},T)\ ,\\
\\
\mbox{because\ \ \ \ \ \ \ \ \ \ \ \ \ \ \  }\widetilde{\Xi}_1^q\circ
\mathfrak{t}_{\binom{1}{2}(t)}
&=\mathfrak{t}_{\binom{1}{1}(t)}\circ\mu_{\binom{1}{1}}^q\circ\mu_{\binom{1}{2}}^q
\circ\mu_{\binom{1}{1}}^q\ .
\end{array}$$
But, unfortunately, it is clear that we have $\widetilde{\Xi}_{1}^q\neq{\Xi}_{1}^q$.

\section*{Acknowledgments}
This paper and its sequel \cite{RB2} come from my Ph D. thesis. I am
grateful to my advisor M. Semenov-Tian-Shansky, who has
suggested to me the cluster combinatorics study of dual Poisson-Lie groups, for
his support and patience. I would like also
to thank Vladimir Fock  for useful discussions and and Jiang-Hua Lu for
having kindly explained to me the morphisms in \cite{EL}.

\end{document}